\documentclass[11pt]{article}
\usepackage{amsmath,amssymb,amsthm,latexsym,amscd,stmaryrd,mathrsfs,longtable,color,url}

\newtheoremstyle{standard}%
{9pt}%
{9pt}%
{\it}
{}%
{\bfseries}%
{}
{ }%
{#3}%

\newcommand{\db}[1]{(\!({#1})\!)}

\newcommand{\wak}{k}
\newcommand{\centralc}{C}
\newcommand{\subL}{P}

\numberwithin{equation}{section}

\newcommand{\N}{{\mathbb N}}
\newcommand{\Z}{{\mathbb Z}}
\newcommand{\Q}{{\mathbb Q}}

\newcommand{\C}{{\mathbb C}}

\newcommand{\wi}{i}
\newcommand{\wj}{j}
\newcommand{\wh}{h}
\newcommand{\Har}{H}

\newcommand{\wn}{n}

\newcommand{\wx}{x}

\newcommand{\wl}{l}

\newcommand{\rankL}{d}
\newcommand{\mN}{N}
\newcommand{\mW}{W}
\newcommand{\cocy}{\varphi}
\newcommand{\zone}{z}
\newcommand{\mz}{z}

\newcommand{\module}{M}

\newcommand{\mn}{p}
\newcommand{\hei}{{\mathfrak h}}

\newcommand{\sU}{{\mathscr U}}

\newcommand{\nor}{\begin{subarray}{c}\circ\\\circ\end{subarray}}

\newcommand{\fg}{{\mathfrak g}}
\newcommand{\fh}{{\mathfrak h}}

\newcommand{\seq}[1]{\langle{#1}\rangle}
\newcommand{\ul}[1]{{#1}}
\newcommand{\lu}{u}
\newcommand{\lw}{w}
\newcommand{\lv}{v}

\newcommand{\mK}{K}
\newcommand{\lom}{r}
\newcommand{\lJ}{s}
\newcommand{\lE}{t}
\newcommand{\ExB}{E}

\newcommand{\lao}{r}
\newcommand{\laJ}{s}
\newcommand{\laE}{t}

\newcommand{\sv}{P}
\newcommand{\vac}{{\mathbf 1}}
\newcommand{\lattice}{L}

\newcommand{\nS}{S}

\newcommand{\wc}{c}
\newcommand{\lvE}{\ExB_{t}\lv}

\DeclareMathOperator{\pr}{pr} 
 
\DeclareMathOperator{\Ker}{Ker}

\DeclareMathOperator{\Ext}{Ext} 
\DeclareMathOperator{\id}{id}

\DeclareMathOperator{\Span}{Span} 
\DeclareMathOperator{\wt}{wt}

\DeclareMathOperator{\rank}{rank} 

\DeclareMathOperator{\tw}{tw}

\newtheorem{lemma}{Lemma}[section]
\newtheorem{theorem}[lemma]{Theorem}
\newtheorem{proposition}[lemma]{Proposition}
\newtheorem{corollary}[lemma]{Corollary}
\theoremstyle{definition}
\newtheorem{definition}[lemma]{Definition}

\newtheorem{remark}[lemma]{Remark}

\theoremstyle{standard}

\title{The irreducible weak modules for the fixed point subalgebra of the vertex algebra associated to
a non-degenerate even lattice by an automorphism of order $2$}
\author{Kenichiro Tanabe\footnote{Research was partially supported by the Grant-in-aid
(No. 18K03198) for Scientific Research, JSPS.}\\\\
Department of Mathematics\\
Hokkaido University\\
Kita 10, Nishi 8, Kita-Ku, Sapporo, Hokkaido, 060-0810\\
Japan\\\\
ktanabe@math.sci.hokudai.ac.jp}
\date{}
\begin{document}
\maketitle

\begin{abstract}
Let $V_{\lattice}$ be the vertex algebra  associated to a non-degenerate even lattice $\lattice$,
$\theta$ the automorphism of $V_{\lattice}$ induced from the $-1$-isometry of $\lattice$, and
$V_{\lattice}^{+}$ the fixed point subalgebra of $V_{\lattice}$ under the action of $\theta$.
We classify the irreducible weak $V_{\lattice}^{+}$-modules
and show that any irreducible weak $V_{\lattice}^{+}$-module 
is isomorphic to a weak submodule of some irreducible weak $V_{\lattice}$-module or 
to a submodule of some irreducible $\theta$-twisted $V_{\lattice}$-module.
\end{abstract}



\clearpage
\tableofcontents
\section{\label{section:introduction}Introduction}
Let $V$ be a vertex algebra, $G$ a finite automorphism group of $V$,
and $V^G$ the fixed point subalgebra of $V$ under the action of $G$: $V^{G}=\{u\in V\ |\ gu=u\mbox{ for all }g\in G\}$.
The  fixed point subalgebras play an important role in the study of vertex algebras, particularly in the 
construction of interesting examples.
For example,  let $V_{\lattice}$ be the vertex algebra associated to a non-degenerate even lattice $\lattice$ of finite rank
and $\theta$ the automorphism of $V_{\lattice}$ of order $2$ induced from the $-1$-isometry of $\lattice$.
We write $V_{\lattice}^{\pm}=\{a \in V_{\lattice}\ |\ \theta(a)=\pm a\}$ for simplicity.
Then, the moonshine vertex algebra 
$V^{\natural}$ is constructed from $V_{\Lambda}^{+}$ and an irreducible $V_{\Lambda}^{+}$-module
in \cite{FLM} where $\Lambda$ is the Leech lattice.
We remark that $V_{\Lambda}^{+}$ is also a fixed point subalgebra of $V^{\natural}$
under the action of an automorphism of order $2$.

One of the main problems about $V^G$ is to describe the $V^G$-modules in terms of $V$ and $G$.
If $V$ is a vertex operator algebra, it is conjectured that under some conditions on $V$,
every irreducible $V^G$-module is a submodule of some irreducible $g$-twisted 
$V$-module for some $g\in G$ (cf. \cite{DVVV1989}).
The conjecture is confirmed for many examples including $V_{\lattice}^{+}$
where $L$ is a positive definite even lattice (cf. \cite{AD2004,DLaTYY2004,DN1999-1,DN1999-2,DN2001,TY2007,TY2013}).
In those examples, they classify the irreducible $V^{G}$-modules directly
by investigating the Zhu algebra, which is an associative $\C$-algebra introduced in \cite{Z1996},
since \cite[Theorem 2.2.1]{Z1996} says that
for a vertex operator algebra $V$
there is a one to one correspondence between the set of all isomorphism classes of irreducible $\N$-graded weak $V$-modules 
and that of irreducible modules for the Zhu algebra associated to $V$, where we note that
an arbitrary $V$-module is automatically an $\N$-graded weak $V$-module.
We recall some results on the representations of $V_{L}$ and $V_{L}^{+}$.
It is well known that the vertex algebra $V_{L}$ is a vertex operator algebra 
if and only if $\lattice$ is positive definite.
It is shown in \cite[Theorem 3.1]{Dong1993}
that
$\{V_{\lambda+\lattice}\ |\ \lambda+\lattice\in \lattice^{\perp}/L\}$ is a complete set of representatives of equivalence classes of
the irreducible weak $V_{\lattice}$-modules (see Definition \ref{definition:weak-module} for the definition of a weak module),
where $L^{\perp}$ is the dual lattice of $\lattice$. 
It is also shown in \cite[Theorem 3.16]{DLM1997} that
every weak $V_{\lattice}$-module is 
completely reducible.
The corresponding results for $\theta$-twisted weak $V_{\lattice}$-modules are obtained in \cite{Dong1994}.
When the lattice $\lattice$ is positive definite, 
the irreducible $V_{\lattice}^{+}$-modules 
are classified in \cite{AD2004,DN1999-2},
the fusion rules are determined in \cite{Abe2001, ADL2005}
and it is established that $V_{\lattice}^{+}$ is
$C_2$-cofinite in \cite{ABD2004, Yamsk2004} and rational in \cite{Abe2005, DJL2012}.
Thus, in this case it follows from \cite[Theorem 4.5]{ABD2004} that
$V_{\lattice}^{+}$ is regular and every irreducible weak $V_{\lattice}^{+}$-module
is an irreducible $V_{\lattice}^{+}$-module.
Here, it is worth mentioning that if a vertex operator algebra $V$
is simple and $C_2$-cofinite, then so is $V^G$ for any
finite solvable automorphism group $G$ of $V$ by \cite{Mi2015}.

We consider the case that $\lattice$ is not positive definite.
In this case, vertex algebras $V_{\lattice}$ and $V_{\lattice}^{+}$ are also related to $V^{\natural}$. 
In fact, in \cite{B1992}, Borcherds  constructs the monster Lie algebra 
from the tensor product $V^{\natural}\otimes_{\C}V_{II_{1,1}}$ 
where $II_{1,1}$ is the even unimodular Lorentzian lattice of rank $2$
and uses this Lie algebra to prove the moonshine conjecture stated in \cite{CN1979}.
Moreover, let $II_{25,1}$ be the even unimodular Lorentzian lattice of signature $(25,1)$
and $\Gamma$ a Niemeier lattice, where we recall that the Leech lattice $\Lambda$ is one of the Niemeier lattices.
It is known that the direct sum $\Gamma\oplus II_{1,1}$
is isomorphic to $II_{25,1}$, which implies that
 $V_{II_{25,1}}\cong V_{\Gamma}\otimes V_{II_{1,1}}$
and hence  $V_{II_{25,1}}^{+}\cong V_{\Gamma}^{+}\otimes V_{II_{1,1}}^{+}\oplus
V_{\Gamma}^{-}\otimes V_{II_{1,1}}^{-}$.
Thus,  $V_{\Gamma}^{+}$'s are related to each other through $V_{II_{25,1}}^{+}$.
Since $V_{\Lambda}^{+}$ is a fixed point subalgebra of $V^{\natural}$ as mentioned above
and there are several interesting connections between $II_{25,1}$ and 
 the Niemeier lattices obtained in 
\cite{B1985} and \cite{CS1982-1,CS1982-2} (see also \cite[Chapters 24 and 26]{CS1999}),
one may expect that 
the study of $V_{II_{25,1}}$ and its subalgebras including $V_{II_{25,1}}^{+}$ provides a better understanding of the moonshine vertex algebra $V^{\natural}$ and related algebras.
This is one of the motivations to study representations of $V_{\lattice}^{+}$.
As mentioned above $V_{\lattice}$ is not a vertex operator algebra in this case, however,
we note that the conjecture about representations for $V^G$ above makes sense even for (weak) modules for a vertex algebra $V$.
For $V_{\lattice}^{+}$-modules, 
the irreducible $V_{\lattice}^{+}$-modules 
are classified by using the Zhu algebra in \cite{J2006, Yamsk2008} and it is established that $V_{\lattice}^{+}$ is $C_2$-cofinite in
\cite{JY2010} and rational in \cite{Yamsk2009}.
However, the study of weak $V_{\lattice}^{+}$-modules has not progressed
in spite of the fact that $V_{\lattice}^{+}$ itself is a weak module
but not a module for $V_{\lattice}^{+}$,
because of the absence of useful tool like the Zhu algebras for weak modules.

The following is the main result of this paper, which implies that
for any non-degenerate even lattice $\lattice$ of finite rank,
every irreducible weak $V_{\lattice}^{+}$-module
is isomorphic to a weak submodule of some irreducible weak $V_{\lattice}$-module or to a submodule of some 
 irreducible $\theta$-twisted $V_{\lattice}$-module.
Namely, the conjecture above holds for irreducible weak $V_{\lattice}^{+}$-modules.

\begin{theorem}
\label{theorem:classification-weak-module}
Let $\lattice$ be a non-degenerate even lattice of finite rank with a bilinear form $\langle\ ,\ \rangle$.
The following is a complete set of representatives of equivalence classes of the irreducible weak $V_{\lattice}^{+}$-modules:
\begin{enumerate}
\item
$V_{\lambda+\lattice}^{\pm}$, $\lambda+\lattice\in \lattice^{\perp}/\lattice$ with $2\lambda\in \lattice$.   
\item
$V_{\lambda+\lattice}\cong V_{-\lambda+\lattice}$, $\lambda+\lattice\in \lattice^{\perp}/\lattice$ with $2\lambda\not\in \lattice$.   
\item
$V_{\lattice}^{T_{\chi},\pm}$ for any irreducible $\hat{\lattice}/P$-module $T_{\chi}$ with central character $\chi$.
\end{enumerate}
\end{theorem}

Here, 
$V_{\lambda+\lattice}^{\pm}=\{\lu\in V_{\lambda+\lattice}\ |\ \theta(\lu)=\pm \lu\}$
for $\lambda+\lattice\in \lattice^{\perp}/\lattice$ with $2\lambda\in \lattice$,
$\hat{\lattice}$ is the canonical central extension of $\lattice$
by the cyclic group $\langle\kappa\rangle$ of order $2$ with  the commutator map
$c(\alpha,\beta)=\kappa^{\langle\alpha,\beta\rangle}$ for $\alpha,\beta\in\lattice$,
$\subL=\{\theta(a) a^{-1}\ |\ a\in\hat{\lattice}\}$,
$V_{\lattice}^{T_{\chi}}$ is a $\theta$-twisted $V_{\lattice}$-module,
and $V_{\lattice}^{T_{\chi},\pm}=\{\lu\in V_{\lattice}^{T_{\chi}}\ |\ \theta(\lu)=\pm \lu\}$.
Note that in Theorem \ref{theorem:classification-weak-module}, 
$V_{\lattice}^{T_{\chi},\pm}$ in (3) are $V_{\lattice}^{+}$-modules,
however, if $\lattice$ is not positive definite,
then $V_{\lambda+\lattice}^{\pm}$ in (1) and $V_{\lambda+\lattice}$ in (2) are not  
$V_{\lattice}^{+}$-modules (see \eqref{eq:dimC(Vlambda+lattice)n=+infty}).
If $\lattice$ is positive definite, then Theorem \ref{theorem:classification-weak-module}
is the same as \cite[Theorem 7.7]{AD2004} and \cite[Theorem 5.13]{DN1999-2}.
Using Theorem \ref{theorem:classification-weak-module}, in \cite[Theorem 1.1]{Tanabe-nondeg-rep} we show that every weak $V_{\lattice}^{+}$-module is completely reducible
for any non-degenerate even lattice $\lattice$ of finite rank.

Let us explain the basic idea briefly in the case that $\rank L=1$.
We write $L=\Z\alpha$ and we further assume that $\langle\alpha,\alpha\rangle\neq 2$ to simplify the argument.
When $\langle\alpha,\alpha\rangle=2$, we only need to change the set of generators for $V_{\lattice}^{+}$
in the following argument.
Let $V$ be a vertex algebra and $(\module,Y_{\module})$ a weak $V$-module.
For $\lu\in V$ and $\lv\in\module$, we write the expansion
of $Y_{\module}(u,x)v$ by $Y_{\module}(u,x)v=\sum_{i\in\Z}u_ivx^{-i-1}$ 
and define $\epsilon(\lu,\lv)\in\Z\cup\{-\infty\}$ by
\begin{align}
\lu_{\epsilon(\lu,\lv)}\lv&\neq 0
\mbox{ and }\lu_{i}\lv=0\mbox{ for all }i>\epsilon(\lu,\lv)
\end{align}
if $Y_{\module}(u,x)v\neq 0$ and $\epsilon(\lu,\lv)=-\infty$
if $Y_{\module}(u,x)v= 0$.
It is known that the vertex operator algebra $M(1)$ associated to the
Heisenberg algebra is a subalgebra of $V_{\lattice}$
and the fixed point subalgebra $M(1)^{+}=M(1)^{\langle\theta\rangle}$ under the action of $\theta$
is a subalgebra of $V_{\lattice}^{+}$.
The irreducible $M(1)^{+}$-modules 
are classified in \cite{DN1999-1, DN2001}
and the fusion rules for $M(1)^{+}$-modules are determined in \cite{Abe2000, ADL2005}.
The vertex operator algebra $M(1)^{+}$ 
is generated by the conformal vector (Virasoro element) $\omega$ and a homogeneous element $\Har$ (or $J$) of weight $4$,
 and $V_{\lattice}^{+}$ is generated by $M(1)^{+}$ and \textcolor{black}{$E=e^{\alpha}+\theta(e^{\alpha})$} (see \eqref{eq:definition-omega-J-H}).
We can find some relations for $\omega,\Har$ in $M(1)^{+}$ and for $\omega,\Har, E$ in $V_{\lattice}^{+}$ with the help of computer algebra system 
Risa/Asir\cite{Risa/Asir} (Lemma \ref{lemma:relations-M(1)-V(lattice)+}).
Let $\lv$ 
be a non-zero element of a weak $V_{\lattice}^{+}$-module $\module$.
If $\epsilon(\omega,\lv)\geq 1$, then taking suitable actions of the relations on $\lv$
and using the commutation relations,
we obtain relations in 
$\Span_{\C}\{\omega_{\epsilon(\omega,\lv)}^i\Har_{\epsilon(\Har,\lv)}^{j}\lv
\ |\ i,j\in\Z_{\geq 0}\}$ or in
$\Span_{\C}\{\omega_{\epsilon(\omega,\lv)}^i\Har_{\epsilon(\Har,\lv)}^{j}
\ExB_{\epsilon(\ExB,\lv)}\lv\ |\ i,j\in\Z_{\geq 0}\}$
with the help of computer algebra system 
Risa/Asir (Lemmas \ref{lemma:m1+wJ} and \ref{lemma:r=1-s=3}).
Using these relations in $\module$, we can get
a simultaneous eigenvector $\lu\in\module$ for $\omega_1$ and $\Har_{3}$ with
$\epsilon(\omega,\lu)\leq 1$ and $\epsilon(\Har,\lu)\leq 3$, namely
we have an irreducible module $\C \lu$ for the Zhu algebra $A(M(1)^{+})$ of 
$M(1)^{+}$ (Lemma \ref{lemma:r=1-s=3}).
Moreover, we have some conditions on $\epsilon(\ExB,\lu)$ (Lemmas \ref{lemma:structure-Vlattice-M1}, \ref{lemma:structure-Vlattice-M1-norm2}, and \ref{lemma:structure-Vlattice-M1-norm1-2}).
Using results of extension groups for $M(1)^{+}$-modules (Section \ref{section:Extension groups for M(1)+}), we have an irreducible $M(1)^{+}$-module $\mK$ in the $M(1)^{+}$-submodule of $\module$ generated by $\C \lu$ (Lemma \ref{lemma:M(1)directsum}).
Since $V_{\lattice}^{+}$ is a direct sum of irreducible
$M(1)^{+}$-modules, for any irreducible $M(1)^{+}$-submodule $\mN$ of $V_{\lattice}^{+}$,
the $V_{\lattice}^{+}$-module structure of $\module$
induces an intertwining operator $I(\mbox{ },x) : \mN\times K\rightarrow\module\db{x}$ for weak $M(1)^{+}$-modules.
By using results of extension groups for $M(1)^{+}$-modules (Section \ref{section:Extension groups for M(1)+}), 
the same argument as above shows that 
there exists an $M(1)^{+}$-module which is a
direct sum of irreducible $M(1)^{+}$-modules in the image of $I(\mbox{ },x)$ (Section \ref{section:Intertwining operators for M(1)+}).
Thus, we obtain a weak irreducible $V_{\lattice}^{+}$-submodule $\mW$ of $\module$
which is isomorphic to a submodule of a $\theta$-twisted irreducible $V_{\lattice}$-module or 
is a direct sum of  pairwise non-isomorphic irreducible $M(1)^{+}$-modules (Lemma \ref{lemma:M(1)directsum}).
In the latter case,
by a standard argument we can determine the possible weak $V_{\lattice}^{+}$-module structures for such an $M(1)^{+}$-module (Lemma \ref{lemma:st-unique-two})
and hence we obtain the desired result.
For general $\lattice$,
we divide our analysis into four cases based on the norm of an element of $\C\otimes_{\Z}\lattice$
and carry out the procedure above by an enormous amount of computation.

Complicated computation has been done by a computer algebra system Risa/Asir\cite{Risa/Asir}.
Throughout this paper, the word \lq\lq a direct computation\rq\rq \ means a direct computation with the help of Risa/Asir.

The organization of the paper is as follows. 
In Section \ref{section:preliminary} we recall some basic properties of the
vertex algebra $M(1)$ associated to the Heisenberg algebra and 
the vertex algebra $V_{\lattice}$ associated to a non-degenerate even lattice $\lattice$.
In Section \ref{section:Modules for the Zhu algera of} 
we construct an irreducible module for the Zhu algebra
of $M(1)^{+}$
in a weak $V_{\lattice}^{+}$-module in the case that the rank of $\lattice$ is $1$.
In Section \ref{section:Modules for the Zhu algera of general}
we do the same for the general even non-degenerate lattice $L$.
In Section \ref{section:Extension groups for M(1)+}
we investigate extension groups for $M(1)^{+}$-modules.
In Section \ref{section:Intertwining operators for M(1)+}
we study intertwining operators for a triple of 
weak $M(1)^{+}$-modules.
In Section \ref{section:Proof of Theorem} we give a proof
of Theorem \ref{theorem:classification-weak-module}.
In Section \ref{section:appendix} (Appendix) we put 
computations of $a_{i}b$ for some $a,b\in V_{\lattice}^{+}$
and $i=0,1,\ldots$ to obtain a set of generators of a given $M(1)^{+}$-module. 
In Section \ref{section:notation} we list some notation.

\section{\label{section:preliminary}Preliminary}
We assume that the reader is familiar with the basic knowledge on
vertex algebras as presented in \cite{B1986,FLM,LL,Li1996}. 

Throughout this paper, $\mn$ is a non-zero complex number,
$\N$ denotes the set of all non-negative integers,
$\Z$ denotes the set of all integers, 
$\lattice$ is a non-degenerate even lattice of finite rank $\rankL$ with a bilinear form $\langle\ ,\ \rangle$,
$(V,Y,{\mathbf 1})$ is a vertex algebra.
Recall that $V$ is the underlying vector space, 
$Y(\mbox{ },\wx)$ is the linear map from $V\otimes_{\C}V$ to $V\db{x}$, and
${\mathbf 1}$ is the vacuum vector.
Throughout this paper, we always assume that $V$ has an element $\omega$ such that $\omega_{0}a=a_{-2}\vac$ for all $a\in V$.
For a vertex operator algebra $V$, this condition automatically holds since $V$ has the conformal vector (Virasoro element).
For $i\in\Z$, we define
\begin{align}
	\Z_{< i}&=\{j\in\Z\ |\ j< i\}\mbox{ and }\Z_{> i}=\{j\in\Z\ |\ j> i\}.
\end{align}
The notion of a module for $V$ has been introduced in several papers, however, if $V$ is a vertex operator algebra,
then the notion of a module for $V$ viewed as a vertex algebra is different from the notion of a module for $V$ viewed as 
a vertex operator algebra (cf. \cite[Definitions 4.1.1 and 4.1.6]{LL}).
To avoid confusion, throughout this paper, we refer to a module for a vertex operator algebra defined in \cite[Definition 4.1.6]{LL} as a {\it module}
and to a module for a vertex algebra defined in \cite[Definition 4.1.1]{LL} as a {\it weak module}.
The reason why we use the terminology \lq\lq weak module\rq\rq\ is that when $V$ is a vertex operator algebra, a module for $V$ viewed as a vertex algebra is called 
a weak $V$-module (cf. \cite[p.157]{Li1996}, \cite[p.150]{DLM1997}, and \cite[Definition 2.3]{ABD2004}).
We write down the definition of a weak $V$-module:
\begin{definition}
\label{definition:weak-module}
A {\it weak $V$-module} $\module$ is a vector space over $\C$ equipped with a linear map
\begin{align}
\label{eq:inter-form}
Y_{\module}(\ , x) : V\otimes_{\C}\module&\rightarrow \module\db{x}\nonumber\\
a\otimes u&\mapsto  Y_{\module}(a, x)\lu=\sum_{n\in\Z}a_{n}\lu x^{-n-1}
\end{align}
such that the following conditions are satisfied:
\begin{enumerate}
\item $Y_{\module}(\vac,x)=\id_{\module}$.
\item
For $a,b\in V$ and $\lu\in \module$,
\begin{align}
\label{eq:inter-borcherds}
&x_0^{-1}\delta(\dfrac{x_1-x_2}{x_0})Y_{\module}(a,x_1)Y_{\module}(b,x_2)\lu-
x_0^{-1}\delta(\dfrac{x_2-x_1}{-x_0})Y_{\module}(b,x_2)Y_{\module}(a,x_1)\lu\nonumber\\
&=x_1^{-1}\delta(\dfrac{x_2+x_0}{x_1})Y_{\module}(Y(a,x_0)b,x_2)\lu.
\end{align}
\end{enumerate}
\end{definition}
For $n\in\C$ and a weak $V$-module $\module$, we define $M_{n}=\{\lu\in V\ |\ \omega_1 \lu=n \lu\}$.
For $a\in V_{n}\ (n\in\C)$, $\wt a$ denotes $n$.
For a vertex algebra $V$ which admits a decomposition $V=\oplus_{n\in\Z}V_n$ and a subset $U$ of a weak $V$-module, we 
define
\begin{align}
\label{eq:OmegaV(U)=BigluinU}
\Omega_{V}(U)&=\Big\{\lu\in U\ \Big|\ 
\begin{array}{l}
a_{i}\lu=0\ \mbox{for all homogeneous }a\in V\\
\mbox{and }i>\wt a-1.
\end{array}\Big\}.
\end{align}
For a vertex algebra $V$ which admits a decomposition $V=\oplus_{n\in\Z}V_n$, a weak $V$-module $\mN$
 is called {\it $\N$-graded} if $N$ admits a decomposition $N=\oplus_{n=0}^{\infty}N(n)$
such that $a_{i}N(n)\subset N(\wt a-i-1+n)$ for all homogeneous $a\in V$, $i\in\Z$, and $n\in\Z_{\geq 0}$, where 
we define $N(n)=0$ for all $n<0$. For a triple of weak $V$-modules $\module, \mN,\mW$,
$\lu\in\module, \lv\in\mW$, and an intertwining operator $I(\ ,x)$ from $\module\times \mW$ to $\mN$, 
we write the expansion of $I(u,x)v$ by
\begin{align}
I(\lu,x)\lv
&=\sum_{i\in\C}\lu_i\lv x^{-i-1}
=\sum_{i\in\C}I(\lu_i;i)\lv x^{-i-1}\in \mN\{x\}.
\end{align}
In this paper, we consider only the case that the image of $I(\mbox{ },x)$ is contained in
$\mN\db{x}$,
namely $I(\ ,x) : \module\times \mW\rightarrow \mN\db{x}$. 
For a subset $X$ of $\mW$,
\begin{align}
M\cdot X\mbox{ denotes }\Span_{\C}\{a_{i}\lu\ |\ a\in \module, i\in\Z, \lu\in X\}\subset N. 
\end{align}
For an intertwining operator $I(\mbox{ },x) : \module\times \mW\rightarrow \mN\db{x}$,
$\lu\in\module$, and $\lv\in \mW$, we define  $\epsilon_{I}(\lu,\lv)=\epsilon(\lu,\lv)\in\Z\cup\{-\infty\}$ by
\begin{align}
\label{eqn:max-vanish}
\lu_{\epsilon_{I}(\lu,\lv)}\lv&\neq 0\mbox{ and }\lu_{i}\lv
=0\mbox{ for all }i>\epsilon_{I}(\lu,\lv)
\end{align}
if $I(\lu,x)\lv\neq 0$ and $\epsilon_{I}(\lu,\lv)=-\infty$
if $I(\lu,x)\lv= 0$.
If $\module$ is irreducible, then $\epsilon_{I}(\lu,\lv)\in\Z$ by \cite[Proposition 11.9]{DL}.
We will frequently use the following easy formula:
\begin{lemma}
\label{lemma:comm-change}
Let $\module, \mW, \mN$ be three weak $V$-modules
and  $I(\ ,x) : \module\times \mW\rightarrow \mN\db{x}$ an intertwining operator. 
For $a\in V$, $b\in \module$, $m\in\Z_{\geq -1}$, $k\in\Z_{\leq -1}$, and $n\in\Z$,
\begin{align}
\label{eq:abm}
(a_{k}b)_{n}
&=\sum_{\begin{subarray}{l}i\leq m\\i+j-k=n\end{subarray}}\binom{-i-1}{-k-1}a_{i}b_{j}+\sum_{\begin{subarray}{l}i\geq m+1\\i+j-k=n\end{subarray}}\binom{-i-1}{-k-1}b_{j}a_{i}\nonumber\\
&\quad{}-\sum_{i=0}^{m}\binom{-i-1}{-k-1}\sum_{\wl=0}^{\infty}\binom{i}{\wl}(a_{\wl}b)_{n+k-\wl}\nonumber\\
&=\sum_{\begin{subarray}{l}i\leq m\\i+j-k=n\end{subarray}}
\binom{-i-1}{-k-1}a_{i}b_{j}+\sum_{\begin{subarray}{l}i\geq m+1\\i+j-k=n\end{subarray}}\binom{-i-1}{-k-1}b_{j}a_{i}\nonumber\\
&\quad{}+(-1)^{k}\sum_{\wl=0}^{\infty}\binom{\wl-k-1}{-k-1}
\binom{m-k}{\wl-k}(a_{\wl}b)_{n+k-\wl}.
\end{align}
\end{lemma}
\begin{proof}
The first expression follows from
\begin{align}
(a_{k}b)_{n}
&=\sum_{\begin{subarray}{l}i<0\\i+j-k=n\end{subarray}}\binom{-i-1}{-k-1}a_{i}b_{j}+\sum_{\begin{subarray}{l}i\geq 0\\i+j-k=n\end{subarray}}\binom{-i-1}{-k-1}b_{j}a_{i}\nonumber\\
&=\sum_{\begin{subarray}{l}i\leq m\\i+j-k=n\end{subarray}}\binom{-i-1}{-k-1}a_{i}b_{j}+\sum_{\begin{subarray}{l}i\geq m+1\\i+j-k=n\end{subarray}}\binom{-i-1}{-k-1}b_{j}a_{i}
-\sum_{\begin{subarray}{l}0\leq i\leq m\\i+j-k=n\end{subarray}}\binom{-i-1}{-k-1}[a_{i},b_{j}].
\end{align}
The last expression follows from the fact that
$\sum_{i=0}^{m}\binom{-i-1}{-k-1}\binom{i}{\wl}
=(-1)^{k+1}\binom{\wl-k-1}{-k-1}
\binom{m-k}{\wl-k}$ for $l\in\Z_{\geq 0}$.
\end{proof}

We recall the vertex operator algebra $M(1)$ associated to the Heisenberg algebra and 
the vertex algebra $V_{\lattice}$ associated to a non-degenerate even lattice $\lattice$.
Let $\hei$ be a finite dimensional vector space equipped with a non-degenerate symmetric bilinear form
$\langle \mbox{ }, \mbox{ }\rangle$.
Set a Lie algebra
\begin{align}
\hat{\hei}&=\hei\otimes \C[t,t^{-1}]\oplus \C \centralc
\end{align} 
with the Lie bracket relations 
\begin{align}
[\beta\otimes t^{m},\gamma\otimes t^{n}]&=m\langle \beta,\gamma\rangle\delta_{m+n,0}\centralc,&
[\centralc,\hat{\hei}]&=0
\end{align}
for $\beta,\gamma\in \hei$ and $m,n\in\Z$.
For $\beta\in \hei$ and $n\in\Z$, $\beta(n)$ denotes $\beta\otimes t^{n}\in\widehat{H}$. 
Set two Lie subalgebras of $\fh$:
\begin{align}
\widehat{\hei}_{{\geq 0}}&=\bigoplus_{n\geq 0}\hei \otimes t^n\oplus \C \centralc&\mbox{ and }&&
\widehat{\hei}_{<0}&=\bigoplus_{n\leq -1}\hei\otimes t^n.
\end{align}
For $\beta\in\fh$,
$\C e^{\beta}$ denotes the one dimensional $\widehat{\fh}_{\geq 0}$-module uniquely determined
by the condition that for $\gamma\in\fh$
\begin{align}
\gamma(i)\cdot e^{\beta}&
=\left\{
\begin{array}{ll}
\langle\gamma,\beta\rangle e^{\beta}&\mbox{ for }i=0\\
0&\mbox{ for }i>0
\end{array}
\right.&&\mbox{ and }&
\centralc\cdot e^{\beta}&=e^{\beta}.
\end{align}
We take an $\widehat{\fh}$-module 
\begin{align}
\label{eq:untwist-induced}
M(1,\beta)&=\sU (\widehat{\fh})\otimes_{\sU (\widehat{\fh}_{\geq 0})}\C e^{\beta}
\cong \sU(\widehat{\fh}_{<0})\otimes_{\C}\C e^{\beta}
\end{align}
where $\sU(\fg)$ is the universal enveloping algebra of a Lie algebra $\fg$.
Then, $M(1)=M(1,0)$ has a vertex operator algebra structure with
the conformal vector
\begin{align}
\label{eq:conformal-vector}
\omega&=\dfrac{1}{2}\sum_{i=1}^{\dim\fh}h_i(-1)h_i^{\prime}(-1)\vac
\end{align}
where $\{h_1,\ldots,h_{\dim\fh}\}$ is a basis of $\fh$ and
$\{h_1^{\prime},\ldots,h_{\dim\fh}^{\prime}\}$ is its dual basis.
Moreover, $M(1,\beta)$ is an irreducible $M(1)$-module for any $\beta\in\fh$. 
The vertex operator algebra $M(1)$ is called the {vertex operator algebra associated to
 the Heisenberg algebra} $\oplus_{0\neq n\in\Z}\fh\otimes t^{n}\oplus \C \centralc$. 

Let $\lattice$ be a non-degenerate even lattice.
We define $\fh=\C\otimes_{\Z}\lattice$
and denote by $\lattice^{\perp}$ the dual of $\lattice$: $\lattice^{\perp}=\{\gamma\in\fh\ |\ \langle\beta,\gamma\rangle\in\Z\mbox{ for all }\beta\in\lattice\}$. 
Taking $M(1)$ for $\fh$,
we define $V_{\lambda+\lattice}=\oplus_{\beta\in\lambda+\lattice}M(1,\beta)$
for $\lambda+\lattice\in \lattice^{\perp}/\lattice$.
Then, $V_{\lattice}$ admits a unique vertex algebra structure compatible with the action of $M(1)$ and 
for each $\lambda+\lattice\in\lattice^{\perp}/\lattice$
the vector space
$V_{\lambda+\lattice}$ is an irreducible weak $V_{\lattice}$-module which admits the following decomposition:
\begin{align}
V_{\lambda+\lattice}&=\bigoplus_{n\in\langle\lambda,\lambda\rangle/2+\Z}(V_{\lambda+\lattice})_{n} \mbox{ where }
(V_{\lambda+\lattice})_{n}=\{a\in V_{\lambda+\lattice}\ |\ \omega_{1}a=na\}.
\end{align}
Note that if $\lattice$ is positive definite, then $\dim_{\C}(V_{\lambda+\lattice})_{n}<+\infty$
for all $n\in \lambda+\lattice$
and $(V_{\lambda+\lattice})_{\langle\lambda,\lambda\rangle/2+i}=0$ for sufficiently small $i\in\Z$.
If $\lattice$ is not positive definite, then  
\begin{align}
\label{eq:dimC(Vlambda+lattice)n=+infty}
\dim_{\C}(V_{\lambda+\lattice})_{n}=+\infty
\end{align}
for all $n\in \langle\lambda,\lambda\rangle/2+\Z$, which implies that $V_{\lambda+\lattice}$ is not a $V_{\lattice}$-module.
For $\alpha\in \fh$, we write
\begin{align}
E(\alpha)&=
\textcolor{black}{e^{\alpha}+\theta(e^{\alpha})}
\end{align}

Let  $\hat{\lattice}$ be the canonical central extension of $\lattice$
by the cyclic group $\langle\kappa\rangle$ of order $2$ with  the commutator map
$c(\alpha,\beta)=\kappa^{\langle\alpha,\beta\rangle}$ for $\alpha,\beta\in\lattice$:
\begin{align}
	0\rightarrow \langle\kappa\rangle\overset{}{\rightarrow} \hat{\lattice}\overset{-}{\rightarrow} \lattice\rightarrow 0.
\end{align}
Then, the $-1$-isometry of $\lattice$ induces an automorphism $\theta$ of $\hat{\lattice}$ of order $2$
and an automorphism, by abuse of notation we also denote by $\theta$, of $V_{\lattice}$ of order $2$.
In $M(1)$, we have
\begin{align}
\label{eq:theta}
\theta(h^1(-i_1)\cdots h^n(-i_n)\vac)&=(-1)^{n}h^1(-i_1)\cdots h^n(-i_n)\vac
\end{align}
for $n\in\Z_{\geq 0}$, $h^1,\ldots,h^{n}\in\fh$, and $i_1,\ldots,i_n\in\Z_{>0}$.
\textcolor{black}{For a weak $V_{\lattice}$-module $\module$,
we define a weak $V_{\lattice}$-module $(\module\circ \theta,Y_{\module\circ \theta})$
by $\module\circ\theta=\module$ and 
\begin{align}
Y_{\module\circ \theta}(a,x)&=Y_{\module}(\theta(a),x)
\end{align}
for  $a\in V_{\lattice}$.
Then
$V_{\lambda+\lattice}\circ\theta\cong V_{-\lambda+\lattice}$
for $\lambda\in \lattice^{\perp}$.
Thus, for $\lambda\in \lattice^{\perp}$ with $2\lambda\in\lattice$
we define
\begin{align}
V_{\lambda+\lattice}^{\pm}&=\{u\in V_{\lambda+\lattice}\ |\ \theta(u)=\pm u\}.
\end{align}
}
Next, we recall the construction of $\theta$-twisted modules for $M(1)$ and $V_{\lattice}$ following \cite{FLM}.
Set a Lie algebra
\begin{align}
\hat{\hei}[-1]&=\hei\otimes t^{1/2}\C[t,t^{-1}]\oplus \C \centralc
\end{align} 
with the Lie bracket relations 
\begin{align}
[\centralc,\hat{\hei}[-1]]&=0&\mbox{and}&&
[\alpha\otimes t^{m},\beta\otimes t^{n}]&=m\langle\alpha,\beta\rangle\delta_{m+n,0}\centralc
\end{align}
for $\alpha,\beta\in \hei$ and $m,n\in1/2+\Z$.
For $\alpha\in \hei$ and $n\in1/2+\Z$, $\alpha(n)$ denotes $\alpha\otimes t^{n}\in\widehat{\hei}$. 
Set two Lie subalgebras of $\hat{\hei}[-1]$:
\begin{align}
\widehat{\hei}[-1]_{{\geq 0}}&=\bigoplus_{n\in 1/2+\N}\hei\otimes t^n\oplus \C \centralc&\mbox{ and }&&
\widehat{\hei}[-1]_{{<0}}&=\bigoplus_{n\in 1/2+\N}\hei\otimes t^{-n}.
\end{align}
Let $\C \vac_{\tw}$ denote a unique one dimensional $\widehat{\hei}[-1]_{{\geq 0}}$-module 
such that 
\begin{align}
h(i)\cdot \vac_{\tw}&
=0\quad\mbox{ for }h\in\fh\mbox{ and }i\in \frac{1}{2}+\N,\nonumber\\
\centralc\cdot \vac_{\tw}&=\vac_{\tw}.
\end{align}
We take an $\widehat{\hei}[-1]$-module 
\begin{align}
\label{eq:twist-induced}
M(1)(\theta)
&=\sU (\widehat{\hei}[-1])\otimes_{\sU (\widehat{\hei}[-1]_{\geq 0})}\C u_{\ul{\zeta}}
\cong\sU (\widehat{\hei}[-1]_{<0})\otimes_{\C}\C u_{\ul{\zeta}}.
\end{align}
We define for $\alpha\in \hei$, 
\begin{align}
	\alpha(x)&=\sum_{i\in 1/2+\Z}\alpha(i)x^{-i-1}
	\end{align}
and for $u=\alpha_1(-\wi_1)\cdots \alpha_k(-\wi_k)\vac\in M(1)$, 
\begin{align}
	Y_{0}(u,x)&=\nor
\dfrac{1}{(\wi_1-1)!}	(\dfrac{d^{\wi_1-1}}{dx^{\wi_1-1}}\alpha_1(x))
	\cdots
\dfrac{1}{(\wi_k-1)!}	(\dfrac{d^{\wi_k-1}}{dx^{\wi_k-1}}\alpha_k(x))\nor.
\end{align}
Here, for $\beta_1,\ldots,\beta_{n}\in \fh$ and $i_1,\ldots,i_n\in1/2+\Z$, we define 
$\nor \beta_1(i_1)\cdots\beta_{n}(i_n)\nor$ inductively by
\begin{align}
\label{eq:nomal-ordering}
\nor \beta_1(i_1)\nor&=\beta_1(i_1)\qquad\mbox{ and}\nonumber\\
\nor \beta_1(i_1)\cdots\beta_{n}(i_n)\nor&=
\left\{
\begin{array}{ll}
\nor \beta_{2}(i_2)\cdots\beta_{n}(i_n)\nor \beta_1(i_1)&\mbox{if }i_1\geq 0,\\
\beta_{1}(i_1)\nor \beta_{2}(i_2)\cdots\beta_{n}(i_n)\nor &\mbox{if }i_1<0.
\end{array}\right.
\end{align}
Let $h^{[1]},\ldots,h^{[\dim\fh]}$ be an orthonormal basis of $\fh$.
We define $c_{mn}\in\Q$ for $ m,n\in \Z_{\geq 0}$ by
\begin{align}
	\sum_{m,n=0}^{\infty}c_{mn}x^{m}y^{n}&=-\log (\dfrac{(1+x)^{1/2}+(1+y)^{1/2}}{2})
	\end{align}
and
	\begin{align}
		\Delta_{x}&=\sum_{m,n=0}^{\infty}c_{mn}\sum_{i=1}^{\dim\fh}h^{[i]}(m)h^{[i]}(n)x^{-m-n}.
		\end{align}
Then, for $u\in M(1)$ we define a vertex operator $Y_{M(1)(\theta)}$ by
\begin{align}
Y_{M(1)(\theta)}(u,x)&=Y_{0}(e^{\Delta_{x}}u,x).
\end{align}
Then, \cite[Theorem 9.3.1]{FLM}
shows that 
$(M(1)(\theta),Y_{M(1)(\theta)})$ is an irreducible $\theta$-twisted $M(1)$-module.
Set a submodule $\subL=\{\theta(a) a^{-1}\ |\ a\in\hat{\lattice}\}$ of $\hat{\lattice}$.
Let $T_{\chi}$ be the irreducible $\hat{L}/\subL$-module associated to a
central character $\chi$ such that $\chi(\kappa)=-1$.
We set
\begin{align}
V_{\lattice}^{T_{\chi}}&=M(1)(\theta)\otimes T_{\chi}.
\end{align} 
Then, \cite[Theorem 9.5.3]{FLM} shows that 
$V_{\lattice}^{T_{\chi}}$ admits an irreducible $\theta$-twisted $V_{\lattice}$-module structure compatible with the action of $M(1)$.
We define the action of $\theta$ on $V_{\lattice}^{T_{\chi}}$ by
\begin{align}
\theta(h^1(-i_1)\cdots h^n(-i_n)\lu)&=(-1)^{n}h^1(-i_1)\cdots h^n(-i_n)\lu
\end{align}
for $n\in\Z_{\geq 0}$, $h^1,\ldots,h^{n}\in\fh$, $i_1,\ldots,i_n\in 1/2+\Z_{>0}$, and $\lu\in T_{\chi}$.
We set
\begin{align}
V_{\lattice}^{T_{\chi},\pm}&=\{\lu\in V_{\lattice}^{T_{\chi}}\ |\ \theta(u)=\pm u\}.
\end{align}

We recall the {\it Zhu algebra} $A(V)$ of a vertex operator algebra $V$ from \cite[Section 2]{Z1996}.
For homogeneous $a\in V$ and $b\in V$, we define
\begin{align}
\label{eq:zhu-ideal-multi}
a\circ b&=\sum_{i=0}^{\infty}\binom{\wt a}{i}a_{i-2}b\in V
\end{align}
and 
\begin{align}
\label{eq:zhu-bimodule-left}
a*b&=\sum_{i=0}^{\infty}\binom{\wt a}{i}a_{i-1}b\in V.
\end{align}
We extend \eqref{eq:zhu-ideal-multi} and \eqref{eq:zhu-bimodule-left} for an arbitrary $a\in V$ by linearity.
We also define
$O(V)=\Span_{\C}\{a\circ b\ |\ a,b\in V\}$.
Then, the quotient space
\begin{align}
\label{eq:zhu-bimodule}
A(V)&=M/O(V)
\end{align}
called the {\em Zhu algebra} of $V$, is an associative $\C$-algebra with multiplication  
\eqref{eq:zhu-bimodule-left} by \cite[Theorem 2.1.1]{Z1996}.

\section{Modules for the Zhu algebra of $M(1)^{+}$ in a weak $V_{\lattice}^{+}$-module:
the case that $\rank L=1$}
\label{section:Modules for the Zhu algera of}

In this section, under the condition that the rank of $L$ is $1$,
we shall show that there exists an irreducible $A(M(1)^{+})$-module in an arbitrary non-zero 
weak $V_{\lattice}^{+}$-module.

Throughout this section, $\mn$ is a non-zero complex number, 
$\hei$ is a one dimensional vector space equipped with a non-degenerate symmetric bilinear form
$\langle \mbox{ }, \mbox{ }\rangle$, $\wh,\alpha\in\hei$ such that
\begin{align}
\langle \wh,\wh\rangle&=1\mbox{ and }
\langle\alpha,\alpha\rangle=\mn,
\end{align}
$\module,\mN,\mW$ are weak $M(1)^{+}$-modules,
and 
$I(\mbox{ },x) : M(1,\alpha)\times \mW\rightarrow N\db{x}$ is  a non-zero intertwining operator.
We define
\begin{align}
\omega&=\dfrac{1}{2}h(-1)^2\vac,\nonumber\\
\Har&=\dfrac{1}{3}h(-3)h(-1)\vac-\dfrac{1}{3}h(-2)^2\vac,\nonumber\\
J&=h(-1)^4\vac-2h(-3)h(-1)\vac+\dfrac{3}{2}h(-2)^2\vac\nonumber\\
&=-9\Har+4\omega_{-1}^2\vac-3\omega_{-3}\vac,\nonumber\\
\ExB&=\ExB(\alpha)=e^{\alpha}+\textcolor{black}{\theta(e^{\alpha})}.
\label{eq:definition-omega-J-H}
\end{align}
Since \begin{align}
	0&=\omega_{-2}h(-1)\vac-2\omega_{-1}h(-2)\vac+3h(-4)\vac\label{eq:w-2h-2wh3h}
\end{align}
and $[\omega_{0},a_{i}]=-ia_{i-1}$ for all $a\in M(1)$ and $i\in\Z$,
\begin{align}
\label{eq:h(i)vacinSpanBig}
h(j)\vac\in \Span\Big\{\omega_{-i_1}\cdots\omega_{-i_m}h(-k)\vac\ \Big|\ 
\begin{array}{l}m\in\Z_{\geq 0},i_1,\ldots,i_m\in\Z_{>0}\\
\mbox{ and }k=1,2,3\end{array}\Big\}
\end{align}
for all $j\in\Z$. 
For $i,j\in\Z$, a direct computation shows that
\begin{align}
	[\omega_{i},\omega_{j}]&=(i-j)\omega_{i+j-1}+\delta_{i+j-2,0}\dfrac{i(i-1)(i-2)}{12},\label{eq:wiwj}\\
	[\omega_{i},J_{j}]&=(3i-j)J_{i+j-1},\label{eq:wiJj}\\
	[\omega_i,\Har_{j}]&
	=(3i-j)\Har_{i+j-1}
	+\frac{i(i-1)(3i+j-6)}{6}\omega_{i+j-3}\nonumber\\
	&\quad{}+\dfrac{-1}{3}\binom{i}{5}\delta_{i+j-4,0},\label{eq:wiHj}\\
	[\omega_{i},\ExB_{j}]&=((-1+\dfrac{\mn}{2})i-j)\ExB_{i+j-1} \label{eq:wiEji},\\
	[h(i),\omega_{j}]&=h(i+j-1),\\
	[h(i),\Har_{j}]
	&=\big(\frac{i(i+j-2)(5i+j-5)}{6}+\binom{i}{3}\big)h(i+j-3),
\end{align}
and
\begin{align}
\label{eq:h(-2)h(-1)=omega0} 
	h(-2)h(-1)&=
	\omega_{0 } \omega,\nonumber\\
	h(-3)h(-1)&=
	\Har
	+\frac{1}{3}\omega_{0 }^{2}\omega,\nonumber\\
	h(-2)h(-2)&=
	-2\Har
	+\frac{1}{3}\omega_{0 }^{2}\omega,\nonumber\\
	h(-3)h(-2)&=\frac{-1}{2}\omega_{0 } \Har
	+\frac{1}{12}\omega_{0 }^{3}\omega,\nonumber\\
	h(-3)h(-3)&=
	\frac{1}{3}\omega_{-2 }^2\vac
	+\frac{4}{5}\omega_{-1 } \Har
	+\frac{-1}{15}\omega_{0 }^{2}\omega_{-1 } \omega
	+\frac{-3}{10}\omega_{0 }^{2}\Har
	+\frac{1}{45}\omega_{0 }^{4}\omega.
\end{align}
It follows from \eqref{eq:h(i)vacinSpanBig}, \eqref{eq:wiHj} and \eqref{eq:h(-2)h(-1)=omega0} 
that $M(1)^{+}$ is spanned by the elements 
$\omega_{-i_1}\cdots \omega_{-i_m}\Har_{-j_1}\cdots \Har_{-j_n}\vac$
where $m,n\in\Z_{\geq 0}$ and $i_1,\ldots,i_m,j_1,\ldots,j_n\in\Z_{>0}$,
which is already shown in \cite[Theorem 2.7]{DG1998}.

We have
\begin{align}
\label{eq:J1E=frac2}
J_{2}\ExB&=2(2\mn-1)\omega_{0 } \ExB,\quad
J_{3}\ExB=(\mn^2-\dfrac{\mn}{2})\ExB,\quad
J_{i}\ExB=0\mbox{ for }i\geq 4,
\end{align}
and 
\begin{align}
\label{eq:Har1ExB=dfrac2mn2mn1omega1ExB}
\Har_{2}\ExB&=\dfrac{1}{3}\omega_{0 } \ExB,\quad
\Har_{i}\ExB=0\mbox{ for }i\geq 3.
\end{align}
If $\mn\neq 1/2$, then
\begin{align}
\label{eq:J1ExB=frac2mn(4mn-11)-1}
J_{1}\ExB&=\frac{2\mn(4\mn-11)}{2\mn-1}\omega_{-1 } \ExB
+\frac{8\mn+5}{2\mn-1}\omega_{0 }^2 \ExB,\nonumber\\
\Har_{1}\ExB&=
\dfrac{2\mn}{2\mn-1}\omega_{-1 } \ExB
+\dfrac{-1}{2\mn-1}\omega_{0 }^2 \ExB
\end{align}
and if $\mn\neq 2,1/2$, then
\begin{align}
\label{eq:J1ExB=frac2mn(4mn-11)-2}
J_{0}\ExB&=\frac{2 (\mn-8) (2 \mn-1)}{\mn-2}\omega_{-2 } \ExB
+\frac{4 (4 \mn^2-\mn+4)}{(\mn-2) (2 \mn-1)}\omega_{0 } \omega_{-1 } \ExB
\nonumber\\&\quad{}
+\frac{-18}{(\mn-2) (2 \mn-1)}\omega_{0 }^3 \ExB\nonumber,\\
\Har_{0}\ExB&=
\dfrac{2\mn}{\mn-2}\omega_{-2 } \ExB
+\dfrac{-4\mn}{(2\mn-1)(\mn-2)}\omega_{0 } \omega_{-1 } \ExB
+\dfrac{2}{(2\mn-1)(\mn-2)}\omega_{0 }^3 \ExB.
\end{align}
If $\mn=2$, then by Lemma \ref{lemma:comm-change} and \eqref{eq:omega[1]0Har[1]0ExB=omega0} in Section \ref{section:appendix} 
for $m\geq 0$,
\begin{align}
\label{eq:[omegai(Har0ExB)j]}
&		[\omega_{i},(\Har_{0}\ExB)_{j}]\nonumber\\
		&=(3i-j)(\Har_{0}\ExB)_{i+j-1}\nonumber\\
		&\quad{}+8\binom{i}{2}\big( 
		\sum_{k\leq m}\omega_{k}\ExB_{i+j-3-k}+\sum_{k\geq m+1}\ExB_{i+j-3-k}\omega_{k }\nonumber\\
&\qquad{}+((m+1)(i+j-3)-\binom{m+1}{2})\ExB_{i+j-4}\big)\nonumber\\
		&\quad{}-2\binom{i}{2}(i+j-2)(i+j-3) \ExB_{i+j-4}\nonumber\\
		&\quad{}-6\binom{i}{3}(i+j-3)\ExB_{i+j-4}+12\binom{i}{4}\ExB_{i+j-4}.
\end{align}
If $\mn=1/2$, then by \eqref{eq:norm1-2-omega[1]0Har[1]1ExB} in Section \ref{section:appendix},
\begin{align}
\label{eq:[omegaiHar1ExBj]}
[\omega_{i},(\Har_{1}\ExB)_{j}]
&=(\frac{5}{4}i-j)(\Har_{1}\ExB)_{i+j-1}\nonumber\\
&\quad{}+\binom{i}{2}(i+j-2)\ExB_{i+j-3}+
\binom{i}{3}\ExB_{i+j-3}.
\end{align}
For $n\in\Z$, $m\in\Z_{\geq -1}$, and $k\in\Z_{<0}$,
using Lemma \ref{lemma:comm-change} and \eqref{eq:wiEji},
we expand each of $(\omega_{k}E)_{\wn}$ and $(\omega_{-1}^2E)_{n}$
so that 
the resulting expression is  a linear combination of elements of the form 
\begin{align}
\omega_{i_1}\cdots \omega_{i_r}\ExB_{l}\omega_{j_1}\cdots \omega_{j_s}
\end{align}
where $r,s\in\Z_{\geq 0}$, $l\in\Z$, 
$i_1,\ldots,i_r\leq m
$,
and $j_1,\ldots,j_s\geq m+1$
as follows:
\begin{align}
\label{eq:(omega-12E)wn-1}
&(\omega_{k}E)_{\wn}\nonumber\\
&=\sum_{i\leq m}\binom{-i-1}{-k-1}\omega_{i}E_{\wn+k-i}+\sum_{i\geq m+1}\binom{-i-1}{-k-1}E_{\wn+k-i}\omega_{i}\nonumber\\
&\quad{}+(-1)^{k}((-\wn-k)\binom{m-k}{-k}+
\binom{-k}{-k-1}\binom{m-k}{1-k}\dfrac{\mn}{2})E_{\wn+k-1}\nonumber\\
\end{align}
and
\begin{align}
	\label{eq:(omega-12E)wn-2}
&	(\omega_{-1}^2E)_{n}\nonumber\\
&=
\sum_{\begin{subarray}{l}i<0,j<0,\\k=n-i-j-2\end{subarray}}\omega_{i}\omega_{j}E_{k}
+	2\sum_{\begin{subarray}{l}i<0,0\leq j,\\k=n-i-j-2\end{subarray}}\omega_{i}E_{k}\omega_{j}
	+\sum_{\begin{subarray}{l}0\leq i,0\leq j,\\k=n-i-j-2\end{subarray}}E_{k}\omega_{j}\omega_{i}\nonumber\\
&=\sum_{\begin{subarray}{l}i<0,j<0,\\k=n-i-j-2\end{subarray}}\omega_{i}\omega_{j}E_{k}\nonumber\\
&\quad{}+	2\sum_{\begin{subarray}{l}i<0,0\leq j\leq m,\\k=n-i-j-2\end{subarray}}
\big(\omega_{i}\omega_{j}E_{k}
-((-1+\frac{\mn}{2})j-k)\omega_{i}E_{j+k-1}\big)+2\sum_{\begin{subarray}{l}i<0,m+1\leq j,\\k=n-i-j-2\end{subarray}}\omega_{i}E_{k}\omega_{j}\nonumber\\
&\quad{}+\sum_{\begin{subarray}{l}0\leq i,j\leq m,\\k=n-i-j-2\end{subarray}}\Big(
\omega_{j}\omega_{i}E_{k}+
((1-\frac{p}{2})j+k)((1-\frac{p}{2})i+j+k-1)E_{j+k-2}
\nonumber\\
&\qquad{}+((1-\frac{p}{2})j+k)\omega_{i}E_{j+k-1}+((1-\frac{p}{2})i+k)\omega_{j}E_{i+k-1}\Big)\nonumber\\
&\quad{}+\sum_{\begin{subarray}{l}m+1\leq i,0\leq j\leq m,\\k=n-i-j-2\end{subarray}}\big(((1-\frac{p}{2})j+k)E_{j+k-1}\omega_{i}+\omega_{j}E_{k}\omega_{i}\big)\nonumber\\
&\quad{}+(m+1)((1-\frac{p}{2})m+n-m-3)E_{n-4}+(m+1)\omega_{m}E_{n-m-3}\nonumber\\
&\quad{}+\sum_{\begin{subarray}{l}0\leq i\leq m,m+1\leq j\nonumber\\(i,j)\neq (m+1,0),\\k=n-i-j-2\end{subarray}}
E_{k}(j-i)\omega_{i+j-1}
+\sum_{\begin{subarray}{l}0\leq i\leq m,m+1\leq j,\nonumber\\k=n-i-j-2\end{subarray}}((1-\frac{p}{2})i+k)E_{i+k-1}\omega_{j}+\omega_{i}E_{k}\omega_{j}\big)\nonumber\\
&\quad{}+\sum_{\begin{subarray}{l}m+1\leq i,j,\\k=n-i-j-2\end{subarray}}E_{k}\omega_{j}\omega_{i}.
\end{align}
For $i\in\Z$ and a subset $X$ of a weak $M(1)^{+}$-module $\mK$,
$\langle\omega_i\rangle X$ denotes the subspace of $\mK$ spanned by the elements
$\omega_i^{j}\lu, j\in\Z_{\geq 0}, \lu\in X$. 

The following lemmas follows from \eqref{eq:wiwj}--\eqref{eq:wiEji}, \eqref{eq:[omegai(Har0ExB)j]},
and \eqref{eq:[omegaiHar1ExBj]}.

\begin{lemma}\label{lemma:ojzero}
Let $\lu$ be  
an element of a weak $M(1)^{+}$-module $(\mK,Y_{\mK})$ with $\epsilon(\omega,\lu)=\epsilon_{Y_{\mK}}(\omega,\lu)\geq 1$.
\begin{enumerate}
\item For any $i\geq 0$, $j>\epsilon(\omega,\lu)$, and $k\geq 2$, 
\begin{align}
\omega_{j}\omega_{\epsilon(\omega,\lu)}^{i}\lu&=
\omega_{j}\omega_{\epsilon(\omega,\lu)}^{i}J_{\epsilon(J,\lu)}\lu=0,\nonumber\\
\omega_{k}^{i}J_{\epsilon(J,\lu)}\lu&=J_{\epsilon(J,\lu)}\omega_{k}^{i}\lu.
\end{align}
\item
For any non-zero $\lv\in \langle\omega_{\epsilon(\omega,\lu)}\rangle\{\lu,J_{\epsilon(J,\lu)}\lu\}$,
$\epsilon(\omega,\lv)\leq \epsilon(\omega,\lu)$. 
\end{enumerate}
\end{lemma}

\begin{lemma}\label{lemma:EBasis-zero}
Let $\lu\in\module$ with $\epsilon(\omega,\lu)=\epsilon_{I}(\omega,\lu)\geq 1$
and let $a$ be one of $E,\Har_{0}E$, or $\Har_{1}E$.
\begin{enumerate}
\item For any $i\geq 0$, $j>\epsilon(\omega,\lu)$, and $k\geq 2$, 
\begin{align}
\omega_{j}\omega_{\epsilon(\omega,\lu)}^{i}a_{\epsilon(a,\lu)}\lu&=0,\nonumber\\
\omega_{k}^{i}a_{\epsilon(a,\lu)}\lu&=a_{\epsilon(a,\lu)}\omega_{k}^{i}\lu.
\end{align}
\item
For a non-zero $\lv\in \langle\omega_{\epsilon(\omega,\lu)}\rangle a_{\epsilon(a,\lu)}\lu$,
$\epsilon(\omega,\lv)\leq\epsilon(\omega,\lu)$. 
\end{enumerate}
\end{lemma}

By using the commutation relation $[\Har_{i},\ExB_{j}]=\sum_{k=0}^{\infty}\binom{i}{k}(\Har_{k}\ExB)_{i+j-k}$
for $i,j\in\Z$,
the following result follows from
Lemma \ref{lemma:comm-change} and \eqref{eq:Har1ExB=dfrac2mn2mn1omega1ExB}--\eqref{eq:J1ExB=frac2mn(4mn-11)-2}.
\begin{lemma}\label{lemma:bound-H0E}
Assume $\mn\neq 1/2$.
Let $\lu\in \mW$ with $\epsilon(\omega,\lu)\geq 1$
and $\epsilon(\Har,\lu)\leq 2\epsilon(\omega,\lu)+1$.
Then, $\epsilon(\Har_{0}\ExB,\lu)\leq \epsilon(\ExB,\lu)+2\epsilon(\omega,\lu)+1$.
\end{lemma}

If $\mn=1/2$, then a direct computation shows that
\begin{align}
0&=\omega_{-1 } E-\omega_{0 }^2 E,\label{eq:norm1-2-0}\\
\Har_{0}E&=\frac{-2}{3}\omega_{-2 } E+\frac{4}{3}\omega_{0 }(\Har_{1}E),\label{eq:norm1-2-1}\\
0&=
8\omega_{-3 } E+12\Har_{-1 } E
+3\omega_{-1 } (\Har_{1}E)+4\omega_{0 } \omega_{-2 } E
-11\omega_{0 }^2(\Har_{1}E).
\label{eq:norm1-2-2}
\end{align}
\begin{lemma}
\label{lemma:bound-H0-1E}
Let $U$ be an $A(M(1)^{+})$-submodule of $\Omega_{M(1)^{+}}(\mW)$,
$\lu\in U$, and $\lE\in\Z$ such that $\epsilon(\ExB,\lv)\leq \lE$ for all non-zero $\lv\in U$. 
Then $\epsilon(\Har_{0}\ExB,\lu)\leq \lE+3$ and 
$\epsilon(\Har_{1}\ExB,\lu)\leq \lE+2$.
\end{lemma}
\begin{proof}
For $\mn\neq 1/2$,
the result follows from
Lemma \ref{lemma:bound-H0E} and \eqref{eq:J1ExB=frac2mn(4mn-11)-1}.
Assume $\mn=1/2$.
For $i,j\in\Z$ and $r\in\Z_{\geq 0}$, it follows from 
\eqref{eq:Har1ExB=dfrac2mn2mn1omega1ExB}, \eqref{eq:norm1-2-1}, and Lemma \ref{lemma:comm-change} that
\begin{align}
\label{eq:[HariExBj]lu}
&[\Har_{i},\ExB_{j}]\lu\nonumber\\
&=(\Har_{0}\ExB)_{i+j}\lu+i(\Har_{1}\ExB)_{i+j-1}\lu+\binom{i}{2}(\Har_{2}\ExB)_{i+j-2}\lu\nonumber\\
&=\frac{-1}{3}(i+4j)(\Har_{1}E)_{i+j-1}\lu
-\binom{i}{2}\frac{i+j-2}{3}\ExB_{i+j-3}\lu\nonumber\\
&\quad{}+\frac{-2}{3}
\big(\sum_{k\leq \lom}(-k-1)\omega_{k}\ExB_{i+j-2-k}+\sum_{k\geq \lom+1}(-k-1)\ExB_{i+j-2-k}\omega_{k}\nonumber\\
&\qquad{}+(\binom{r+2}{2}(-i-j+2)+\binom{r+2}{3})E_{i+j-3}\big)\lu.
\end{align}
Using \eqref{eq:[HariExBj]lu} with $i=3$ and $r=1$,
we have $\epsilon(\Har_{1}\ExB,\lu)\leq \lE+2$.
By \eqref{eq:norm1-2-1}, $\epsilon(\Har_{0}\ExB,\lu)\leq \lE+3$.
\end{proof}

A direct computations shows the following result.
\begin{lemma}
\label{lemma:relations-M(1)-V(lattice)+}
The following elements of $V_{\lattice}^{+}$ are zero:
\begin{align}
\sv^{(8),H}&=
-2376\omega_{-2 } \omega_{-2 } \omega_{-1 } \vac
+3168\omega_{-3 } \omega_{-1 } \omega_{-1 } \vac
-6256\omega_{-3 } \omega_{-3 } \vac
-11799\omega_{-4 } \omega_{-2 } \vac
\nonumber\\&\quad{}
+30456\omega_{-5 } \omega_{-1 } \vac
+2310\omega_{-7 } \vac
-9504\omega_{-1 } \omega_{-1 } \Har_{-1 } \vac
-6024\omega_{-3 } \Har_{-1 } \vac
\nonumber\\&\quad{}
-13419\omega_{-2 } \Har_{-2 } \vac
-6516\omega_{-1 } \Har_{-3 } \vac
+11868\Har_{-5 } \vac+5040\Har_{-1 }^2 \vac,\\
\sv^{(8),J}&=
-29056\omega_{-1}^4\vac
-118960\omega_{-2}^2\omega_{-1}\vac+
39040\omega_{-3}\omega_{-1}^2\vac
-39480\omega_{-3}^2\vac\nonumber\\
&\quad{}
-32120\omega_{-4}\omega_{-2}\vac+
497760\omega_{-5}\omega_{-1}\vac+
230360\omega_{-7}\vac\nonumber\\
&\quad{}+
5024\omega_{-1}^2J_{-1}\vac
-8536\omega_{-3}J_{-1}\vac+
8939\omega_{-2}J_{-2}\vac\nonumber\\
&\quad{}
-2444\omega_{-1}J_{-3}\vac+
1572J_{-5}+560J_{-1}^2\vac,\\
\sv^{(9)}&=30J_{-6}\vac-30\omega_{-1}J_{-4}\vac+27\omega_{-2}J_{-3}\vac-39\omega_{-3}J_{-2}\vac\nonumber\\
&\quad{}+16\omega_{-1}^2J_{-2}\vac+52\omega_{-4}J_{-1}\vac-32\omega_{-2}\omega_{-1}J_{-1}\vac,
\end{align}
\begin{align}
\sv^{(10),H}&=
919328\omega_{-9 } \vac
-545856\omega_{-5 } \omega_{-1 } \omega_{-1 } \vac
\nonumber\\&\quad{}
-529536\omega_{-4 } \omega_{-4 } \vac
+545352\omega_{-4 } \omega_{-2 } \omega_{-1 } \vac
\nonumber\\&\quad{}
+520160\omega_{-3 } \omega_{-3 } \omega_{-1 } \vac
-524968\omega_{-3 } \omega_{-2 } \omega_{-2 } \vac
\nonumber\\&\quad{}
-10240\omega_{-3 } \omega_{-1 } \omega_{-1 } \omega_{-1 } \vac
+7680\omega_{-2 } \omega_{-2 } \omega_{-1 } \omega_{-1 } \vac
\nonumber\\&\quad{}
+1937712\omega_{-5 } \Har_{-1 } \vac
-845376\omega_{-3 } \omega_{-1 } \Har_{-1 } \vac
\nonumber\\&\quad{}
-381048\omega_{-2 } \omega_{-2 } \Har_{-1 } \vac
+30720\omega_{-1 } \omega_{-1 } \omega_{-1 } \Har_{-1 } \vac
\nonumber\\&\quad{}
-720081\omega_{-4 } \Har_{-2 } \vac
-128280\omega_{-2 } \omega_{-1 } \Har_{-2 } \vac
\nonumber\\&\quad{}
-435576\omega_{-3 } \Har_{-3 } \vac
+234528\omega_{-1 } \omega_{-1 } \Har_{-3 } \vac
\nonumber\\&\quad{}
+345849\omega_{-2 } \Har_{-4 } \vac
-1211160\omega_{-1 } \Har_{-5 } \vac
\nonumber\\&\quad{}
+2360970\Har_{-7 } \vac
+70875\Har_{-2 } \Har_{-2 } \vac
\nonumber\\&\quad{}
+734184\omega_{-7 } \omega_{-1 } \vac
+898766\omega_{-6 } \omega_{-2 } \vac,
\end{align}
\begin{align}
\sv^{(10),J}&=8192\omega_{-1 }^5\vac-2048\omega_{-1 }^{3}J_{-1 } \vac\nonumber\\
&\quad{}+758496\omega_{-9 } \vac
-1728\omega_{-5 } \omega_{-3 } \vac
\nonumber\\&\quad{}
-15232\omega_{-5 } \omega_{-1 } \omega_{-1 } \vac
-60848\omega_{-4 } \omega_{-4 } \vac
\nonumber\\&\quad{}
-134224\omega_{-4 } \omega_{-2 } \omega_{-1 } \vac
-6912\omega_{-3 } \omega_{-3 } \omega_{-1 } \vac
\nonumber\\&\quad{}
-136872\omega_{-3 } \omega_{-2 } \omega_{-2 } \vac
-112640\omega_{-3 } \omega_{-1 } \omega_{-1 } \omega_{-1 } \vac
\nonumber\\&\quad{}
-69280\omega_{-2 } \omega_{-2 } \omega_{-1 } \omega_{-1 } \vac
-6092\omega_{-4 } J_{-2 } \vac
\nonumber\\&\quad{}
+6272\omega_{-3 } \omega_{-1 } J_{-1 } \vac
+360\omega_{-2 } \omega_{-2 } J_{-1 } \vac
\nonumber\\&\quad{}
+152\omega_{-2 } \omega_{-1 } J_{-2 } \vac
+1856\omega_{-3 } J_{-3 } \vac
\nonumber\\&\quad{}
+9408\omega_{-1 } \omega_{-1 } J_{-3 } \vac
+12656\omega_{-2 } J_{-4 } \vac
\nonumber\\&\quad{}
-29968\omega_{-1 } J_{-5 } \vac
+43320J_{-7 } \vac
\nonumber\\&\quad{}
+525J_{-2 } J_{-2 } \vac
+1309248\omega_{-7 } \omega_{-1 } \vac
\nonumber\\&\quad{}
+352992\omega_{-6 } \omega_{-2 } \vac,
\end{align}
\begin{align}
\label{eq:big(2(mn-2)(-27 + 54mn - 44mn2+ 40mn3)omega-3-1}
Q^{(4)}
&=
2(\mn-2)(-27 + 54 \mn - 44 \mn^2 + 40 \mn^3)\omega_{-3 }\ExB
\nonumber\\&\quad{}
-12\mn (\mn-2) (-3 + 4 \mn)\omega_{-1 }^2\ExB
\nonumber\\&\quad{}
-6 \mn( \mn-2 )  (-9 + 2 \mn) (-1 + 2 \mn)\Har_{-1 }\ExB
\nonumber\\&\quad{}
+(-72\mn^3-96\mn^2+210\mn-90)\omega_{0 } \omega_{-2 }\ExB
\nonumber\\&\quad{}
+(120\mn^2-48\mn+36)\omega_{0 }^2\omega_{-1 }\ExB
\nonumber\\&\quad{}
+(-48\mn-9)\omega_{0 }^4\ExB,
\end{align}
\begin{align}
Q^{(5,1)}&=
3 (\mn-2) (10 \mn^2-29 \mn+32) (10 \mn^2-4 \mn+3)\omega_{-4 } E
\nonumber\\&\quad{}
-12 \mn (3 \mn-4) (10 \mn^2-4 \mn+3)\omega_{-2 } \omega_{-1 } E
\nonumber\\&\quad{}
-3 (\mn-8) (\mn-2) (2 \mn-1) (10 \mn^2-4 \mn+3)\Har_{-2 } E
\nonumber\\&\quad{}
+8 (2 \mn-7) (15 \mn^3-22 \mn^2+8 \mn-6)\omega_{0 } \omega_{-3 } E
\nonumber\\&\quad{}
+24 \mn^2 (8 \mn-9)\omega_{0 } \omega_{-1 } \omega_{-1 } E
\nonumber\\&\quad{}
-12 (\mn-2) (2 \mn-1) (6 \mn^2-5 \mn+6)\omega_{0 } \Har_{-1 } E
\nonumber\\&\quad{}
-6 (2 \mn^3-32 \mn^2+29 \mn+12)\omega_{0 } \omega_{0 } \omega_{-2 } E
\nonumber\\&\quad{}
-6 (8 \mn-9)\omega_{0 } \omega_{0 } \omega_{0 } \omega_{0 } \omega_{0 } E,\\
Q^{(5,2)}&=
3(\mn-2)(10\mn^2-29\mn+32)(12\mn^3+16\mn^2-35\mn+15)\omega_{-4 } \ExB
\nonumber\\&\quad{}
-12\mn(3\mn-4)(12\mn^3+16\mn^2-35\mn+15)\omega_{-2 } \omega_{-1 } \ExB
\nonumber\\&\quad{}
-3(\mn-8)(\mn-2)(2\mn-1)(12\mn^3+16\mn^2-35\mn+15)\Har_{-2 } \ExB
\nonumber\\&\quad{}
+2(136\mn^5-316\mn^4-1266\mn^3+3409\mn^2-2470\mn+624)\omega_{0 } \omega_{-3 } \ExB
\nonumber\\&\quad{}
+12\mn(20\mn^3-3\mn^2-44\mn+24)\omega_{0 } \omega_{-1 } \omega_{-1 } \ExB
\nonumber\\&\quad{}
-6(\mn-2)(2\mn-1)(14\mn^3+21\mn^2-74\mn+60)\omega_{0 } \Har_{-1 } \ExB
\nonumber\\&\quad{}
-12(2\mn^3-32\mn^2+29\mn+12)\omega_{0 } \omega_{0 } \omega_{0 } \omega_{-1 } \ExB
\nonumber\\&\quad{}
-3(16\mn^2+61\mn-102)\omega_{0 } \omega_{0 } \omega_{0 } \omega_{0 } \omega_{0 } \ExB,
\end{align}
\begin{align}
Q^{(6)}&=
 2 (3696 \mn^8-22564 \mn^7+66284 \mn^6-84937 \mn^5+56207 \mn^4
\nonumber\\&\qquad{}
-91528 \mn^3+11774 \mn^2+29190 \mn-13500)\omega_{-5 } E
\nonumber\\&\quad{}
-4 \mn (352 \mn^6+2152 \mn^5-8282 \mn^4+7951 \mn^3-11696 \mn^2\nonumber\\&\qquad{}
+6304 \mn-1542)\omega_{-3 } \omega_{-1 } E
\nonumber\\&\quad{}
-3 \mn (1584 \mn^6-5572 \mn^5+6456 \mn^4-6877 \mn^3+5214 \mn^2\nonumber\\&\qquad{}
-3040 \mn+642)\omega_{-2 } \omega_{-2 } E
\nonumber\\&\quad{}
+720 \mn^3 (\mn-2) (4 \mn-1)\omega_{-1 } \omega_{-1 } \omega_{-1 } E
\nonumber\\&\quad{}
-24 \mn (\mn-2) (2 \mn-1) (44 \mn^4-98 \mn^3+157 \mn^2-88 \mn+48)\omega_{-1 } \Har_{-1 } E
\nonumber\\&\quad{}
-3 (\mn-2) (2 \mn-25) (2 \mn-1)^2 (44 \mn^4-13 \mn^3+62 \mn^2-48 \mn+18)\Har_{-3 } E
\nonumber\\&\quad{}
+3 (1760 \mn^7-9382 \mn^6+1391 \mn^5+28130 \mn^4-14380 \mn^3\nonumber\\&\qquad{}
+29762 \mn^2-25851 \mn+7650)\omega_{0 } \omega_{-4 } E
\nonumber\\&\quad{}
+12 \mn (352 \mn^5-1459 \mn^4+2396 \mn^3-2894 \mn^2+1254 \mn-225)\omega_{0 } \omega_{-2 } \omega_{-1 } E
\nonumber\\&\quad{}
-3 (\mn-2) (2 \mn-1) (352 \mn^5+101 \mn^4+86 \mn^3-614 \mn^2+804 \mn-225)\omega_{0 } \Har_{-2 } E
\nonumber\\&\quad{}
+12 (88 \mn^6+1104 \mn^5-4136 \mn^4+3714 \mn^3-3944 \mn^2+2670 \mn-675)\omega_{0 } \omega_{0 } \omega_{-3 } E
\nonumber\\&\quad{}
-6 (352 \mn^5-1099 \mn^4+686 \mn^3-689 \mn^2+804 \mn-225)\omega_{0 } \omega_{0 } \omega_{0 } \omega_{-2 } E
\nonumber\\&\quad{}
-90 (\mn-2) (4 \mn-1)\omega_{0 } \omega_{0 } \omega_{0 } \omega_{0 } \omega_{0 } \omega_{0 } E.
\end{align}
If $\mn=2$, then we have the four following relations:
\begin{align}
0&=6\omega_{-2 } E
-4\omega_{0 } \omega_{-1 } E
+\omega_{0 }^{3}E,\label{eq:6omega2E4omega0omega1E-1}\\
0&=
180\omega_{-3 }E  -48\omega_{-1 }^2 E +72H_{-1 }E -63\omega_{0 } (H_{0}E)\nonumber\\&\quad{}
 +8\omega_{0 }^2\omega_{-1 } E +\omega_{0 }^4 E,
\label{eq:6omega2E4omega0omega1E-2}
\\
0&=
9450\omega_{-4 } E -900\omega_{-1 } (H_{0}E)+6750H_{-2 } E-768\omega_{0 } \omega_{-1 }^{2} E \nonumber\\&\quad{}
 -3168\omega_{0 } \Har_{-1 } E 
+297\omega_{0 }^{2} (H_{0}E) +128\omega_{0 }^{3}\omega_{-1 } E+16\omega_{0 }^{5}E,
\label{eq:6omega2E4omega0omega1E-3}\\
0&=
584199000\omega_{-6 } \ExB 
-117085500\Har_{-4 } \ExB 
\nonumber\\&\quad{}
+98941500\omega_{-3 } (\Har_{0}\ExB) 
-27594000\omega_{-1 }^2 (\Har_{0}\ExB) 
\nonumber\\&\quad{}
+34587000\Har_{-1 } (\Har_{0}\ExB) 
-13132800\omega_{0 } \omega_{-1 }^3 \ExB 
\nonumber\\&\quad{}
-60739200\omega_{0 } \omega_{-1 } \Har_{-1 } \ExB 
+277223400\omega_{0 } \Har_{-3 } \ExB 
\nonumber\\&\quad{}
-85188900\omega_{0 } \omega_{-2 } (\Har_{0}\ExB) 
+206053320\omega_{0 }^2 \omega_{-4 } \ExB 
\nonumber\\&\quad{}
-8524040\omega_{0 }^2 \omega_{-1 } (\Har_{0}\ExB) 
-27546608\omega_{0 }^3\omega_{-3 } \ExB 
\nonumber\\&\quad{}
-51990312\omega_{0 }^3\Har_{-1 } \ExB 
+17161013\omega_{0 }^4(\Har_{0}\ExB) 
\nonumber\\&\quad{}
-820800\omega_{0 }^5\omega_{-1 } \ExB 
+410400\omega_{0 }^7 \ExB.
\label{eq:6omega2E4omega0omega1E-4}
\end{align}
\renewcommand{\arraystretch}{2.0}
\renewcommand{\arraystretch}{1}
\end{lemma}

\begin{lemma}\label{lemma:m1+wJ}
Let $\lu$ be a non-zero element of a weak $M(1)^{+}$-module $(\mK,Y_{\mK})$ such that
$\epsilon(\omega,\lu)=\epsilon_{Y_{\mK}}(\omega,\lu)\geq 2$ and  $\epsilon(\omega,\lu)\leq \epsilon(\omega,\lv)$ for 
all non-zero $\lv\in \mK$.
Then $\epsilon(J,\lu)=2\epsilon(\omega,\lu)+1$,
\begin{align}
\label{eq:J2e=4omegae}
J_{2\epsilon(\omega,\lu)+1}\lu&=4\omega_{\epsilon(\omega,\lu)}^2\lu,
\end{align}
and 
\begin{align}
\label{eq:epsilon(Har,lu)leq 2epsilon}
\epsilon(\Har,\lu)&\leq 2\epsilon(\omega,\lu).
\end{align} 
\end{lemma}
\begin{proof}
We write 
\begin{align}
\lao&=\epsilon(\omega,\lu)\mbox{ and }\laJ=\epsilon(J,\lu)
\end{align}
for simplicity. 
It follows from Lemma \ref{lemma:ojzero} (2) and the condition of $\lu$ that for any non-zero $\lv\in \langle\omega_{\lao}
\rangle\{\lu,J_{\laJ}\lu\}$ and 
$i\in\Z_{\geq 0}$, 
\begin{align}
\label{eq:omiv}
\omega_{\lao}^i\lv&\neq 0.
\end{align}
Since the same argument as in \cite[(3.23)]{Tanabe2017}
shows that
\begin{align}
\label{eq:J-ind}
0&=\dfrac{1}{16}\sv^{(9)}_{\laJ+2\lao+3}\lu=(-\laJ+2\lao+1)J_{\laJ}\omega_{\lao}^2\lu\nonumber\\
&=(-\laJ+2\lao+1)\omega_{\lao}^2J_{\laJ}\lu
\end{align}
by Lemma \ref{lemma:ojzero} (1), $\laJ=2\lao+1$ by \eqref{eq:omiv}.
Since
\begin{align}
0&=P^{(10),J}_{5\lao+4}\lu
=(8192\omega_{-1}^{5}\vac-2048\omega_{-1}^3J_{-1}\vac)_{5\lao+4}\lu\nonumber\\
&=2048(4\omega_{\lao}^{5}-J_{2\lao+1}\omega_{\lao}^3)\lu\nonumber\\
&=2048\omega_{\lao}^{3}(4\omega_{\lao}^{2}-J_{2\lao+1})\lu
\end{align}
by Lemma \ref{lemma:ojzero} (1), \eqref{eq:J2e=4omegae} holds by \eqref{eq:omiv}.
It follows from \eqref{eq:definition-omega-J-H} that $\Har_{i}\lu=0$ for all $i\geq 2\epsilon(\omega,\lu)+1$ and hence 
$\epsilon(\Har,\lu)\leq 2\epsilon(\omega,\lu)$.
\end{proof}

\begin{lemma}\label{lemma:r=1-s=3}
	Let $\lattice$ be a non-degenerate even lattice of rank $1$
	and $\module$ a non-zero weak $V_{\lattice}^{+}$-module.
	Then, there exists a non-zero $\lu \in\Omega_{M(1)^{+}}(\module)$ that
	satisfies one of the following conditions:
	\begin{enumerate}
		\item $\epsilon(\omega,\lu)=\epsilon(J,\lu)=\epsilon(E,\lu)=-1$. In this case $V_{L}^{+}\cdot \lu\cong V_{L}^{+}$.
		\item $\Har_{3}\lu=0$.
		\item $\omega_1\lu=\lu$ and $\Har_{3}\lu=\lu$.
		\item $\omega_1\lu=(1/16)\lu$ and $\Har_{3}\lu=(-1/128)\lu.$
		\item $\omega_1\lu=(9/16)\lu$ and $\Har_{3}\lu=(15/128)\lu.$
	\end{enumerate}
\end{lemma}
\begin{proof}
We write $\lattice=\Z \alpha$. Throughout the proof of this lemma,
$\mn=\langle\alpha,\alpha\rangle\in 2\Z\setminus\{0\}$.
	For $\lu\in \mW$ with $\epsilon(\omega,\lu)<0$,
	since $\omega_0\lu=0$, it follows from \cite[Proposition 4.7.7]{LL}
	that $V_{\lattice}^{+}\cdot \lu\cong V_{\lattice}^{+}$ and hence
	$\epsilon(\omega,\lu)=\epsilon(J,\lu)=\epsilon(E,\lu)=-1$.
	
	We assume $\epsilon(\omega,\lv)\geq 0$ for all $\lv\in\module$.
	We take a non-zero $\lu\in\module$ with  $\epsilon(\omega,\lu)$ as small as possible, namely $0\leq \epsilon(\omega,\lu)\leq \varepsilon(\omega,\lv)$ for all $\lv\in \module$.
	We write 
	\begin{align}
		\lom&=\epsilon(\omega,\lu),\quad \lJ=\epsilon(J,\lu),\mbox{ and }\lE=\epsilon(E,\lu)
	\end{align}
	for simplicity. 
	Suppose $\lao\geq 2$. Then, Lemma \ref{lemma:m1+wJ}  shows that $\laJ=2\lao+1$ and
	$
	\Har_{i}\lu=0
	$
	for all $i\geq 2\lao+1$.
By Lemma \ref{lemma:EBasis-zero} (2), 
for any non-zero $\lv\in \langle\omega_{\lao}\rangle\{\lu,J_{\laJ}\lu,E_{\laE}\lu, 
(\Har_{0}E)_{\epsilon(\Har_{0}E,\lu)}\lu\}$ and 
	$i\in\Z_{\geq 0}$, 
	\begin{align}
		\label{eq:omiv2}
		\omega_{\lao}^i\lv&\neq 0.
	\end{align}
Assume $\mn\neq 2$.
By \eqref{eq:(omega-12E)wn-1} and \eqref{eq:(omega-12E)wn-2} with $m=\lom$, 
	\begin{align}
		\label{eq:0=(omega-3ExB)lE+2lom+2lu=}
0&=(\omega_{-3}\ExB)_{\lE+2\lom+2}\lu=(\omega_0\omega_{-2}\ExB)_{\lE+2\lom+2}\lu\nonumber\\
&=(\omega_0^2\omega_{-1}\ExB)_{\lE+2\lom+2}\lu=
		(\omega_0^4\ExB)_{\lE+2\lom+2}\lu
		\end{align}
	and
	\begin{align}
	\label{eq:(omega-12ExB)lE+2lom+2lu=}
	(\omega_{-1}^2\ExB)_{\lE+2\lom+2}\lu&=\omega_{\lom}^{2}\ExB_{\lE}\lu.
	\end{align}
Using Lemma \ref{lemma:comm-change}, 
\eqref{eq:Har1ExB=dfrac2mn2mn1omega1ExB}, \eqref{eq:J1ExB=frac2mn(4mn-11)-1}, and \eqref{eq:J1ExB=frac2mn(4mn-11)-2},
we expand $(\Har_{-1}\ExB)_{\lE+2\lom+2}$ so that 
the resulting expression is  a linear combination of elements of the form 
\begin{align}
a^{(1)}_{i_1}\cdots a^{(l)}_{i_l}\ExB_{m}b^{(1)}_{j_1}\cdots b^{(n)}_{j_n}
	\end{align}
where $l,n\in\Z_{\geq 0}$, $m\in\Z$, and 
\begin{align}
(a^{(1)},i_1),\ldots,(a^{(l)},i_l)&\in\{(\omega,k)\ |\ k\leq \lom\}\cup\{(\Har,k)\ |\ k\leq 2\lom\},\nonumber\\
(b^{(1)},j_1),\ldots,(b^{(n)},j_n)&\in\{(\omega,k)\ |\ k\geq \lom+1\}\cup\{(\Har,k)\ |\ k\geq 2\lom+1\},
\end{align}
as was done in \eqref{eq:(omega-12E)wn-1} and \eqref{eq:(omega-12E)wn-2}.
Then, taking the action of the obtained expansion of $(\Har_{-1}\ExB)_{\lE+2\lom+2}$ on $\lu$
and using \eqref{eq:epsilon(Har,lu)leq 2epsilon} and \eqref{eq:0=(omega-3ExB)lE+2lom+2lu=}, we have 
	\begin{align}
	\label{eq:(Har-1ExB)lE+2lom+2lu=}
	(\Har_{-1}\ExB)_{\lE+2\lom+2}\lu&=\ExB_{\lE}\Har_{2\lom+1}\lu=0.
\end{align}
By \eqref{eq:big(2(mn-2)(-27 + 54mn - 44mn2+ 40mn3)omega-3-1}, \eqref{eq:0=(omega-3ExB)lE+2lom+2lu=}, 
\eqref{eq:(omega-12ExB)lE+2lom+2lu=}, and \eqref{eq:(Har-1ExB)lE+2lom+2lu=},
	\begin{align}
		\label{eqn:zeroQ4}
		0&=Q^{(4)}_{\lE+2\lom+2}\lu
		=-12\mn (\mn-2) (-3 + 4 \mn)\omega_{\lom}^2E_{\lE}\lu,
	\end{align}
which contradicts \eqref{eq:omiv2}. 

Assume $\mn=2$.
By Lemma \ref{lemma:bound-H0E}, $\epsilon(\Har_{0}\ExB,\lu)\leq \lE+2\lom+1$.
	By \eqref{eq:6omega2E4omega0omega1E-2}, Lemma \ref{lemma:m1+wJ} and the results in Section \ref{section:normal-2},
the same argument as above shows
	\begin{align}
		0&=
		(180\omega_{-3 } E -48\omega_{-1 }^2 E +72\Har_{-1 } E -63\omega_{0 } (\Har_{0}E)\nonumber\\&\quad{}
		+8\omega_{0 }^2\omega_{-1 } E +\omega_{0 }^4 E)_{\lE+2\lom+2}\lu\nonumber\\
		&=(-48\omega_{\lom}^2E_{\lE}+72E_{\lE}\Har_{2\lom+1}+63(\lE+2\lom+2)(\Har_{0}E)_{\lE+2\lom+1})\lu\nonumber\\
		&=(-48\omega_{\lom}^2E_{\lE}+63(\lE+2\lom+2)(\Har_{0}E)_{\lE+2\lom+1})\lu
	\end{align}
	and hence $(\Har_{0}E)_{\lE+2\lom+1}\lu\neq 0$ by \eqref{eq:omiv2}.
	By \eqref{eq:6omega2E4omega0omega1E-3} and results in Section \ref{section:normal-2},
	\begin{align}
		0&=
		(9450\omega_{-4 } E -900\omega_{-1 } (H_{0}E)+6750H_{-2 } E-768\omega_{0 } \omega_{-1 }^{2} E \nonumber\\&\quad{}
		-3168\omega_{0 } \Har_{-1 } E 
		+297\omega_{0 }^{2} (H_{0}E) +128\omega_{0 }^{3}\omega_{-1 } E+16\omega_{0 }^{5}E)_{\lE+3\lom+2}\lu\nonumber\\
		&=-900\omega_{\lom}(\Har_{0}E)_{\lE+2\lom+1}\lu,
	\end{align}
	which also contradicts \eqref{eq:omiv2}. We conclude that $\lao\leq 1$.
	
	Suppose $\laJ\geq 4$. 
By using  \cite[(2.29)]{Tanabe2017} and \eqref{eq:wiJj}, the same argument as in \cite[Lemma 3.3]{Tanabe2017} shows that
	$\epsilon(\omega,J_{\laJ}\lu)\leq 1$ and
	$J_{j}J_{\laJ}\lu=0$ for all $j\geq \laJ+1$.
	By the same argument as in \cite[(3.25)]{Tanabe2017},
	\begin{align}
		(J_{-1}J)_{2\laJ+1}\lu
		&=J_{\laJ}^2\lu
	\end{align}
	and hence
	\begin{align}
		\label{eq:jjv}
		0&=P^{(8),J}_{2\laJ+1}\lu=J_{\laJ}^2\lu=J_{\laJ}(J_{\laJ}\lu),
	\end{align}
	which means $\epsilon(J,J_{\laJ}\lu)<\laJ=\epsilon(J,\lu)$.
	Replacing $\lu$ by $J_{\laJ}\lu$ repeatedly, we get a non-zero $\lu\in \module$ such that
	$\lao\leq 1$ and $\laJ\leq 3$. Thus, $\lu\in \Omega_{M(1)^{+}}(\module)$ and 
	in particular, $\epsilon(\Har,\lu)\leq 3$.
	Deleting the terms including $\omega_1^{i}\Har_3^2\lu\ (i=0,1,\ldots)$ from 
	the following simultaneous equations 
	\begin{align}
		\label{eq:p87u}
		0&=P^{(8),H}_7\lu
		=-72(132 \omega_{1 }^2
		-65 \omega_{1 } 
		+3
		-70\Har_{3 })\Har_{3 }\lu\mbox{ and }\\
		0&=P^{(10),H}_9\lu \nonumber\\
		&=240 \Har_{3} (-207 + 4725 \Har_{3} + 4472 \omega_{1} - 9118 \omega_{1}^2 + 128 \omega_{1}^3)\lu,
	\end{align}
	we have
	\begin{align}
		\label{eq:P8-P10}
		0
		&=( \omega_{1}-1 ) ( 16 \omega_{1}-1 ) (16 \omega_{1}-9  ) \Har_{3}\lu.
	\end{align}
	By \eqref{eq:p87u} and \eqref{eq:P8-P10}, the proof is complete.
\end{proof}
\begin{remark}
	If $\mn>0$, then Lemma \ref{lemma:r=1-s=3} also 
	follows from \cite[Theorem 7.7]{Abe2005}, \cite[Theorem 5.13]{DN1999-2}, and \cite[Theorem 2.7]{Mi2004d}.
\end{remark}
\begin{remark}
As we have seen in the proof of Lemma \ref{lemma:r=1-s=3},
starting from an arbitrary non-zero element in $\module$,
we can get $\lu$ in Lemma \ref{lemma:r=1-s=3} inductively.
\end{remark}

\begin{lemma}
\label{lemma:structure-Vlattice-M1}
Assume $\mn\neq 2, 1/2$.
Let $\lu$ be a non-zero element of $\Omega_{M(1)^{+}}(W)$.
We write 
\begin{align}
\lE&=\epsilon(\ExB,\lu)
\end{align}
for simplicity.
We set
\begin{align}
\lv&=(\omega_1-\dfrac{(\lE+1)^2}{2\mn})\lu.
\end{align}
We have 
\begin{align}
\lvE&=(\omega_1-\dfrac{(\lE+1-\mn)^2}{2\mn})E_{\lE}\lu.
\end{align}
\begin{enumerate}
\item Assume $\Har_{3}\lu=0$.
If $\lvE\neq 0$, then 
$\lE=\mn-2$ and 
\begin{align}
\label{eq:omega1lv=lv}
\omega_1(\lvE)&=\lvE.
\end{align}
If $\lu$ is an eigenvector of $\omega_1$ and $\lv\neq 0$, then 
$\lE=\mn-2$ and 
\begin{align}
\omega_1\lu=\dfrac{\mn}{2}\lu.
\end{align}
\item
If $\omega_1\lu=\lu$ and $\Har_{3}\lu=\lu$, then
$\lE=0$.
\item
Assume $\omega_1\lu=(1/16)\lu$ and $\Har_{3}\lu=(-1/128)\lu$. 
Then
$\lE=\mn/2-1$ or $(\mn-1)/2$.
In particular if $\mn$ is an even integer, then $\lE=\mn/2-1$.
\item
Assume $\omega_1\lu=(9/16)\lu$ and $\Har_{3}\lu=(15/128)\lu$. Then
$\lE=\mn/2-1$ or $(\mn-3)/2$.
In particular if $\mn$ is an even integer, then $\lE=\mn/2-1$.
\end{enumerate}
\end{lemma}
\begin{proof}
We first expand each of $Q^{(4)}_{\lE+4},Q^{(5,1)}_{\lE+5}, Q^{(5,2)}_{\lE+5}$, and $Q^{(6)}_{\lE+6}$
so that the resulting expression is a linear combination of elements of the form 
\begin{align}
	a^{(1)}_{i_1}\cdots a^{(l)}_{i_l}\ExB_{m}b^{(1)}_{j_1}\cdots b^{(n)}_{j_n}
\end{align}
where $l,n\in\Z_{\geq 0}$, $m\in\Z$, and 
\begin{align}
	(a^{(1)},i_1),\ldots,(a^{(l)},i_l)&\in\{(\omega,k)\ |\ k\leq 1\}\cup\{(\Har,k)\ |\ k\leq 2\},\nonumber\\
	(b^{(1)},j_1),\ldots,(b^{(n)},j_n)&\in\{(\omega,k)\ |\ k\geq 2\}\cup\{(\Har,k)\ |\ k\geq 3\},
\end{align}
as was done in the proof of \eqref{eq:0=(omega-3ExB)lE+2lom+2lu=}--\eqref{eq:(Har-1ExB)lE+2lom+2lu=} in Lemma \ref{lemma:r=1-s=3}.
Then, taking the action of each expansion on $\lu$, we have
\begin{align}
\label{eq:w18w48-1}
0&=((\lE+1-\mn)^2-2\mn\omega_{1})
\nonumber\\
&\quad{}\times((16 \mn+3) \lE^2+(-24 \mn^2+58 \mn) \lE
\nonumber\\
&\qquad\quad{}+8 \mn^3+(-8  \omega_{1}-37) \mn^2+(22  \omega_{1}+42) \mn-12  \omega_{1})\ExB_{\lE}\lu
\nonumber\\&\quad{}+
2 \mn (\mn-2) (2 \mn-9) (2 \mn-1)\ExB_{\lE}\Har_{3}\lu,\\
\label{eq:w18w48-2}
0&=
(((\lE+1-\mn)^2-2\mn\omega_{1}))\nonumber\\
&\quad{}\times\big((8 \mn-9) \lE^3+(88 \mn^2+26 \mn-45) \lE^2
\nonumber\\
&\qquad\quad{}+(-146 \mn^3+(16 \omega_{1}+375) \mn^2+(-18 \omega_{1}-43) \mn-54) \lE
\nonumber\\
&\qquad\quad{}+50 \mn^4
+(-60 \omega_{1}-240) \mn^3+(184 \omega_{1}+306) \mn^2
\nonumber\\
&\qquad\qquad{}+(-140 \omega_{1}-40) \mn+24 \omega_{1}-24\big)\ExB_{\lE}\lu
\nonumber\\&\quad{}+
2 (\mn-2) (2 \mn-1) ((6 \mn^2-5 \mn+6) \lE+10 \mn^3-54 \mn^2+10 \mn+6)\ExB_{\lE}\Har_{3}\lu,\\
\label{eq:w18w48-3}
0&=
(((\lE+1-\mn)^2-2\mn\omega_{1}))\nonumber\\
&\quad{}\times\big((16 \mn^2+61 \mn-102) \lE^3+(216 \mn^3+326 \mn^2-67 \mn-510) \lE^2
\nonumber\\
&\qquad\quad{}+(-352 \mn^4+(40 \omega_{1}+349) \mn^3+(-6 \omega_{1}+1772) \mn^2
\nonumber\\
&\qquad\qquad{}+(-88 \omega_{1}-1218) \mn+48 \omega_{1}-540) \lE+120 \mn^5
\nonumber\\
&\qquad\qquad{}+(-144 \omega_{1}-376) \mn^4+
(200 \omega_{1}-405) \mn^3
\nonumber\\
&\qquad\qquad{}+(646 \omega_{1}+1914) \mn^2+(-1180 \omega_{1}-1000) \mn+480 \omega_{1}-240\big)\ExB_{\lE}\lu
\nonumber\\&\quad{}+
2 (\mn-2) (2 \mn-1) \big((14 \mn^3+21 \mn^2-74 \mn+60) \lE+24 \mn^4
\nonumber\\
&\qquad\qquad{}-90 \mn^3-221 \mn^2+220 \mn+60\big)\ExB_{\lE}\Har_{3}\lu,
\end{align}
\begin{align}
\label{eq:w18w48-4}
0&=
5 ((\lE+1-\mn)^2-2\mn\omega_{1}) \nonumber\\
&\quad{}\times\big((12 \mn^2-27 \mn+6) \lE^4\nonumber\\
&\qquad\quad{}+(24 \mn^3+174 \mn^2-501 \mn+114) \lE^3\nonumber\\
&\qquad\quad{}+(1056 \mn^5-1443 \mn^4+(24 \omega_{1}+2373) \mn^3+(-54 \omega_{1}-1317) \mn^2
\nonumber\\
&\qquad\qquad{}+(12 \omega_{1}-1548) \mn+411) \lE^2\nonumber\\
&\qquad\quad{}+(-1584 \mn^6+(352 \omega_{1}+5873) \mn^5+(-814 \omega_{1}-8768) \mn^4\nonumber\\
&\qquad\qquad{}+(1220 \omega_{1}+12088) \mn^3+(-1568 \omega_{1}-7732) \mn^2
\nonumber\\
&\qquad\qquad{}+(756 \omega_{1}-561) \mn-90 \omega_{1}+414) \lE\nonumber\\
&\qquad\quad{}+528 \mn^7+(-704 \omega_{1}-3058) \mn^6+(2820 \omega_{1}+6912) \mn^5\nonumber\\
&\qquad\qquad{}+(48 \omega_{1}^2-5042 \omega_{1}-10000) \mn^4+(-108 \omega_{1}^2+7242 \omega_{1}+11310) \mn^3\nonumber\\
&\qquad\qquad{}+(24 \omega_{1}^2-7052 \omega_{1}-5950) \mn^2+(3060 \omega_{1}+42) \mn-360 \omega_{1}+180\big)\ExB_{\lE}\lu
\nonumber\\&\quad{}+
(\mn-2) (2 \mn-1) \big((1056 \mn^5-582 \mn^4+1428 \mn^3-1932 \mn^2+1992 \mn-450) \lE \nonumber\\
&\qquad\quad{}+792 \mn^6+(176 \omega_{1}-6224) \mn^5+(-392 \omega_{1}+8666) \mn^4
 \nonumber\\
&\qquad\quad{}+(628 \omega_{1}-13729) \mn^3+(-352 \omega_{1}+11014) \mn^2+(192 \omega_{1}+120) \mn-450\big)\ExB_{\lE}\Har_{3}\lu.
\end{align}
We also expand each of $Q^{(4)}_{\lE+4},Q^{(5,1)}_{\lE+5}, Q^{(5,2)}_{\lE+5}$ and $Q^{(6)}_{\lE+6}$
so that the resulting expression is a linear combination of elements of the form 
\begin{align}
	a^{(1)}_{i_1}\cdots a^{(l)}_{i_l}\ExB_{m}b^{(1)}_{j_1}\cdots b^{(n)}_{j_n}
\end{align}
where $l,n\in\Z_{\geq 0}$, $m\in\Z$, and 
\begin{align}
	(a^{(1)},i_1),\ldots,(a^{(l)},i_l)&\in\{(\omega,k)\ |\ k\leq 0\}\cup\{(\Har,k)\ |\ k\leq 2\},\nonumber\\
	(b^{(1)},j_1),\ldots,(b^{(n)},j_n)&\in\{(\omega,k)\ |\ k\geq 1\}\cup\{(\Har,k)\ |\ k\geq 3\}.
\end{align}
Then, taking the action of each expansion on $\lu$, we have
\begin{align}
\label{eq:ExBlE((1+lE)2-2mnomega1)-0}
0&=
\ExB_{\lE}((1+\lE)^2-2 \mn \omega_{1})
\nonumber\\&\quad{}\times{} \big((16 \mn+3)\lE^{2}
+
(-16 \mn^2+36 \mn+12)\lE
\nonumber\\&\qquad\quad{}+
4 \mn^3+(-8 \omega_1-18) \mn^2+(22 \omega_1+14) \mn-12 \omega_1+12\big)\lu
\nonumber\\&\quad{}+\ExB_{\lE}\Har_{3}\big(
8 \mn^4-56 \mn^3+98 \mn^2-36 \mn\big)\lu,
\end{align}
\begin{align}
\label{eq:ExBlE((1+lE)2-2mnomega1)-1}
0&=
\ExB_{\lE}((1+\lE)^2-2 \mn \omega_{1})
\nonumber\\&\quad{}\times{} \big((8 \mn-9)\lE^{3}
\nonumber\\&\qquad\quad{}+
(72 \mn^2+44 \mn-45)\lE^{2}
\nonumber\\&\qquad\quad{}+
(-78 \mn^3+(16 \omega_1+166) \mn^2+(-18 \omega_1+115) \mn-78)\lE
\nonumber\\&\qquad\quad{}+
20 \mn^4+(-60 \omega_1-88) \mn^3+(184 \omega_1+52) \mn^2+(-140 \omega_1+112) \mn+24 \omega_1-48\big)\lu
\nonumber\\&\quad{}+\ExB_{\lE}\Har_{3}\big(
(24 \mn^4-80 \mn^3+98 \mn^2-80 \mn+24)\lE
\nonumber\\&\qquad\quad{}+
40 \mn^5-316 \mn^4+620 \mn^3-292 \mn^2-20 \mn+24\big)\lu,\\
\label{eq:ExBlE((1+lE)2-2mnomega1)-2}
0&=
\ExB_{\lE}((1+\lE)^2-2 \mn \omega_{1})
\nonumber\\&\quad{}\times{} \big((16 \mn^2+61 \mn-102)\lE^{3}
\nonumber\\&\qquad\quad{}+
(176 \mn^3+332 \mn^2+21 \mn-558)\lE^{2}
\nonumber\\&\qquad\quad{}+
(-188 \mn^4+(40 \omega_1+106) \mn^3+(-6 \omega_1+1088) \mn^2\nonumber\\
&\qquad\qquad{}+(-88 \omega_1+74) \mn+48 \omega_1-1068)\lE
\nonumber\\&\qquad\quad{}+
48 \mn^5+(-144 \omega_1-132) \mn^4+(200 \omega_1-282) \mn^3\nonumber\\
&\qquad\qquad{}+(646 \omega_1+678) \mn^2+(-1180 \omega_1+420) \mn+480 \omega_1-720\big)\lu
\nonumber\\&\quad{}+\ExB_{\lE}\Har_{3}\big(
(56 \mn^5-56 \mn^4-450 \mn^3+1064 \mn^2-896 \mn+240)\lE
\nonumber\\&\qquad\quad{}+
96 \mn^6-600 \mn^5+112 \mn^4+2730 \mn^3-2844 \mn^2+280 \mn+240\big)\lu,
	\end{align}
\begin{align}
\label{eq:ExBlE((1+lE)2-2mnomega1)-3}
0&=
\ExB_{\lE}((1+\lE)^2-2 \mn \omega_{1})
\nonumber\\&\quad{}\times{} \big((60 \mn^2-135 \mn+30)\lE^{4}
\nonumber\\&\qquad\quad{}+
(1140 \mn^2-2565 \mn+570)\lE^{3}
\nonumber\\&\qquad\quad{}+
(3520 \mn^5-2845 \mn^4+(120 \omega_1+4970) \mn^3+(-270 \omega_1+1675) \mn^2\nonumber\\
&\qquad\qquad{}+(60 \omega_1-11580) \mn+2505)\lE^{2}
\nonumber\\&\qquad\quad{}+
(-3520 \mn^6+(1760 \omega_1+11230) \mn^5+(-4550 \omega_1-10490) \mn^4\nonumber\\
&\qquad\qquad{}+(7180 \omega_1+13010) \mn^3+(-8080 \omega_1+6570) \mn^2\nonumber\\
&\qquad\qquad{}+(3780 \omega_1-22110) \mn-450 \omega_1+4320)\lE
\nonumber\\&\qquad\quad{}+
(880 \mn^7+(-3520 \omega_1-4660) \mn^6+(14340 \omega_1+7480) \mn^5\nonumber\\
&\qquad\qquad{}+(240 \omega_1^2-26230 \omega_1-5875) \mn^4+(-540 \omega_1^2+37410 \omega_1+2050) \mn^3\nonumber\\
&\qquad\qquad{}+(120 \omega_1^2-35500 \omega_1+13280) \mn^2+(15300 \omega_1-15990) \mn-1800 \omega_1+2700)
\big)\lu
\nonumber\\&\quad{}+\ExB_{\lE}\Har_{3}\big(
(1760 \mn^7-4780 \mn^6+4310 \mn^5-7540 \mn^4+13100 \mn^3-13060 \mn^2+5850 \mn-900)\lE
\nonumber\\&\qquad\quad{}+
(1760 \mn^8+(352 \omega_1-17592) \mn^7+(-1664 \omega_1+53484) \mn^6\nonumber\\
&\qquad\qquad{}+(3568 \omega_1-89118) \mn^5+(-4628 \omega_1+114333) \mn^4+(3400 \omega_1-86520) \mn^3\nonumber\\
&\qquad\qquad{}+(-1664 \omega_1+22384) \mn^2+(384 \omega_1+2106) \mn-900)
\big)\lu.
\end{align}
Note that $\omega_{1}$'s are on the left side of $\ExB_{\lE}$ in \eqref{eq:w18w48-1}--\eqref{eq:w18w48-4}
but are on the right side of $\ExB_{\lE}$ in \eqref{eq:ExBlE((1+lE)2-2mnomega1)-0}--\eqref{eq:ExBlE((1+lE)2-2mnomega1)-3}.

\begin{enumerate}
\item
Assume $\Har_{3}\lu=0$
, $\lvE=(\omega_1-(\lE+1-\mn)^2/(2\mn))E_{\lE}\lu\neq 0$, and $\lE\neq \mn-2$.
By \eqref{eq:w18w48-1},
\begin{align}
\label{eq:dfrac3lE2+42mn +58lEmn}
0&=(2(\mn-2)(4\mn-3)\omega_1\nonumber\\
&\qquad{}-(3 \lE^2+42 \mn +58 \lE \mn +16 \laE^2 \mn -37 \mn ^2-24 \lE \mn ^2+8 \mn ^3))
\ExB_{\lE}\lv.
\end{align}
Then, by \eqref{eq:w18w48-2} and \eqref{eq:w18w48-3},
\begin{align}
0&=(10 \mn^2-4 \mn+3)(2 t-\mn+2) ((8 \mn-9) t-4 \mn^2+16 \mn-12),\nonumber\\
0&=(12 \mn^3+16 \mn^2-35 \mn+15) (2 t-\mn+2) ((8 \mn-9) t-4 \mn^2+16 \mn-12)
\end{align}
and hence 
\begin{align}
\label{eq:0=mn92mn2mnlE}
0&=(2 t-\mn+2) ((8 \mn-9) t-4 \mn^2+16 \mn-12).
\end{align}
By \eqref{eq:w18w48-3} and \eqref{eq:dfrac3lE2+42mn +58lEmn},
\begin{align}
\label{eq:(4032 mn5-2880 mn4+3360 mn3-2091 mn2+774 mn-108}
0&=
(4032 \mn^5-2880 \mn^4+3360 \mn^3-2091 \mn^2+774 \mn-108) t^3\nonumber\\
&\quad{}+(11264 \mn^7-38336 \mn^6+71692 \mn^5-87577 \mn^4\nonumber\\
&\qquad{}+82959 \mn^3-43593 \mn^2+12078 \mn-1431) t^2\nonumber\\
&\quad{}+(-11264 \mn^8+62768 \mn^7-146980 \mn^6+228928 \mn^5\nonumber\\
&\qquad{}-266274 \mn^4+218205 \mn^3-103356 \mn^2+26244 \mn-2916) t\nonumber\\
&\quad{}+2816 \mn^9-22304 \mn^8+72152 \mn^7-138308 \mn^6\nonumber\\
&\qquad{}+193478 \mn^5-203897 \mn^4+147876 \mn^3-64152 \mn^2+15282 \mn-1620.
\end{align}
If $2 t-\mn+2=0$, then $\lE\neq 0$ and by \eqref{eq:(4032 mn5-2880 mn4+3360 mn3-2091 mn2+774 mn-108}
\begin{align}
0&=t^2(4t+3)(352t^4+1356t^3+2080t^2+1452t+385),
\end{align}
a contradiction.
If $(8 \mn-9) t-4 \mn^2+16 \mn-12=0$, then as polynomials in $\mn$,
the right-hand side of \eqref{eq:(4032 mn5-2880 mn4+3360 mn3-2091 mn2+774 mn-108}
divided by $(8 \mn-9) t-4 \mn^2+16 \mn-12$ leaves a remainder of 
\begin{align}
\label{eq:(10560 mn-11880) t6+(76320 mn-169245/2) t5}
0&=(10560 \mn-11880) t^6+(76320 \mn-169245/2) t^5\nonumber\\
&\quad{}+(471705/2 \mn-2059785/8) t^4+(3187935/8 \mn-6397605/16) t^3\nonumber\\
&\quad{}+(3012585/8 \mn-5600025/16) t^2+(202005 \mn-356265/2) t\nonumber\\
&\quad{}+45900 \mn-39285.
\end{align}
Moreover, $(8 \mn-9) t-4 \mn^2+16 \mn-12$ divided by  the right-hand side of 
\eqref{eq:(10560 mn-11880) t6+(76320 mn-169245/2) t5} leaves a remainder of 
\begin{align}
0&=
\frac{-9 (t+1) (4 t+3)^2 (7 t+4) (7 t^2+10 t+7)}
{(5632 t^6+40704 t^5+125788 t^4+212529 t^3+200839 t^2+107736 t+24480)^2} \nonumber\\
&\quad{}\times (94556 t^4+495381 t^3+1010052 t^2+931824 t+326592)
\end{align}
and hence $\lE=-1$, which implies $\mn=1/2$ by \eqref{eq:(10560 mn-11880) t6+(76320 mn-169245/2) t5}, a contradiction.
 Thus, 
$\lE=\mn-2$
by \eqref{eq:0=mn92mn2mnlE} and \eqref{eq:omega1lv=lv} holds by \eqref{eq:dfrac3lE2+42mn +58lEmn}.
\item Assume $\omega_{1}\lu=\lu$ and $\Har_{3}\lu=\lu$.
By \eqref{eq:ExBlE((1+lE)2-2mnomega1)-0} 
and \eqref{eq:ExBlE((1+lE)2-2mnomega1)-1},
$tg_1(\mn)=tg_2(\mn)=0$ where
\begin{align}
\label{eq:g1(mn)=3(lE+1)2(lE+4)+}
g_1(\mn)&=3 (\lE+1)^2 (\lE+4)+
2 (4 \lE+7) (2 \lE^2+5 \lE+6)\mn\nonumber\\
&\quad{}-2 (8 \lE^2+45 \lE+70)\mn^{2}+4 (\lE+10)\mn^{3},\nonumber\\
g_2(\mn)&=-3 (\lE+1) (\lE+2) (3 \lE^2+12 \lE+17)\nonumber\\
&\quad{}+(8 \lE^4+60 \lE^3+211 \lE^2+300 \lE+117)\mn\nonumber\\
&\quad{}+2 (36 \lE^3+155 \lE^2+292 \lE+279)\mn^{2}\nonumber\\
&\quad{}-2 (39 \lE^2+224 \lE+409)\mn^{3}\nonumber\\
&\quad{}+20 (\lE+11)\mn^{4}.
\end{align}
We note that the degrees of $g_1(\mn)$ and $g_2(\mn)$
in $\mn$ are at most $3$ and $4$ respectively.
Assume $\lE\neq 0$. By \eqref{eq:g1(mn)=3(lE+1)2(lE+4)+}, 
\begin{align}
0&=2 (10 +\lE)^{2}g_1(\mn)-(-240 + 176 \lE +  51 \lE^2 +\lE^3+ 10(11 + \lE)(10+\lE) \mn)
g_2(\mn)
\nonumber\\
&= -3 (\lE+1) (7 \lE^5+212 \lE^4+1837 \lE^3+6152 \lE^2+9064 \lE+5840)\nonumber\\
&\quad{}-6 (79 \lE^5+465 \lE^4+90 \lE^3-3352 \lE^2-7666 \lE-5060)\mn\nonumber\\
&\quad{}+6 (61 \lE^4+409 \lE^3+866 \lE^2-880 \lE-2400)\mn^{2},
\end{align}
which is a polynomial of degree at most $2$ in $\mn$. 
Repeating this procedure to decrease the degrees of polynomials in $\mn$,
we finally obtain
\begin{align}
	\label{eq:g1(lE)=-3(-1 + lE)}
0&=(-1 + \lE) (1 + \lE) (2 + \lE)^2 (3 + \lE) 
	(-20 + \lE + 33 \lE^2) \nonumber\\
	&\quad{}\times (433 + 235 \lE + 	67 \lE^2 + 33 \lE^3)\nonumber\\
	&\quad{}\times  (-2400 - 880 \lE + 	866 \lE^2 + 409 \lE^3 + 61 \lE^4)^2 \nonumber\\
&\quad{}\times 	(5020 + 14072 \lE + 18476 \lE^2 + 
	13940 \lE^3 + 6272 \lE^4 + 1580 \lE^5 + 	175 \lE^6).
	\end{align}
Since $\lE$ is an integer, it follows from
 \eqref{eq:ExBlE((1+lE)2-2mnomega1)-0}, 
\eqref{eq:ExBlE((1+lE)2-2mnomega1)-1}, 	and \eqref{eq:g1(lE)=-3(-1 + lE)}
that $(\lE,\mn)=(-1,0),(1,2),(-3,2)$, or $(-2,1/2)$, a contradiction.
 Thus, $\lE=0$.
If $\omega_1\lu=(1/16)\lu$ (resp. $(9/16)\lu$) and $H_{3}\lu=(-1/128)\lu$ (resp. $(15/128)\lu$), then 
the same argument as above shows the results.
\end{enumerate}
\end{proof}

\begin{lemma}
\label{lemma:structure-Vlattice-M1-norm2}
Assume $\mn=2$.
Let $\lu$ be a non-zero element of $\Omega_{M(1)^{+}}(\mW)$.
We write 
\begin{align}
\lE&=\epsilon(\ExB,\lu)
\end{align}
for simplicity.
Then, 
\begin{align}
	\label{eq:(H0ElE1lu=dfrac194+7lE}
	(H_{0}E)_{\lE+3}\lu&=
	\dfrac{-1}{9 (4+7 \lE)}E_{\lE}
	(36+72 \Har_{3}+142 \lE+123 \lE^2+18 \lE^3\nonumber\\
	&\qquad{}+\lE^4+156 \omega_{1}-40 \lE \omega_1+8 \lE^2\omega_1-48 \omega_1^2)\lu.
\end{align}
We set
\begin{align}
\lv&=(\omega_1-\frac{(\lE+1)^2}{4})\lu.
\end{align}
We have
\begin{align}
E_{\lE}\lv=(\omega_1-\frac{(\lE-1)^2}{4})E_{\lE}\lu.
\end{align}
\begin{enumerate}
\item
Assume $\Har_{3}\lu=0$.
If $\ExB_{\lE}\lv\neq 0$, then $\lE=0$ and
\begin{align}
\omega_1\ExB_{0}\lv&=\ExB_{0}\lv.
\end{align}
If $\lu$ is an eigenvector of $\omega_1$ and $\lv\neq 0$, then 
\begin{align}
\label{eq:omega1omega11lu=0}
\omega_1\lu&=\lu.
\end{align}
\item
Let $(\zeta,\xi)\in\{(1,1),(1/16,-1/128),(9/16,15/128)\}$.
If $\omega_1\lu=\zeta\lu$ and $\Har_{3}\lu=\xi\lu$, then
$\lE=0$. 
\end{enumerate}
\end{lemma}
\begin{proof}
	Taking the $(\lE+3)$-th action of \eqref{eq:6omega2E4omega0omega1E-1},
the $(\lE+4)$-th action of \eqref{eq:6omega2E4omega0omega1E-2}, 
the $(\lE+5)$-th action of \eqref{eq:6omega2E4omega0omega1E-3}, and 
the $(\lE+7)$-th action of \eqref{eq:6omega2E4omega0omega1E-4} on $\lu$,
we have 
\begin{align}
0&=\lE((1-\lE)^2-4\omega_1)E_{\lE}\lu
\label{eq:lE1lE24omega1ElElu-0}\\
&
=E_{\lE}\lE((1+\lE)^2-4\omega_1)\lu,
\label{eq:lE1lE24omega1ElElu-0-1}\\
0&=9(4 + 7 \lE)(H_{0}E)_{\lE+3}\lu+72 E_{\lE}\Har_{3}\lu\nonumber\\
&\quad{}+(36 + 298 \lE + 35 \lE^2 + 26 \lE^3 + \lE^4 \nonumber\\
&\qquad{}+ 156 \omega_{1} - 136 \lE \omega_{1} + 
 8 \lE^2 \omega_{1} - 48 \omega_{1}^2)E_{\lE}\lu
\label{eq:lE1lE24omega1ElElu-1}\\
&=\ExB_{\lE}\Big((t+1) ( \lE ^3+17  \lE ^2+106  \lE +36) +4 (2  \lE ^2-10  \lE +39)\omega_{1 } -48 \omega_{
1 }^2\Big)\lu\nonumber\\&\quad{}
+9 (7  \lE +4)(H_{0}E)_{ \lE +3 }\lu +72E_{ \lE  } H_{3 }\lu,
\label{eq:lE1lE24omega1ElElu-4}\\
0&=9 (-520 - 959 \lE + 33 \lE^2 - 100 \omega_{1})(H_{0}E)_{\lE+3}\lu+72  (-155 + 44 \lE) E_{\lE}\Har_{3}\lu\nonumber\\
&\quad{}-8 (585 + 4492 \lE + 1576 \lE^2 - 210 \lE^3 + 62 \lE^4 \nonumber\\
&\qquad{}+ 2 \lE^5 + 2685 \omega_{1} + 32 \lE \omega_{1} - 192 \lE^2 \omega_{1} + 16 \lE^3 \omega_{1} \nonumber\\
&\qquad{}- 480 \omega_{1}^2 -   96 \lE \omega_{1}^2)E_{\lE}\lu,
\label{eq:lE1lE24omega1ElElu-3}\\
&=E_{ \lE  }\Big(
-8 ( \lE +1) (2  \lE ^4+44  \lE ^3-158  \lE ^2+1222  \lE +585)\nonumber\\
&\qquad{}-8 (16  \lE ^3+992  \lE +2685) \omega_{1 }
+768 ( \lE +5) \omega_{1 }^2\Big)\lu\nonumber\\
&\quad{}+72 (44  \lE -155)E_{ \lE  } H_{3 }\lu \nonumber\\
&\quad{}+(H_{0}E)_{ \lE +3 }\big(9 (33  \lE ^2-859  \lE -520) -900\omega_{1 }\big)\lu,
\label{eq:lE1lE24omega1ElElu-5}
\end{align}
\begin{align}
0&=
-8 (52787700 t^7+1097587588 t^6+5494080415 t^5-89625113568 t^4+68909700044 t^3\nonumber\\
&\qquad{}+2468574039524 t^2+3786872840265 t+493804109430)\ExB_{t } \lu
\nonumber\\&\quad{}
+16 (52787700 t^5+336785348 t^4+16075086171 t^3-110729180408 t^2\nonumber\\
&\qquad{}-710794593411 t-1166720253525)\omega_{1 } \ExB_{t } \lu
\nonumber\\&\quad{}
+3 (5886227459 t^4+64230119866 t^3-465363710675 t^2-2778231175402 t\nonumber\\
&\qquad{}-1316641482180)(\Har_{0}\ExB)_{t+3 } \lu
\nonumber\\&\quad{}
+72 (743028209 t^3+17731219498 t^2+23020475889 t-140972980110)\ExB_{t } \Har_{3 } \lu
\nonumber\\&\quad{}
-60 (146187286 t^2+9549468148 t+19688188167)\omega_{1 } (\Har_{0}\ExB)_{t+3 } \lu
\nonumber\\&\quad{}
-28394226000\omega_{1 } \omega_{1 } (\Har_{0}\ExB)_{t+3 } \lu
\nonumber\\&\quad{}
+1536 (93467197 t^2+449390927 t+1282501650)\omega_{1 } \omega_{1 } \ExB_{t } \lu
\nonumber\\&\quad{}
+14515200 (931 t+661)\omega_{1 } \omega_{1 } \omega_{1 } \ExB_{t } \lu
\nonumber\\&\quad{}
+67132800 (931 t+661)\omega_{1 } \ExB_{t } \Har_{3 } \lu
\nonumber\\&\quad{}
+35590023000(\Har_{0}\ExB)_{t+3 } \Har_{3 } \lu \label{eq:lE1lE24omega1ElElu-6}\\
&=
-8 (t+1) (52787700 t^6+1150375288 t^5+5017275823 t^4-78748712473 t^3\nonumber\\
&\qquad{}-158883687883 t^2+959628223785 t+493804109430)\ExB_{t }\lu \nonumber\\&\quad{}
+16 (52787700 t^5+336785348 t^4+663193947 t^3-195213260792 t^2\nonumber\\
&\qquad{}-957034910211 t-1166720253525)\ExB_{t } \omega_{1 }\lu \nonumber\\&\quad{}
+3 (5886227459 t^4+67153865586 t^3-283839089715 t^2-2384467412062 t\nonumber\\
&\qquad{}-1316641482180)(\Har_{0}\ExB)_{t+3 }\lu \nonumber\\&\quad{}
+72 (743028209 t^3+16863155098 t^2+22404159489 t-140972980110)\ExB_{t } \Har_{3 }\lu \nonumber\\&\quad{}
-60 (146187286 t^2+8602993948 t+19688188167)(\Har_{0}\ExB)_{t+3 } \omega_{1 }\lu \nonumber\\&\quad{}
-28394226000(\Har_{0}\ExB)_{t+3 } \omega_{1 }^2\lu\nonumber\\&\quad{}
+1536 (67073347 t^2+430651577 t+1282501650)\ExB_{t } \omega_{1 }^2\lu  \nonumber\\&\quad{}
+14515200 (931 t+661)\ExB_{t } \omega_{1 }^3\lu \nonumber\\&\quad{}
+67132800 (931 t+661)\ExB_{t } \Har_{3 } \omega_{1 }\lu \nonumber\\&\quad{}
+35590023000(\Har_{0}\ExB)_{t+3 } \Har_{3 }\lu.\label{eq:lE1lE24omega1ElElu-7}
\end{align}
Note that $\omega_{1}$'s are on the left side of $\ExB_{\lE}$ in 
\eqref{eq:lE1lE24omega1ElElu-0},
\eqref{eq:lE1lE24omega1ElElu-1}, \eqref{eq:lE1lE24omega1ElElu-3}, and \eqref{eq:lE1lE24omega1ElElu-6}
but are on the right side of $\ExB_{\lE}$ in 
\eqref{eq:lE1lE24omega1ElElu-0-1},
\eqref{eq:lE1lE24omega1ElElu-4}, \eqref{eq:lE1lE24omega1ElElu-5}, and \eqref{eq:lE1lE24omega1ElElu-7}.
By Lemma \ref{lemma:bound-H0E} and \eqref{eq:lE1lE24omega1ElElu-1}, we have \eqref{eq:(H0ElE1lu=dfrac194+7lE}.
Deleting the terms including $(H_{0}E)_{\lE+3}\lu$
from the simultaneous equations \eqref{eq:lE1lE24omega1ElElu-1}, \eqref{eq:lE1lE24omega1ElElu-3}, and \eqref{eq:lE1lE24omega1ElElu-6},
we have 
\begin{align}
0&=360(-4 + 2 \lE + 11 \lE^2 + 4 \omega_{1})E_{\lE}\Har_{3}\lu\nonumber\\
&\quad{}+((1 - \lE)^2 - 4\omega_{1})
\big(
-\lE (-2596 - 5354 \lE + 745 \lE^2 + 29 \lE^3)\nonumber\\
&\qquad{}-4 (60 - 341 \lE + 82 \lE^2)\omega_{1}-240 \omega_{1}^2
\big)E_{\lE}\lu,\nonumber\\
0&=
-360 (-4 + 2 \lE + 11 \lE^2 + 4 \omega_{1})\ExB_{\lE}\Har_{3}\lu\nonumber\\
&\quad{}+((1 - \lE)^2 - 4\omega_{1}) \big(29 \lE^4+745 \lE^3+(328 \omega_{1}-5354) \lE^2\nonumber\\
&\qquad{}+(-1364 \omega_{1}-2596) \lE-240 \omega_{1}^2+240 \omega_{1}\big)\ExB_{\lE} \lu.
\label{eq:lE47lE1lE24omega1-2}
\end{align}
By \eqref{eq:lE1lE24omega1ElElu-0}, $\lE=0$ or $((1 - \lE)^2-4\omega_{1})E_{\lE}\lu=0$.
If $\lE=0$ and  $\Har_{3}\lu=0$, then by \eqref{eq:lE47lE1lE24omega1-2},
\begin{align}
0&=(\omega_1-1)(\omega_1-\dfrac{1}{4})E_{\lE}\lu,
\end{align}
which finishes (1).

By using \eqref{eq:lE1lE24omega1ElElu-0-1}, \eqref{eq:lE1lE24omega1ElElu-4}, \eqref{eq:lE1lE24omega1ElElu-5}, 
 and \eqref{eq:lE1lE24omega1ElElu-7},
the same argument as above shows (2).
\end{proof}

\begin{lemma}
	\label{lemma:structure-Vlattice-M1-norm1-2}
	Assume $\mn=1/2$.
	Let $U$ be an $A(M(1)^{+})$-submodule of $\Omega_{M(1)^{+}}(\mW)$,
	$\lu$ a simultaneous eigenvector of $\{\omega_1,\Har_{3}\}$ in $U$.
	We write 
	\begin{align}
		\lE&=\epsilon(\ExB,\lu)
	\end{align}
	for simplicity. Then,
	\begin{align}
		\label{eq:omega1lu=(1+lE)2lu}
		\omega_1\lu&=(1+\lE)^2\lu.
	\end{align}
	If	$\lE\neq -1$, then
	\begin{align}
		\label{eq:(Har0E)lE+3lu=ExB-0}
		(\Har_{1}E)_{\lE+2}\lu&=\ExB_{\lE}\big(\frac{3}{(1 + \lE)
			(3 + 2\lE)}\Har_{3} + (1 + \lE)\big)\lu,
	\end{align} 
	and if $\lE=-1$, then
	\begin{align}
		\label{eq:(Har0E)lE+3lu=ExB-1}
		\Har_{3}\lu&=0.
	\end{align} 
\end{lemma}
\begin{proof}
	Taking the $(\lE+2)$-th action of \eqref{eq:norm1-2-0} on $\lu$, we have
	\begin{align}
		\omega_{1}\ExB_{\lE}\lu&=(\lE+\frac{1}{2})^2\ExB_{\lE}\lu
	\end{align}
	and hence \eqref{eq:omega1lu=(1+lE)2lu} holds.
	Taking the $(\lE+3)$-th action of \eqref{eq:norm1-2-1}
	and the $(\lE+4)$-the action of \eqref{eq:norm1-2-2} on $\lu$, we have
	\begin{align}
		0&=
		8(\lE+1)E_{\lE}\omega_{1 }\lu -(\lE+1)(11\lE+15)(\Har_{1}E)_{\lE+2}\lu \nonumber\\&\quad{}
		+4(\lE+1)^2E_{\lE}\lu +12E_{\lE}\Har_{3 } \lu 
		+3(\Har_{1}E)_{\lE+2}\omega_{1 }\lu,
	\end{align}
	and hence \eqref{eq:(Har0E)lE+3lu=ExB-0}
	and \eqref{eq:(Har0E)lE+3lu=ExB-1} by \eqref{eq:omega1lu=(1+lE)2lu}.
\end{proof}

\section{Modules for the Zhu algebra of $M(1)^{+}$ in a weak $V_{\lattice}^{+}$-module: the general case}
\label{section:Modules for the Zhu algera of general}
Let $L$ be a non-degenerate even lattice.
In this section, we shall show that there exists an irreducible $A(M(1)^{+})$-module in an arbitrary non-zero 
weak $V_{\lattice}^{+}$-module.

Throughout this section $M$ is a weak $V_{\lattice}^{+}$-module.
Let $h^{[1]},\ldots,h^{[\rankL]}$ be an orthonormal basis of $\fh$.
For $i=1,\ldots, \rankL$,
we define
\begin{align}
\label{eq:def-oega-i-H-i}
\omega^{[i]}&=\dfrac{1}{2}h^{[i]}(-1)^2,\nonumber\\
\omega&=\omega^{[1]}+\cdots+\omega^{[\rankL]},\nonumber\\
J^{[i]}&=h^{[i]}(-1)^4\vac-2h^{[i]}(-3)h^{[i]}(-1)\vac+\dfrac{3}{2}h^{[i]}(-2)^2\vac,\nonumber\\
\Har^{[i]}&=\dfrac{1}{3}h^{[i]}(-3)h^{[i]}(-1)\vac-\dfrac{1}{3}h^{[i]}(-2)^2\vac.
\end{align}
We recall the following notation and some results from \textcolor{black}{\cite[Sections 4 and 5]{DN2001}}:
for any pair of distinct elements $i,j\in\{1,\ldots,\rankL\}$ and $l,m\in\Z_{>0}$,
\begin{align}
\nS_{ij}(l,m)&=h^{[i]}(-l)h^{[j]}(-m),\nonumber\\
E^{u}_{ij}&=5\nS_{ij}(1,2)+25\nS_{ij}(1,3)+36\nS_{ij}(1,4)+16\nS_{ij}(1,5),\nonumber\\
E^{t}_{ij}&=-16\nS_{ij}(1,2)+145\nS_{ij}(1,3)+19\nS_{ij}(1,4)+8\nS_{ij}(1,5),\nonumber\\
\Lambda_{ij}
&=45\nS_{ij}(1,2)+190\nS_{ij}(1,3)+240\nS_{ij}(1,4)+96\nS_{ij}(1,5).
\end{align}
It follows from \cite[Proposition 5.3.14]{DN2001} that in $A(M(1)^{+})$, $A^{u}=\oplus_{i,j}\C E^{u}_{ij}$ and 
$A^{t}=\oplus_{i,j}\C E^{t}_{ij}$ are two-sided ideals, each of  which is isomorphic to the $\rankL\times \rankL$ matrix algebra
and $A^{u}A^{t}=A^{t}A^{u}=0$. 
By \cite[Proposition 5.3.15]{DN2001},
$A(M(1)^{+})/(A^{u}+A^{t})$ is a commutative algebra generated by
the images of $\omega^{[i]},\Har^{[i]}$ and $\Lambda_{jk}$ where
$i=1,\ldots,\rankL$ and $j,k\in \{1,\ldots,\rankL\}$ with $j\neq k$.

For $\lambda\in\fh$, $i,j\in\{1,\ldots,\rankL\}$ with $i\neq j$,
\begin{align}
\label{eq:norm2Si1(1,1)1lu=-Si1(1,2)2lu}
\omega^{[i]}_{1}e^{\lambda}
&=\frac{\langle\lambda,h^{[i]}\rangle^2}{2}e^{\lambda},\nonumber\\
\Har^{[i]}_{3}e^{\lambda}&=0,\nonumber\\
S_{ij}(1,1)_{1}e^{\lambda}&=-S_{ij}(1,2)_{2}e^{\lambda}
=S_{ij}(1,3)_{3}=\langle\lambda,h^{[i]}\rangle\langle\lambda,h^{[j]}\rangle e^{\lambda}
\end{align}
\begin{align}
\label{eq:norm2Si1(1,1)1lu=-Si1(1,2)2lu-2}
\omega^{[i]}_{1}h^{[j]}(-1)\vac
&=\delta_{ij}h^{[j]}(-1)\vac,\nonumber\\
\Har^{[i]}_{3}h^{[j]}(-1)\vac&=\delta_{ij}h^{[j]}(-1)\vac,\nonumber\\
S_{ij}(1,1)_{1}h^{[j]}(-1)\vac&=h^{[i]}(-1)\vac,\nonumber\\
S_{ij}(1,2)_{2}h^{[j]}(-1)\vac&=-2h^{[i]}(-1)\vac,\nonumber\\
S_{ij}(1,3)_{3}h^{[j]}(-1)\vac&=3h^{[i]}(-1)\vac,
\end{align}
and 
\begin{align}
\label{eq:norm2Si1(1,1)1lu=-Si1(1,2)2lu-twist-0}
\omega^{[i]}_{1}\vac_{\tw}&=\delta_{ij}\frac{1}{16} \vac_{\tw},\nonumber\\
\Har^{[i]}_{3} \vac_{\tw}&=\delta_{ij}\frac{-1}{128} \vac_{\tw},\nonumber\\
S_{ij}(1,1)_{1} \vac_{\tw}&=
S_{ij}(1,2)_{2} \vac_{\tw}=
S_{ij}(1,3)_{3} \vac_{\tw}=0,
\end{align}
\begin{align}
\label{eq:norm2Si1(1,1)1lu=-Si1(1,2)2lu-twist-1}
\omega^{[i]}_{1}h^{[j]}(-\frac{1}{2})\vac_{\tw}&=\delta_{ij}\frac{9}{16}h^{[j]}(-\frac{1}{2})\vac_{\tw},\nonumber\\
\Har^{[i]}_{3}h^{[j]}(-\frac{1}{2})\vac_{\tw}&=\delta_{ij}\frac{15}{128}h^{[j]}(-\frac{1}{2})\vac_{\tw},\nonumber\\
S_{ij}(1,1)_{1}h^{[j]}(-\frac{1}{2})\vac_{\tw}&=\frac{1}{2}h^{[i]}(-\frac{1}{2})\vac_{\tw},\nonumber\\
S_{ij}(1,2)_{2}h^{[j]}(-\frac{1}{2})\vac_{\tw}&=\frac{-3}{4}h^{[i]}(-\frac{1}{2})\vac_{\tw},\nonumber\\
S_{ij}(1,3)_{3}h^{[j]}(-\frac{1}{2})\vac_{\tw}&=\frac{15}{16}h^{[i]}(-\frac{1}{2})\vac_{\tw}.
\end{align}
A direct computation shows that
\begin{align}
\label{eq:s14-15}
\nS_{ij}(1,4)&=\dfrac{-1}{3}\omega^{[j]}_{-2}\nS_{ij}(1,1)
+\dfrac{2}{3}\omega^{[j]}_{-1}\nS_{ij}(1,2),\nonumber\\
\nS_{ij}(1,5)&=
\dfrac{4}{93}(\omega^{[j]}_{-1})^2\nS_{ij}(1,1)
+\dfrac{-1}{93}J^{[j]}_{-1}\nS_{ij}(1,1)\nonumber\\
&\quad{}+\dfrac{9}{62}\omega^{[j]}_{-2}\nS_{ij}(1,2)
+\dfrac{4}{93}\omega^{[j]}_{-1}\nS_{ij}(1,3)\nonumber\\
&=
\frac{1}{7}\Har^{[j]}_{-1 } S_{ij}(1,1)
+\frac{1}{21}\omega^{[j]}_{-2 } S_{ij}(2,1)
+\frac{-2}{21}\omega^{[j]}_{-1 } S_{ij}(3,1),
\end{align}

\begin{align}
	\label{eq:[h[j](l),Sij(1,1)m]}
	[h^{[j]}(l),S_{ij}(1,1)_{m}]&=h^{[i]}(l+m-1),\nonumber\\
	[h^{[j]}(l),S_{ij}(1,2)_{m}]&=2h^{[i]}(l+m-2),\nonumber\\
	[h^{[j]}(l),S_{ij}(1,3)_{m}]&=3h^{[i]}(l+m-3),\nonumber\\
	[h^{[i]}(l),S_{ij}(1,2)_{m}]&=(-l-m+1)h^{[j]}(l+m-2),\nonumber\\
	[h^{[i]}(l),S_{ij}(1,3)_{m}]&=\binom{-l-m+1}{2}h^{[j]}(l+m-3),
\end{align}
\begin{align}
S_{ij}(2,1)&=
\omega_{0 } S_{ij}(1,1)
-S_{ij}(1,2),
\nonumber\\
S_{ij}(3,1)
&=\frac{1}{2}\omega_{0 }^{2}S_{ij}(1,1)
-\omega_{0 } S_{ij}(1,2)
+S_{ij}(1,3),\nonumber\\
S_{ij}(2,2)
&=
\omega_{0 } S_{ij}(1,2)
-2S_{ij}(1,3),\nonumber\\
S_{ij}(3,2)
&=-\omega^{[j]}_{-2 } S_{ij}(1,1)
+2\omega^{[j]}_{-1 } S_{ij}(1,2)
\nonumber\\&\quad{}
+\frac{1}{2}\omega_{0 }^{2}S_{ij}(1,2)
-2\omega_{0 } S_{ij}(1,3),\nonumber\\
S_{ij}(3,3)
&=
\frac{-1}{2}\omega_{0 } \omega^{[j]}_{-2 } S_{ij}(1,1)
+\frac{3}{2}\omega^{[i]}_{-2 } S_{ij}(1,2)
\nonumber\\&\quad{}
-\omega_{0 } \omega^{[i]}_{-1 } S_{ij}(1,2)
+\omega_{0 } \omega^{[j]}_{-1 } S_{ij}(1,2)
\nonumber\\&\quad{}
+\frac{1}{4}\omega_{0 }^{3}S_{ij}(1,2)
+2\omega^{[i]}_{-1 } S_{ij}(1,3)
\nonumber\\&\quad{}
-\omega_{0 }^{2}S_{ij}(1,3).
\end{align}

For distinct $i,j,k\in\{1,\ldots,\rankL\}$ and $l,m\in\Z$,
a direct computation shows that
\begin{align}
[\omega^{[j]}_{l},\nS_{ij}(1,1)_{m}]&=\nS_{ij}(1,2)_{l+m}+l\nS_{ij}(1,1)_{l+m-1},\nonumber\\
[\nS_{ij}(1,1)_{l},\nS_{ij}(1,1)_{m}]&=(l-m)(\omega^{[i]}+\omega^{[j]})_{l+m-1}+\binom{l}{3}\delta_{l+m-3,-1},\nonumber\\
[S_{kj}(1,1)_{l},S_{ij}(1,1)_{m}]&=S_{ik}(1,2)_{l+m}+lS_{ik}(1,1)_{l+m-1}.
\label{eq:omega-s11}
\end{align}
Let $\lu$ be a non-zero element of $\module$ such that $\epsilon(\omega^{[i]},\lu)\leq 1$ and $\epsilon(\Har^{[i]},\lu)\leq 3$ for all $i=1,\ldots,d$.
Let $i,j$ be a pair of distinct elements of $\{1,\ldots,\rankL\}$.
We write $\epsilon_{S}=\epsilon(S_{\wi\wj}(1,1),\lu)$ for simplicity. A direct computation shows that
\begin{align}
	\label{eq:action[H4-S]}
	&[\Har^{[\wj]}_3,S_{\wi\wj}(1,1)_{\epsilon_{S}}]\lu\nonumber\\
	&=\Big(S_{\wi\wj}(1,1)_{\epsilon_{S}}(1+4\omega^{[\wj]}_{1})+S_{\wi\wj}(1,2)_{\epsilon_{S}+1}(5+4\omega^{[\wj]}_{1})+4S_{\wi\wj}(1,3)_{\epsilon_{S}+2}\Big)\lu,\nonumber\\
	&[\Har^{[\wj]}_3,S_{\wi\wj}(1,2)_{\epsilon_{S}+1}]\lu\nonumber\\
	&=
	\Big(S_{ij}(1,1)_{\epsilon_{S} }(\frac{-4}{7}\omega^{[\wj]}_{1 } +\frac{24}{7}\Har^{[\wj]}_{3 })
	+S_{ij}(1,2)_{\epsilon_{S}+1 } (\dfrac{2}{7}+\frac{-20}{7}\omega^{[\wj]}_{1 })\nonumber\\
	&\quad{}+S_{ij}(1,3)_{\epsilon_{S}+2 }(\dfrac{2}{7}+\frac{-16}{7}\omega^{[\wj]}_{1 })\Big)\lu,\nonumber\\
	&[\Har^{[\wj]}_3,S_{\wi\wj}(1,3)_{\epsilon_{S}+2}]\lu\nonumber\\
	&=
	\Big(S_{ij}(1,1)_{\epsilon_{S} }(\frac{-598}{203}\omega^{[\wj]}_{1 } +\frac{144}{203}\Har^{[\wj]}_{3 }) \nonumber\\
	&\quad{}+S_{ij}(1,2)_{\epsilon_{S}+1 } (\frac{37}{203}+\frac{1534}{203}\omega^{[\wj]}_{1 } +\frac{60}{29}\Har^{[\wj]}_{3 } )\nonumber\\
	&\quad{}+S_{ij}(1,3)_{\epsilon_{S}+2 } (\frac{37}{203}+\frac{936}{203}\omega^{[\wj]}_{1 } )\Big)\lu.
\end{align}

\begin{lemma}
For a non-degenerate even lattice $\lattice$ of rank $\rankL$,
there exists a sequence 
	of elements $\beta_1,\ldots,\beta_{\rankL}\in \lattice$ such that
	$\langle\beta_i,\beta_i\rangle\neq 0$ and 
	$\langle\beta_j,\beta_k\rangle=0$ for all $i\in\{1,\ldots,\rankL\}$
	and $j,k\in\{1,\ldots,\rankL\}$ with $j\neq k$. 
\end{lemma}
\begin{proof}
	Let $\gamma_1,\ldots,\gamma_d$ be a basis of $\Q\otimes_{\Z}\lattice$
	such that $\langle\gamma_i,\gamma_i\rangle\neq 0$ and 
	$\langle\gamma_j,\gamma_k\rangle=0$ for all $i\in\{1,\ldots,\rankL\}$
	and $j,k\in\{1,\ldots,\rankL\}$ with $j\neq k$. . 
	Since $\gamma_1,\ldots,\gamma_d\in \Q\otimes_{\Z}\lattice$,
	there exists a non-zero integer $m_i$ such that 
	$m\gamma_i\in \lattice$ for all $i=1,\ldots,d$.
	Then, the elements $\beta_i=m\gamma_i\ (i=1,\ldots,d)$ satisfy the condition.
\end{proof}

Let $\Lambda=\oplus_{i=1}^{d}\Z\beta_i$ be a sublattice of a non-degenerate even lattice $\lattice$ of rank $\rankL$
such that $\langle\beta_i,\beta_i\rangle\neq 0$ and 
$\langle\beta_i,\beta_j\rangle=0$ for any distinct pair of elements $i,j\in\{1,\ldots,\rankL\}$. 
We have
\begin{align}
	V_{\Z\beta_1}^{+}\otimes \cdots\otimes V_{\Z\beta_d}^{+} \subset V_{\lattice}^{+}
\end{align}
and take an orthonormal basis $h^{[1]},\ldots,h^{[\rankL]}$ of $\fh$ defined by
\begin{align}
\label{eq:h[i]=frac1sqrtlanglebeta[i]}
h^{[i]}&=\frac{1}{\sqrt{\langle\beta^{[i]},\beta^{[i]}\rangle}}\beta^{[i]}\quad (i=1,\ldots,\rankL).
\end{align}
Since $[h^{[i]}_{l},h^{[j]}_{m}]=0$ for any distinct pair of elements $i,j\in\{1,\ldots,\rankL\}$ and $l,m\in\Z$, 
it follows Lemma \ref{lemma:r=1-s=3} that there exists a simultaneous eigenvector $\lu$ of
 $\{\omega^{[i]}_{1},\Har^{[i]}_{3}\}_{i=1}^{\rankL}$ in a weak $V_{\lattice}^{+}$-module $\module$
such that $\epsilon(\omega^{[i]},\lu)\leq 1$ and
$\epsilon(\Har^{[i]},\lu)\leq 3$ for all $i=1,\ldots,\rankL$.

\begin{lemma}
	\label{lemma:bound-index-S}
	Let $U$ be a subspace of a weak $M(1)^{+}$-module.
	\begin{enumerate}
		\item
		Let $i,j\in\{1,\ldots,\rankL\}$ with $i\neq j$ and
		$\wak\in\Z$ such that
		$\wak\geq\epsilon(S_{ij}(1,1),\lu)$ for all non-zero $\lu\in U$.
		If $U$ is stable under the action of $\omega_{1}^{[j]}$, then
		\begin{align}
			\label{eq:sab-bound}
			\epsilon(S_{ij}(1,m+1),\lu)\leq \wak+m
		\end{align}
		for all $m\in\Z_{\geq 0}$.
		\item
		Assume $\alpha\in\C h^{[1]}$.
		Let $i\in\{2,\ldots,\rankL\}$ and
		$\lE\in\Z$ such that
		$\lE\geq\epsilon(\ExB(\alpha),\lu)$ for all non-zero $\lu\in U$.
		If $U$ is stable under the action of $S_{i1}(1,1)_{1}$, then
		\begin{align}
			\label{eq:sab-bound-1}
			\epsilon(S_{i1}(1,1)_{0}\ExB(\alpha),\lu)\leq \lE+1.
		\end{align}
		
	\end{enumerate}
	
\end{lemma}
\begin{proof}
	\begin{enumerate}
		\item
		For $n\in\Z$ and $m\in\Z_{>0}$, since 
		\begin{align}
			[\omega^{[j]}_1,S_{ij}(1,m)_{n}]&=(\omega^{[j]}_0S_{ij}(1,m))_{n+1}+(\omega^{[j]}_1S_{ij}(1,m))_{n} \mbox{ and }\nonumber\\
			\omega^{[j]}_1S_{ij}(1,m)&=mS_{ij}(1,m),
		\end{align}
		we have
		\begin{align}
			\label{eq:s12s13-bound}
			S_{ij}(1,m+1)_{n+1}&=\dfrac{1}{m}(\omega^{[j]}_0S_{ij}(1,m))_{n+1}\nonumber\\
			&=\dfrac{1}{m}[\omega^{[j]}_1,S_{ij}(1,m)_{n}]-S_{ij}(1,m)_{n},
		\end{align}
		which implies \eqref{eq:sab-bound}.
		\item
		Since $S_{i1}(1,1)_{m}\ExB(\alpha)=0$ for all $m\in\Z_{>0}$,
		the same argument as above shows the result.
	\end{enumerate}
\end{proof}

\begin{lemma}
\label{lemma:Zhu-Omega}
There exists an irreducible $A(M(1)^{+})$-submodule 
of
 $\Omega_{M(1)^{+}}(M)$.
\end{lemma}
\begin{proof}
Let $h^{[1]},\ldots,h^{[\rankL]}$ be the orthonormal basis of $\fh$
defined by \eqref{eq:h[i]=frac1sqrtlanglebeta[i]}.
We can take a simultaneous eigenvector $\lu$ of 
	$\{\omega^{[i]}_{1},\Har^{[i]}_{3}\}_{i=1}^{\rankL}$
	such that $\epsilon(\omega^{[i]},\lu)\leq 1$ and
	$\epsilon(\Har^{[i]},\lu)\leq 3$ for all $i=1,\ldots,\rankL$.
		We take a pair of distinct $i,j$ so that $\epsilon(S_{ij}(1,1),\lu)\geq \epsilon(S_{lm}(1,1),\lu)$
	for any pair of distinct elements $l,m\in\{1,\ldots,\rankL\}$.
We may assume $(i,j)=(2,1)$.  
We consider a non-zero subspace $W=\sum_{m=1}^{3}\C S_{21}(1,m)_{\epsilon_{S}+m-1}\lu$ of $\module$.
By Lemma \ref{lemma:bound-index-S}, \eqref{eq:omega-s11} and computations in
Section \ref{section:appendix-The general case},
for all non-zero $\lw\in \mW$, $i=1,\ldots,\rankL$, 
we have $\epsilon(\omega^{[i]},\lw)\leq 1$ and $\epsilon(\Har^{[i]},\lw)\leq 3$.
		Since $\lu$ is a simultaneous eigenvector of $\{\omega^{[i]}_{1},\Har^{[i]}_{3}\}_{i=1}^{\rankL}$, it follows from \eqref{eq:omega-s11}, \eqref{eq:action[H4-S]}, and computations in
Section \ref{section:appendix-The general case} that 
		$\mW$ is invariant under the actions of $\{\omega^{[i]}_{1},\Har^{[i]}_{3}\}_{i=1}^{\rankL}$.
	
	A direct computation shows that
	\begin{align}
		\label{eqn:sm1s-1}
		\nS_{21}(1,1)_{-1}\nS_{21}(1,1)
		&=\dfrac{2}{3}\omega^{[1]}_{-3 } \vac+\dfrac{2}{3}\omega^{[2]}_{-3 } \vac
		+\Har^{[1]} +\Har^{[2]} 
		+4\omega^{[1]}_{-1 } \omega^{[2]},\\
		\label{eqn:sm1s-2}
\nS_{21}(1,1)_{-1}\nS_{21}(1,2)
		&=2\omega^{[2]}_{-1 } \omega^{[1]}_{-2 } \vac
		+\omega_{0 } \Har^{[2]}_{-1 } \vac
		+\frac{-1}{2}\omega_{0 } \Har^{[1]}_{-1 } \vac
		\nonumber\\&\quad{}
		+\frac{1}{6}\omega_{0 }^3\omega^{[2]}_{-1 } \vac
		+\frac{1}{12}\omega_{0 }^3\omega^{[1]}_{-1 } \vac,\\
\label{eqn:sm1s-3}
\nS_{21}(1,1)_{-1}\nS_{21}(1,3)
&=\frac{2}{5}\omega^{[2]}_{-1 } \Har^{[2]}_{-1 } \vac
		+\frac{1}{6}\omega^{[2]}_{-2 } \omega^{[2]}_{-2 } \vac
		+\frac{4}{3}\omega^{[2]}_{-1 } \omega^{[1]}_{-3 } \vac
		\nonumber\\&\quad{}
		+\frac{1}{3}\omega^{[1]}_{-2 } \omega^{[1]}_{-2 } \vac
		+2\omega^{[2]}_{-1 } \Har^{[1]}_{-1 } \vac
		+\frac{4}{5}\omega^{[1]}_{-1 } \Har^{[1]}_{-1 } \vac
		\nonumber\\&\quad{}
		+\frac{7}{20}\omega_{0 }^{2}\Har^{[2]}_{-1 } \vac
		+\frac{-1}{30}\omega_{0 }^{2}\omega^{[2]}_{-1 } \omega^{[2]}_{-1 } \vac
		+\frac{-1}{15}\omega_{0 }^{2}\omega^{[1]}_{-1 } \omega^{[1]}_{-1 } \vac
		\nonumber\\&\quad{}
		+\frac{-3}{10}\omega_{0 }^{2}\Har^{[1]}_{-1 } \vac
		+\frac{19}{360}\omega_{0 }^{4}\omega^{[2]}_{-1 } \vac
		+\frac{1}{45}\omega_{0 }^{4}\omega^{[1]}_{-1 } \vac
\end{align}
We write $\epsilon_{S}=\epsilon(\nS_{21}(1,1),\lu)$ for simplicity and suppose \textcolor{black}{$\epsilon_{S}\geq 1$}.
	By \eqref{eq:abm} with $m=\epsilon_{S}-1$ and \eqref{eq:Sij(1,1)0Sij(1,1)-1}, for $n\in\Z$ we have
	\begin{align}
		\label{eq:s12s12-1}
	&(\nS_{21}(1,1)_{-1}\nS_{21}(1,1))_{n}\lu\nonumber\\
&=S_{21}(1,1)_{n-\epsilon_{S}-1}S_{21}(1,1)_{\epsilon_{S}}\lu+\epsilon_{S}(n-\epsilon_{S})(\omega^{[1]}+\omega^{[2]})_{n-2}\lu\nonumber\\
&\quad{}-\binom{\epsilon_{S}}{4}\delta_{n-4,-1}\lu.
\end{align}
For $n\in\Z$,
since 
\begin{align}
		\label{eq:s12s12-2}
&\Big(\dfrac{2}{3}\omega^{[1]}_{-3 } \vac+\dfrac{2}{3}\omega^{[2]}_{-3 } \vac
	+\Har^{[1]}+\Har^{[2]}+4\omega^{[1]}_{-1 } \omega^{[2]} \Big)_{n}\lu\nonumber\\
&=\dfrac{2}{3}\binom{n}{2}(\omega^{[1]}+\omega^{[2]})_{n-2}\lu+(\Har^{[1]}+\Har^{[2]})_{n}\lu\nonumber\\
&\quad{}+4\sum_{\epsilon_{S}\leq i}
\omega^{[1]}_{n-1-i}\omega^{[2]}_{i}\lu+4\sum_{i<\epsilon_{S}}
\omega^{[2]}_{i}\omega^{[1]}_{n-1-i}\lu
	\end{align}
 and \eqref{eqn:sm1s-1}, we have
	\begin{align}
		\label{eq:s12s12-3}
&S_{21}(1,1)_{n-\epsilon_{S}-1}S_{21}(1,1)_{\epsilon_{S}}
\lu\nonumber\\
&=(\dfrac{2}{3}\binom{n}{2}-\epsilon_{S}(n-\epsilon_{S}))(\omega^{[1]}+\omega^{[2]})_{n-2}\lu+(\Har^{[1]}+\Har^{[2]})_{n}\lu\nonumber\\
&\quad{}+4\sum_{\epsilon_{S}\leq i}
\omega^{[1]}_{n-1-i}\omega^{[2]}_{i}\lu+4\sum_{i<\epsilon_{S}}\omega^{[2]}_{i}\omega^{[1]}_{n-1-i}\lu+\binom{\epsilon_{S}}{4}\delta_{n-4,-1}\lu.
	\end{align}
Assume $\epsilon_{S}\geq 2$. It follows from \eqref{eq:s12s12-3} that
$S_{21}(1,1)_{i}S_{21}(1,1)_{\epsilon_{S}}\lu =0$ for $i\geq \epsilon_{S}$.
By using \eqref{eqn:sm1s-2}, \eqref{eqn:sm1s-3}, \eqref{eq:Sij(1,1)0Sij(1,1)-2}, and \eqref{eq:Sij(1,1)0Sij(1,1)-3},
the same argument as above shows that 
\begin{align}
	S_{21}(1,1)_{i}S_{21}(1,2)_{\epsilon_{S}+1}\lu&=S_{21}(1,1)_{i}S_{21}(1,3)_{\epsilon_{S}+2}\lu=0
\end{align}
for $i\geq \epsilon_{S}$.
Thus we can take a simultaneous eigenvector $\lv$ of 
$\{\omega^{[i]}_{1},\Har^{[i]}_{3}\}_{i=1}^{\rankL}$ in $\mW$
such that $\epsilon(S_{21}(1,1),\lv)<\epsilon(S_{21}(1,1),\lu)$.
By \eqref{eq:omega-s11} we have
 $\epsilon(S_{jk}(1,1),\lv)\leq \epsilon(S_{jk}(1,1),\lu)$ for 
any pair of distinct elements $j,k\in\{1,\ldots,\rankL\}$ with $j>k$ and $(j,k)\neq (2,1)$.
It follows from 	Lemma \ref{lemma:bound-index-S}
that $\epsilon(S_{ij}(1,m),\lv)\leq \epsilon(S_{ij}(1,1),\lv)+m-1$ for $m\in\Z_{>0}$
and any pair of distinct element $i,j\in\{1,\ldots,\rankL\}$.
Replacing $\lu$ by this $\lv$ repeatedly, we get a non-zero element $\lu\in\Omega_{M(1)^{+}}(M)$.
It follows from Lemma \ref{lemma:bound-index-S}, \eqref{eq:omega-s11}, \eqref{eq:s12s12-3}, 
 and computations in Section \ref{section:appendix-The general case}
that $S_{ij}(1,m)_{m}\lu\in\Omega_{M(1)^{+}}(M)$
for any pair of distinct elements $i,j\in\{1,\ldots,\rankL\}$ and $m=1,2,3$.

	If $A^{u}\lu\neq 0$ (resp. $A^{t}\lu\neq 0$), then \textcolor{black}{it follows from} \cite[Section 6.2]{DN2001} that 
	$A^{u}\lu\cong M(1)^{-}(0)$ (resp. $A^{t}\lu\cong M(\theta)^{-}(0)$)
	as $A(M(1)^{+})$-modules.
	Suppose $(A^{u}+A^{t})\lu=0$. Then, since
	$\lu$ is a simultaneous eigenvector of $\{\omega^{[i]}_{1},\Har^{[i]}_{3}\}_{i=1}^{\rankL}$,
	it follows from 
	\cite[Proposition 5.3.13, Proposition 5.3.15, (6.1.15), and (6.1.16)]{DN2001} that
	$A(M(1)^{+})\lu$ is finite dimensional. This completes the proof.
\end{proof}

\section{Extension groups for $M(1)^{+}$}
\label{section:Extension groups for M(1)+}
In this section we study some weak modules for $M(1)^{+}$ with rank $\rankL$. 
Results in this section will be used in Section \ref{section:Intertwining operators for M(1)+}.
When $\rankL=1$, some results in this section have 
been obtained in \cite[Section 5]{Abe2005}.
In some parts of the following argument we shall use techniques in \cite[Section 5]{Abe2005}.
Throughout this section, $\module$ is an $M(1)^{+}$-modules,
$W$ is an irreducible $M(1)^{+}$-modules, 
and $N$ is a weak $M(1)^{+}$-module.
In this section,  we consider the following exact sequence
\begin{align}
	0\rightarrow W\overset{}{\rightarrow} N\overset{\pi}{\rightarrow} M\rightarrow 0
\label{eq:exact-seq}
	\end{align}
of weak $M(1)^{+}$-modules.
We shall use  the symbols in \eqref{eq:def-oega-i-H-i}.
Let $B$ be an irreducible $A(M(1)^{+})$-submodule of $\module(0)$.
For $\zeta=(\zeta^{[1]},\ldots,\zeta^{[\rankL]}), \xi=(\xi^{[1]},\ldots,\xi^{[\rankL]})\in \C^{\rankL}$, 
let $\lv\in B$ such that 
\begin{align}
	\label{eq:eigenvalue-omega-H}
(\omega^{[i]}_{1}-\zeta^{[i]})\lv&=(\Har^{[i]}_{3}\lv-\xi^{[i]})\lv=0
\end{align}
for $i=1,\ldots,\rankL$.
Since the actions of $\omega^{[i]}_{1}$ and $\Har^{[i]}_{3}$ on $W$ are semisimple, 
we can take $\lu\in N$ so that 
\begin{align}
\label{eq:pi(lu)lvmboxand} 
\pi(\lu)=\lv\mbox{ and }
(\omega^{[i]}_{1}-\zeta^{[i]})^2\lu=(\Har^{[i]}_{3}-\xi^{[i]})^2\lu=0
\end{align}
for all $i=1,\ldots, \rankL$.
If $(W,B) \not\cong (M(1)^{+},M(1)^{-}(0))$, then it follows from \cite[Lemma 4.8]{Abe2005} that 
\begin{align}
\label{eq:epsilon(omegailu)leq1}
\epsilon(\omega^{[i]},\lu)\leq 1\mbox{ and }\epsilon(\Har^{[i]},\lu)\leq 3 
\end{align}
for all $i=1,\ldots, \rankL$.
We note that
$[\omega^{[i]}_{1},\Har^{[j]}_{3}]=[\Har^{[i]}_{3},\Har^{[j]}_{3}]=0$ for all $i,j=1,\ldots, \rankL$ .
For $\zeta=(\zeta^{[1]},\ldots,\zeta^{[\rankL]}), \xi=(\xi^{[1]},\ldots,\xi^{[\rankL]})\in \C^{\rankL}$, 
we define 
\begin{align}
W_{\zeta,\xi}=\bigcap_{j=1}^{\rankL}\Ker (\omega^{[j]}_{3}-\zeta^{[j]})
\cap
\bigcap_{j=1}^{\rankL}\Ker (\Har^{[j]}_{3}-\xi^{[j]})\cap W.
\end{align}
Since
$(\omega^{[i]}_{1}-\zeta^{[i]})\lu$ and $(\Har^{[i]}_{3}-\xi^{[i]})\lu$ are elements of $W$
and the actions of $\omega^{[i]}_{1}$ and $\Har^{[i]}_{3}$ on $W$  are semisimple for all $i=1,\ldots, \rankL$,
we have 
\begin{align}
\label{eq:keromega-kerH}
(\omega^{[i]}_{1}-\zeta^{[i]})\lu, (\Har^{[i]}_{3}-\xi^{[i]})\lu\in W_{\zeta,\xi}
\end{align}
for all $i=1,\ldots,\rankL$.

For a pair of distinct elements $i,j\in\{1,\ldots,\rankL\}$, a direct computation shows that
\begin{align}
0&=
6\omega^{[i]}_{-2 } S_{ij}(1,1)
+2\omega^{[j]}_{-2 } S_{ij}(1,1)
\nonumber\\&\quad{}
-4\omega_{0 } \omega^{[i]}_{-1 } S_{ij}(1,1)
+\omega_{0 } \omega_{0 } \omega_{0 } S_{ij}(1,1)
\nonumber\\&\quad{}
+4\omega^{[i]}_{-1 } S_{ij}(1,2)
-4\omega^{[j]}_{-1 } S_{ij}(1,2)
\nonumber\\&\quad{}
-3\omega_{0 } \omega_{0 } S_{ij}(1,2)
+6\omega_{0 } S_{ij}(1,3),\label{eq:s11-3}\\
0&=
32\omega^{[i]}_{-3 } S_{ij}(1,1)
-24\Har^{[i]}_{-1 } S_{ij}(1,1)
\nonumber\\&\quad{}
-8\omega^{[j]}_{-3 } S_{ij}(1,1)
+24\Har^{[j]}_{-1 } S_{ij}(1,1)
\nonumber\\&\quad{}
-120\omega_{0 } \omega^{[i]}_{-2 } S_{ij}(1,1)
+36\omega_{0 } \omega^{[j]}_{-2 } S_{ij}(1,1)
\nonumber\\&\quad{}
+72\omega_{0 } \omega_{0 } \omega^{[i]}_{-1 } S_{ij}(1,1)
-9\omega_{0 } \omega_{0 } \omega_{0 } \omega_{0 } S_{ij}(1,1)
\nonumber\\&\quad{}
+12\omega^{[i]}_{-2 } S_{ij}(1,2)
+12\omega^{[j]}_{-2 } S_{ij}(1,2)
\nonumber\\&\quad{}
-72\omega_{0 } \omega^{[i]}_{-1 } S_{ij}(1,2)
-72\omega_{0 } \omega^{[j]}_{-1 } S_{ij}(1,2)
\nonumber\\&\quad{}
+18\omega_{0 } \omega_{0 } \omega_{0 } S_{ij}(1,2),\label{eq:s11-4-1}\\
0&=
8\omega^{[j]}_{-3 } S_{ij}(1,1)
-24\Har^{[j]}_{-1 } S_{ij}(1,1)
\nonumber\\&\quad{}
+54\omega_{0 } \omega^{[i]}_{-2 } S_{ij}(1,1)
-36\omega_{0 } \omega^{[j]}_{-2 } S_{ij}(1,1)
\nonumber\\&\quad{}
-36\omega_{0 } \omega_{0 } \omega^{[i]}_{-1 } S_{ij}(1,1)
+9\omega_{0 } \omega_{0 } \omega_{0 } \omega_{0 } S_{ij}(1,1)
\nonumber\\&\quad{}
+54\omega^{[i]}_{-2 } S_{ij}(1,2)
-12\omega^{[j]}_{-2 } S_{ij}(1,2)
\nonumber\\&\quad{}
+72\omega_{0 } \omega^{[j]}_{-1 } S_{ij}(1,2)
-18\omega_{0 } \omega_{0 } \omega_{0 } S_{ij}(1,2)
\nonumber\\&\quad{}
+72\omega^{[i]}_{-1 } S_{ij}(1,3),\label{eq:s11-4-2}\\
0&=
14\omega^{[j]}_{-3 } S_{ij}(1,1)
+12\Har^{[j]}_{-1 } S_{ij}(1,1)
\nonumber\\&\quad{}
-3\omega^{[j]}_{-2 } S_{ij}(1,2)
-36\omega^{[j]}_{-1 } S_{ij}(1,3).\label{eq:s11-4-3}
\end{align}

The following result is well-known:
\begin{lemma}
\label{lemma:m1plusfinitegenerated}
The vertex operator algebra $M(1)^{+}$ is strongly generated by
$\omega^{[i]},J^{[i]}$, and $S_{lm}(1,r)$ $(1\leq i\leq \rankL, 1\leq m<l\leq \rankL, r=1,2,3)$
in the sense of \cite[p.111]{Kac1998}.
\end{lemma}

The following result is a direct consequence of Lemma \ref{lemma:m1plusfinitegenerated}:
\begin{lemma}
\label{lemma:m1plusmodulegenerated}
Let $\mK$ be an $M(1)^{+}$-module such that $\mK=M(1)^{+}\cdot \mK(0)$.
Then, $\mK$ is spanned 
by
$a^{(1)}_{i_1}\cdots a^{(n)}_{i_n}b$ where $n\in\Z_{\geq 0}$,
$b\in \mK(0)$, 
$a^{(j)}\in\{\omega^{[k]},J^{[k]}\ |\ k=1,\ldots,\rankL\}\cup
\{S_{lm}(1,r)\ |\ 1\leq m<l\leq\rankL, r=1,2,3\}$ 
and $i_j\in\Z_{\leq \wt a^j-2}$
for $j=1,\ldots,n$.
\end{lemma}

\begin{lemma}
Let $U$ be a subspace of $\mW$ which is stable under the actions of	$\{\omega^{[i]}_{1},\Har^{[i]}_{3}\}_{i=1}^{\rankL}$. 
Assume $\epsilon_{I}(\omega^{[i]},\lu)\leq 1$ and $\epsilon_{I}(\Har^{[i]},\lu)\leq 3$
	for any non-zero $\lu\in U$ and $i=1,\ldots,\rankL$.
	Let $i,j\in \{1,\ldots,\rankL\}$ with $i\neq j$
	and $\epsilon_{S}\in\Z$ such that
	$\epsilon_{S}\geq \epsilon_{I}(S_{ij},\lu)$ for all non-zero $\lu\in U$.
Then, for a non-zero $\lu\in U$
\begin{align}
\label{eq:(epsilonS-1)big((18zeta[i]+3)epsilonS5-0}
0&=
-\epsilon_{S} (\epsilon_{S}+1)^2S_{ij}(1,1)_{\epsilon_{S} }\lu
-(\epsilon_{S}+2) (3 \epsilon_{S}+1)S_{ij}(1,2)_{\epsilon_{S}+1 }\lu
\nonumber\\&\quad{}
+4 \epsilon_{S}S_{ij}(1,1)_{\epsilon_{S} } \omega^{[i]}_{1 }\lu
-4S_{ij}(1,1)_{\epsilon_{S} } \omega^{[j]}_{1 }\lu
\nonumber\\&\quad{}
-2 (3 \epsilon_{S}+1)S_{ij}(1,3)_{\epsilon_{S}+2 }\lu
+4S_{ij}(1,2)_{\epsilon_{S}+1 } \omega^{[i]}_{1 }\lu
\nonumber\\&\quad{}
-4S_{ij}(1,2)_{\epsilon_{S}+1 } \omega^{[j]}_{1 }\lu,\\
\label{eq:(epsilonS-1)big((18zeta[i]+3)epsilonS5-1}
0&=
-\epsilon_{S} (\epsilon_{S}+1) (3 \epsilon_{S}^2+3 \epsilon_{S}-2)S_{ij}(1,1)_{\epsilon_{S} }\lu
-2 \epsilon_{S} (3 \epsilon_{S}^2+3 \epsilon_{S}-2)S_{ij}(1,2)_{\epsilon_{S}+1 }\lu
\nonumber\\&\quad{}
+8 \epsilon_{S} (3 \epsilon_{S}-1)S_{ij}(1,1)_{\epsilon_{S} }\omega^{[i]}_{1 } \lu
+8 (3 \epsilon_{S}-1)S_{ij}(1,2)_{\epsilon_{S}+1 }\omega^{[i]}_{1 } \lu
\nonumber\\&\quad{}
-8S_{ij}(1,1)_{\epsilon_{S} }\Har^{[i]}_{3 } \lu
+8 (3 \epsilon_{S}-1)S_{ij}(1,1)_{\epsilon_{S} }\omega^{[j]}_{1 } \lu
\nonumber\\&\quad{}
+8 (3 \epsilon_{S}-1)S_{ij}(1,2)_{\epsilon_{S}+1 }\omega^{[j]}_{1 } \lu
+8S_{ij}(1,1)_{\epsilon_{S} }\Har^{[j]}_{3 } \lu,\\
\label{eq:(epsilonS-1)big((18zeta[i]+3)epsilonS5-2}
0&=
2 (3 \epsilon_{S}^3+21 \epsilon_{S}^2+42 \epsilon_{S}+14)S_{ij}(1,2)_{\epsilon_{S}+1 }\lu
-8 (3 \epsilon_{S}-7)S_{ij}(1,1)_{\epsilon_{S} }\omega^{[j]}_{1 } \lu
\nonumber\\&\quad{}
+4 (18 \epsilon_{S}+7)S_{ij}(1,3)_{\epsilon_{S}+2 }\lu
-8 (3 \epsilon_{S}-7)S_{ij}(1,2)_{\epsilon_{S}+1 }\omega^{[j]}_{1 } \lu
\nonumber\\&\quad{}
-8S_{ij}(1,1)_{\epsilon_{S} }\Har^{[j]}_{3 } \lu
+3 \epsilon_{S} (\epsilon_{S}+1)^2 (\epsilon_{S}+4)S_{ij}(1,1)_{\epsilon_{S} }\lu
\nonumber\\&\quad{}
-12 \epsilon_{S} (\epsilon_{S}+4)S_{ij}(1,1)_{\epsilon_{S} }\omega^{[i]}_{1 } \lu
-36S_{ij}(1,2)_{\epsilon_{S}+1 }\omega^{[i]}_{1 } \lu
\nonumber\\&\quad{}
+24S_{ij}(1,3)_{\epsilon_{S}+2 }\omega^{[i]}_{1 } \lu,\\
\label{eq:(epsilonS-1)big((18zeta[i]+3)epsilonS5-3}
0&=
-S_{ij}(1,2)_{\epsilon_{S}+1 }\lu
-5S_{ij}(1,1)_{\epsilon_{S} }\omega^{[j]}_{1 } \lu
\nonumber\\&\quad{}
-S_{ij}(1,3)_{\epsilon_{S}+2 }\lu
-11S_{ij}(1,2)_{\epsilon_{S}+1 }\omega^{[j]}_{1 } \lu
\nonumber\\&\quad{}
+2S_{ij}(1,1)_{\epsilon_{S} }\Har^{[j]}_{3 } \lu
-6S_{ij}(1,3)_{\epsilon_{S}+2 }\omega^{[j]}_{1 } \lu.
\end{align}
If $\lu$ is a simultaneous eigenvector of $\{ \omega^{[i]}_{1},\omega^{[j]}_{1},\Har^{[i]}_{3},\Har^{[j]}_{3}\}$
with eigenvalues $\{\zeta^{[i]},\zeta^{[j]},\xi^{[i]},\xi^{[j]}\}$:
\begin{align}
	(\omega^{[i]}_{1}-\zeta^{[i]})\lu&=
	(\omega^{[j]}_{1}-\zeta^{[j]})\lu\nonumber\\
=(\Har^{[i]}_{3}-\xi^{[i]})\lu	&=(\Har^{[j]}_{3}-\xi^{[j]})\lu=0,
\label{eq:omega-zeta-H4-0}
\end{align}
then
\begin{align}
0&=(\epsilon_{S}-1)\big((18\zeta^{[i]}+3)\epsilon_{S}^{5}+
(-54\zeta^{[i]}+6)\epsilon_{S}^{4}\nonumber\\
&\qquad{}+
((-216\zeta^{[i]}-36)\zeta^{[j]}+216 (\zeta^{[i]})^2-78\zeta^{[i]}+1)\epsilon_{S}^{3}\nonumber\\
&\qquad{}+
((744\zeta^{[i]}+4)\zeta^{[j]}+24 (\zeta^{[i]})^2+22\zeta^{[i]}-2)\epsilon_{S}^{2}\nonumber\\
&\qquad{}+
((-1152 (\zeta^{[i]})^2-192\zeta^{[i]})\zeta^{[j]}-48 (\zeta^{[i]})^2+12\zeta^{[i]})\epsilon_{S}^{1}\nonumber\\
&\qquad{}+
(384 (\zeta^{[i]})^2-16\zeta^{[i]})\zeta^{[j]}\big)\nonumber\\
 	&\qquad{}+\big((144\zeta^{[i]}+24) \epsilon_{S}^2+(-192\zeta^{[i]}+8) \epsilon_{S}\nonumber\\
 	&\qquad\quad{}+(-192\zeta^{[i]}-32)\zeta^{[j]}+192 (\zeta^{[i]})^2-48\zeta^{[i]}\big)\xi^{[i]}\nonumber\\
 	&\qquad+\big(-72 \epsilon_{S}^4-96 \epsilon_{S}^3+(288\zeta^{[j]}+144\zeta^{[i]}) \epsilon_{S}^2\nonumber\\
 	&\qquad\quad{}+(192\zeta^{[i]}+8) \epsilon_{S}+192\zeta^{[i]}\zeta^{[j]}-192 (\zeta^{[i]})^2+16\zeta^{[j]}\big)\xi^{[i]},\label{eq:s11-zeta-1}\\
0&=
(\epsilon_{S}-1)\big(
18\zeta^{[j]}+3)\epsilon_{S}^{5}+
(-54\zeta^{[j]}+6)\epsilon_{S}^{4}\nonumber\\
&\qquad{}+
(216(\zeta^{[j]})^2+(-216\zeta^{[i]}-78)\zeta^{[j]}-36\zeta^{[i]}+1)\epsilon_{S}^{3}\nonumber\\
&\qquad{}+
(24 (\zeta^{[j]})^2+(744\zeta^{[i]}+22)\zeta^{[j]}+4\zeta^{[i]}-2)\epsilon_{S}^{2}\nonumber\\
&\qquad{}+
((-1152\zeta^{[i]}-48) (\zeta^{[j]})^2+(-192\zeta^{[i]}+12)\zeta^{[j]})\epsilon_{S}^{1}\nonumber\\
&\qquad{}+
(384\zeta^{[i]} (\zeta^{[j]})^2-16\zeta^{[i]}\zeta^{[j]})\big)\nonumber\\
&\qquad{}+\big(-72 \epsilon_{S}^4-96 \epsilon_{S}^3+(144\zeta^{[j]}+288\zeta^{[i]}) \epsilon_{S}^2\nonumber\\
&\qquad\quad{}+(192\zeta^{[j]}+8) \epsilon_{S}-192 (\zeta^{[j]})^2+(192\zeta^{[i]}+16)\zeta^{[j]}\big)\xi^{[i]}\nonumber\\
&\qquad{}+\big((144\zeta^{[j]}+24) \epsilon_{S}^2+(-192\zeta^{[j]}+8) \epsilon_{S}\nonumber\\
&\qquad\quad{}+192(\zeta^{[j]})^2+(-192\zeta^{[i]}-48)\zeta^{[j]}-32\zeta^{[i]}\big)\xi^{[j]}.\label{eq:s11-zeta-2}
 \end{align}
\end{lemma}
\begin{proof}
	By Lemma \ref{lemma:bound-index-S}, $\epsilon(S_{ij}(1,m+1),\lu)\leq \epsilon_{S}+m$ for 
	$m=0,1,\ldots$ and all non-zero $\lu\in U$.
By taking the $(\epsilon_{S}+3)$-th action of \eqref{eq:s11-3} on $\lu$ and using the results in Section \ref{section:normal-total}, 
the same argument as in the proof of \eqref{eq:0=(omega-3ExB)lE+2lom+2lu=}--\eqref{eq:(Har-1ExB)lE+2lom+2lu=} in Lemma \ref{lemma:r=1-s=3} shows \eqref{eq:(epsilonS-1)big((18zeta[i]+3)epsilonS5-0}.
		Taking the $(\epsilon_{S}+4)$-th actions of \eqref{eq:s11-4-1}--\eqref{eq:s11-4-3} on $\lu$,
		we have \eqref{eq:(epsilonS-1)big((18zeta[i]+3)epsilonS5-1}--\eqref{eq:(epsilonS-1)big((18zeta[i]+3)epsilonS5-3}.
		Deleting the terms including $S_{ij}(1,3)_{\epsilon_{S}+2}\lu$
from \eqref{eq:(epsilonS-1)big((18zeta[i]+3)epsilonS5-0},
\eqref{eq:(epsilonS-1)big((18zeta[i]+3)epsilonS5-2}, and 
\eqref{eq:(epsilonS-1)big((18zeta[i]+3)epsilonS5-3}, we  have two relations.
Deleting the terms including $S_{ij}(1,2)_{\epsilon_{S}+2}\lu$ from these two relations and 
\eqref{eq:(epsilonS-1)big((18zeta[i]+3)epsilonS5-1}, we have 
	\eqref{eq:s11-zeta-1} and \eqref{eq:s11-zeta-2}.
		\end{proof}

\begin{lemma}
\label{lemma:M1lambda-submodule}
Let 
\begin{align}
	0\rightarrow \mW\overset{}{\rightarrow} N\overset{\pi}{\rightarrow} \module\rightarrow 0
\label{eq:exact-seq-lambda}
	\end{align}
be an exact sequence of weak $M(1)^{+}$-modules
where $\mW$ is an irreducible $M(1)^{+}$-module,
$N$ is a weak $M(1)^{+}$-module and $\module=\oplus_{i\in \gamma+\Z_{\geq 0}}\module_{i}$ is an $M(1)^{+}$-module.
Let $B$ be an irreducible $A(M(1)^{+})$-submodule of $\module_{\gamma}$
that is not isomorphic to $\mW(0)$ 
and $\lv$
 a simultaneous eigenvector in $B$ of $\{\omega^{[i]}_{1},\Har^{[i]}_{3}\}_{i=1}^{\rankL}$
with eigenvalues $\{\zeta^{[i]},\xi^{[i]}\}_{i=1}^{\rankL}$:
\begin{align}
(\omega^{[i]}_1-\zeta^{[i]})\lv=(\Har^{[i]}_3-\xi^{[i]})\lv=0.
\end{align}
If $(W,B)\not\cong (M(1)^{+},M(1)^{-}(0))$,
then there exists a preimage $\lu\in N_{\gamma}$ of $\lv$ under the canonical projection $\mN_{\gamma}\rightarrow \module_{\gamma}$
such that
\begin{align}
(\omega^{[i]}_1-\zeta^{[i]})\lu=(\Har^{[i]}_3-\xi^{[i]})\lu=0
\end{align}
for $i=1,\ldots,\rankL$.
\end{lemma}
\begin{proof}
Using \cite[Proposition 4.3]{Abe2005} and eigenvalues of $\omega^{[i]}_{1}$ and $\Har^{[i]}_{3}$  for $i=1,\ldots,\rankL$ on 
irreducible $M(1)^{+}$-modules in \cite[Table 1]{AD2004},
we see that the result holds
if $\mW=M(\theta)^{\pm}$ or $B=M(\theta)^{\pm}(0)$.
We discuss the other cases.
We take $\lu\in\mN$ that satisfies \eqref{eq:pi(lu)lvmboxand}.

Let $B=\C e^{\lambda}$ for some $\lambda\in\fh\setminus\{0\}$.
In this case $\zeta^{[i]}=\langle \lambda,h^{[i]}\rangle^2/2$ and
$\xi^{[i]}=0$ for $i=1,\ldots,\rankL$.
We note that at least one of $\zeta^{[1]},\ldots,\zeta^{[\rankL]}$ is not zero.
Let $W=M(1,\mu)$ such that $\mu\in\fh\setminus\{0,\pm\lambda\}$.
Since $\cap_{j=1}^{n}\Ker \Har^{[j]}_{3}\cap M(1,\mu)=\C e^{\mu}$ by \cite[Proposition 4.3]{Abe2005}, 
		\begin{align}
			M(1,\mu)_{\zeta,(0,\ldots,0)}&=\bigcap_{j=1}^{\rankL}\Ker (\omega^{[j]}_{1}-\zeta^{[j]})\cap\C e^{\mu}.
			\label{eq:cap-omega-mu}
		\end{align}
Assume  
\begin{align}
\label{eq:omega-i-zeta-or-H-0}
(\omega^{[i]}_{1}-\zeta^{[i]})\lu\neq 0\mbox{ or }\Har^{[i]}_{3}\lu\neq 0
\mbox{ for some }i\in \{1,\ldots,\rankL\}.
\end{align}
It follows from  \eqref{eq:keromega-kerH} that $M(1,\mu)_{\zeta,(0,\ldots,0)}\neq 0$ 
		and hence 
		$\langle\lambda,h^{[j]}\rangle=\pm\langle\mu,h^{[j]}\rangle$ for all $j=1,\ldots,\rankL$ by \eqref{eq:cap-omega-mu}.
Since $\gamma=\langle\lambda,\lambda\rangle=\langle\mu,\mu\rangle$ and $\lambda\neq\pm\mu$, 
$N(0)\cong M(1,\lambda)(0)\oplus M(1,\mu)(0)$ as $A(M(1)^{+})$-modules, which leads to the result.
If $W=M(1)^{\pm}$, then the result follows from the fact that 
$M(1)^{\pm}_{\zeta,(0,\ldots,0)}=0$.

If $B=\C \vac=M(1)^{+}(0)$, then the same argument as above shows the result.

Let $B=M(1)^{-}(0)$,
$W=M(1,\lambda)$ such that $\lambda\in\fh\setminus\{0\}$, and $\lv=h^{[j]}(-1)\vac$ for some $j\in \{1,\ldots,\rankL\}$.
Since $\xi^{[i]}=\delta_{ij}$ for $i=1,\ldots,\rankL$, it follows from \cite[Proposition 4.3]{Abe2005} that
\begin{align}
M(1,\lambda)_{\zeta,\xi}&\subset \C h^{[j]}(-1)e^{\lambda}.
\end{align}
Suppose there exists $i\in\{1,\ldots,\rankL\}$ such that
$(\omega^{[i]}-\delta_{ij})\lu\neq 0$ or $(\Har^{[i]}-\delta_{ij})\lu\neq 0$.
Then, $M(1,\lambda)_{\zeta,\xi}\neq 0$ and hence $\delta_{jk}=\langle \lambda, h^{[k]}\rangle^2/2+\delta_{jk}$
for all $k=1,\ldots,\rankL$, which contradicts that $\lambda\neq 0$.
The proof is complete.
\end{proof}

\begin{lemma}
\label{lemma:M1lambda-submodule-2}
Let 
\begin{align}
	0\rightarrow \mW\overset{}{\rightarrow} N\overset{\pi}{\rightarrow} \module\rightarrow 0
\label{eq:exact-seq-lambda-2}
	\end{align}
be an exact sequence of weak $M(1)^{+}$-modules
where $\mW$ is an irreducible $M(1)^{+}$-module,
$N$ is a weak $M(1)^{+}$-module and $\module=\oplus_{i\in \gamma+\Z_{\geq 0}}\module_{i}$ is an $M(1)^{+}$-module.
Let $B$ be an irreducible $A(M(1)^{+})$-submodule of $\module_{\gamma}$
and $\lv$
a simultaneous eigenvector in $B$ for $\{\omega^{[i]}_{1},\Har^{[i]}_{3}\}_{i=1}^{\rankL}$ 
with eigenvalues $\{\zeta^{[i]},\xi^{[i]}\}_{i=1}^{\rankL}$:
\begin{align}
(\omega^{[i]}_1-\zeta^{[i]})\lv=(\Har^{[i]}_3-\xi^{[i]})\lv=0.
\end{align}
Let $\lw\in \mN_{\gamma}$
 such that 
$(\omega^{[i]}_1-\zeta^{[i]})\lw=(\Har^{[i]}_3-\xi^{[i]})\lw=0$
 for all $i=1,\ldots,\rankL$.
If $(W,B)\not\cong (M(1)^{+}, M(1)^{-}(0))$, then $\lw \in \Omega_{M(1)^{+}}(N_{\gamma})$.
If $(W,B)\cong (M(1)^{+},M(1)^{-}(0))$, then  $\epsilon(S_{ij}(1,1),\lw)\leq 2$ for any pair of distinct elements $i,j\in\{1,\ldots,\rankL\}$.
\end{lemma}
\begin{proof}
If $(W,B)\not\cong (M(1)^{+}, M(1)^{-}(0))$, then by Lemma \ref{lemma:bound-index-S}
and \eqref{eq:epsilon(omegailu)leq1}, 
it is enough to show that $\epsilon(S_{ij}(1,1),\lw)\leq 1$
for any pair of distinct elements $i,j\in\{1,\ldots, \rankL\}$.
We write $\epsilon_{S}=\epsilon(S_{ij}(1,1),\lw)$	for simplicity.

\begin{enumerate}
\item
Let $B\cong M(1,\lambda)(0)$ for some $\lambda\in\fh\setminus\{0\}$. In this case $\xi^{[i]}=0$ for all $i=1,\ldots,\rankL$.
		Assume $\langle\lambda,\lambda\rangle\neq 0$.
		Then, we may assume $\lambda\in \C h^{[1]}$ and hence
		$\langle\lambda,h^{[i]}\rangle=\zeta^{[i]}=0$ for all $i=2,\ldots,\rankL$.
	By \eqref{eq:s11-zeta-1} and \eqref{eq:s11-zeta-2}, 
\begin{align}
0&=\epsilon_{S}^2 (\epsilon_{S}-1) ((\epsilon_{S}+1)(3 \epsilon_{S}^2+3 \epsilon_{S}-2)+4(-9\epsilon_{S}+1) \zeta^{[1]}),\nonumber\\
0&=\epsilon_{S} (\epsilon_{S}-1) ((18 \zeta^{[1]}+3) \epsilon_{S}^4+(-54 \zeta^{[1]}+6) \epsilon_{S}^3+(216 (\zeta^{[1]})^2-78 \zeta^{[1]}+1)
 \epsilon_{S}^2\nonumber\\
&\qquad{}+(24 (\zeta^{[1]})^2+22 \zeta^{[1]}-2) \epsilon_{S}-48 (\zeta^{[1]})^2+12 \zeta^{[1]})
 \label{eq:s11-k-2}
\end{align}
for $i=2,\ldots,\rankL$ and hence $\epsilon_{S}=0$ or $1$ for $S_{i1}(1,1)$.
		For the other $i,j$, since $\zeta^{[i]}=\zeta^{[j]}=0$,
		the same argument as above shows that $\epsilon_{S}=0$ or $1$.
	
Assume $\langle\lambda,\lambda\rangle=0$. Then, we may assume $0\neq \langle \lambda,h^{[1]}\rangle^2
		=-\langle \lambda,h^{[2]}\rangle^2$ 
		and $\langle\lambda,h^{[j]}\rangle=0$ for all $j=3,4,\ldots,\rankL$. 
By substituting $\zeta^{[2]}=-\zeta^{[1]}$ into \eqref{eq:s11-zeta-1} and \eqref{eq:s11-zeta-2}, the same argument as above shows
that 
$\epsilon_{S}=0,1$ for $S_{21}(1,1)$.
		For the other $i,j$, since one of $\zeta^{[i]}$ or $\zeta^{[j]}$ is $0$,
		the same argument as above also shows that $\epsilon_{S}=0$ or  $1$.
\item
Let $B\cong M(1)^{-}(0)$. 
If $\rankL=1$, then the result is shown in \textcolor{black}{\cite[Theorem 5.5]{Abe2005}}.
Assume $\rankL\geq 2$.
If $(\zeta^{[i]},\xi^{[i]})=(\zeta^{[j]},\xi^{[j]})=(0,0)$,
then it follows \eqref{eq:s11-zeta-1} that
\begin{align}
	0&=\epsilon_{S}^2(\epsilon_{S}-1)(\epsilon_{S}+1)(3\epsilon_{S}^2+3\epsilon_{S}-2)
	\end{align}
and hence $\epsilon_{S}\leq 1$.

Assume $(\zeta^{[i]},\xi^{[i]})=(0,0)$ and $(\zeta^{[j]},\xi^{[j]})=(1,1)$. It follows \eqref{eq:s11-zeta-1} that
\begin{align}
\label{eq:epsilonS13epsilonS615}
0=(\epsilon_{S}-2) (\epsilon_{S}-1) (3 \epsilon_{S}^4+12 \epsilon_{S}^3-11 \epsilon_{S}^2-20 \epsilon_{S}-16)
\end{align}
and hence $\epsilon_{S}=1$ or $2$.
We further assume that $\epsilon_{S}=2$ for some $i,j$. 
By \eqref{eq:(epsilonS-1)big((18zeta[i]+3)epsilonS5-0}--\eqref{eq:(epsilonS-1)big((18zeta[i]+3)epsilonS5-3},
\begin{align}
S_{ij}(1,2)_{3}\lu&=-2S_{ij}(1,1)_{2}\lu
\mbox{ and }S_{ij}(1,3)_{4}\lu=3S_{ij}(1,1)_{2}\lu.
\end{align}
By	\eqref{eq:omega-s11} and \eqref{eq:action[H4-S]} 
and computations in Section \ref{section:appendix-The general case},
$S_{ij}(1,1)_{2}\lu\in \Omega_{M(1)^{+}}(\mW)$ and 
$\C S_{ij}(1,1)_{2}\lu\cong \C \vac$ as $A(M(1)^{+})$-modules,
which leads $W\cong M(1)^{+}$ since $W$ is irreducible.

The same argument as above shows the results for the case that $B\cong M(1)^{+}(0)$ or $M(\theta)^{\pm}(0)$.
\end{enumerate}
\end{proof}
	\begin{lemma}
\label{lemma:Ext-split-lambda-mu}
For a pair of non-isomorphic irreducible $M(1)^{+}$-modules $\module,\mW$
such that $(M,W)\not\cong (M(1)^{+},M(1)^{-})$ and $(M(1)^{-},M(1)^{+})$,
$\Ext^{1}_{M(1)^{+}}(\module,W)=0$.
\end{lemma}
	\begin{proof}
Let $\mN$ be a weak $M(1)^{+}$-module and 
\begin{align}
\label{eq:rightarrowM(1mu)oversetiotarightarrowN}
	0\rightarrow \mW\overset{}{\rightarrow} N\overset{\pi}{\rightarrow} \module\rightarrow 0
	\end{align}
an exact sequence of weak $M(1)^{+}$-modules.
We write $\module =\oplus_{i\in\gamma+\Z_{\geq 0}}\module_{i}$ with $\module_{\gamma}\neq 0$
and $W=\oplus_{i\in \delta+\Z_{\geq 0}}\mW_{i}$ with $\mW_{\delta}\neq 0$.
By Lemmas \ref{lemma:M1lambda-submodule} and \ref{lemma:M1lambda-submodule-2}, there exists $\lu\in \Omega_{M(1)^{+}}(\mN)$ such that
$0\neq \pi(\lu)\in \module_{\gamma}$.

Assume $\mW\cap (M(1)^{+}\cdot \lu)\neq 0$.
Since $\mW_{\delta}\subset \mW\cap (M(1)^{+}\cdot \lu)$,
$\delta\in \gamma+\Z_{\geq 0}$.
Since $\Ext^{1}_{M(1)^{+}}(\mW,\module)\neq 0$ 
by the assumption, 
\cite[Proposition 2.5]{Abe2005}, and \cite[Proposition 3.5]{ADL2005},
the same argument as above shows that $\gamma\in \delta+\Z_{\geq 0}$
and hence $\gamma=\delta$.
Since $N(0)=N_{\gamma}\cong \module_{\gamma}\oplus \mW_{\delta}$ as $A(M(1)^{+})$-modules,
the sequence \eqref{eq:rightarrowM(1mu)oversetiotarightarrowN} splits, a contradiction.

\end{proof}
	
\begin{lemma}
\label{lemma:Ext-M-M}
For $M=M(1)^{+},M(1)^{-},M(\theta)^{+}$, and $M(\theta)^{-}$,
$\Ext^{1}_{M(1)^{+}}(M,M)=0$.
\end{lemma}
\begin{proof}
	Let $\mW$ be an $M(1)^{+}$-module such that $\mW\cong \module$,
	$\mN$ a weak $M(1)^{+}$-module, and 
	\begin{align}
		\label{eq:rightarrowM(1mu)oversetiotarightarrowN-1}
		0\rightarrow \mW\overset{}{\rightarrow} N\overset{\pi}{\rightarrow} \module\rightarrow 0
	\end{align}
	an exact sequence of weak $M(1)^{+}$-modules.
	We take $\lu\in \mN$ and $\lv\in\module$ as in \eqref{eq:eigenvalue-omega-H}
and	\eqref{eq:pi(lu)lvmboxand}.	
In the case of $M=M(1)^{+}$, the same argument as in the proof of 
\cite[Proposition 5.1]{Abe2005} shows that $\Ext^{1}_{M(1)^{+}}(M(1)^{+},M(1)^{+})=0$.

For $M=M(1)^{-}$ or $M(\theta)^{\pm}$, it is enough to show that
$N(0)\cong \mW(0)\oplus M(0)\cong M(0)\oplus M(0)$ as $A(M(1)^{+})$-modules.
In the Zhu algebra $A(M(1)^{+})$, we have
\begin{align}
\label{eq:omega-H-commute}
\omega^{[i]}*\Har^{[i]}\equiv \Har^{[i]}*\omega^{[i]}
\end{align}
and recall that the following relations from \cite[(6.1.10) and (6.1.11)]{DN2001}:
\begin{align}
\label{eq:omega-1-omega-1/16-omega-9/16-1}
(132(\omega^{[i]})^2-65\omega^{[i]}-70\Har^{[i]}+3)*\Har^{[i]}&\equiv 0,\\
\label{eq:omega-1-omega-1/16-omega-9/16-2}
(\omega^{[i]}-\vac)*(\omega^{[i]}-\dfrac{1}{16}\vac)*(\omega^{[i]}-\dfrac{9}{16}\vac)*\Har^{[i]}&\equiv 0
\end{align}  
 for $i=1,\ldots, \rankL$. Here, we note that $\Har_{a}$ in \cite[Section 6]{DN2001}
is equal to the image of $-9\Har^{[a]}$ under the projection $M(1)^{+}\rightarrow A(M(1)^{+})$
for $a=1,\ldots,\rankL$.
\begin{enumerate}
\item Let $M=M(\theta)^{+}$.
Since $S_{ij}(1,1)_{1}\vac_{\tw}=0$ for any pair of distinct elements $i,j\in\{1,\ldots,\rankL\}$,
$S_{ij}(1,1)_{1}\lu\in \C v$ in $W$.
We note that $\omega^{[i]}_{1}\vac_{\tw}=(1/16)\vac_{\tw}$ and 
$\Har^{[i]}_{3}\vac_{\tw}=(-1/128)\vac_{\tw}$. 
By \eqref{eq:omega-1-omega-1/16-omega-9/16-2}, $\omega^{[i]}_{1}\lw=(1/16)\lw$ and 
$\Har^{[i]}_{3}\lw=(-1/128)\lw$ for all $\lw \in N(0)$.
By \eqref{eq:s11-zeta-1},
\begin{align}
0&=
(\epsilon_{S}-1) (22 \epsilon_{S}^2-8 \epsilon_{S}+1) (6 \epsilon_{S}^3+6 \epsilon_{S}^2-7 \epsilon_{S}+1)
\label{eq:tw-0-S11}
\end{align}
and hence $\epsilon_{S}\leq 1$. Thus, $S_{ij}(1,k)_{k}\lu=0$ for all $k\in \Z_{\geq 0}$ by \eqref{eq:sab-bound}
and hence $N(0)\cong M(\theta)^{+}(0)\oplus M(\theta)^{+}(0)$.

\item Let $M=M(1)^{-}$. We consider $N(0)$.
Since $N(0)/W(0)\cong M(1)^{-}(0)$, $A^{u}\cdot N(0)\neq 0$.
Since $A^{t}\cdot N(0)\subset W(0), A^{u}A^{t}=0$, and $A^{u}\cdot \lw\neq 0$ for any non-zero $\lw\in W(0)$,
we  have $A^{t}\cdot N(0)=0$.
The same argument shows that $\Lambda_{ij}\cdot N(0)=0$ for any pair of distinct elements $i,j\in\{1,\ldots,\rankL\}$.
We note that the eigenvalues for $\omega^{[i]}|_{N(0)}$ are $0$ and $1$,
and  those for $\Har^{[i]}|_{N(0)}$ are also $0$ and $1$.

\textcolor{black}{Let $\lv\in N$} be a simultaneous generalized eigenvector of $\{\omega^{[i]}_{1},\Har^{[i]}_{3}\}_{i=1}^{\rankL}$.
We fix $i=1,\ldots,\rankL$.
Assume $\Har^{[i]}_{3}\lv\neq 0$. By \eqref{eq:omega-1-omega-1/16-omega-9/16-2}, $\omega^{[i]}_{1}\Har^{[i]}_{3}\lv=\Har^{[i]}_{3}\lv$.
By \eqref{eq:omega-1-omega-1/16-omega-9/16-1}, $(\Har^{[i]}_3-1)\Har^{[i]}_{3}\lv=0$, which implies
$\omega^{[i]}_{1}\lv=\lv$.

Assume $\Har^{[i]}_{3}\lv=0$ and $(\omega^{[i]}_{1}-1)^2\lv=0$. Since
$(\omega^{[i]}_{1}-1)\lv\in \mW(0)$ and
there is no non-zero vector $\lw\in M(1)^{-}(0)$ such that $\omega^{[i]}_{1}\lw=\lw$ and $\Har^{[i]}_{3}\lw=0$,
$(\omega^{[i]}_{1}-1)\lv=0$.

Assume $\Har^{[i]}_{3}\lv=0$ and $(\omega^{[i]}_{1})^2\lv=0$.
The argument as above shows that there exists $k$ such that $k\neq i$, $\Har^{[k]}_{3}\lv=\lv$ and  $\omega^{[k]}_{1}\lv=\lv$. 
Since $\omega^{[k]}_{1}\omega^{[i]}_{1}\lv=\omega^{[i]}_{1}\lv$, we have $E^{u}_{kk}\omega^{[i]}_{1}\lv=\omega^{[i]}_{1}\lv$ in $\mW(0)$.
By \cite[Lemma 5.2.2]{DN2001}, $\omega^{[i]}_{1}\lv=0$. 
Thus $A(M(1)^{+})\cdot \lv=A^{u}\cdot \lv$. 
Since $A^{u}$ is isomorphic to the matrix algebra, 
$A^{u}\cdot \lv$ is an irreducible  $A(M(1)^{+})$-module. Thus $N(0)\cong M(1)^{-}(0)\oplus M(1)^{-}(0)$.

\item In the case of $M=M(\theta)^{-}$, the same argument as in (2) above shows that 
$N(0)\cong M(\theta)^{-}(0)\oplus M(\theta)^{-}(0)$.

\end{enumerate}
\end{proof}

By Lemmas \ref{lemma:Ext-split-lambda-mu}, \ref{lemma:Ext-M-M}, \cite[Proposition 2.5]{Abe2005}, and  \cite[Proposition 3.5]{ADL2005}, 
we have the following result:
\begin{proposition}
\label{proposition:Ext-total}
If a pair $(M,W)$ of irreducible $M(1)^{+}$-modules satisfies one of the following conditions,
then $\Ext^{1}_{M(1)^{+}}(M,W)=\Ext^{1}_{M(1)^{+}}(W,M)=0$.
 
\begin{enumerate}
\item $M\cong M(1,\lambda)$ with $\lambda\in\fh\setminus\{0\}$ and $W\not\cong M(1,\lambda)$.
\item $M\cong M(\theta)^{\pm}$.
\item $M\cong M(1)^{+}$ and $W\not\cong M(1)^{-}$.
\item $M\cong M(1)^{-}$ and $W\not\cong M(1)^{+}$.
\end{enumerate}
\end{proposition}

The following result is a direct consequence of Lemmas \ref{lemma:M1lambda-submodule} and \ref{lemma:M1lambda-submodule-2}.
\begin{corollary}\label{corollary:verma-irreducible}
	Let $\Omega$ be an irreducible $A(M(1)^{+})$-module such that 
	$\Omega\not\cong M(1)^{+}_{0}=\C\vac$.
Then the generalized Verma module for $M(1)^{+}$ associated
with  $\Omega$ is irreducible.
\end{corollary}
\begin{proof}
Let $\mN=\oplus_{i\in\delta+\Z_{\geq 0}}\mN_{i}$ with $\mN_{\delta}=\Omega$ be the generalized Verma module for $M(1)^{+}$ associated
with $\Omega$ and $\mW=\oplus_{i\in \gamma+\Z_{\geq 0}}\mW_{i}$ the maximal submodule of $\module$ such that $\Omega\cap \mW=0$.
We take $\gamma$ so that $W_{\gamma}\neq 0$ if $W\neq 0$.
We note that $\gamma-\delta\in\Z_{>0}$.
Taking the restricted dual of the exact sequence 
$	0\rightarrow \mW\overset{}{\rightarrow} \mN\overset{}{\rightarrow} \mN/\mW\rightarrow 0$,
we have the following exact sequence
\begin{align}
0\rightarrow (\mN/\mW)^{\prime}\overset{}{\rightarrow} 
\mN^{\prime}\overset{}{\rightarrow}\mW^{\prime} \rightarrow 0.
\end{align}
We note that $(\mN/\mW)^{\prime}\not\cong M(1)^{+}$ by \cite[Proposition 3.5]{ADL2005}.
Assume $\mW_{\gamma}\neq 0$ and let $B$ be an irreducible $A(M(1)^{+})$-submodule  
of $W_{\gamma}^{\prime}$.
By Lemmas \ref{lemma:M1lambda-submodule} and \ref{lemma:M1lambda-submodule-2},
there exists a non-zero
 $\lu^{\prime}\in \Omega_{M(1)^{+}}(N_{\gamma}^{\prime})$.
For any homogeneous $a\in M(1)^{+}$ such that $\omega_{2}a=0$
and $i\in\Z_{\geq \wt a}$,
it follows from \cite[5.2.4]{FHL} that
\begin{align}
\label{eq:langleailuprime}
0&=\langle a_{i}\lu^{\prime},\lw\rangle=(-1)^{\wt a}\langle \lu^{\prime},a_{2\wt a-i-2}\lw\rangle.
\end{align}
for all $\lw\in \mN$.
Thus, it follows from Lemma \ref{lemma:m1plusmodulegenerated} that
$\lu^{\prime}=0$, a contradiction.
\end{proof}

\begin{lemma}
\label{lemma:generalizedVermaModule-M(1)-}
Let $\mW$ be the generalized Verma module $W$ associated to the $A(M(1)^{+})$-module $\C \vac$
and $\pi : \mW\rightarrow M(1)^{+}$ the canonical projection.
Then, $\Ker \pi\cong (M(1)^{-})^{\oplus k}$ for some $k\in \{1,\ldots,\rankL\}$
as $M(1)^{+}$-modules.
\end{lemma}
\begin{proof}
The same argument as in \cite[(6.1)]{Abe2005}
shows that there is a non-split exact sequence
\begin{align}
\label{eq:non-split-verma-M1plus}
0\rightarrow M(1)^{-}\overset{}{\rightarrow} \mN\overset{}{\rightarrow} M(1)^{+}\rightarrow 0
\end{align}
of $M(1)^{+}$-modules. Thus, $\Ker \pi \neq 0$.
Let $\lu\in\mW$ such that $\pi(\lu)=\vac$.
Note that $u\in \Omega_{M(1)^{+}}(N)$. Thus,
\begin{align}
\label{eq:P(8)H6lu=144}
0&=P^{(8),H}_{6}\lu=144(\omega^{[i]}_{0}-3\Har^{[i]}_{2})\lu. 
\end{align}
for all $i=1,\ldots,\rankL$.
Taking the $3$rd action of \eqref{eq:s11-3} on $\lu$, we have
\begin{align}
\label{eq:Si1121luSi1132}
0&=S_{ij}(1,2)_{1}\lu+S_{ij}(1,3)_{2}\lu
\end{align}
for all $i=2,\ldots,\rankL$. 
By \eqref{eq:P(8)H6lu=144}, \eqref{eq:Si1121luSi1132}, 
and 
\begin{align}
S_{ij}(1,1)_{0}\lu&=-(\omega_{0}S_{ij}(1,1))_{1}\lu=-S_{ij}(2,1)_{1}\lu-S_{ij}(1,2)_{1}\lu
\end{align}
for any pair of distinct elements $i,j\in\{1,\ldots,\rankL\}$,
$N_1$ is spanned by $\{\omega^{[j]}_{0}\lu,S_{ij}(1,2)_{1}\ |\ i,j=1,\ldots,\rankL, i\neq j\}$. 
For distinct $i,j,k\in\{1,\ldots,\rankL\}$, by \eqref{eq:P(8)H6lu=144}, \eqref{eq:Si1121luSi1132},
and computations in Section \ref{section:appendix-The general case},
\begin{align}
\omega_{1}^{[j]}\omega^{[j]}_{0}\lu&=\omega^{[j]}_{0}\lu,\nonumber\\
\omega_{1}^{[i]}\omega^{[j]}_{0}\lu&=0,\nonumber\\
\Har_{3}^{[j]}\omega^{[j]}_{0}\lu&=\omega^{[j]}_{0}\lu,\nonumber\\
\Har_{3}^{[i]}\omega^{[j]}_{0}\lu&=0,\nonumber\\
S_{ij}(1,1)_{1}\omega^{[j]}_{0}\lu&=-S_{ij}(1,2)_{1}\lu,\nonumber\\
S_{ij}(1,2)_{2}\omega^{[j]}_{0}\lu&=2S_{ij}(1,2)_{1}\lu,\nonumber\\
S_{ij}(1,3)_{3}\omega^{[j]}_{0}\lu&=-3S_{ij}(1,2)_{1}\lu,\nonumber\\
\omega^{[i]}_{1}S_{ij}(1,2)_{1}\lu&=S_{ij}(1,2)_{1}\lu,\nonumber\\
\omega^{[j]}_{1}S_{ij}(1,2)_{1}\lu&=0,\nonumber\\
\Har_{3}^{[i]}S_{ij}(1,2)_{1}\lu&=S_{ij}(1,2)_{1 }\lu,\nonumber\\
\Har_{3}^{[j]}S_{ij}(1,2)_{1}\lu&=0,\nonumber\\
S_{ij}(1,1)_{1}S_{ij}(1,2)_{1}\lu&=\textcolor{black}{-\omega^{[j]}_{0}\lu},\nonumber\\
S_{ij}(1,2)_{2}S_{ij}(1,2)_{1}\lu&=0,\nonumber\\
S_{ij}(1,3)_{3}S_{ij}(1,2)_{1}\lu&=0,\nonumber\\
S_{kj}(1,1)_{1}S_{ij}(1,2)_{1}\lu&=0,\nonumber\\
S_{kj}(1,2)_{2}S_{ij}(1,2)_{1}\lu&=0,\nonumber\\
S_{kj}(1,3)_{3}S_{ij}(1,2)_{1}\lu&=0.
\end{align}
Thus, for each $j=1,\ldots,\rankL$,
$U_j=\Span_{\C}\{\omega^{[j]}_{0}\lu,S_{ij}(1,2)_{1}\lu\ |\ i\neq j\}$
is an $A(M(1)^{+})$-module which is isomorphic to $0$ or 
$M(1)^{-}(0)$. 
Since $\sum_{j=1}^{\rankL}U_j=N_1$, $(W/(M(1)^{+}\cdot (\sum_{j=1}^{\rankL}U_j))_{1}=0$.
Thus $(W/(M(1)^{+}\cdot (\sum_{j=1}^{\rankL}U_j))\cong M(1)^{+}$ and hence
$\Ker \pi=M(1)^{+}\cdot (\sum_{j=1}^{\rankL}U_j)$.
Now the result follows from Corollary \ref{corollary:verma-irreducible}.
\end{proof}

\section{Intertwining operators for $M(1)^{+}$}
\label{section:Intertwining operators for M(1)+}
In this section, we study some intertwining operators for $M(1)^{+}$.
Throughout this section, $\mn$ is a non-zero complex number,
$\alpha$ is a non-zero element of $\fh$ with $\langle\alpha,\alpha\rangle=\mn$,
$K$ is an $M(1)^{+}$-module such that $K=M(1)^{+}\cdot K(0)$, $\mW$ is a weak $M(1)^{+}$-module,
$I(\ ,x) : M(1,\alpha)\times \mK\rightarrow W\db{x}$
is a non-zero intertwining operator, and 
$h^{[1]},\ldots,h^{[\rankL]}$ is an orthonormal basis of $\fh$
such that $\alpha\in \C h^{[1]}$ if $\langle\alpha,\alpha\rangle\neq 0$
and $\alpha\in \C h^{[1]}+\C h^{[2]}$ if $\langle\alpha,\alpha\rangle=0$.
We write 
\begin{align}
E=E(\alpha)
\end{align}
for simplicity and set $\mN= M(1,\alpha)\cdot K$.
 
\begin{lemma}
\label{lemma:N-Basis}
Let $\lE\in\Z$ such that $\lE\geq \epsilon_{I}(E,\lu)$ for all non-zero $\lu\in K(0)$.
The vector space
\begin{align}
\label{eq:B=SpanC(ajExB)}
B=\Span_{\C}\{(a_{j}\ExB)_{\lE+\wt a-j-1}\lu\ |\ \lu\in K(0), \mbox{homogeneous }a\in M(1)^{+}\mbox{ and }j\in\Z\}
\end{align}
is a subspace of $\Omega_{M(1)^{+}}(\mW)$ and is 
an $A(M(1)^{+})$-module.
Moreover the following result holds.
\begin{enumerate}
	\item If $\mn\neq 0, 1/2$ and $1$, then
\begin{align}
	B&=
	\Span_{\C}\{E_{\lE }\lu,(S_{i1}(1,1)_{0}E)_{\lE +1}\lu\ |\ u\in \mK(0), i=2,\ldots,\rankL\}.
\end{align} 
	\item If $\mn=1/2$, then
\begin{align}
	B&=
	\Span_{\C}\{E_{\lE }\lu,(\Har^{[1]}_{1}E)_{\lE +2}\lu,(S_{i1}(1,1)_{0}E)_{\lE +1}\lu\ |\ u\in \mK(0), i=2,\ldots,\rankL\}.
\end{align} 
	\item If $\mn=1$, then
\begin{align}
	B&=
	\Span_{\C}\{E_{\lE }\lu,(S_{i1}(1,1)_{0}E)_{\lE +1}\lu,(S_{i1}(1,2)_{0}E)_{\lE +2}\lu\ |\ u\in \mK(0), i=2,\ldots,\rankL\}.
\end{align} 
\item If $\mn=0$, then
\begin{align}
\label{eq:B=SpanCElElu(omega[1]0E)}
	B&=\Span_{\C}\{E_{\lE }\lu,(\omega^{[1]}_{0}E)_{\lE+1}\lu,(S_{i1}(1,1)_{0}E)_{\lE +1}\lu\ |\ u\in \mK(0), i=3,\ldots,\rankL\}\nonumber\\
&=\Span_{\C}\{E_{\lE }\lu,(S_{i1}(1,1)_{0}E)_{\lE +1}\lu\ |\ u\in \mK(0), i=2,\ldots,\rankL\}.
\end{align} 
\end{enumerate}
The weak submodule $M(1)^{+}\cdot B$ of $\mW$ is spanned 
by the elements
\begin{align}
\label{eq:omega0ra(1)icdotsa(n)}
\textcolor{black}{\omega_{0}^{r}}a^{(1)}_{i_1}\cdots a^{(n)}_{i_n}b
\end{align}
where $r,n\in\Z_{\geq 0}$,
\begin{align}
	\label{eq:a(j)inomega[k],Har[k]-1}
a^{(j)}&\in\{\omega^{[k]},\Har^{[k]}\ |\ k=1,\ldots,\rankL\}\cup 
\{S_{lm}(1,r)\ |\ 1\leq m<l\leq\rankL, r=1,2,3\},\\
	\label{eq:a(j)inomega[k],Har[k]-2}
b&\in B,
\end{align} 
and $i_j\in\Z_{<0}$ for $j=1,\ldots,n$.
\end{lemma}
\begin{proof}
The first part of the results is clear.
In Appendix \ref{section:appendix},
for an element $a$ in the set \eqref{eq:a(j)inomega[k],Har[k]-1}, $b\in B$, 
and $i=0,1,\ldots$, 
we see that $a_{i}b$ is expressed as a linear combination of elements in the form of \eqref{eq:omega0ra(1)icdotsa(n)},
which induces the other results.
In the case that $\mn=2$, we use \eqref{eq:(H0ElE1lu=dfrac194+7lE}.
By \eqref{eq:S21(11)0ExB=-sqrt-1omega0E+2sqrt-1(omega[1]0ExB)}, the two spaces are equal in \eqref{eq:B=SpanCElElu(omega[1]0E)}.
\end{proof}

\subsection{The case that $\langle\alpha,\alpha\rangle\neq 0,1/2,1,2$}
\label{section:Thecasethatlanglealphaalpharangleneq0}
Throughout this subsection, $\mn$ is a complex number
except $0,1/2,1$ and $2$, and
\begin{align}
\langle\alpha,\alpha\rangle&=\mn.
\end{align}
We assume $\alpha\in\C h^{[1]}$ and hence $\langle\alpha, h^{[1]}\rangle^2=\mn$ and $\langle h^{[i]},\lambda\rangle=0$ for all $i=2,\ldots,\rankL$.
For any distinct pair of elements $i,j\in\{1,\ldots,\rankL\}$, a direct computation shows that
\begin{align}
\label{eq:(mn-1)(2mn2-mn+2)Sij(1,1)-2E-1}
0&= (\mn-1) (2 \mn^2-\mn+2) S_{i1}(1,1)_{-2 } E
+(-4) (\mn-1)^{2}\omega^{[i]}_{-1 } S_{i1}(1,1)_{0}E
\nonumber\\&\quad{}
-(\mn+2) (2 \mn-1) \omega_{0 } S_{i1}(1,1)_{-1 } E
+ (4 \mn-1) \omega_{0 }^{2} S_{i1}(1,1)_{0}E
\nonumber\\&\quad{}
+\mn (\mn-1) (2 \mn-5) S_{i1}(2,1)_{-1}E,\\
\label{eq:(mn-1)(2mn2-mn+2)Sij(1,1)-2E-2}
0&=
\mn (4 \mn-1)S_{i1}(1,1)_{-3 } \ExB 
\nonumber\\&\quad{}
-2 \mn (4 \mn-1)\omega^{[i]}_{-1 } S_{i1}(1,1)_{-1 } \ExB 
\nonumber\\&\quad{}
+(\mn-3) (4 \mn-1)\omega^{[i]}_{-2 } (S_{i1}(1,1)_{0}\ExB) 
\nonumber\\&\quad{}
-(\mn+1)\omega_{0 } S_{i1}(1,1)_{-2 } \ExB 
\nonumber\\&\quad{}
+2 (4 \mn-1)\omega_{0 } \omega^{[i]}_{-1 } (S_{i1}(1,1)_{0}\ExB) 
\nonumber\\&\quad{}
-2\omega_{0 } \omega^{[1]}_{-1 } (S_{i1}(1,1)_{0}\ExB) 
\nonumber\\&\quad{}
+\omega_{0 } \omega_{0 } S_{i1}(1,1)_{-1 } \ExB 
\nonumber\\&\quad{}
-\mn (4 \mn-1)S_{i1}(1,2)_{-2 } \ExB 
\nonumber\\&\quad{}
+3 \mn\omega_{0 } S_{i1}(1,2)_{-1 } \ExB 
\nonumber\\&\quad{}
+\mn (4 \mn-1)S_{i1}(1,3)_{-1 } \ExB,\\
\label{eq:(mn-1)(2mn2-mn+2)Sij(1,1)-2E-3}
0&=
3 \mn (\mn-1) (4 \mn-1)S_{i1}(1,1)_{-3 } \ExB 
\nonumber\\&\quad{}
-4 \mn (\mn-2) (4 \mn-1)\omega^{[i]}_{-1 } S_{i1}(1,1)_{-1 } \ExB 
\nonumber\\&\quad{}
+2 (\mn-3) (\mn-2) (4 \mn-1)\omega^{[i]}_{-2 } (S_{i1}(1,1)_{0}\ExB) 
\nonumber\\&\quad{}
+3 (\mn-1) (4 \mn-1)\omega^{[1]}_{-2 } (S_{i1}(1,1)_{0}\ExB) 
\nonumber\\&\quad{}
-6 (\mn-1) (\mn+1)\omega_{0 } S_{i1}(1,1)_{-2 } \ExB 
\nonumber\\&\quad{}
+4 (\mn-2) (4 \mn-1)\omega_{0 } \omega^{[i]}_{-1 } (S_{i1}(1,1)_{0}\ExB) 
\nonumber\\&\quad{}
-12 (\mn-1)\omega_{0 } \omega^{[1]}_{-1 } (S_{i1}(1,1)_{0}\ExB) 
\nonumber\\&\quad{}
+6 (\mn-1)\omega_{0 } \omega_{0 } S_{i1}(1,1)_{-1 } \ExB 
\nonumber\\&\quad{}
-3 (\mn-1)^2 (4 \mn-1)S_{i1}(1,2)_{-2 } \ExB 
\nonumber\\&\quad{}
+3 (\mn-1) (2 \mn+1)\omega_{0 } S_{i1}(1,2)_{-1 } \ExB,
\end{align}
\begin{align}
\label{eq:(mn-1)(2mn2-mn+2)Sij(1,1)-2E-5}
0&=
72 \mn (\mn-1)\omega^{[i]}_{-1 } S_{i1}(1,1)_{-2 } \ExB
\nonumber\\&\quad{}
+3 \mn (\mn+2) (2 \mn-1)\omega^{[i]}_{-2 } S_{i1}(1,1)_{-1 } \ExB
\nonumber\\&\quad{}
-2 (14 \mn^3+\mn^2-8 \mn-16)\omega^{[i]}_{-3 } (S_{i1}(1,1)_{0}\ExB)_{-1 } \lu
\nonumber\\&\quad{}
-12 (\mn-1) (\mn+2) (2 \mn-1)\Har^{[i]}_{-1 } (S_{i1}(1,1)_{0}\ExB)_{-1 } \lu
\nonumber\\&\quad{}
+72 \mn (\mn-1)\omega^{[i]}_{-1 } \omega^{[1]}_{-1 } (S_{i1}(1,1)_{0}\ExB)_{-1 } \lu
\nonumber\\&\quad{}
-3 (2 \mn^2-21 \mn+22)\omega_{0 } \omega^{[i]}_{-2 } (S_{i1}(1,1)_{0}\ExB)_{-1 } \lu
\nonumber\\&\quad{}
-36 (\mn-1)\omega_{0 } \omega_{0 } \omega^{[i]}_{-1 } (S_{i1}(1,1)_{0}\ExB)_{-1 } \lu
\nonumber\\&\quad{}
-36 \mn (\mn-1) (2 \mn+1)\omega^{[i]}_{-1 } S_{i1}(1,2)_{-1 } \ExB.
\end{align}
\begin{lemma}
Let $\lE\in\Z$ such that $\lE\geq \epsilon_{I}(\ExB,\lu)$ for all non-zero $\lu\in K(0)$.
Then, for $\lu\in K(0)$ and $i=2,\ldots,\rankL$,
\begin{align}
\label{eq:(lE+1)((4mn-1)lE-2mn2+7mn-2)(Si1(1,1)-0}
	0&=
	(\lE+1)((4\mn-1)\lE-2\mn^2+7\mn-2)(S_{i1}(1,1)_{0}E)_{\lE+1 } \lu
	\nonumber\\&\quad{}
	+\big((2\mn^2+3\mn-2)\lE-2\mn^2+7\mn-2\big)E_{\lE } S_{i1}(1,1)_{1 } \lu
	\nonumber\\&\quad{}
	 -4(\mn-1)^2(S_{i1}(1,1)_{0}E)_{\lE+1 } \omega^{[1]}_{1 } \lu
	\nonumber\\&\quad{}
	+\mn(\mn-1)(2\mn-5)E_{\lE } S_{i1}(1,2)_{2 } \lu,\\
\label{eq:(lE+1)((4mn-1)lE-2mn2+7mn-2)(Si1(1,1)-1}
0&=
(t+1) ((3 \mn-1) t-\mn^2+5 \mn-2)(S_{i1}(1,1)_{0}\ExB)_{t+1 } \lu
\nonumber\\&\quad{}
+\big((\mn-1) t^2+(-3 \mn^2+10 \mn-5) t+8 \mn^3-19 \mn^2+17 \mn-4\big)\ExB_{t } S_{i1}(1,1)_{1 } \lu
\nonumber\\&\quad{}
-2 (\mn-1) (4 \mn-1) (t+1)(S_{i1}(1,1)_{0}\ExB)_{t+1 } \omega^{[i]}_{1 } \lu
\nonumber\\&\quad{}
-2 \mn (\mn-1) (4 \mn-1)\ExB_{t } S_{i1}(1,1)_{1 } \omega^{[i]}_{1 } \lu
\nonumber\\&\quad{}
-\mn (\mn-1) (3 t-13 \mn+13)\ExB_{t } S_{i1}(1,2)_{2 } \lu
\nonumber\\&\quad{}
+2 (\mn-1) (t-3 \mn+4)(S_{i1}(1,1)_{0}\ExB)_{t+1 } \omega^{[1]}_{1 } \lu
\nonumber\\&\quad{}
+\mn (\mn-1) (4 \mn-1)\ExB_{t } S_{i1}(1,3)_{3 } \lu,
\end{align}
\begin{align}
\label{eq:(lE+1)((4mn-1)lE-2mn2+7mn-2)(Si1(1,1)-2}
0&=
(t+1) ((10 \mn^2-30 \mn+8) t-4 \mn^3+29 \mn^2-53 \mn+16)(S_{i1}(1,1)_{0}\ExB)_{t+1 } \lu
\nonumber\\&\quad{}
+\big((6 \mn^2-12 \mn+6) t^2+(-8 \mn^3+36 \mn^2-70 \mn+30) t\nonumber\\
&\qquad{}+16 \mn^4-72 \mn^3+125 \mn^2-105 \mn+24\big)\ExB_{t } S_{i1}(1,1)_{1 } \lu
\nonumber\\&\quad{}
-4 (\mn-2) (\mn-1) (4 \mn-1) (t+1)(S_{i1}(1,1)_{0}\ExB)_{t+1 } \omega^{[i]}_{1 } \lu
\nonumber\\&\quad{}
-4 \mn (\mn-2) (\mn-1) (4 \mn-1)\ExB_{t } S_{i1}(1,1)_{1 } \omega^{[i]}_{1 } \lu
\nonumber\\&\quad{}
-(\mn-1) ((6 \mn^2-3 \mn-3) t-20 \mn^3+69 \mn^2-58 \mn-3)\ExB_{t } S_{i1}(1,2)_{2 } \lu
\nonumber\\&\quad{}
+6 (\mn-1)^2 (2 t-4 \mn+9)(S_{i1}(1,1)_{0}\ExB)_{t+1 } \omega^{[1]}_{1 } \lu,
	\end{align}
\begin{align}
\label{eq:(lE+1)((4mn-1)lE-2mn2+7mn-2)(Si1(1,1)-4}
0&=
(t+1) ((144 \mn^3+12 \mn^2-60 \mn+12) t^2
\nonumber\\&\qquad{}+(-232 \mn^4+878 \mn^3+303 \mn^2-503 \mn+94) t
\nonumber\\&\qquad{}+88 \mn^5-586 \mn^4+913 \mn^3+593 \mn^2-716 \mn+140)(S_{i1}(1,1)_{0}\ExB)_{t+1 } \lu
\nonumber\\&\quad{}
+\big((144 \mn^4+12 \mn^3-60 \mn^2+12 \mn) t^2
\nonumber\\&\qquad{}+(-280 \mn^5+506 \mn^4+999 \mn^3-599 \mn^2-182 \mn+96) t
\nonumber\\&\qquad{}+176 \mn^6-220 \mn^5-960 \mn^4+2167 \mn^3-485 \mn^2-342 \mn+96\big)\ExB_{t } S_{i1}(1,1)_{1 } \lu
\nonumber\\&\quad{}
-2 (\mn-1)^2 (4 \mn-1) (t+1) ((6 \mn-6) t+14 \mn^2+39 \mn-26)(S_{i1}(1,1)_{0}\ExB)_{t+1 } \omega^{[i]}_{1 } \lu
\nonumber\\&\quad{}
-2 \mn (\mn-1)^2 (4 \mn-1) (26 \mn^2+27 \mn-26)\ExB_{t } S_{i1}(1,1)_{1 } \omega^{[i]}_{1 } \lu
\nonumber\\&\quad{}
-\mn (\mn-1) ((96 \mn^3+120 \mn^2-132 \mn+24) t\nonumber\\
&\qquad{}-88 \mn^4-790 \mn^3+1695 \mn^2-227 \mn-158)\ExB_{t } S_{i1}(1,2)_{2 } \lu
\nonumber\\&\quad{}
-24 (\mn-1)^2 ((4 \mn^2-\mn) t-4 \mn^3+25 \mn^2+8 \mn-8)(S_{i1}(1,1)_{0}\ExB)_{t+1 } \omega^{[1]}_{1 } \lu
\nonumber\\&\quad{}
+24 \mn (\mn-1)^2 (4 \mn-1)^2\ExB_{t } S_{i1}(1,3)_{3 } \lu
\nonumber\\&\quad{}
-4 (\mn-1)^3 (\mn+2) (2 \mn-1) (4 \mn-1)(S_{i1}(1,1)_{0}\ExB)_{t+1 } H^{[i]}_{3 } \lu
\nonumber\\&\quad{}
-24 \mn^2 (\mn-1)^2 (4 \mn-1)\ExB_{t } S_{i1}(1,1)_{1 } \omega^{[1]}_{1 } \lu
\nonumber\\&\quad{}
+24 \mn (\mn-1)^3 (4 \mn-1)(S_{i1}(1,1)_{0}\ExB)_{t+1 } \omega^{[1]}_{1 } \omega^{[i]}_{1 } \lu
\nonumber\\&\quad{}
-12 \mn (\mn-1)^3 (2 \mn+1) (4 \mn-1)\ExB_{t } S_{i1}(1,2)_{2 } \omega^{[i]}_{1 } \lu.
\end{align}
\end{lemma}
\begin{proof}
By Lemma \ref{lemma:bound-index-S}, $\epsilon(S_{i1}(1,1)_{0}E,\lu)\leq \lE+1$ for all non-zero $\lu\in K(0)$.
By taking the $(\lE+3)$-th action of \eqref{eq:(mn-1)(2mn2-mn+2)Sij(1,1)-2E-1} on $\lu$ and using the results in Section \ref{section:Norm-nonzero}, 
the same argument as in the proof of \eqref{eq:0=(omega-3ExB)lE+2lom+2lu=}--\eqref{eq:(Har-1ExB)lE+2lom+2lu=} in Lemma \ref{lemma:r=1-s=3} shows \eqref{eq:(lE+1)((4mn-1)lE-2mn2+7mn-2)(Si1(1,1)-0}.
Taking the $(\lE+4)$-th actions  of \eqref{eq:(mn-1)(2mn2-mn+2)Sij(1,1)-2E-2} and \eqref{eq:(mn-1)(2mn2-mn+2)Sij(1,1)-2E-3}, and  the $(\lE+5)$-th action of \eqref{eq:(mn-1)(2mn2-mn+2)Sij(1,1)-2E-5}
 on $\lu$,
we have \eqref{eq:(lE+1)((4mn-1)lE-2mn2+7mn-2)(Si1(1,1)-1}--\eqref{eq:(lE+1)((4mn-1)lE-2mn2+7mn-2)(Si1(1,1)-4}.
\end{proof}
\subsubsection{The case that $K(0)\cong \C e^{\lambda}$ for some $\lambda\in\fh\setminus\{0\}$}

We need some preparation for the case that $\mK(0)\cong \C e^{\lambda}$ as $A(M(1)^{+})$-modules for some $\lambda\in\fh\setminus\{0\}$.
Let $\lu_{\lambda}$ be the element of $\mK(0)$ corresponding to $e^{\lambda}$.
We write 
\begin{align}
\lE=\epsilon_{I}(\ExB,\lu_{\lambda})
\end{align}
for simplicity. For $i=2,\ldots,\rankL$, since $S_{i1}(1,1)_{1}\lu_{\lambda}=\langle\lambda,h^{[i]}\rangle\langle\lambda,h^{[1]}\rangle \lu_{\lambda}$,
\begin{align}
\epsilon(S_{i1}(1,1)_{0}E,\lu_{\lambda})&\leq \lE+1
\end{align}
by Lemma \ref{lemma:bound-index-S} and
\begin{align}
\epsilon(\Har^{[1]}_{0}E,\lu_{\lambda})&\leq \lE+3
\end{align}
by Lemma \ref{lemma:bound-H0-1E}.
Thus, we take the vector space $B$
in Lemma \ref{lemma:N-Basis} for this $\lE$.

For $i=0,1,\ldots$, let us define $M(1)^{+}$-modules $N^{\seq{i}}$, $A(M(1)^{+})$-modules $B^{\seq{i}}$, integers $\lE^{\seq{i}}$, and intertwining operators $I^{\seq{i}} :  M(1,\alpha)\times \mK\rightarrow N^{\seq{i}}\db{x}$
inductively as follows: $N^{\seq{0}}=\mN$, $B^{\seq{0}}=B$, $\lE^{\seq{0}}=\lE$, and $I^{\seq{0}}=I$. 
For $N^{\seq{i}}$,  $B^{\seq{i}}$, and $I^{\seq{i}}$, 
define the $M(1)^{+}$-module
\begin{align}
\label{eq:Nseqi+1=Nseqi}
N^{\seq{i+1}}&=N^{\seq{i}}/(M(1)^{+}\cdot B^{\seq{i}}).
\end{align}
Suppose $N^{\seq{i+1}}\neq 0$. Using the natural projection $\pr : N^{\seq{i}}\rightarrow N^{\seq{i+1}}$,
we define the non-zero intertwining operator 
\begin{align}
I^{\seq{i+1}}(\mbox{ },x)&=\pr \circ I^{\seq{i}}(\mbox{ },x) : M(1,\alpha)\times \mK\rightarrow N^{\seq{i+1}}\db{x}.
\end{align}
We write
\begin{align}
\lE^{\seq{i+1}}&=\epsilon_{I^{\seq{i+1}}}(\ExB,\lu_{\lambda})\mbox{ and }\nonumber\\
I^{\seq{i+1}}(u,x)v&=\sum_{j\in\Z}I^{\seq{i+1}}(u;j)v x^{-j-1}
\end{align}
for $u\in M(1,\alpha)$ and $\lu\in K$
and define 
\begin{align}
\label{eq:BigIseqi+1Useqi+1}
&B^{\seq{i+1}}\nonumber\\
&=\Span_{\C}\Big\{I^{\seq{i+1}}(a_{j}\ExB;\lE^{\seq{i+1}}+\wt a-j-1)\lu_{\lambda}\ \Big|\ \begin{array}{l}\mbox{homogeneous }a\in M(1)^{+}\\\mbox{ and }j\in\Z\end{array}\Big\}.
\end{align}
If $N^{\seq{i+1}}=0$, then we define $N^{\seq{j}}=0$, $B^{\seq{j}}=0$, and $I^{\seq{j}}=0$ for all $j\geq i+1$.
Note that 
\begin{align}
\label{eq:LEtildeLE}
\lE=\lE^{\seq{0}}>\lE^{\seq{1}}>\lE^{\seq{2}}>\cdots.
\end{align}

A direct computation shows that for any pair of distinct elements $i,j\in\{2,\ldots,\rankL\}$, 
\begin{align}
S_{i j}(1,1)_{0}E&=S_{i j}(1,1)_{1}E=0,\nonumber\\
S_{i j}(1,1)_{1}E_{\lE}e^{\lambda}&=\langle h^{[i]},\lambda\rangle\langle h^{[j]},\lambda\rangle E_{\lE}e^{\lambda},\nonumber\\
S_{i 1}(1,1)_{1}E_{\lE}e^{\lambda}
&=(S_{i 1}(1,1)_{0}E)_{\lE+1}e^{\lambda}+
\langle h^{[i]},\lambda\rangle\langle h^{[1]},\lambda\rangle E_{\lE}e^{\lambda}.\nonumber\\
\end{align}
Thus, if $(S_{i j}(1,1)_{0}E)_{\lE+1}\lu_{\lambda}=\wc \langle h^{[i]},\lambda\rangle
\langle h^{[j]},\lambda\rangle 
E_{\lE}\lu_{\lambda}$ for some $c\in\C$, then
\begin{align}
	\label{eq:Si1111ElEelambda}
	S_{i j}(1,1)_{1}E_{\lE}\lu_{\lambda}
	&=(S_{i j}(1,1)_{0}E)_{\lE+1}\lu_{\lambda}+
	\langle h^{[i]},\lambda\rangle\langle h^{[j]},\lambda\rangle E_{\lE}\lu_{\lambda}\nonumber\\
	&=(c+1)\langle h^{[i]},\lambda\rangle\langle h^{[j]},\lambda\rangle E_{\lE}\lu_{\lambda}\nonumber\\
	&=\langle h^{[i]},\lambda\rangle\langle h^{[j]},\lambda+\wc\lambda\rangle E_{\lE}\lu_{\lambda}\nonumber\\
	&=\langle h^{[i]},\lambda+\wc\langle h^{[j]},\lambda\rangle h^{[j]}\rangle\langle h^{[j]},\lambda+\wc\langle h^{[j]},\lambda\rangle h^{[j]}
	\rangle E_{\lE}\lu_{\lambda}.
\end{align}

\begin{lemma}
\label{lemma:intertwining-lambda}
For the vector space $B$ with $\lE=\epsilon_{I}(\ExB,\lu_{\lambda})$ in Lemma \ref{lemma:N-Basis},
the following results hold:
\begin{enumerate}
\item If $\langle\alpha,\lambda\rangle=0$, then $\lE=-1$ and
\begin{align}
B&\cong \C e^{\alpha+\lambda},
\C e^{\alpha-\lambda},\mbox{ or }
\C e^{\alpha+\lambda}\oplus 
\C e^{\alpha-\lambda}
\end{align}
as $A(M(1)^{+})$-modules and 
\begin{align}
\label{eq:nonzero-mN-structure-1}
\mN&=M(1)^{+}\cdot B\nonumber\\
&\cong M(1,\lambda+\alpha),M(1,\lambda-\alpha),\mbox{ or }M(1,\lambda+\alpha)\oplus M(1,\lambda-\alpha)
\end{align}
as  $M(1)^{+}$-modules.
\item
If $\langle\alpha,\lambda\rangle\neq 0$ 
and $(1+\lE)^2- \langle \alpha,\lambda\rangle^2=0$, then
\begin{align}
B\cong \C e^{\alpha+\lambda}\mbox{ or }\C e^{\alpha-\lambda}
\end{align}
as $A(M(1)^{+})$-modules.
Furthermore, if $\lambda\neq \pm\alpha$, 
\begin{align}
\label{eq:nonzero-mN-structure-2}
\mN&\cong M(1,\lambda+\alpha),M(1,\lambda-\alpha),\mbox{ or }M(1,\lambda+\alpha)\oplus M(1,\lambda-\alpha)
\end{align}
as  $M(1)^{+}$-modules.
\item
If 
$(1+\lE)^2- \langle \alpha,\lambda\rangle^2\neq 0$, then $\lE=\mn-2$, $\lambda=\pm\alpha$, and $B\cong M(1)^{-}(0)$ as $A(M(1)^{+})$-modules.
\end{enumerate}
\end{lemma}
\begin{proof}
By Lemma \ref{lemma:N-Basis}, $B$ is spanned by $E_{\lE}\lu_{\lambda}$ and 
$(S_{21}(1,1)_{0}E)_{\lE+1}\lu_{\lambda},\ldots, (S_{\rankL 1}(1,1)_{0}E)_{\lE+1}\lu_{\lambda}$.

\begin{enumerate}
\item Assume $\langle\alpha,\lambda\rangle=0$. Since
$\omega^{[1]}_{0}\lu_{\lambda}=\omega^{[1]}_{1}\lu_{\lambda}=0$ and $\omega^{[1]}_{0}E=\omega_{0}E$,
it follows from Lemma \ref{lemma:structure-Vlattice-M1} (1) that 
$(\omega^{[1]}_{1}-(\lE+1)^2/(2\mn))\lu_{\lambda}=0$ and hence $\lE=-1$. 
Since $\lattice$ is non-degenerate, 
we may assume $\langle h^{[2]},\lambda\rangle\neq 0$.
For $i=3,\ldots,\rankL$, by Lemma \ref{lemma:comm-change} and
\begin{align}
\label{eq:Si2(1,1)0S21(1,1)0E=dfrac-1mn-1}
&S_{i2}(1,1)_{0}S_{21}(1,1)_{0}E
=\dfrac{-1}{\mn-1}\omega_{0}S_{\wi 1}(1,1)_{0}E+\dfrac{\mn}{\mn-1}S_{\wi 1}(1,1)_{-1}E,
\nonumber\\
&S_{i2}(1,1)_{-1}S_{21}(1,1)_{0}E\nonumber\\
&= \dfrac{1}{2\mn(\mn-1)^2}\omega_0^{2}S_{\wi 1}(1,1)_{0}E+\dfrac{1}{2(\mn-1)^2}\omega_0S_{\wi 1}(1,1)_{-1}E\nonumber\\
&\quad{}+\dfrac{1}{2-2\mn}S_{\wi 1}(1,2)_{-1}E+
\dfrac{-1}{2-2\mn}S_{\wi 1}(1,1)_{-2}E+2 \omega^{[2]}_{-1}S_{i1}(1,1)_{0}E,\nonumber\\
\end{align}
which are obtained by a direct computation,
we have
\begin{align}
\label{eq:(S21(1,1)0E)0Si2(1,1)1lulambdanonumber}
&(S_{21}(1,1)_{0}E)_{0}S_{i2}(1,1)_{1}\lu_{\lambda}\nonumber\\
&=(S_{i2}(1,1)_{-1}S_{21}(1,1)_{0}E)_{2}\lu_{\lambda}+(S_{i2}(1,1)_{0}S_{21}(1,1)_{0}E)_{1}\lu_{\lambda}\nonumber\\
&=\langle \lambda,h^{[2]}\rangle^2
(S_{i1}(1,1)_{0}E)_{0}\lu_{\lambda}
\end{align}
and hence
\begin{align}
\label{eq:langlelambdah[i]rangle(21(11)0E}
\langle \lambda,h^{[i]}\rangle
(S_{21}(1,1)_{0}E)_{0}\lu_{\lambda}
&=\langle \lambda,h^{[2]}\rangle
(S_{i1}(1,1)_{0}E)_{0}\lu_{\lambda}.
\end{align}
Thus, $U$ is spanned by $E_{-1}\lu_{\lambda}$ and $(S_{21}(1,1)_{0}E)_{0}\lu_{\lambda}$.
By using results in Section \ref{section:Norm-nonzero} and 
the commutation relation $[a_i,b_j]=\sum_{k=0}^{\infty}\binom{i}{k}(a_{k}b)_{i+j-k}$
for $a,b\in V_{\lattice}^{+}$ and $i,j\in\Z$, a direct computation shows that
\begin{align}
\label{eq:omega[1]1E-1lulambda}
\omega^{[1]}_{1}E_{ -1}\lu_{\lambda}&=\frac{\mn}{2}E_{ -1}\lu_{\lambda},\nonumber\\
\omega^{[2]}_{1}E_{ -1}\lu_{\lambda}&=E_{ -1}\omega^{[2]}_{1}\lu_{\lambda},\nonumber\\
\omega^{[1]}_{1}(S_{2 1}(1,1)_{0}E)_{ 0}\lu_{\lambda}&=\frac{\mn}{2}(S_{2 1}(1,1)_{0}E)_{ -1}\lu_{\lambda},\nonumber\\
\omega^{[2]}_{1}(S_{2 1}(1,1)_{0}E)_{ 0}\lu_{\lambda}&=(S_{2 1}(1,1)_{0}E)_{ 0}\omega^{[2]}_{1}\lu_{\lambda},\nonumber\\
\Har^{[i]}_{3}E_{ -1}\lu_{\lambda}&=0\mbox{ for }i=1,\ldots,\rankL,\nonumber\\
\Har^{[i]}_{3}(S_{2 1}(1,1)_{0}E)_{ 0}\lu_{\lambda}&=0\mbox{ for }i=1,\ldots,\rankL,\nonumber\\
S_{2 1}(1,1)_{1}E_{ -1}\lu_{\lambda}&=(S_{2 1}(1,1)_{0}E)_{ 0}\lu_{\lambda},\nonumber\\
S_{2 1}(1,1)_{1}(S_{2 1}(1,1)_{0}E)_{ 0}\lu_{\lambda}&=2\mn E_{ -1}\omega^{[2]}_{1}\lu_{\lambda},\nonumber\\
S_{2 1}(1,2)_{2}E_{ -1}\lu_{\lambda}&=-(S_{2 1}(1,1)_{0}E)_{ 0}\lu_{\lambda},\nonumber\\
S_{2 1}(1,2)_{2}(S_{2 1}(1,1)_{0}E)_{ 0}\lu_{\lambda}&=-2\mn E_{ -1}\omega^{[2]}_{1}\lu_{\lambda},\nonumber\\
S_{2 1}(1,3)_{3}E_{ -1}\lu_{\lambda}&=(S_{2 1}(1,1)_{0}E)_{ 0}\lu_{\lambda},\nonumber\\
S_{2 1}(1,3)_{3}(S_{2 1}(1,1)_{0}E)_{ 0}\lu_{\lambda}&=2\mn E_{ -1}\omega^{[2]}_{1}\lu_{\lambda}.
\end{align}
The actions of the other $S_{ij}(1,k)_{k}, k=1,2,3,$ can be obtained by using \eqref{eq:langlelambdah[i]rangle(21(11)0E}
and \eqref{eq:omega[1]1E-1lulambda}.
Thus,
\begin{align}
\C \big(E_{-1}\lu_{\lambda}\pm \langle h^{[1]},\alpha\rangle
\langle h^{[2]},\lambda\rangle(S_{2 1}(1,1)_{0}E)_{0}\lu_{\lambda}\big)\cong 0
\mbox{ or }\C e^{\lambda\pm\alpha}
\end{align}
as $A(M(1)^{+})$-modules.
We have obtained the desired result.
\item
Suppose $\langle \alpha,\lambda\rangle\neq 0$ and $(1+\lE)^2- \langle \alpha,\lambda\rangle^2=0$. Then, $\lE\neq -1$
and it follows from \eqref{eq:norm2Si1(1,1)1lu=-Si1(1,2)2lu} and \eqref{eq:(lE+1)((4mn-1)lE-2mn2+7mn-2)(Si1(1,1)-0} that
\begin{align}
\label{eq:0=12mn22mnmn22lEmnlE-1}
0&=
(-1+2 \mn) (-2-2 \mn+\mn^2-(2+\mn) \lE)\nonumber\\
&\quad{}\times ((1+\lE)(S_{i 1}(1,1)_{0}E)_{\lE+1}\lu_{\lambda}+\mn E_{\lE}S_{\wi 1}(1,1)_{1}\lu_{\lambda})
\end{align}
for $i=2,\ldots,\rankL$.

Assume $-2-2 \mn+\mn^2-(2+\mn) \lE=0$.  Since $\lE$ is an integer, $\mn\neq 1/4,-2$.
By \eqref{eq:norm2Si1(1,1)1lu=-Si1(1,2)2lu}, \eqref{eq:(lE+1)((4mn-1)lE-2mn2+7mn-2)(Si1(1,1)-1}, and \eqref{eq:(lE+1)((4mn-1)lE-2mn2+7mn-2)(Si1(1,1)-2},
we have
\begin{align}
\label{eq:0=12mn22mnmn22lEmnlE-2}
0&=\mn (4 \mn-1) \big((-1+\mn)(S_{i 1}(1,1)_{0}E)_{\lE+1}+(2+\mn)E_{\lE}S_{i1}(1,1)_1\big)\nonumber\\
&\quad{}\times ((\mn+2)^2\omega^{[i]}_{1}-1)\lu_{\lambda}&\mbox{ and }\nonumber\\
0&=\mn (\mn-2) (4 \mn-1) \big((-1+\mn)(S_{i 1}(1,1)_{0}E)_{\lE+1}+(2+\mn)E_{\lE}S_{i1}(1,1)_1\big)\nonumber\\
&\quad{}\times (4(\mn+2)^2\omega^{[i]}_{1}-(5 - 2\mn))\lu_{\lambda}.
\end{align}
Thus
\begin{align}
	0&=\mn (\mn-2) (\mn+2)^2 (2 \mn-1) (4 \mn-1) \nonumber\\
	&\quad{}\times  \big((-1+\mn)(S_{i 1}(1,1)_{0}E)_{\lE+1}+(2+\mn)E_{\lE}S_{i1}(1,1)_1\big)\lu\nonumber\\
	&= (\mn-2) (\mn+2)^3 (2 \mn-1) (4 \mn-1) \nonumber\\
&\quad{}\times 	\big((1+\lE)(S_{i 1}(1,1)_{0}E)_{\lE+1}\lu_{\lambda}+\mn E_{\lE}S_{i1}(1,1)_1\lu_{\lambda}\big),
\end{align}
where we have used $\mn (-1+\mn)=(\mn+2)(1+\lE)$,
and hence 
\begin{align}
	\label{eq:(1+lE)(Si1(1,1)0E)lE+1}
	0&=
	(1+\lE)(S_{i 1}(1,1)_{0}E)_{\lE+1}\lu_{\lambda}+\mn E_{\lE}S_{i1}(1,1)_1\lu_{\lambda}.
	\end{align}
By \eqref{eq:0=12mn22mnmn22lEmnlE-1} and \eqref{eq:(1+lE)(Si1(1,1)0E)lE+1},
\begin{align}
\label{eq:0=12mn22mnmn22lEmnlE-3}
(S_{i 1}(1,1)_{0}E)_{\lE+1}\lu_{\lambda}&=\frac{-\mn}{1+\lE}E_{\lE}S_{\wi 1}(1,1)_{1}\lu_{\lambda}
\end{align}
and hence by	\eqref{eq:norm2Si1(1,1)1lu=-Si1(1,2)2lu} and \eqref{eq:Si1111ElEelambda}
\begin{align}
&S_{i 1}(1,1)_{1}E_{\lE}\lu_{\lambda}\nonumber\\
&=\langle h^{[i]},\lambda+\dfrac{-\mn}{1+\lE}\langle h^{[1]},\lambda\rangle h^{[1]}\rangle\langle h^{[1]},\lambda+\dfrac{-\mn}{1+\lE}\langle h^{[1]},\lambda\rangle h^{[1]}
\rangle E_{\lE}\lu_{\lambda}.
\end{align}

Since $(-\mn\langle h^{[1]},\lambda\rangle)^2/(1+\lE)^2=p=\langle\alpha,\alpha\rangle$,
$B\cong \C e^{\lambda+\alpha}$ or $\C e^{\lambda-\alpha}$
as $A(M(1)^{+})$-modules.

\item
Suppose  $(1+\lE)^2\neq \langle \alpha,\lambda\rangle^2=\mn\langle h^{[1]},\lambda\rangle^2$.
Since 
\begin{align}
\label{eq:omega[1]1-dfrac(lE+1)22mn)lvlambda}
(\omega^{[1]}_{1}-\dfrac{(\lE+1)^2}{2\mn})\lu_{\lambda}
&=\dfrac{\langle \alpha, \lambda\rangle^2-(\lE+1)^2}{2\mn}\lu_{\lambda},
\end{align}
it follows from Lemmas \ref{lemma:structure-Vlattice-M1} (1) and \ref{lemma:structure-Vlattice-M1-norm2} (1)  that $\lE=\mn-2$ and $\langle h^{[1]},\lambda\rangle^2=\mn$.
It follows from \eqref{eq:norm2Si1(1,1)1lu=-Si1(1,2)2lu} and \eqref{eq:(lE+1)((4mn-1)lE-2mn2+7mn-2)(Si1(1,1)-0} that
$\langle h^{[\wi]},\lambda\rangle=0$ for all $i=2,\ldots,\rankL$
and hence $\lambda=\pm\alpha$.
A direct computation shows that
for $i=2,\ldots,\rankL$ and
$j=2,\ldots,\rankL$ with $j\neq i$ 
\begin{align}
\omega^{[1]}_{1}E_{\lE}\lu_{\lambda}&=\Har^{[1]}_{3}E_{\lE}\lu_{\lambda}=E_{\lE}\lu_{\lambda},\nonumber\\
\omega^{[i]}_{1}E_{\lE}\lu_{\lambda}&=\Har^{[i]}_{3}E_{\lE}\lu_{\lambda}=0,\nonumber\\
S_{\wi 1}(1,1)_{1}E_{\lE}\lu_{\lambda}&=(S_{\wi 1}(1,1)_{0}E)_{\lE+1}\lu_{\lambda},\nonumber\\
S_{\wi 1}(1,2)_{2}E_{\lE}\lu_{\lambda}&=-2(S_{\wi 1}(1,1)_{0}E)_{\lE+1}\lu_{\lambda},\nonumber\\
S_{\wi 1}(1,3)_{3}E_{\lE}\lu_{\lambda}&=3(S_{\wi 1}(1,1)_{0}E)_{\lE+1}\lu_{\lambda},\nonumber\\
\omega^{[1]}_{1}(S_{\wi 1}(1,1)_{0}E)_{\lE+1}\lu_{\lambda}&=\Har^{[1]}_{3}(S_{\wi 1}(1,1)_{0}E)_{\lE+1}\lu_{\lambda}=0,\nonumber\\
\omega^{[i]}_{1}(S_{\wi 1}(1,1)_{0}E)_{\lE+1}\lu_{\lambda}&=\Har^{[i]}_{3}(S_{\wi 1}(1,1)_{0}E)_{\lE+1}\lu_{\lambda}=
(S_{\wi 1}(1,1)_{0}E)_{t+1 }\lu_{\lambda},\nonumber\\ 
S_{\wi 1}(1,1)_{1}(S_{\wi 1}(1,1)_{0}E)_{\lE+1}\lu_{\lambda}
&=E_{\lE}\lu_{\lambda},\nonumber\\
S_{\wi 1}(1,2)_{2}(S_{\wi 1}(1,1)_{0}E)_{\lE+1}\lu_{\lambda}&=
S_{\wi 1}(1,3)_{3}(S_{\wi 1}(1,1)_{0}E)_{\lE+1}\lu_{\lambda}=0,\nonumber\\
\omega^{[j]}_{1}(S_{\wi 1}(1,1)_{0}E)_{\lE+1}\lu_{\lambda}&=0,\nonumber\\
\textcolor{black}{S_{j1}(1,1)_{1}(S_{i1}(1,1)_{0}E)_{\lE+1}\lu_{\lambda}}&=0.
\end{align}
Thus, $B\cong M(1)^{-}(0)$ as $A(M(1)^{+})$-modules.
\end{enumerate}

We shall use the symbols prepared in \eqref{eq:Nseqi+1=Nseqi}--\eqref{eq:BigIseqi+1Useqi+1}.
By (1),(2), and (3) above, 
if $\langle\alpha,\lambda\rangle=0$, then $N^{\seq{1}}=0$, namely, $N=M(1)^{+}\cdot B$,
and if $\lambda\neq \pm\alpha$, then $B\cong \C e^{\lambda\pm\alpha}$, $B^{\seq{1}}\cong \C e^{\lambda\mp\alpha}$, and 
$N^{\seq{2}}=0$.
By Corollary \ref{corollary:verma-irreducible}, \eqref{eq:nonzero-mN-structure-1}--\eqref{eq:nonzero-mN-structure-2} hold.
\end{proof}

\subsubsection{The case that $K(0)\cong M(1)^{-}(0)$}
\label{section:The case that K(0)cong M(1)-(0)}
Assume $\mK(0)\cong M(1)^{-}(0)$ as $A(M(1)^{+})$-modules.
For $i=1,\ldots,\rankL$,
$u^{[i]}$ denotes the element of $\mK(0)$ 
corresponding to $h^{[i]}(-1)\vac$.
In this case, it follows from Lemma \ref{lemma:structure-Vlattice-M1} (1) that
\begin{align}
\label{eq:epsilon(ExBh[1](-1)vac)}
\epsilon(\ExB,u^{[1]})&=0,\mbox{ and }\nonumber\\
\epsilon(\ExB,u^{[i]})&=-1\mbox{ for }i=2,\ldots,\rankL.
\end{align}
Thus, we can take the vector space $B$ in Lemma \ref{lemma:N-Basis} for $\lE=0$.
\begin{lemma}
\label{lemma:intertwining-M1minus}
If $K(0)\cong M(1)^{-}(0)$, then
$B\cong\C e^{\alpha}$ as $A(M(1)^{+})$-modules,
and $\mN=M(1)^{+}\cdot B\cong M(1,\alpha)$.
\end{lemma}
\begin{proof}
For $i=2,\ldots,\rankL$, since 
$S_{i1}(1,1)_{1}u^{[1]}=u^{[i]}$, it follows from  Lemma \ref{lemma:bound-index-S} and \eqref{eq:epsilon(ExBh[1](-1)vac)}
that
\begin{align}
\epsilon(S_{i1}(1,1)_{0}E,u^{[1]})\leq 1.
\end{align}
By \eqref{eq:(lE+1)((4mn-1)lE-2mn2+7mn-2)(Si1(1,1)-0},
\begin{align}
0&=
-3 (-2+\mn) (-1+2 \mn)(S_{i1}(1,1)_{0}E)_{1}u^{[1]}\nonumber\\
&\quad{}-E_{0}(2 + 3 \mn - 12 \mn^2 + 4 \mn^3) u^{[i]}.
\label{eq:big(32mn12mnSi110E1-1}
\end{align}
It follows from \eqref{eq:(lE+1)((4mn-1)lE-2mn2+7mn-2)(Si1(1,1)-2}, \eqref{eq:(lE+1)((4mn-1)lE-2mn2+7mn-2)(Si1(1,1)-4},
and 
\eqref{eq:big(32mn12mnSi110E1-1} that for $i=2,\ldots,\rankL$,
\begin{align}
0&=(\mn-1) (4 \mn-1) (2 \mn^2-5)\ExB_{0} u^{[i]},\nonumber
\\
0&=
(2 \mn+1) (23 \mn^2-37 \mn-22)\ExB_{0}u^{[i]}
\label{eq:big(32mn12mnSi110E1-2}
\end{align}
and hence 
\begin{align}
E_{0}u^{[i]}&=(S_{i1}(1,1)_{0}E)_{1}u^{[1]}=0.
\end{align}
By \eqref{eq:norm2Si1(1,1)1lu=-Si1(1,2)2lu-2} and the computation in Section \ref{section:Norm-nonzero},
for $i=2,\ldots,\rankL$ and $j=1,\ldots,\rankL$,
we have
\begin{align}
	\omega^{[1]}_{1}\ExB_{0}u^{[1]}&=\dfrac{\mn}{2}\ExB_{0}u^{[1]},\nonumber\\
	\omega^{[i]}_{1}\ExB_{0}u^{[1]}&=\Har^{[j]}_{3}\ExB_{0}u^{[1]}=0,\nonumber\\	
	S_{i1}(1,1)_{1}\ExB_{0}u^{[1]}&=	S_{i1}(1,2)_{2}\ExB_{0}u^{[1]}=	S_{i1}(1,3)_{3}\ExB_{0}u^{[1]}=0
	\end{align}
and hence $B\cong \C e^{\alpha}$ as $A(M(1)^{+})$-modules.
The same argument as in the proof of Lemma \ref{lemma:intertwining-lambda}
shows that $\mN= M(1)^{+}\cdot B\cong M(1,\alpha)$.
\end{proof}

\subsection{The case that $\langle\alpha,\alpha\rangle=2$}
Throughout this subsection, 
\begin{align}
\langle\alpha,\alpha\rangle&=2.
\end{align}
We assume $\alpha\in\C h^{[1]}$ and hence $\langle\alpha, h^{[1]}\rangle^2=2$ and $\langle h^{[i]},\lambda\rangle=0$ for all $i=2,\ldots,\rankL$.

	A direct computation shows that
	\begin{align}
		\label{eq:8Si1(11)2E-4omega[1]-1-0}
		0&=
		8S_{i1}(1,1)_{-2 } \ExB 
		-4\omega^{[1]}_{-1 } (S_{i1}(1,1)_{0}\ExB) 
		-12\omega_{0 } S_{i1}(1,1)_{-1 } \ExB 
		\nonumber\\&\quad{}
		+7\omega_{0 } \omega_{0 } (S_{i1}(1,1)_{0}\ExB) 
		-2S_{i1}(1,2)_{-1 } \ExB,\\
\label{eq:8Si1(11)2E-4omega[1]-1-1}
0&=
198S_{i1}(1,1)_{-4 } \ExB 
-1158\omega^{[i]}_{-2 } S_{i1}(1,1)_{-1 } \ExB 
\nonumber\\&\quad{}
-4444\omega^{[i]}_{-3 } (S_{i1}(1,1)_{0}\ExB) 
-1158\Har^{[i]}_{-1 } (S_{i1}(1,1)_{0}\ExB) 
\nonumber\\&\quad{}
-180\omega^{[1]}_{-1 } S_{i1}(1,1)_{-2 } \ExB 
+2448\omega^{[i]}_{-1 } \omega^{[1]}_{-1 } (S_{i1}(1,1)_{0}\ExB) 
\nonumber\\&\quad{}
+270\omega^{[1]}_{-2 } S_{i1}(1,1)_{-1 } \ExB 
-810\Har^{[1]}_{-1 } (S_{i1}(1,1)_{0}\ExB) 
\nonumber\\&\quad{}
-621S_{i1}(1,1)_{-1 } (\Har_{0}\ExB) 
-180\omega_{0 } S_{i1}(1,1)_{-3 } \ExB 
\nonumber\\&\quad{}
+2415\omega_{0 } \omega^{[i]}_{-2 } (S_{i1}(1,1)_{0}\ExB) 
-135\omega_{0 } \omega^{[1]}_{-2 } (S_{i1}(1,1)_{0}\ExB) 
\nonumber\\&\quad{}
+45\omega_{0 } \omega_{0 } S_{i1}(1,1)_{-2 } \ExB 
-612\omega_{0 } \omega_{0 } \omega^{[i]}_{-1 } (S_{i1}(1,1)_{0}\ExB) 
\nonumber\\&\quad{}
+342S_{i1}(1,2)_{-3 } \ExB 
-3672\omega^{[i]}_{-1 } S_{i1}(1,2)_{-1 } \ExB 
\nonumber\\&\quad{}
+1782\omega_{0 } S_{i1}(1,3)_{-1 } \ExB
	\end{align}
for $i=2,\ldots,\rankL$.
\begin{lemma}
Let $\lE\in\Z$ such that $\lE\geq \epsilon_{I}(\ExB,\lu)$ for all non-zero $\lu\in K(0)$.
Then, for $\lu\in K(0)$ and $i=2,\ldots,\rankL$,
	\begin{align}
		\label{eq:(t+1)(7t+4)(Si1(1,1)0E)t+1lu-0}
		0&=
		(t+1) (7 t+4)(S_{i1}(1,1)_{0}E)_{t+1 } \lu
		+4 (3 t+1)E_{t } S_{i1}(1,1)_{1 } \lu
		\nonumber\\&\quad{}
		-4(S_{i1}(1,1)_{0}E)_{t+1 } \omega^{[1]}_{1 } \lu
		-2E_{t } S_{i1}(1,2)_{2 } \lu,\\
		\label{eq:(t+1)(7t+4)(Si1(1,1)0E)t+1lu-1}
0&=
(t+1) (7665 t^2+43205 t+18582)(S_{i1}(1,1)_{0}\ExB)_{t+1 } \lu\nonumber\\
&\quad{}+2 (6543 t^2+39413 t-14474)\ExB_{t } S_{i1}(1,1)_{1 } \lu
\nonumber\\&\quad{}
-14 (t+1) (102 t-2989)(S_{i1}(1,1)_{0}\ExB)_{t+1 } \omega^{[i]}_{1 } \lu
+77980\ExB_{t } S_{i1}(1,1)_{1 } \omega^{[i]}_{1 } \lu
\nonumber\\&\quad{}
-10 (27 t+9481)\ExB_{t } S_{i1}(1,2)_{2 } \lu
-2702(S_{i1}(1,1)_{0}\ExB)_{t+1 } \Har^{[i]}_{3 } \lu
\nonumber\\&\quad{}
-6 (1735 t+3097)(S_{i1}(1,1)_{0}\ExB)_{t+1 } \omega^{[1]}_{1 } \lu
-11844\ExB_{t } S_{i1}(1,1)_{1 } \omega^{[1]}_{1 } \lu
\nonumber\\&\quad{}
+5712(S_{i1}(1,1)_{0}\ExB)_{t+1 } \omega^{[1]}_{1 } \omega^{[i]}_{1 } \lu
-1890(S_{i1}(1,1)_{0}\ExB)_{t+1 } \Har^{[1]}_{3 } \lu
\nonumber\\&\quad{}
-1449(\Har_{0}\ExB)_{t+3 } S_{i1}(1,1)_{1 } \lu
-8568\ExB_{t } S_{i1}(1,2)_{2 } \omega^{[i]}_{1 } \lu
\nonumber\\&\quad{}
-126 (33 t+301)\ExB_{t } S_{i1}(1,3)_{3 } \lu.
	\end{align}
\end{lemma}
\begin{proof}
Taking the $(\lE+3)$-th action of \eqref{eq:8Si1(11)2E-4omega[1]-1-0}
and the $(\lE+5)$-th action of \eqref{eq:8Si1(11)2E-4omega[1]-1-1}
on $\lu$,	we have the results.
\end{proof}

\begin{lemma}
	\label{lemma:intertwining-lambda-norm2}
Assume $K(0)\cong \C e^{\lambda}$ for some $\lambda\in\fh\setminus\{0\}$.
Let $\lu_{\lambda}$ denote the element of $K(0)$ corresponding to $e^{\lambda}$. 
	For the vector space $B$ with $\lE=\epsilon_{I}(\ExB,\lu_{\lambda})$ in Lemma \ref{lemma:N-Basis},
	the following results hold:
	\begin{enumerate}
		\item If $\langle\alpha,\lambda\rangle=0$, then $\lE=-1$ and
		\begin{align}
			B&\cong \C e^{\alpha+\lambda},
			\C e^{\alpha-\lambda},\mbox{ or }
			\C e^{\alpha+\lambda}\oplus 
			\C e^{\alpha-\lambda}
		\end{align}
		as $A(M(1)^{+})$-modules and 
		\begin{align}
			\label{eq:norm2-mN-structure-1}
			\mN&=M(1)^{+}\cdot B\nonumber\\
			&\cong M(1,\lambda+\alpha),M(1,\lambda-\alpha),\mbox{ or }M(1,\lambda+\alpha)\oplus M(1,\lambda-\alpha)
		\end{align}
		as  $M(1)^{+}$-modules.
		\item
		If $\langle\alpha,\lambda\rangle\neq 0$ 
		and $(1+\lE)^2- \langle \alpha,\lambda\rangle^2=0$, then
		\begin{align}
			B\cong \C e^{\alpha+\lambda}\mbox{ or }\C e^{\alpha-\lambda}
		\end{align}
		as $A(M(1)^{+})$-modules.
		Furthermore, if $\lambda\neq \pm\alpha$, 
		\begin{align}
			\label{eq:norm2-mN-structure-2}
			\mN&\cong M(1,\lambda+\alpha),M(1,\lambda-\alpha),\mbox{ or }M(1,\lambda+\alpha)\oplus M(1,\lambda-\alpha)
		\end{align}
		as  $M(1)^{+}$-modules.
		\item
		If 
		$(1+\lE)^2- \langle \alpha,\lambda\rangle^2\neq 0$, then $\lE=0$, $\lambda=\pm\alpha$, and $B\cong M(1)^{-}(0)$ as $A(M(1)^{+})$-modules.
	\end{enumerate}
\end{lemma}
\begin{proof}

	By \eqref{eq:lE1lE24omega1ElElu-0-1}, 
\begin{align}
\label{eq:(1+lE)2-langlealpha,lambdarangle2=0}
(1+\lE)^2-\langle \alpha,\lambda\rangle^2=0\mbox{ or }\lE=0 
\end{align}
and in the later case
\begin{align}
	\label{eq:2=langlelambda,h[1]rangle2}
2&=\langle \lambda,h^{[1]}\rangle^2=\langle \alpha,\lambda\rangle^2/2.
\end{align}
\begin{enumerate}
\item
Assume $\langle\alpha,\lambda\rangle=\langle h^{[1]},\lambda\rangle=0$. 
By \eqref{eq:(1+lE)2-langlealpha,lambdarangle2=0} and \eqref{eq:2=langlelambda,h[1]rangle2}, 
$(1+\lE)^2-\langle \alpha,\lambda\rangle^2=0$ and hence $\lE=-1$.
A direct computation shows that \eqref{eq:Si2(1,1)0S21(1,1)0E=dfrac-1mn-1},
\eqref{eq:(S21(1,1)0E)0Si2(1,1)1lulambdanonumber}, and
\eqref{eq:langlelambdah[i]rangle(21(11)0E} hold even for the case that $\mn=2$.
\item 
Assume $\langle\alpha,\lambda\rangle\neq 0$ and $(1+\lE)^2-\langle \alpha,\lambda\rangle^2=0$.
Then, $\lE\neq -1$. 
By \eqref{eq:(t+1)(7t+4)(Si1(1,1)0E)t+1lu-0},
	\begin{align}
	(S_{i1}(1,1)_{0}\ExB)_{\lE+1 }\lu_{\lambda}
	&=\dfrac{-2\langle\lambda,h^{[1]}\rangle\langle\lambda,h^{[i]}\rangle}{1+\lE}
	\ExB_{\lE}\lu_{\lambda}
\end{align}
for all $i=2,\ldots,\rankL$ and hence $B=\C \ExB_{\lE}\lu_{\lambda}$.
\item 
Assume  $(1+\lE)^2-\langle \alpha,\lambda\rangle^2\neq 0$.
Then, by \eqref{eq:(1+lE)2-langlealpha,lambdarangle2=0} and \eqref{eq:2=langlelambda,h[1]rangle2},
$\lE=0$ and $\langle\lambda,h^{[1]}\rangle^2=2$.
By \eqref{eq:(t+1)(7t+4)(Si1(1,1)0E)t+1lu-0},
\begin{align}
\langle\lambda,h^{[i]}\rangle&=0
\end{align}
for all $i=2,\ldots,\rankL$ and hence $\lambda=\pm\alpha$.

\end{enumerate}
The same argument as in the proof of Lemma \ref{lemma:intertwining-lambda} shows the other results.
\end{proof}

The same argument as in the proof of Lemma \ref{lemma:intertwining-M1minus}
shows the following result:
\begin{lemma}
\label{lemma:intertwining-M1minus-norm2}
If $K(0)\cong M(1)^{-}(0)$, then
$B\cong\C e^{\alpha}$ as $A(M(1)^{+})$-modules,
and $\mN=M(1)^{+}\cdot B\cong M(1,\alpha)$.
\end{lemma}
\begin{proof}
For $i=1,\ldots,\rankL$,
$u^{[i]}$ denotes the element of $\mK(0)$ 
corresponding to $h^{[i]}(-1)\vac$.
The same argument as in Section \ref{section:The case that K(0)cong M(1)-(0)}
shows that $\lE=0$ and $\epsilon(S_{i1}(1,1)_{0}E,u^{[1]})\leq 1$ for $i=2,\ldots,\rankL$.
For $i=2,\ldots,\rankL$,
by \eqref{eq:norm2Si1(1,1)1lu=-Si1(1,2)2lu-2} and \eqref{eq:(t+1)(7t+4)(Si1(1,1)0E)t+1lu-0},
\begin{align}
\ExB_{0}u^{[i]}&=0
\end{align}
and by \eqref{eq:(t+1)(7t+4)(Si1(1,1)0E)t+1lu-1},
\begin{align}
(S_{i1}(1,1)_{0}\ExB)_{1}\lu^{[1]}&=0.
\end{align}
Thus, $B=\C E_{0}u^{[1]}$ by \eqref{eq:(H0ElE1lu=dfrac194+7lE}.  
The same argument as in the proof of Lemma \ref{lemma:intertwining-lambda}
shows that  $B\cong \C e^{\alpha}$
as $A(M(1)^{+})$-modules and $\mN= M(1)^{+}\cdot B\cong M(1,\alpha)$ as $M(1)^{+}$-modules.
\end{proof}
\subsection{The case that $\langle\alpha,\alpha\rangle=1/2$}
Throughout this subsection, 
\begin{align}
\label{langlealpha,alpharangle=frac12}
\langle\alpha,\alpha\rangle&=\frac{1}{2}.
\end{align}
We assume $\alpha\in\C h^{[1]}$ and hence $\langle\alpha, h^{[1]}\rangle^2=1/2$ and
$\langle \alpha,h^{[i]}\rangle=0$ for all $i=2,\ldots,\rankL$.
	A direct computation shows that for $i=2,\ldots,\rankL$
\begin{align}
		\label{eq:2Si1(1,1)-3ExB-2omega[1]-1-0} 
	0&=
	S_{i1}(1,1)_{-2 } \ExB+\omega^{[1]}_{-1 } (S_{i1}(1,1)_{0}\ExB)
	-\omega_{0 }^2(S_{i1}(1,1)_{0}\ExB)-S_{i1}(1,2)_{-1 } \ExB,\\
	\label{eq:2Si1(1,1)-3ExB-2omega[1]-1-1} 
	0&=2S_{i1}(1,1)_{-3 } \ExB 
	-2\omega^{[i]}_{-1 } S_{i1}(1,1)_{-1 } \ExB 
	-5\omega^{[i]}_{-2 } (S_{i1}(1,1)_{0}\ExB) 
	\nonumber\\&\quad{}
	-\omega^{[1]}_{-1 } S_{i1}(1,1)_{-1 } \ExB 
	-5\omega_{0 } S_{i1}(1,1)_{-2 } \ExB 
	+4\omega_{0 } \omega^{[i]}_{-1 } (S_{i1}(1,1)_{0}\ExB) 
	\nonumber\\&\quad{}
	-4\omega_{0 } \omega^{[1]}_{-1 } (S_{i1}(1,1)_{0}\ExB) 
	+3\omega_{0 }^2 S_{i1}(1,1)_{-1 } \ExB 
	+3\omega_{0 } S_{i1}(1,2)_{-1 } \ExB 
	\nonumber\\&\quad{}
	+S_{i1}(1,3)_{-1 } \ExB,\\
	\label{eq:8omega[1]-3ExB+12}
	0&=
	8\omega^{[1]}_{-3 } \ExB+12\Har^{[1]}_{-1 } \ExB
	+3\omega^{[1]}_{-1 } (\Har^{[1]}_{1}\ExB)+4\omega_{0 } \omega^{[1]}_{-2 } \ExB
	-11\omega_{0 }^2 (\Har^{[1]}_{1}\ExB),\\
	\label{eq:2Si1(1,1)-3ExB-2omega[1]-1-2} 	0&=
	15S_{i1}(1,1)_{-3 } \ExB 
	-22\omega^{[i]}_{-1 } S_{i1}(1,1)_{-1 } \ExB 
	-55\omega^{[i]}_{-2 } (S_{i1}(1,1)_{0}\ExB) 
	\nonumber\\&\quad{}
	-9\omega^{[1]}_{-1 } S_{i1}(1,1)_{-1 } \ExB 
	-6S_{i1}(1,1)_{-1 } (\Har^{[1]}_{1}\ExB) 
	-27\omega_{0 } S_{i1}(1,1)_{-2 } \ExB 
	\nonumber\\&\quad{}
	+44\omega_{0 } \omega^{[i]}_{-1 } (S_{i1}(1,1)_{0}\ExB) 
	-12\omega_{0 } \omega^{[1]}_{-1 } (S_{i1}(1,1)_{0}\ExB) 
	+15\omega_{0 } \omega_{0 } S_{i1}(1,1)_{-1 } \ExB 
	\nonumber\\&\quad{}
	+15\omega_{0 } S_{i1}(1,2)_{-1 } \ExB. 
\end{align}
\begin{lemma}
	Let $\lE\in\Z$ such that $\lE\geq \epsilon_{I}(\ExB,\lu)$ for all non-zero $\lu\in K(0)$.
	Then, for $\lu\in K(0)$ and $i=2,\ldots,\rankL$,
\begin{align}
		\label{eq:-2(t+1)(9t+8)(Si1(1,1)0ExB)t+1lu-0}
		0&=-(t+1)^2(S_{i1}(1,1)_{0}\ExB)_{\lE+1}\lu-\ExB_{\lE} S_{i1}(1,1)_{1 }\lu \nonumber\\&\quad{}
	+(S_{i1}(1,1)_{0}\ExB)_{\lE+1}\omega^{[1]}_{1 }\lu -\ExB_{\lE} S_{i1}(1,2)_{2 }\lu, \\
	\label{eq:-2(t+1)(9t+8)(Si1(1,1)0ExB)t+1lu-1}
	0&=
	-2 (t+1) (9 t+8)(S_{i1}(1,1)_{0}\ExB)_{t+1 } \lu
	+2 (3 t^2+5 t-10)\ExB_{t } S_{i1}(1,1)_{1 } \lu
	\nonumber\\&\quad{}
	-8 (t+1)(S_{i1}(1,1)_{0}\ExB)_{t+1 } \omega^{[i]}_{1 } \lu
	-4\ExB_{t } S_{i1}(1,1)_{1 } \omega^{[i]}_{1 } \lu
	\nonumber\\&\quad{}
	-(6 t+29)\ExB_{t } S_{i1}(1,2)_{2 } \lu
	+2 (4 t+17)(S_{i1}(1,1)_{0}\ExB)_{t+1 } \omega^{[1]}_{1 } \lu
	\nonumber\\&\quad{}
	-2\ExB_{t } S_{i1}(1,1)_{1 } \omega^{[1]}_{1 } \lu
	+2\ExB_{t } S_{i1}(1,3)_{3 } \lu,\\
\label{eq:8(t+1)ExBomega[1]1lu-(t+1)(11 t+15)}
	0&=
	8 (t+1)\ExB_{\lE}\omega^{[1]}_{1 }\lu -(t+1) (11 t+15)(\Har^{[1]}_{1}\ExB)_{\lE+2}\lu \nonumber\\&\quad{}
	+4 (t+1)^2\ExB_{\lE}+12\ExB_{\lE}\Har^{[1]}_{3 }\lu  	+3(\Har^{[1]}_{1}\ExB)_{\lE+2}\omega^{[1]}_{1 }\lu,\\
	\label{eq:-2(t+1)(9t+8)(Si1(1,1)0ExB)t+1lu-2}
	0&=
	-2 (t+1) (25 t+18)(S_{i1}(1,1)_{0}\ExB)_{t+1 } \lu
	+2 (15 t^2+31 t-28)\ExB_{t } S_{i1}(1,1)_{1 } \lu
	\nonumber\\&\quad{}
	-88 (t+1)(S_{i1}(1,1)_{0}\ExB)_{t+1 } \omega^{[i]}_{1 } \lu
	-44\ExB_{t } S_{i1}(1,1)_{1 } \omega^{[i]}_{1 } \lu
	\nonumber\\&\quad{}
	-5 (6 t+25)\ExB_{t } S_{i1}(1,2)_{2 } \lu
	+6 (4 t+19)(S_{i1}(1,1)_{0}\ExB)_{t+1 } \omega^{[1]}_{1 } \lu
	\nonumber\\&\quad{}
	-18\ExB_{t } S_{i1}(1,1)_{1 } \omega^{[1]}_{1 } \lu
	-12(\Har^{[1]}_{1}\ExB)_{t+2 } S_{i1}(1,1)_{1 } \lu.
\end{align}
\end{lemma}
\begin{proof}
Taking the $(\lE+3)$-th action of \eqref{eq:2Si1(1,1)-3ExB-2omega[1]-1-0} and 
the $(\lE+4)$-th actions of \eqref{eq:2Si1(1,1)-3ExB-2omega[1]-1-1} --
	\eqref{eq:2Si1(1,1)-3ExB-2omega[1]-1-2} on $\lu$, we have the results.
	\end{proof}
\begin{lemma}
	\label{lemma:intertwining-lambda-norm1-2}
Assume $K(0)\cong \C e^{\lambda}$ for some $\lambda\in\fh\setminus\{0\}$.
Let $\lu_{\lambda}$ denote the element of $K(0)$ corresponding to $e^{\lambda}$.
Let $B$ be the vector space with $\lE=\epsilon_{I}(\ExB,\lu_{\lambda})$ in Lemma \ref{lemma:N-Basis}.
Then
\begin{align}
(1+\lE)^2-\langle\lambda,\alpha\rangle^2&=0.
\end{align}
Moreover,	the following results hold:
	\begin{enumerate}
		\item If $\langle\alpha,\lambda\rangle=0$, then $\lE=-1$ and
		\begin{align}
			B&\cong \C e^{\alpha+\lambda},
			\C e^{\alpha-\lambda},\mbox{ or }
			\C e^{\alpha+\lambda}\oplus 
			\C e^{\alpha-\lambda}
		\end{align}
		as $A(M(1)^{+})$-modules and 
		\begin{align}
			\label{eq:norm1-2-mN-structure-1}
			\mN&=M(1)^{+}\cdot B\nonumber\\
			&\cong M(1,\lambda+\alpha),M(1,\lambda-\alpha),\mbox{ or }M(1,\lambda+\alpha)\oplus M(1,\lambda-\alpha)
		\end{align}
		as  $M(1)^{+}$-modules.
		\item
		If $\langle\alpha,\lambda\rangle\neq 0$, then
		\begin{align}
			B\cong \C e^{\alpha+\lambda}\mbox{ or }\C e^{\alpha-\lambda}
		\end{align}
		as $A(M(1)^{+})$-modules and
		\begin{align}
			\label{eq::norm1-2-mN-structure-2}
			\mN&\cong M(1,\lambda+\alpha),M(1,\lambda-\alpha),\mbox{ or }M(1,\lambda+\alpha)\oplus M(1,\lambda-\alpha)
		\end{align}
		as  $M(1)^{+}$-modules.
	\end{enumerate}
\end{lemma}
\begin{proof}
By \eqref{eq:omega1lu=(1+lE)2lu} 
\begin{align}
\label{eq:2(1+lE)2=langlelambda}
(1+\lE)^2&=\frac{\langle \lambda,h^{[1]}\rangle^2}{2}=\langle \lambda,\alpha\rangle^2.
\end{align}

\begin{enumerate}
\item
Assume $\langle\lambda,\alpha\rangle=\langle\lambda,h^{[1]}\rangle=0$.
Then, by \eqref{eq:2(1+lE)2=langlelambda}, $\lE=-1$.
A direct computation shows that \eqref{eq:Si2(1,1)0S21(1,1)0E=dfrac-1mn-1},
\eqref{eq:(S21(1,1)0E)0Si2(1,1)1lulambdanonumber}, and
\eqref{eq:langlelambdah[i]rangle(21(11)0E} hold even for the case that $\mn=1/2$.
\item
Assume $\langle\lambda,\alpha\rangle\neq 0$.
Then, by \eqref{eq:2(1+lE)2=langlelambda}, $\lE\neq -1$. Substituting \eqref{eq:(Har0E)lE+3lu=ExB-0} into 
\eqref{eq:-2(t+1)(9t+8)(Si1(1,1)0ExB)t+1lu-1} and \eqref{eq:-2(t+1)(9t+8)(Si1(1,1)0ExB)t+1lu-2},
and then deleting the terms including $\omega^{[i]}_{1}\lu_{\lambda}$ from the obtained relations,
we have
\begin{align}
0&=(2\lE+3)(4\lE+5)(2(1+\lE)(S_{i1}(1,1)_{0}\ExB)_{\lE+1}+\ExB_{\lE})\lu_{\lambda}
\end{align}
and hence 
\begin{align}
(S_{i1}(1,1)_{0}\ExB)_{\lE+1}\lu_{\lambda}&=\frac{-1}{2(1+\lE)}\ExB_{\lE}\lu_{\lambda}.
\end{align}
for all $i=2,\ldots,\rankL$. Thus, $B=\C \ExB_{\lE}\lu_{\lambda}$.
\end{enumerate}
The same argument as in the proof of Lemma \ref{lemma:intertwining-lambda} shows the other results.
\end{proof}
\begin{lemma}
\label{lemma:intertwining-M1minus-norm1-2}
If $K(0)\cong M(1)^{-}(0)$, then
$B\cong\C e^{\alpha}$ as $A(M(1)^{+})$-modules,
and $\mN=M(1)^{+}\cdot B\cong M(1,\alpha)$.
\end{lemma}
\begin{proof} 
	For $i=1,\ldots,\rankL$,
	$u^{[i]}$ denotes the element of $\mK(0)$ 
	corresponding to $h^{[i]}(-1)\vac$.
	The same argument as in Section \ref{section:The case that K(0)cong M(1)-(0)}
	shows that $\lE=0$ and $\epsilon(S_{i1}(1,1)_{0}E,u^{[1]})\leq 1$
	for $i=2,\ldots,\rankL$.
	By \eqref{eq:norm2Si1(1,1)1lu=-Si1(1,2)2lu-2},
	\eqref{eq:-2(t+1)(9t+8)(Si1(1,1)0ExB)t+1lu-0}--\eqref{eq:-2(t+1)(9t+8)(Si1(1,1)0ExB)t+1lu-2}
	\begin{align}
		\label{eq:(Si1(1,1)0E)1lu=(Si1(1,2)0E)1lu}
		0&=(S_{i1}(1,1)_{0}E)_{1}u^{[1]}
		=(S_{i1}(1,2)_{0}E)_{1}u^{[1]}
		=\ExB_{0}u^{[i]}
	\end{align}
	for all $i=2,\ldots,\rankL$ and hence
	$B=\C \ExB_{0}u^{[1]}$.
	The same argument as in the proof of Lemma \ref{lemma:intertwining-M1minus}
	shows that $U\cong \C e^{\alpha}$ and $N\cong M(1,\alpha)$.
	\end{proof}

\subsection{The case that $\langle\alpha,\alpha\rangle=1$}
Throughout this subsection, 
\begin{align}
\label{langlealpha,alpharangle=1}
\langle\alpha,\alpha\rangle&=1.
\end{align}
We assume $\alpha\in\C h^{[1]}$ and hence $\langle\alpha, h^{[1]}\rangle^2=1$ and
$\langle \alpha,h^{[i]}\rangle=0$ for all $i=2,\ldots,\rankL$.
A direct computation show that
\begin{align}
\label{eq:=Si1(1,1)-2ExB-Si1(1,2)-1ExB-1}
0&=S_{i1}(1,1)_{-1 } \ExB-\omega_{0 } (S_{i1}(1,1)_{0}\ExB),\\
\label{eq:=Si1(1,1)-2ExB-Si1(1,2)-1ExB-2}
0&=S_{i1}(1,1)_{-2 } \ExB 
-S_{i1}(1,2)_{-1 } \ExB 
+\omega_{0 } (S_{i1}(1,2)_{0}\ExB),\\
\label{eq:=Si1(1,1)-2ExB-Si1(1,2)-1ExB-3}
0&=27\omega^{[1]}_{-2 } (S_{i1}(1,1)_{0}\ExB) 
-8\omega_{0 } S_{i1}(1,1)_{-2 } \ExB 
\nonumber\\&\quad{}
-2\omega_{0 } \omega^{[1]}_{-1 } (S_{i1}(1,1)_{0}\ExB) 
+\omega_{0 } \omega_{0 } \omega_{0 } (S_{i1}(1,1)_{0}\ExB) 
\nonumber\\&\quad{}
+21S_{i1}(1,2)_{-2 } \ExB 
-12\omega^{[1]}_{-1 } (S_{i1}(1,2)_{0}\ExB) 
\nonumber\\&\quad{}
-18\omega_{0 } S_{i1}(1,2)_{-1 } \ExB 
-15S_{i1}(1,3)_{-1 } \ExB,\\
\label{eq:=Si1(1,1)-2ExB-Si1(1,2)-1ExB-4}
0&=S_{i1}(1,1)_{-3 } \ExB
-9\omega^{[1]}_{-2 } (S_{i1}(1,1)_{0}\ExB)
\nonumber\\&\quad{}
+2\omega_{0 } S_{i1}(1,1)_{-2 } \ExB
-8S_{i1}(1,2)_{-2 } \ExB
\nonumber\\&\quad{}
+4\omega^{[1]}_{-1 } (S_{i1}(1,2)_{0}\ExB)
+7\omega_{0 } S_{i1}(1,2)_{-1 } \ExB
\nonumber\\&\quad{}
+6S_{i1}(1,3)_{-1 } \ExB
\end{align}
for all $i=2,\ldots,\rankL$.
\begin{lemma}
Let $\lE\in\Z$ such that $\lE\geq \epsilon_{I}(\ExB,\lu)$ for all non-zero $\lu\in K(0)$.
Then, for $\lu\in K(0)$ and $i=2,\ldots,\rankL$,
\begin{align}
\label{eq:(1+lE)(Si1(1,1)0ExB)lE+1lu-1}
0&=(1+\lE)(S_{i1}(1,1)_{0}\ExB)_{\lE+1}\lu+\ExB_{\lE}S_{i1}(1,1)_{1}\lu,\\
\label{eq:(1+lE)(Si1(1,1)0ExB)lE+1lu-2}
0&=(1+\lE)(S_{i1}(1,2)_{0}\ExB)_{\lE+2}\lu+\ExB_{\lE}(2S_{i1}(1,1)_{1}+S_{i1}(1,2)_{2})\lu,\\
\label{eq:(1+lE)(Si1(1,1)0ExB)lE+1lu-3}
0&=
-(t^3+9 t^2+21 t+12)(S_{i 1}(1,1)_{0}\ExB)_{t+1 } \lu
\nonumber\\&\quad{}
+(5 t-22)\ExB_{t } S_{i 1}(1,1)_{1 } \lu
-(4 t+5)(S_{i 1}(1,2)_{0}\ExB)_{t+2 } \lu
\nonumber\\&\quad{}
+2 t(S_{i 1}(1,1)_{0}\ExB)_{t+1 } \omega^{[1]}_{1 } \lu
+2 (9 t-7)\ExB_{t } S_{i 1}(1,2)_{2 } \lu
\nonumber\\&\quad{}
-12(S_{i 1}(1,2)_{0}\ExB)_{t+2 } \omega^{[1]}_{1 } \lu
-15\ExB_{t } S_{i 1}(1,3)_{3 } \lu,\\
\label{eq:(1+lE)(Si1(1,1)0ExB)lE+1lu-4}
0&=
-(t+2)(S_{i1}(1,1)_{0}\ExB)_{t+1 } \lu
-2(7t+24)\ExB_{t } S_{i1}(1,1)_{1 } \lu
\nonumber\\&\quad{}
-(24t+23)(S_{i1}(1,2)_{0}\ExB)_{t+2 } \lu
+2(S_{i1}(1,1)_{0}\ExB)_{t+1 } \omega^{[1]}_{1 } \lu
\nonumber\\&\quad{}
-3(7t+4)\ExB_{t } S_{i1}(1,2)_{2 } \lu
+12(S_{i1}(1,2)_{0}\ExB)_{t+2 } \omega^{[1]}_{1 } \lu
\nonumber\\&\quad{}
+18\ExB_{t } S_{i1}(1,3)_{3 } \lu.
\end{align}

\end{lemma}
\begin{proof}
Taking the $(\lE+2)$-th action of \eqref{eq:=Si1(1,1)-2ExB-Si1(1,2)-1ExB-1}
the $(\lE+3)$-th action of \eqref{eq:=Si1(1,1)-2ExB-Si1(1,2)-1ExB-2},
and
the $(\lE+4)$-th actions of \eqref{eq:=Si1(1,1)-2ExB-Si1(1,2)-1ExB-3} and
\eqref{eq:=Si1(1,1)-2ExB-Si1(1,2)-1ExB-4} on $\lu$,
we have the results.
\end{proof}
\begin{lemma}
\label{lemma:intertwining-lambda-norm1}
Assume $K(0)\cong \C e^{\lambda}$ for some $\lambda\in\fh\setminus\{0\}$.
Let $\lu_{\lambda}$ denote the element of $K(0)$ corresponding to $e^{\lambda}$.
Let $B$ be the vector space with $\lE=\epsilon_{I}(\ExB,\lu_{\lambda})$ in Lemma \ref{lemma:N-Basis}.
Then,	the following results hold:
	\begin{enumerate}
		\item If $\langle\alpha,\lambda\rangle=0$, then $\lE=-1$ and
		\begin{align}
			B&\cong \C e^{\alpha+\lambda},
			\C e^{\alpha-\lambda},\mbox{ or }
			\C e^{\alpha+\lambda}\oplus 
			\C e^{\alpha-\lambda}
		\end{align}
		as $A(M(1)^{+})$-modules and 
		\begin{align}
			\label{eq:norm1-mN-structure-1}
			\mN&=M(1)^{+}\cdot B\nonumber\\
			&\cong M(1,\lambda+\alpha),M(1,\lambda-\alpha),\mbox{ or }M(1,\lambda+\alpha)\oplus M(1,\lambda-\alpha)
		\end{align}
		as  $M(1)^{+}$-modules.
		\item
		If $\langle\alpha,\lambda\rangle\neq 0$ 
		and $(1+\lE)^2- \langle \alpha,\lambda\rangle^2=0$, then
		\begin{align}
			B\cong \C e^{\alpha+\lambda}\mbox{ or }\C e^{\alpha-\lambda}
		\end{align}
		as $A(M(1)^{+})$-modules.
		Furthermore, if $\lambda\neq \pm\alpha$, 
		\begin{align}
			\label{eq:norm1-mN-structure-2}
			\mN&\cong M(1,\lambda+\alpha),M(1,\lambda-\alpha),\mbox{ or }M(1,\lambda+\alpha)\oplus M(1,\lambda-\alpha)
		\end{align}
		as  $M(1)^{+}$-modules.
		\item
		If 
		$(1+\lE)^2- \langle \alpha,\lambda\rangle^2\neq 0$, then $\lE=-1$, $\lambda=\pm\alpha$, and $B\cong M(1)^{-}(0)$ as $A(M(1)^{+})$-modules.
	\end{enumerate}
\end{lemma}
\begin{proof}

If $\lE\neq -1$, then $B=\C \ExB_{\lE}\lu_{\lambda}$ by 
\eqref{eq:(1+lE)(Si1(1,1)0ExB)lE+1lu-1}
and
\eqref{eq:(1+lE)(Si1(1,1)0ExB)lE+1lu-2}.
\begin{enumerate}
\item Assume $\langle\lambda,\alpha\rangle=0$.
Since $\omega^{[1]}_{1}\lu_{\lambda}=(\langle\lambda,h^{[1]}\rangle^2/2) \lu_{\lambda}$,
it follows from Lemma \ref{lemma:structure-Vlattice-M1} that $\lE=-1$ and hence by \eqref{eq:(1+lE)(Si1(1,1)0ExB)lE+1lu-3},
\begin{align}
(S_{i 1}(1,1)_{0}\ExB)_{0 } \lu_{\lambda}&=(S_{i 1}(1,2)_{0}\ExB)_{1}\lu_{\lambda}
\end{align}
for all $i=2,\ldots,\rankL$.
A direct computation shows that
\begin{align}
S_{k2}(1,1)_{-1}S_{21}(1,1)_{0}\ExB
&=
2\omega^{[2]}_{-1 } (S_{k 1}(1,1)_{0}\ExB) 
+\frac{2}{3}\omega^{[1]}_{-1 } (S_{k 1}(1,1)_{0}\ExB) 
\nonumber\\&\quad{}
+\frac{-1}{3}\omega_{0 } S_{k 1}(1,1)_{-1 } \ExB 
+\frac{-1}{3}S_{k 1}(1,2)_{-1 } \ExB 
\nonumber\\&\quad{}
+\frac{-2}{3}\omega_{0 } (S_{k 1}(1,2)_{0}\ExB),\nonumber\\
S_{k2}(1,1)_{0}S_{21}(1,1)_{0}\ExB
&=-(S_{k 1}(1,2)_{0}E),\nonumber\\
S_{k2}(1,1)_{1}S_{21}(1,1)_{0}\ExB
&=S_{k 1}(1,1)_{0}E,\nonumber\\
S_{k2}(1,1)_{i}S_{21}(1,1)_{0}\ExB
&=0\mbox{ for }i\geq 2.
\end{align}
and 
\begin{align}
(S_{k2}(1,1)_{-1}S_{21}(1,1)_{0}\ExB)_{2}\lu_{\lambda}
&=
2(S_{k 1}(1,1)_{0}\ExB)_{0 } \omega^{[i]}_{1 } \lu_{\lambda}
+\frac{-1}{3}(S_{k 1}(1,1)_{0}\ExB)_{0 } \lu_{\lambda}
\nonumber\\&\quad{}
+\frac{2}{3}(S_{k 1}(1,1)_{0}\ExB)_{0 } \omega^{[1]}_{1 } \lu_{\lambda}
+\frac{-1}{3}\ExB_{-1 } S_{k 1}(1,2)_{2 } \lu_{\lambda}
\nonumber\\&\quad{}
+\frac{4}{3}(S_{k 1}(1,2)_{0}\ExB)_{1 } \lu_{\lambda},\nonumber\\
(S_{k2}(1,1)_{0}S_{21}(1,1)_{0}\ExB)_{2}\lu_{\lambda}&=-(S_{k1}(1,2)_{0}\ExB)_{1}\lu_{\lambda}
\end{align}
for all $k=3,\ldots,\rankL$.
Since
\begin{align}
&(S_{k2}(1,1)_{-1}S_{21}(1,1)_{0}\ExB)_{2}\lu_{\lambda}\nonumber\\
&=
(S_{21}(1,1)_{0}\ExB)_{0}S_{k2}(1,1)_{1}\lu_{\lambda}-
(S_{k2}(1,1)_{0}S_{21}(1,1)_{0}\ExB)_{1}\lu_{\lambda},
\end{align}
we have 
\begin{align}
\langle\lambda,h^{[k]}\rangle\langle\lambda,h^{[2]}\rangle(S_{21}(1,1)_{0}\ExB)_{0}\lu_{\lambda}
&=\langle\lambda,h^{[2]}\rangle^2(S_{k1}(1,1)_{0}\ExB)_{0}\lu_{\lambda}
\end{align}
for all $k=3,\ldots,\rankL$.
The same argument as in the proof of Lemma \ref{lemma:intertwining-lambda} (1) shows the other results.
\item[(2),(3)]The same argument as in the proof of Lemma \ref{lemma:intertwining-lambda} (2) and (3) shows the results.
\end{enumerate}
\end{proof}

\begin{lemma}
\label{lemma:intertwining-M1minus-norm1}
If $K(0)\cong M(1)^{-}(0)$, then
$B\cong\C e^{\alpha}$ as $A(M(1)^{+})$-modules,
and $\mN=M(1)^{+}\cdot B\cong M(1,\alpha)$.
\end{lemma}
\begin{proof} 
For $i=1,\ldots,\rankL$,
$u^{[i]}$ denotes the element of $\mK(0)$ 
corresponding to $h^{[i]}(-1)\vac$.
The same argument as in Section \ref{section:The case that K(0)cong M(1)-(0)}
shows that $\lE=0$ and $\epsilon(S_{i1}(1,1)_{0}E,u^{[1]})\leq 1$
for $i=2,\ldots,\rankL$.
By \eqref{eq:norm2Si1(1,1)1lu=-Si1(1,2)2lu-2},
\eqref{eq:(1+lE)(Si1(1,1)0ExB)lE+1lu-1},
\eqref{eq:(1+lE)(Si1(1,1)0ExB)lE+1lu-2}, and 
\eqref{eq:(1+lE)(Si1(1,1)0ExB)lE+1lu-4}
\begin{align}
\label{eq:(Si1(1,1)0E)1lu=(Si1(1,2)0E)1lu-2}
0&=(S_{i1}(1,1)_{0}E)_{1}u^{[1]}
=(S_{i1}(1,2)_{0}E)_{1}u^{[1]}
=\ExB_{0}u^{[i]}
\end{align}
for all $i=2,\ldots,\rankL$ and hence
$B=\C \ExB_{0}u^{[1]}$.
The same argument as in the proof of Lemma \ref{lemma:intertwining-M1minus}
shows that $U\cong \C e^{\alpha}$ and $N\cong M(1,\alpha)$.
\end{proof}

\subsection{The case that $\langle\alpha,\alpha\rangle=0$}
Throughout this subsection, 
\begin{align}
\langle\alpha,\alpha\rangle=0.
\end{align}
In this case we can take an orthonormal basis $h^{[1]},h^{[2]},\ldots,h^{[\rankL]}$ of $\fh$ so that
\begin{align}
\label{eq:langlealphah[2]rangle=sqrt-1}
0\neq \langle\alpha,h^{[2]}\rangle=\sqrt{-1}\langle\alpha,h^{[1]}\rangle
\mbox{ and }\langle\alpha,h^{[i]}\rangle=0
\end{align}
for all $i=3,\ldots,\rankL$.
Since for any pair $c,s\in\C$ such that $c^2+s^2=1$,
the orthonormal basis 
\begin{align}
c h^{[1]}+sh^{[2]}, -s h^{[1]}+c h^{[2]}, h^{[3]},\ldots,h^{[\rankL]} \in\fh
\end{align}
also satisfies the condition \eqref{eq:langlealphah[2]rangle=sqrt-1},
we have the following result.
\begin{lemma}
\label{lemma:infinite-set}
The cardinality of the set
\begin{align}
\Big\{\langle\alpha,h^{[1]}\rangle\ \Big|\ \mbox{\begin{tabular}{l}
$h^{[1]},\ldots,h^{[\rankL]}$ is an orthonormal basis of $\fh$  \\
such that  $0\neq \langle\alpha,h^{[2]}\rangle=\sqrt{-1}\langle\alpha,h^{[1]}\rangle$\\
and $\langle\alpha,h^{[i]}\rangle=0$ for all $i=3,\ldots,\rankL$\end{tabular}}\Big\}
\end{align}
is infinite. 
\end{lemma}
We write 
\begin{align}
\zone=\langle\alpha,h^{[1]}\rangle\mbox{ and }E=E(\alpha)
\end{align}
for simplicity. 
A direct computation shows that
\begin{align}
\label{eq:z147z121omega22E-1}
0&= \zone^4 (7  \zone^2+1)\omega^{[2]}_{-2 } E 
- \zone^2 (7  \zone^4- \zone^2-4)\omega^{[1]}_{-2 } E \nonumber\\&\quad{}
+2  \zone^4 (3  \zone-1) (3  \zone+1)\omega_{0 } \omega^{[2]}_{-1 } E 
-2  \zone^2 (9  \zone^4+ \zone^2+2)\omega_{0 } \omega^{[1]}_{-1 } E \nonumber\\&\quad{}
+ \zone^2 (3  \zone-1) (3  \zone+1)\omega_{0 }^3 E 
- \zone^2 (7  \zone^4+5  \zone^2+2)\sqrt{-1}S_{21}(1,1)_{-2 } E \nonumber\\&\quad{}
-2  \zone^2 (3  \zone^2-1) (3  \zone^2+1)\sqrt{-1}\omega_{0 } S_{21}(1,1)_{-1 } E 
+10  \zone^4\sqrt{-1}S_{21}(1,2)_{-1 } E \nonumber\\&\quad{}
+4  \zone^2\omega^{[2]}_{-1 } (\omega_{0}^{[1]}E) 
+4  \zone^2\omega^{[1]}_{-1 } (\omega_{0}^{[1]}E) 
+2  \zone^2\omega_{0 }^2 (\omega_{0}^{[1]}E)_{-1},\\
\label{eq:z147z121omega22E-2}
0&=
-2 \zone^4 (2 \zone^2+1)^2 (7 \zone^2+1)\omega^{[2]}_{-3 } E\nonumber\\&\quad{}
+2 \zone^2 (7 \zone^2+1) (4 \zone^6-4 \zone^4+9 \zone^2+3)\omega^{[1]}_{-3 } E\nonumber\\&\quad{}
-8 \zone^2 (7 \zone^2+1) (2 \zone^4+1)\omega^{[1]}_{-1 } \omega^{[1]}_{-1 } E\nonumber\\&\quad{}
+12 \zone^2 (7 \zone^2+1) (2 \zone^4+1)H^{[1]}_{-1 } E\nonumber\\&\quad{}
+2 \zone^2 (16 \zone^6-27 \zone^2-4)\omega_{0 } \omega^{[1]}_{-2 } E\nonumber\\&\quad{}
+8 \zone^6 (8 \zone^4-17 \zone^2-3)\omega_{0 } \omega_{0 } \omega^{[2]}_{-1 } E\nonumber\\&\quad{}
-8 \zone^2 (8 \zone^8-9 \zone^6-11 \zone^4-8 \zone^2-1)\omega_{0 } \omega_{0 } \omega^{[1]}_{-1 } E\nonumber\\&\quad{}
+4 \zone^4 (8 \zone^4-17 \zone^2-3)\omega_{0 } \omega_{0 } \omega_{0 } \omega_{0 } E\nonumber\\&\quad{}
+\zone^2 (7 \zone^2+1) (8 \zone^6+4 \zone^4+5 \zone^2-2)\sqrt{-1}S_{21}(1,1)_{-3 } E\nonumber\\&\quad{}
-2 \zone^4 (2 \zone^2+1) (7 \zone^2+1)\sqrt{-1}\omega^{[2]}_{-1 } S_{21}(1,1)_{-1 } E\nonumber\\&\quad{}
-2 \zone^2 (7 \zone^2+1) (6 \zone^4+\zone^2+2)\sqrt{-1}\omega^{[1]}_{-1 } S_{21}(1,1)_{-1 } E\nonumber\\&\quad{}
-\zone^2 (8 \zone^6+8 \zone^4-19 \zone^2-3)\sqrt{-1}\omega_{0 } S_{21}(1,1)_{-2 } E\nonumber\\&\quad{}
-\zone^2 (2 \zone^2+1) (32 \zone^6-68 \zone^4-5 \zone^2+1)\sqrt{-1}\omega_{0 } \omega_{0 } S_{21}(1,1)_{-1 } E\nonumber\\&\quad{}
-3 \zone^4 (4 \zone^2+1) (7 \zone^2+1)\sqrt{-1}S_{21}(1,2)_{-2 } E\nonumber\\&\quad{}
+\zone^2 (48 \zone^6+5 \zone^2+1)\sqrt{-1}\omega_{0 } S_{21}(1,2)_{-1 } E\nonumber\\&\quad{}
+2 \zone^2 (7 \zone^2+1) (12 \zone^4+5 \zone^2+1)\sqrt{-1}S_{21}(1,3)_{-1 }E\nonumber\\&\quad{}
-\zone^2 (7 \zone^2+1) (8 \zone^2+1)\omega^{[2]}_{-2 } (\omega_{0}^{[1]}E)\nonumber\\&\quad{}
-(7 \zone^2+1) (8 \zone^4-5 \zone^2+6)\omega^{[1]}_{-2 } (\omega_{0}^{[1]}E)\nonumber\\&\quad{}
+2 \zone^2 (32 \zone^4-\zone^2-1)\omega_{0 } \omega^{[2]}_{-1 } (\omega_{0}^{[1]}E)\nonumber\\&\quad{}
+2 (32 \zone^6-15 \zone^4+11 \zone^2+2)\omega_{0 } \omega^{[1]}_{-1 } (\omega_{0}^{[1]}E)\nonumber\\&\quad{}
+(8 \zone^2+1) (4 \zone^4-5 \zone^2-1)\omega_{0 } \omega_{0 } \omega_{0 } (\omega_{0}^{[1]}E),
\end{align}
\begin{align}
\label{eq:z147z121omega22E-3}
0&=
8  \zone^6 (2  \zone^2+1)^2 (7  \zone^2+1)\omega^{[2]}_{-1 } \omega^{[2]}_{-1 } E\nonumber\\&\quad{}
+4  \zone^2 (7  \zone^2+1) (52  \zone^6+38  \zone^4+36  \zone^2+9)\omega^{[1]}_{-3 } E\nonumber\\&\quad{}
-8  \zone^2 (7  \zone^2+1) (4  \zone^8+12  \zone^6+35  \zone^4+15  \zone^2+6)\omega^{[1]}_{-1 } \omega^{[1]}_{-1 } E\nonumber\\&\quad{}
+12  \zone^2 (7  \zone^2+1) (8  \zone^6+34  \zone^4+15  \zone^2+6)H^{[1]}_{-1 } E\nonumber\\&\quad{}
+ \zone^2 (32  \zone^8-164  \zone^6-120  \zone^4-151  \zone^2-8)\omega_{0 } \omega^{[1]}_{-2 } E\nonumber\\&\quad{}
-8  \zone^4 (160  \zone^8+26  \zone^6+39  \zone^4+7  \zone^2+2)\omega_{0 } \omega_{0 } \omega^{[2]}_{-1 } E\nonumber\\&\quad{}
+4  \zone^2 (320  \zone^10+428  \zone^8+442  \zone^6+200  \zone^4+57  \zone^2+2)\omega_{0 } \omega_{0 } \omega^{[1]}_{-1 } E\nonumber\\&\quad{}
-10  \zone^2 (64  \zone^8+16  \zone^6+22  \zone^4+5  \zone^2+1)\omega_{0 } \omega_{0 } \omega_{0 } \omega_{0 } E\nonumber\\&\quad{}
+2  \zone^2 (7  \zone^2+1) (2  \zone^6-29  \zone^4-9  \zone^2-6)\sqrt{-1}S_{21}(1,1)_{-3 } E\nonumber\\&\quad{}
-4  \zone^4 (2  \zone^2+1) (7  \zone^2+1) (4  \zone^4+4  \zone^2+3)\sqrt{-1}\omega^{[2]}_{-1 } S_{21}(1,1)_{-1 } E\nonumber\\&\quad{}
-8  \zone^2 (7  \zone^2+1) (4  \zone^8+10  \zone^6+22  \zone^4+9  \zone^2+3)\sqrt{-1}\omega^{[1]}_{-1 } S_{21}(1,1)_{-1 } E\nonumber\\&\quad{}
- \zone^2 (64  \zone^8+80  \zone^6+138  \zone^4+ \zone^2+2)\sqrt{-1}\omega_{0 } S_{21}(1,1)_{-2 } E\nonumber\\&\quad{}
+2  \zone^2 (2  \zone^2+1) (320  \zone^8+80  \zone^6+166  \zone^4+47  \zone^2+7)\sqrt{-1}\omega_{0 } \omega_{0 } S_{21}(1,1)_{-1 } E\nonumber\\&\quad{}
+2  \zone^6 (7  \zone^2+1) (34  \zone^2+5)\sqrt{-1}S_{21}(1,2)_{-2 } E\nonumber\\&\quad{}
+ \zone^2 (160  \zone^8+848  \zone^6+494  \zone^4+127  \zone^2+6)\sqrt{-1}\omega_{0 } S_{21}(1,2)_{-1 } E\nonumber\\&\quad{}
+6  \zone^2 (7  \zone^2+1) (16  \zone^6+38  \zone^4+15  \zone^2+2)\sqrt{-1}S_{21}(1,3)_{-1 }E\nonumber\\&\quad{}
+6  \zone^2 (7  \zone^2+1) (2  \zone^4-5  \zone^2-1)\omega^{[2]}_{-2 } (\omega_{0}^{[1]}E)\nonumber\\&\quad{}
+6 (7  \zone^2+1) (2  \zone^6-23  \zone^4-10  \zone^2-6)\omega^{[1]}_{-2 } (\omega_{0}^{[1]}E)\nonumber\\&\quad{}
-2  \zone^2 (24  \zone^6-226  \zone^4-87  \zone^2-14)\omega_{0 } \omega^{[2]}_{-1 } (\omega_{0}^{[1]}E)\nonumber\\&\quad{}
-2 (24  \zone^8-478  \zone^6-249  \zone^4-116  \zone^2-12)\omega_{0 } \omega^{[1]}_{-1 } (\omega_{0}^{[1]}E)\nonumber\\&\quad{}
-(640  \zone^8+464  \zone^6+314  \zone^4+73  \zone^2+6)\omega_{0 } \omega_{0 } \omega_{0 } (\omega_{0}^{[1]}E),
\end{align}
\begin{align}
\label{eq:z147z121omega22E-4}
0&=
-12 \zone^4 (2 \zone^2+1)^2 (7 \zone^2+1)H^{[2]}_{-1 } E\nonumber\\&\quad{}
-8 \zone^2 (2 \zone^2+3) (5 \zone^2+1) (7 \zone^2+1)\omega^{[1]}_{-3 } E\nonumber\\&\quad{}
+8 \zone^2 (7 \zone^2+1) (8 \zone^4+18 \zone^2+13)\omega^{[1]}_{-1 } \omega^{[1]}_{-1 } E\nonumber\\&\quad{}
+12 \zone^2 (7 \zone^2+1) (4 \zone^6-4 \zone^4-17 \zone^2-13)H^{[1]}_{-1 } E\nonumber\\&\quad{}
-\zone^2 (512 \zone^6+180 \zone^4-144 \zone^2+7)\omega_{0 } \omega^{[1]}_{-2 } E\nonumber\\&\quad{}
+16 \zone^4 (62 \zone^6+79 \zone^4+15 \zone^2+3)\omega_{0 } \omega_{0 } \omega^{[2]}_{-1 } E\nonumber\\&\quad{}
-4 \zone^2 (248 \zone^8+564 \zone^6+424 \zone^4+142 \zone^2+5)\omega_{0 } \omega_{0 } \omega^{[1]}_{-1 } E\nonumber\\&\quad{}
+8 \zone^2 (62 \zone^6+79 \zone^4+15 \zone^2+3)\omega_{0 } \omega_{0 } \omega_{0 } \omega_{0 } E\nonumber\\&\quad{}
+2 \zone^2 (4 \zone^2+11) (7 \zone^2+1) (2 \zone^4+3 \zone^2+2)\sqrt{-1}S_{21}(1,1)_{-3 } E\nonumber\\&\quad{}
+8 \zone^4 (2 \zone^2+1) (7 \zone^2+1)\sqrt{-1}\omega^{[2]}_{-1 } S_{21}(1,1)_{-1 } E\nonumber\\&\quad{}
+4 \zone^2 (7 \zone^2+1) (12 \zone^4+20 \zone^2+13)\sqrt{-1}\omega^{[1]}_{-1 } S_{21}(1,1)_{-1 } E\nonumber\\&\quad{}
-\zone^2 (208 \zone^6+400 \zone^4+238 \zone^2+15)\sqrt{-1}\omega_{0 } S_{21}(1,1)_{-2 } E\nonumber\\&\quad{}
-4 \zone^2 (2 \zone^2+1) (124 \zone^6+158 \zone^4+23 \zone^2+5)\sqrt{-1}\omega_{0 } \omega_{0 } S_{21}(1,1)_{-1 } E\nonumber\\&\quad{}
-6 \zone^2 (7 \zone^2+1) (8 \zone^6+18 \zone^4+15 \zone^2+3)\sqrt{-1}S_{21}(1,2)_{-2 } E\nonumber\\&\quad{}
-\zone^2 (96 \zone^6+48 \zone^4+50 \zone^2-5)\sqrt{-1}\omega_{0 } S_{21}(1,2)_{-1 } E\nonumber\\&\quad{}
+2 \zone^2 (7 \zone^2+1) (48 \zone^6-26 \zone^2-13)\sqrt{-1}S_{21}(1,3)_{-1 }E\nonumber\\&\quad{}
-2 \zone^2 (2 \zone^2+7) (7 \zone^2+1)\omega^{[2]}_{-2 } (\omega_{0}^{[1]}E)\nonumber\\&\quad{}
-2 (7 \zone^2+1) (2 \zone^4-35 \zone^2-39)\omega^{[1]}_{-2 } (\omega_{0}^{[1]}E)\nonumber\\&\quad{}
-2 \zone^2 (8 \zone^4-34 \zone^2+11)\omega_{0 } \omega^{[2]}_{-1 } (\omega_{0}^{[1]}E)\nonumber\\&\quad{}
-2 (8 \zone^6+162 \zone^4+221 \zone^2+26)\omega_{0 } \omega^{[1]}_{-1 } (\omega_{0}^{[1]}E)\nonumber\\&\quad{}
+496 \zone^6+624 \zone^4+154 \zone^2+13\omega_{0 } \omega_{0 } \omega_{0 } (\omega_{0}^{[1]}E),
\end{align}
\begin{align}
\label{eq:z147z121omega22E-5}
0&=144 \zone^4 (2 \zone^2+1) (7 \zone^2+1)\omega^{[2]}_{-1 } \omega^{[1]}_{-1 } \ExB 
\nonumber\\&\quad{}
+8 \zone^2 (7 \zone^2+1) (80 \zone^4+61 \zone^2+18)\omega^{[1]}_{-3 } \ExB 
\nonumber\\&\quad{}
-24 \zone^2 (7 \zone^2+1) (12 \zone^4+34 \zone^2+9)\omega^{[1]}_{-1 } \omega^{[1]}_{-1 } \ExB 
\nonumber\\&\quad{}
+12 \zone^2 (7 \zone^2+1) (32 \zone^4+100 \zone^2+27)\Har^{[1]}_{-1 } \ExB 
\nonumber\\&\quad{}
+3 \zone^2 (8 \zone^2+3) (96 \zone^4-74 \zone^2-11)\omega_{0 } \omega^{[1]}_{-2 } \ExB 
\nonumber\\&\quad{}
-24 \zone^4 (16 \zone^2+3) (16 \zone^4+5 \zone^2+1)\omega_{0 } \omega_{0 } \omega^{[2]}_{-1 } \ExB 
\nonumber\\&\quad{}
+12 \zone^2 (512 \zone^8+768 \zone^6+530 \zone^4+130 \zone^2+9)\omega_{0 } \omega_{0 } \omega^{[1]}_{-1 } \ExB 
\nonumber\\&\quad{}
-12 \zone^2 (16 \zone^2+3) (16 \zone^4+5 \zone^2+1)\omega_{0 } \omega_{0 } \omega_{0 } \omega_{0 } \ExB 
\nonumber\\&\quad{}
-6 \zone^2 (\zone^2+2) (7 \zone^2+1) (16 \zone^2+3)\sqrt{-1}S_{21}(1,1)_{-3 } \ExB 
\nonumber\\&\quad{}
-12 \zone^2 (7 \zone^2+1) (24 \zone^4+40 \zone^2+9)\sqrt{-1}\omega^{[1]}_{-1 } S_{21}(1,1)_{-1 } \ExB 
\nonumber\\&\quad{}
+3 \zone^2 (16 \zone^2+3) (16 \zone^4+20 \zone^2+1)\sqrt{-1}\omega_{0 } S_{21}(1,1)_{-2 } \ExB 
\nonumber\\&\quad{}
+12 \zone^2 (2 \zone^2+1) (16 \zone^2+3) (16 \zone^4+5 \zone^2+1)\sqrt{-1}\omega_{0 } \omega_{0 } S_{21}(1,1)_{-1 } \ExB 
\nonumber\\&\quad{}
+6 \zone^2 (3 \zone^2+1) (7 \zone^2+1) (16 \zone^2+3)\sqrt{-1}S_{21}(1,2)_{-2 } \ExB 
\nonumber\\&\quad{}
+3 \zone^2 (16 \zone^2+3) (16 \zone^4+20 \zone^2+1)\sqrt{-1}\omega_{0 } S_{21}(1,2)_{-1 } \ExB 
\nonumber\\&\quad{}
+6 \zone^2 (4 \zone^2+3) (7 \zone^2+1) (16 \zone^2+3)\sqrt{-1}S_{21}(1,3)_{-1 } \ExB 
\nonumber\\&\quad{}
+6 \zone^2 (7 \zone^2+1) (16 \zone^2+3)\omega^{[2]}_{-2 } (\omega^{[1]}_{0}\ExB) 
\nonumber\\&\quad{}
+6 (7 \zone^2+1) (16 \zone^4-81 \zone^2-27)\omega^{[1]}_{-2 } (\omega^{[1]}_{0}\ExB) 
\nonumber\\&\quad{}
-6 \zone^2 (8 \zone^2-1) (16 \zone^2+3)\omega_{0 } \omega^{[2]}_{-1 } (\omega^{[1]}_{0}\ExB) 
\nonumber\\&\quad{}
-6 (128 \zone^6-384 \zone^4-185 \zone^2-18)\omega_{0 } \omega^{[1]}_{-1 } (\omega^{[1]}_{0}\ExB) 
\nonumber\\&\quad{}
-3 (4 \zone^2+1) (16 \zone^2+3)^2\omega_{0 } \omega_{0 } \omega_{0 } (\omega^{[1]}_{0}\ExB).
\end{align}
For $k=3,4,\ldots,\rankL$,
\begin{align}
\label{eq:-2z12omega21Sk1110E-1}
0&=-2  \zone^2\omega^{[2]}_{-1 } (S_{k 1}(1,1)_{0}E) 
-2  \zone^2\omega^{[1]}_{-1 } (S_{k 1}(1,1)_{0}E) 
\nonumber\\&\quad{}
- \zone^2\omega_{0 } \omega_{0 } (S_{k 1}(1,1)_{0}E) 
- \zone^2 ( \zone^2+1)S_{k 1}(1,1)_{-2 } E 
\nonumber\\&\quad{}
+ \zone^2 ( \zone^2+1)\omega_{0 } S_{k 1}(1,1)_{-1 } E 
+3  \zone^4S_{k 1}(1,2)_{-1 } E 
\nonumber\\&\quad{}
- \zone^2 ( \zone-1) ( \zone+1)\sqrt{-1}S_{k 2}(1,1)_{-2 } E \nonumber\\&\quad{}
+ \zone^2 ( \zone-1) ( \zone+1)\sqrt{-1}\omega_{0 } S_{k 2}(1,1)_{-1 } E 
\nonumber\\&\quad{}
+3  \zone^4\sqrt{-1}S_{k 2}(1,2)_{-1 } E,\\
\label{eq:-2z12omega21Sk1110E-2}
0&=
-12  \zone^2  (\zone-1)  (\zone+1)  (4  \zone^4+26  \zone^2+3)\omega^{[k]}_{-2 } (S_{k 1}(1,1)_{0}E) 
\nonumber\\&\quad{}
+6  \zone^2  (\zone-1)  (\zone+1)  (4  \zone^4+8  \zone^2+1)\omega^{[2]}_{-2 } (S_{k 1}(1,1)_{0}E) 
\nonumber\\&\quad{}
+6  \zone^2  (\zone-1)  (\zone+1)  (4  \zone^4+8  \zone^2+1)\omega^{[1]}_{-2 } (S_{k 1}(1,1)_{0}E) 
\nonumber\\&\quad{}
+8  \zone^2  (\zone-1)  (\zone+1)  (4  \zone^4+26  \zone^2+3)\omega_{0 } \omega^{[k]}_{-1 } (S_{k 1}(1,1)_{0}E) 
\nonumber\\&\quad{}
-12  \zone^2  (4  \zone^6+10  \zone^4-9  \zone^2-2)\omega_{0 } \omega^{[2]}_{-1 } (S_{k 1}(1,1)_{0}E) 
\nonumber\\&\quad{}
-12  \zone^2  (4  \zone^6+10  \zone^4-9  \zone^2-2)\omega_{0 } \omega^{[1]}_{-1 } (S_{k 1}(1,1)_{0}E) 
\nonumber\\&\quad{}
-6  \zone^2  (20  \zone^4-14  \zone^2-3)\omega_{0 } \omega_{0 } \omega_{0 } (S_{k 1}(1,1)_{0}E) 
\nonumber\\&\quad{}
-8  \zone^4  (\zone-1)  (\zone+1)  (4  \zone^4+26  \zone^2+3)\omega^{[k]}_{-1 } S_{k 1}(1,1)_{-1 } E 
\nonumber\\&\quad{}
+6  \zone^4  (\zone-1)  (\zone+1)  (2  \zone^2-1)  (2  \zone^2+5)S_{k 1}(1,1)_{-3 } E 
\nonumber\\&\quad{}
+24  \zone^4  (\zone-1)  (\zone+1)  (2  \zone^4+4  \zone^2+1)\omega^{[2]}_{-1 } S_{k 1}(1,1)_{-1 } E 
\nonumber\\&\quad{}
+48  \zone^6  (\zone-1)  (\zone+1)  (\zone^2+1)\omega^{[1]}_{-1 } S_{k 1}(1,1)_{-1 } E 
\nonumber\\&\quad{}
+12  \zone^2  (2  \zone^2-1)  (\zone^4-3  \zone^2-1)\omega_{0 } S_{k 1}(1,1)_{-2 } E 
\nonumber\\&\quad{}
+12  \zone^2  (10  \zone^6-2  \zone^4-4  \zone^2-1)\omega_{0 } \omega_{0 } S_{k 1}(1,1)_{-1 } E 
\nonumber\\&\quad{}
-6  \zone^2  (\zone-1)  (\zone+1)  (2  \zone^2-1)  (4  \zone^2+1)S_{k 1}(1,2)_{-2 } E 
\nonumber\\&\quad{}
+6  \zone^2  (2  \zone^2-1)  (10  \zone^4-1)\omega_{0 } S_{k 1}(1,2)_{-1 } E 
\nonumber\\&\quad{}
-8  \zone^4  (\zone-1)  (\zone+1)  (4  \zone^4+26  \zone^2+3)\sqrt{-1}\omega^{[k]}_{-1 } S_{k 2}(1,1)_{-1 } E 
\nonumber\\&\quad{}
+6  \zone^2  (\zone-1)  (\zone+1)  (2  \zone^2+1)  (2  \zone^4+7  \zone^2+2)\sqrt{-1}S_{k 2}(1,1)_{-3 } E 
\nonumber\\&\quad{}
+12  \zone^4  (\zone-1)  (\zone+1)  (4  \zone^4+6  \zone^2+1)\sqrt{-1}\omega^{[2]}_{-1 } S_{k 2}(1,1)_{-1 } E 
\nonumber\\&\quad{}
+12  \zone^4  (\zone-1)  (\zone+1)  (4  \zone^4+2  \zone^2-1)\sqrt{-1}\omega^{[1]}_{-1 } S_{k 2}(1,1)_{-1 } E 
\nonumber\\&\quad{}
+6  \zone^2  (\zone-1)  (\zone+1)  (20  \zone^4-2  \zone^2-1)\sqrt{-1}\omega_{0 } \omega_{0 } S_{k 2}(1,1)_{-1 } E 
\nonumber\\&\quad{}
+18  \zone^4  (2  \zone^2-1)  (2  \zone^2+1)\sqrt{-1}\omega_{0 } S_{k 2}(1,2)_{-1 } E 
\nonumber\\&\quad{}
-12  \zone^4  (\zone-1)  (\zone+1)  (2  \zone^2+1)\sqrt{-1}S_{k 2}(1,3)_{-1 } E.
\end{align}
We also have 
\begin{align}
\label{eq:R_1+R_2+R_3}
0&=R_1+R_2+R_3
\end{align}
where
\begin{align}
R_1&=
-120  \zone^4  (2  \zone^2+1)  (1536  \zone^{10}-2336  \zone^8+2764  \zone^6-964  \zone^4-61  \zone^2+54)\Har^{[2]}_{-1 } (S_{k 1}(1,1)_{0}E)
\nonumber\\&\quad{}
+8  \zone^2  (86016  \zone^{14}-298496  \zone^{12}+331728  \zone^{10}-120784  \zone^8
\nonumber\\&\qquad{}
-1232  \zone^6+6360  \zone^4-559  \zone^2+108)\omega^{[2]}_{-3 } (S_{k 1}(1,1)_{0}E)
\nonumber\\&\quad{}
+16  \zone^2  (21504  \zone^{14}-37824  \zone^{12}+38184  \zone^{10}-13924  \zone^8\nonumber\\&\qquad{}-394  \zone^6+379  \zone^4+31  \zone^2+54)\omega^{[2]}_{-1 } \omega^{[2]}_{-1 } (S_{k 1}(1,1)_{0}E)
\nonumber\\&\quad{}
-36  \zone^6  (16  \zone^2-3)  (3072  \zone^{10}-4672  \zone^8+5592  \zone^6-868  \zone^4-30  \zone^2-61)\Har^{[k]}_{-1 } (S_{k 1}(1,1)_{0}E)
\nonumber\\&\quad{}
+16  \zone^2  (21504  \zone^{14}-37824  \zone^{12}+38184  \zone^{10}-13924  \zone^8\nonumber\\&\qquad{}-394  \zone^6+379  \zone^4+31  \zone^2+54)\omega^{[2]}_{-1 } \omega^{[1]}_{-1 } (S_{k 1}(1,1)_{0}E)
\nonumber\\&\quad{}
+72  \zone^4  (2  \zone^2-9)  (2  \zone^2-1)^2  (16  \zone^2-3)  (32  \zone^4-6  \zone^2-3)\omega^{[1]}_{-3 } (S_{k 1}(1,1)_{0}E)
\nonumber\\&\quad{}
-\zone^2  (16  \zone^2-3)  (236544  \zone^{14}-221504  \zone^{12}+306744  \zone^{10}+57724  \zone^8\nonumber\\&\qquad{}-113190  \zone^6+23743  \zone^4+12420  \zone^2-1080)\omega_{0 } \omega^{[k]}_{-2 } (S_{k 1}(1,1)_{0}E)
\nonumber\\&\quad{}
-8  \zone^2  (70656  \zone^{14}-290496  \zone^{12}+347928  \zone^{10}-138404  \zone^8\nonumber\\&\qquad{}-3002  \zone^6+7745  \zone^4-829  \zone^2+108)\omega_{0 } \omega^{[2]}_{-2 } (S_{k 1}(1,1)_{0}E)
\nonumber\\&\quad{}
-24  \zone^4  (16  \zone^2-3)  (768  \zone^{10}-4528  \zone^8+5008  \zone^6-1312  \zone^4-360  \zone^2+81)\omega_{0 } \omega^{[1]}_{-2 } (S_{k 1}(1,1)_{0}E)
\nonumber\\&\quad{}
+10  \zone^2  (16  \zone^2-3)  (15360  \zone^{14}-14144  \zone^{12}+19704  \zone^{10}+3964  \zone^8\nonumber\\&\qquad{}-7542  \zone^6+1591  \zone^4+828  \zone^2-72)\omega_{0 } \omega_{0 } \omega^{[k]}_{-1 } (S_{k 1}(1,1)_{0}E)
\nonumber\\&\quad{}
+16  \zone^2  (6144  \zone^{14}-69504  \zone^{12}+106464  \zone^{10}-46880  \zone^8\nonumber\\&\qquad{}-6002  \zone^6+2981  \zone^4-187  \zone^2+27)\omega_{0 } \omega_{0 } \omega^{[2]}_{-1 } (S_{k 1}(1,1)_{0}E)
\nonumber\\&\quad{}
-72  \zone^4  (16  \zone^2-3)  (768  \zone^8-1136  \zone^6+324  \zone^4+164  \zone^2-9)\omega_{0 } \omega_{0 } \omega^{[1]}_{-1 } (S_{k 1}(1,1)_{0}E)
\nonumber\\&\quad{}
-36  \zone^2  (2  \zone^2+1)  (16  \zone^2-3)  (512  \zone^8-800  \zone^6+256  \zone^4+48  \zone^2-5)\omega_{0 } \omega_{0 } \omega_{0 } \omega_{0 } (S_{k 1}(1,1)_{0}E)
\nonumber\\&\quad{}
+\zone^4  (16  \zone^2-3)  (82944  \zone^{14}-239808  \zone^{12}+351112  \zone^{10}-190300  \zone^8\nonumber\\&\qquad{}+44182  \zone^6+2409  \zone^4-2916  \zone^2+216)\omega^{[k]}_{-2 } S_{k 1}(1,1)_{-1 } E
\nonumber\\&\quad{}
-2  \zone^4  (16  \zone^2-3)  (76800  \zone^{14}-230464  \zone^{12}+339928  \zone^{10}-188564  \zone^8\nonumber\\&\qquad{}+44242  \zone^6+2531  \zone^4-2916  \zone^2+216)\omega^{[k]}_{-1 } S_{k 1}(1,1)_{-2 } E,
\end{align}
\begin{align}
R_2&=
-144  \zone^4  (2  \zone^2-1)  (16  \zone^2-3)  (832  \zone^{10}-1548  \zone^8+1390  \zone^6-375  \zone^4-32  \zone^2+18)S_{k 1}(1,1)_{-4 } E
\nonumber\\&\quad{}
-8  \zone^2  (43008  \zone^{16}-97152  \zone^{14}+49296  \zone^{12}-48056  \zone^{10}+12140  \zone^8\nonumber\\&\qquad{}+5706  \zone^6-1382  \zone^4-85  \zone^2-54)\omega^{[2]}_{-1 } S_{k 1}(1,1)_{-2 } E
\nonumber\\&\quad{}
+8  \zone^2  (21504  \zone^{16}-200640  \zone^{14}+161640  \zone^{12}+500  \zone^{10}-22958  \zone^8\nonumber\\&\qquad{}+1605  \zone^6+410  \zone^4+85  \zone^2+54)\omega^{[2]}_{-2 } S_{k 1}(1,1)_{-1 } E
\nonumber\\&\quad{}
-48  \zone^6  (2  \zone^2-1)  (16  \zone^2-3)  (384  \zone^8-760  \zone^6+310  \zone^4-10  \zone^2-9)\omega^{[1]}_{-1 } S_{k 1}(1,1)_{-2 } E
\nonumber\\&\quad{}
-72  \zone^4  (2  \zone^2-1)^2  (2  \zone^2+1)  (16  \zone^2-3)  (32  \zone^4-6  \zone^2-3)\omega^{[1]}_{-2 } S_{k 1}(1,1)_{-1 } E
\nonumber\\&\quad{}
-10  \zone^4  (16  \zone^2-3)  (15360  \zone^{14}-14144  \zone^{12}+19704  \zone^{10}+3964  \zone^8\nonumber\\&\qquad{}-7542  \zone^6+1591  \zone^4+828  \zone^2-72)\omega_{0 } \omega^{[k]}_{-1 } S_{k 1}(1,1)_{-1 } E
\nonumber\\&{}\quad{}+144  \zone^4  (2  \zone^2-1)^2  (16  \zone^2-3)  (32  \zone^8-114  \zone^6+156  \zone^4-19  \zone^2-6)\omega_{0 } S_{k 1}(1,1)_{-3 } E
\nonumber\\&\quad{}
-8  \zone^2  (12288  \zone^{16}-117504  \zone^{14}+221184  \zone^{12}-95896  \zone^{10}-19988  \zone^8\nonumber\\&\qquad{}+7728  \zone^6-400  \zone^4+85  \zone^2+54)\omega_{0 } \omega^{[2]}_{-1 } S_{k 1}(1,1)_{-1 } E
\nonumber\\&\quad{}
+96  \zone^6  (16  \zone^2-3)  (192  \zone^8-364  \zone^6+136  \zone^4+41  \zone^2-6)\omega_{0 } \omega^{[1]}_{-1 } S_{k 1}(1,1)_{-1 } E
\nonumber\\&\quad{}
+72  \zone^4  (16  \zone^2-3)  (256  \zone^{10}-656  \zone^8+256  \zone^6-24  \zone^4-10  \zone^2-3)\omega_{0 } \omega_{0 } S_{k 1}(1,1)_{-2 } E
\nonumber\\&\quad{}
+24  \zone^4  (16  \zone^2-3)  (1536  \zone^{10}-864  \zone^8-1408  \zone^6+742  \zone^4+221  \zone^2-9)\omega_{0 } \omega_{0 } \omega_{0 } S_{k 1}(1,1)_{-1 } E
\nonumber\\&\quad{}
+2  \zone^4  (16  \zone^2-3)  (76800  \zone^{14}-230464  \zone^{12}+339928  \zone^{10}-188564  \zone^8\nonumber\\&\qquad{}+44242  \zone^6+2531  \zone^4-2916  \zone^2+216)\omega^{[k]}_{-1 } S_{k 1}(1,2)_{-1 } E
\nonumber\\&\quad{}
+288  \zone^4  (2  \zone^2-1)  (16  \zone^2-3)  (256  \zone^{10}-408  \zone^8+303  \zone^6-62  \zone^4-18  \zone^2+6)S_{k 1}(1,2)_{-3 } E
\nonumber\\&\quad{}
-48  \zone^4  (15360  \zone^{12}-10304  \zone^{10}+5704  \zone^8\nonumber\\&\qquad{}-104  \zone^6-1254  \zone^4+245  \zone^2+27)\omega^{[2]}_{-1 } S_{k 1}(1,2)_{-1 } E
\nonumber\\&\quad{}
+24  \zone^4  (2  \zone^2-1)  (16  \zone^2-3)  (768  \zone^{10}-1712  \zone^8\nonumber\\&\qquad{}+1136  \zone^6-284  \zone^4-27  \zone^2+18)\omega^{[1]}_{-1 } S_{k 1}(1,2)_{-1 } E
\nonumber\\&\quad{}
+72  \zone^4  (16  \zone^2-3)  (512  \zone^{10}-1184  \zone^8\nonumber\\&\qquad{}+1016  \zone^6-84  \zone^4-96  \zone^2+15)\omega_{0 } \omega_{0 } S_{k 1}(1,2)_{-1 } E,
\end{align}
and
\begin{align}
R_{3}&=-24  \zone^4  (2  \zone^2-1)^2  (16  \zone^2-3)  (768  \zone^8-976  \zone^6+802  \zone^4+11  \zone^2-90)S_{k 1}(1,3)_{-2 } E
\nonumber\\&\quad{}
-24  \zone^4  (2  \zone^2-1)  (16  \zone^2-3)  (896  \zone^8-1752  \zone^6+618  \zone^4+263  \zone^2-78)\omega_{0 } S_{k 1}(1,3)_{-1 } E
\nonumber\\&\quad{}
+\zone^4  (16  \zone^2-3)  (82944  \zone^{14}-215232  \zone^{12}+313736  \zone^{10}-145564  \zone^8\nonumber\\&\qquad{}+37238  \zone^6+2169  \zone^4-3404  \zone^2+216)\sqrt{-1}\omega^{[k]}_{-2 } S_{k 2}(1,1)_{-1 } E
\nonumber\\&\quad{}
-2  \zone^4  (16  \zone^2-3)  (76800  \zone^{14}-205888  \zone^{12}+302552  \zone^{10}-143828  \zone^8\nonumber\\&\qquad{}+37298  \zone^6+2291  \zone^4-3404  \zone^2+216)\sqrt{-1}\omega^{[k]}_{-1 } S_{k 2}(1,1)_{-2 } E
\nonumber\\&\quad{}
-24  \zone^4  (2  \zone^2+1)  (16  \zone^2-3)  (3456  \zone^{10}-5416  \zone^8+3176  \zone^6
\nonumber\\&\qquad{}-734  \zone^4-395  \zone^2+162)\sqrt{-1}S_{k 2}(1,1)_{-4 } E
\nonumber\\&\quad{}
-8  \zone^2  (2  \zone^2+1)  (21504  \zone^{14}-56256  \zone^{12}+41832  \zone^{10}-15568  \zone^8\nonumber\\&\qquad{}+2582  \zone^6-38  \zone^4-131  \zone^2+54)\sqrt{-1}\omega^{[2]}_{-1 } S_{k 2}(1,1)_{-2 } E
\nonumber\\&\quad{}
+8  \zone^2  (21504  \zone^{16}-216000  \zone^{14}+214440  \zone^{12}-89908  \zone^{10}+12174  \zone^8\nonumber\\&\qquad{}+4139  \zone^6-921  \zone^4-58  \zone^2-54)\sqrt{-1}\omega^{[2]}_{-2 } S_{k 2}(1,1)_{-1 } E
\nonumber\\&\quad{}
-48  \zone^6  (16  \zone^2-3)  (768  \zone^{10}-2288  \zone^8+1868  \zone^6-422  \zone^4+30  \zone^2-9)\sqrt{-1}\omega^{[1]}_{-1 } S_{k 2}(1,1)_{-2 } E
\nonumber\\&\quad{}
-10  \zone^4  (16  \zone^2-3)  (15360  \zone^{14}-14144  \zone^{12}+19704  \zone^{10}+3964  \zone^8\nonumber\\&\qquad{}-7542  \zone^6+1591  \zone^4+828  \zone^2-72)\sqrt{-1}\omega_{0 } \omega^{[k]}_{-1 } S_{k 2}(1,1)_{-1 } E
\nonumber\\&\quad{}
+144  \zone^4  (2  \zone^2+1)  (16  \zone^2-3)  (64  \zone^{10}+124  \zone^8-182  \zone^6-12  \zone^4+29  \zone^2+9)\sqrt{-1}\omega_{0 } S_{k 2}(1,1)_{-3 } E
\nonumber\\&\quad{}
-8  \zone^2  (12288  \zone^{16}-123648  \zone^{14}+220032  \zone^{12}-133048  \zone^{10}+10836  \zone^8\nonumber\\&\qquad{}+5888  \zone^6-942  \zone^4+23  \zone^2-54)\sqrt{-1}\omega_{0 } \omega^{[2]}_{-1 } S_{k 2}(1,1)_{-1 } E
\nonumber\\&\quad{}
+24  \zone^4  (16  \zone^2-3)  (1536  \zone^{10}-2016  \zone^8-184  \zone^6+586  \zone^4-25  \zone^2-18)\sqrt{-1}\omega_{0 } \omega_{0 } \omega_{0 } S_{k 2}(1,1)_{-1 } E
\nonumber\\&\quad{}
+2  \zone^4  (16  \zone^2-3)  (76800  \zone^{14}-205888  \zone^{12}+302552  \zone^{10}-143828  \zone^8\nonumber\\&\qquad{}+37298  \zone^6+2291  \zone^4-3404  \zone^2+216)\sqrt{-1}\omega^{[k]}_{-1 } S_{k 2}(1,2)_{-1 } E
\nonumber\\&\quad{}
+24  \zone^4  (2  \zone^2+1)  (16  \zone^2-3)  (1536  \zone^{10}-3328  \zone^8+1400  \zone^6
\nonumber\\&\qquad{}
+220  \zone^4-191  \zone^2+18)\sqrt{-1}S_{k 2}(1,2)_{-3 } E
\nonumber\\&\quad{}
+24  \zone^6  (16  \zone^2-3)  (1536  \zone^{10}-2656  \zone^8+3376  \zone^6-1504  \zone^4+150  \zone^2+27)\sqrt{-1}\omega^{[1]}_{-1 } S_{k 2}(1,2)_{-1 } E
\nonumber\\&\quad{}
+24  \zone^4  (2  \zone^2+1)  (16  \zone^2-3)  (256  \zone^8+1168  \zone^6-622  \zone^4+23  \zone^2+12)\sqrt{-1}\omega_{0 } S_{k 2}(1,3)_{-1 } E.
\end{align}
The following results follow from \eqref{eq:z147z121omega22E-1}--\eqref{eq:-2z12omega21Sk1110E-2}:
\begin{lemma}
Let $\lE\in\Z$ such that 
$\lE\geq \epsilon_{I}(\ExB,\lu)$ for all non-zero $\lu\in K(0)$.
Let $k\in \{3,\ldots,\rankL\}$.
Taking the $(\lE+3)$-th action of \eqref{eq:z147z121omega22E-1} on $\lu$,
we have
\begin{align}
\label{eq:z12t+1)2((9z1^2-1)t+14-0}
0&=- \zone^2 (t+1)^2 ((9  \zone^2-1) t+14  \zone^2-2)E_{t}\lu\nonumber\\&\quad{}
+2  \zone^2 (t+1) (t+2)(\omega_{0}^{[1]}E)_{t+1}\lu\nonumber\\&\quad{}
-2  \zone^4 ((9  \zone^2-1) t+14  \zone^2-2)E_{t}\omega^{[2]}_{1 }\lu \nonumber\\&\quad{}
+2  \zone^2 ((9  \zone^4+ \zone^2+2) t+14  \zone^4+2  \zone^2+2)E_{t}\omega^{[1]}_{1 }\lu \nonumber\\&\quad{}
+2  \zone^2 ((9  \zone^4-1) t+14  \zone^4+5  \zone^2-1)\sqrt{-1}E_{t}S_{21}(1,1)_{1 }\lu \nonumber\\&\quad{}
+10  \zone^4\sqrt{-1}E_{t}S_{21}(1,2)_{2 }\lu \nonumber\\&\quad{}
+4  \zone^2(\omega_{0}^{[1]}E)_{t+1}\omega^{[2]}_{1 }\lu 
+4  \zone^2(\omega_{0}^{[1]}E)_{t+1}\omega^{[1]}_{1 }\lu.
\end{align}
Taking the $(\lE+4)$-th actions of \eqref{eq:z147z121omega22E-2}--\eqref{eq:z147z121omega22E-4} on $\lu$, we have
\begin{align}
\label{eq:z12t+1)2((9z1^2-1)t+14-1}
0&=
\zone^2 (\lE+1)^2 ((32 \zone^6-68 \zone^4-12 \zone^2) \lE^2
\nonumber\\&\qquad{}
+(-128 \zone^6-284 \zone^4-63 \zone^2+7) \lE-266 \zone^6-265 \zone^4-59 \zone^2+14)E_{\lE } \lu
\nonumber\\&\quad{}
 -(\lE+1) ((32 \zone^6-36 \zone^4-13 \zone^2-1) \lE^2\nonumber\\&\qquad{}+(20 \zone^6-128 \zone^4-56 \zone^2-4) \lE-74 \zone^6-124 \zone^4-62 \zone^2-4)(\omega^{[1]}_{0}\ExB)_{\lE+1 } \lu
\nonumber\\&\quad{}
+2 \zone^4 ((32 \zone^6-68 \zone^4-12 \zone^2) \lE^2\nonumber\\&\qquad{}+(-128 \zone^6-284 \zone^4-63 \zone^2+7) \lE-266 \zone^6-265 \zone^4-59 \zone^2+14)E_{\lE } \omega^{[2]}_{1 } \lu
\nonumber\\&\quad{}
 -2 \zone^2 ((32 \zone^8-36 \zone^6-44 \zone^4-32 \zone^2-4) \lE^2\nonumber\\&\qquad{}+(-128 \zone^8-292 \zone^6-235 \zone^4-131 \zone^2-24) \lE\nonumber\\&\qquad{}-266 \zone^8-367 \zone^6-231 \zone^4-111 \zone^2-21)E_{\lE } \omega^{[1]}_{1 } \lu
\nonumber\\&\quad{}
 -\zone^2 ((64 \zone^8-104 \zone^6-78 \zone^4-3 \zone^2+1) \lE^2\nonumber\\&\qquad{}+(-256 \zone^8-528 \zone^6-246 \zone^4-\zone^2+11) \lE\nonumber\\&\qquad{}-532 \zone^8-880 \zone^6-409 \zone^4-61 \zone^2+10)\sqrt{-1}E_{\lE } S_{21}(1,1)_{1 } \lu
\nonumber\\&\quad{}
 -2 ((32 \zone^6-15 \zone^4+11 \zone^2+2) \lE-28 \zone^8-20 \zone^6-59 \zone^4-3 \zone^2+2)(\omega^{[1]}_{0}\ExB)_{\lE+1 } \omega^{[1]}_{1 } \lu
\nonumber\\&\quad{}
 -8 \zone^2 (7 \zone^2+1) (2 \zone^4+1)E_{\lE } \omega^{[1]}_{1 } \omega^{[1]}_{1 } \lu
\nonumber\\&\quad{}
 -\zone^2 ((48 \zone^6+5 \zone^2+1) \lE-416 \zone^6-287 \zone^4-77 \zone^2)\sqrt{-1}E_{\lE } S_{21}(1,2)_{2 } \lu
\nonumber\\&\quad{}
 -2 \zone^2 ((32 \zone^4-\zone^2-1) \lE+28 \zone^6-12 \zone^4-17 \zone^2-11)(\omega^{[1]}_{0}\ExB)_{\lE+1 } \omega^{[2]}_{1 } \lu
\nonumber\\&\quad{}
+12 \zone^2 (7 \zone^2+1) (2 \zone^4+1)E_{\lE } H^{[1]}_{3 } \lu
\nonumber\\&\quad{}
+4 \zone^2 (7 \zone^2+1) (2 \zone^4+1)\sqrt{-1}(\omega^{[1]}_{0}\ExB)_{\lE+1 } S_{21}(1,1)_{1 } \lu
\nonumber\\&\quad{}
 -2 \zone^4 (2 \zone^2+1) (7 \zone^2+1)\sqrt{-1}E_{\lE} S_{21}(1,1)_{1 } \omega^{[2]}_{1 } \lu
\nonumber\\&\quad{}
 -2 \zone^2 (7 \zone^2+1) (6 \zone^4+\zone^2+2)\sqrt{-1}E_{\lE } S_{21}(1,1)_{1 } \omega^{[1]}_{1 } \lu
\nonumber\\&\quad{}
+2 \zone^2 (7 \zone^2+1) (12 \zone^4+5 \zone^2+1)\sqrt{-1}E_{\lE } S_{21}(1,3)_{3 } \lu,
\end{align}
\begin{align}
\label{eq:z12t+1)2((9z1^2-1)t+14-2}
0&=-2 \zone^2 (\lE+1)^2 ((320 \zone^8+80 \zone^6+110 \zone^4+25 \zone^2+5) \lE^2\nonumber\\&\qquad{}+(832 \zone^8+422 \zone^6+257 \zone^4+79 \zone^2+9) \lE+504 \zone^8+382 \zone^6+71 \zone^4+41 \zone^2-2)E_{\lE } \lu
\nonumber\\&\quad{}
+(\lE+1) ((640 \zone^8+464 \zone^6+314 \zone^4+73 \zone^2+6) \lE^2\nonumber\\&\qquad{}+(1712 \zone^8+1572 \zone^6+982 \zone^4+276 \zone^2+24) \lE\nonumber\\&\qquad{}+1116 \zone^8+1450 \zone^6+810 \zone^4+272 \zone^2+24)(\omega^{[1]}_{0}\ExB)_{\lE+1 } \lu
\nonumber\\&\quad{}
 -4 \zone^4 ((320 \zone^8+52 \zone^6+78 \zone^4+14 \zone^2+4) \lE^2\nonumber\\&\qquad{}+(832 \zone^8+366 \zone^6+193 \zone^4+57 \zone^2+7) \lE\nonumber\\&\qquad{}+504 \zone^8+354 \zone^6+39 \zone^4+30 \zone^2-3)E_{\lE } \omega^{[2]}_{1 } \lu
\nonumber\\&\quad{}
+2 \zone^2 ((640 \zone^10+856 \zone^8+884 \zone^6+400 \zone^4+114 \zone^2+4) \lE^2\nonumber\\&\qquad{}+(1664 \zone^10+2780 \zone^8+2966 \zone^6+1586 \zone^4+473 \zone^2+44) \lE\nonumber\\&\qquad{}
+1008 \zone^10+2048 \zone^8+2272 \zone^6+1302 \zone^4+416 \zone^2+46)E_{\lE } \omega^{[1]}_{1 } \lu
\nonumber\\&\quad{}
+2 \zone^2 ((640 \zone^10+480 \zone^8+412 \zone^6+260 \zone^4+61 \zone^2+7) \lE^2\nonumber\\&\qquad{}+(1664 \zone^10+1676 \zone^8+1048 \zone^6+599 \zone^4+177 \zone^2+17) \lE\nonumber\\&\qquad{}+1008 \zone^10+1380 \zone^8+1072 \zone^6+467 \zone^4+155 \zone^2+10)\sqrt{-1}E_{\lE } S_{21}(1,1)_{1 } \lu
\nonumber\\&\quad{}
+2 \zone^2 ((24 \zone^6-226 \zone^4-87 \zone^2-14) \lE\nonumber\\&\qquad{}-112 \zone^8-480 \zone^6-408 \zone^4-198 \zone^2-14)(\omega^{[1]}_{0}\ExB)_{\lE+1 } \omega^{[2]}_{1 } \lu
\nonumber\\&\quad{}
+8 \zone^6 (2 \zone^2+1)^2 (7 \zone^2+1)E_{\lE } \omega^{[2]}_{1 } \omega^{[2]}_{1 } \lu
\nonumber\\&\quad{}
+2 ((24 \zone^8-478 \zone^6-249 \zone^4-116 \zone^2-12) \lE\nonumber\\&\qquad{}+112 \zone^10+504 \zone^8-104 \zone^6-132 \zone^4-92 \zone^2-12)(\omega^{[1]}_{0}\ExB)_{\lE+1 } \omega^{[1]}_{1 } \lu
\nonumber\\&\quad{}
 -8 \zone^2 (7 \zone^2+1) (4 \zone^8+12 \zone^6+35 \zone^4+15 \zone^2+6)E_{\lE } \omega^{[1]}_{1 } \omega^{[1]}_{1 } \lu
\nonumber\\&\quad{}
 -\zone^2 ((160 \zone^8+848 \zone^6+494 \zone^4+127 \zone^2+6) \lE\nonumber\\&\qquad{}-676 \zone^8-1678 \zone^6-530 \zone^4-80 \zone^2)\sqrt{-1}E_{\lE } S_{21}(1,2)_{2 } \lu
\nonumber\\&\quad{}
+12 \zone^2 (7 \zone^2+1) (8 \zone^6+34 \zone^4+15 \zone^2+6)E_{\lE } H^{[1]}_{3 } \lu
\nonumber\\&\quad{}
+4 \zone^2 (7 \zone^2+1) (8 \zone^6+34 \zone^4+15 \zone^2+6)\sqrt{-1}(\omega^{[1]}_{0}\ExB)_{\lE+1 } S_{21}(1,1)_{1 } \lu
\nonumber\\&\quad{}
 -4 \zone^4 (2 \zone^2+1) (7 \zone^2+1) (4 \zone^4+4 \zone^2+3)\sqrt{-1}E_{\lE } S_{21}(1,1)_{1 } \omega^{[2]}_{1 } \lu
\nonumber\\&\quad{}
 -8 \zone^2 (7 \zone^2+1) (4 \zone^8+10 \zone^6+22 \zone^4+9 \zone^2+3)\sqrt{-1}E_{\lE } S_{21}(1,1)_{1 } \omega^{[1]}_{1 } \lu
\nonumber\\&\quad{}
+6 \zone^2 (7 \zone^2+1) (16 \zone^6+38 \zone^4+15 \zone^2+2)\sqrt{-1}E_{\lE } S_{21}(1,3)_{3 } \lu,
\end{align}
\begin{align}
\label{eq:z12t+1)2((9z1^2-1)t+14-3}
0&=-(\lE+1) ((496 \zone^6+624 \zone^4+154 \zone^2+13) \lE^2\nonumber\\&\qquad{}+(1360 \zone^6+1844 \zone^4+602 \zone^2+52) \lE+932 \zone^6+1402 \zone^4+614 \zone^2+52)(\omega^{[1]}_{0}\ExB)_{\lE+1 } \lu
\nonumber\\&\quad{}
 -2 \zone^2 ((496 \zone^8+1128 \zone^6+848 \zone^4+284 \zone^2+10) \lE^2\nonumber\\&\qquad{}+(1376 \zone^8+3316 \zone^6+2842 \zone^4+1112 \zone^2+129) \lE\nonumber\\&\qquad{}+896 \zone^8+2284 \zone^6+2202 \zone^4+954 \zone^2+132)E_{\lE } \omega^{[1]}_{1 } \lu
\nonumber\\&\quad{}
+2 \zone^2 (\lE+1)^2 ((248 \zone^6+316 \zone^4+60 \zone^2+12) \lE^2\nonumber\\&\qquad{}+(688 \zone^6+922 \zone^4+255 \zone^2+16) \lE+448 \zone^6+620 \zone^4+208 \zone^2-16)E_{\lE } \lu
\nonumber\\&\quad{}
+4 \zone^4 ((248 \zone^6+316 \zone^4+60 \zone^2+12) \lE^2\nonumber\\&\qquad{}+(688 \zone^6+922 \zone^4+255 \zone^2+16) \lE+448 \zone^6+620 \zone^4+208 \zone^2-16)E_{\lE } \omega^{[2]}_{1 } \lu
\nonumber\\&\quad{}
 -2 \zone^2 ((496 \zone^8+880 \zone^6+408 \zone^4+66 \zone^2+10) \lE^2\nonumber\\&\qquad{}+(1376 \zone^8+2532 \zone^6+1488 \zone^4+323 \zone^2+20) \lE\nonumber\\&\qquad{}+728 \zone^8+1664 \zone^6+1274 \zone^4+392 \zone^2+10)\sqrt{-1}E_{\lE } S_{21}(1,1)_{1 } \lu
\nonumber\\&\quad{}
+\zone^2 ((96 \zone^6+48 \zone^4+50 \zone^2-5) \lE\nonumber\\&\qquad{}+1008 \zone^8+292 \zone^6-470 \zone^4-368 \zone^2-18)\sqrt{-1}E_{\lE } S_{21}(1,2)_{2 } \lu
\nonumber\\&\quad{}
+2 \zone^2 ((8 \zone^4-34 \zone^2+11) \lE+112 \zone^6+264 \zone^4+136 \zone^2-20)(\omega^{[1]}_{0}\ExB)_{\lE+1 } \omega^{[2]}_{1 } \lu
\nonumber\\&\quad{}
+2 ((8 \zone^6+162 \zone^4+221 \zone^2+26) \lE\nonumber\\&\qquad{}-112 \zone^8-272 \zone^6-104 \zone^4+138 \zone^2+26)(\omega^{[1]}_{0}\ExB)_{\lE+1 } \omega^{[1]}_{1 } \lu
\nonumber\\&\quad{}
 -12 \zone^4 (2 \zone^2+1)^2 (7 \zone^2+1)E_{\lE } H^{[2]}_{3 } \lu
\nonumber\\&\quad{}
+8 \zone^2 (7 \zone^2+1) (8 \zone^4+18 \zone^2+13)E_{\lE } \omega^{[1]}_{1 } \omega^{[1]}_{1 } \lu
\nonumber\\&\quad{}
+12 \zone^2 (7 \zone^2+1) (4 \zone^6-4 \zone^4-17 \zone^2-13)E_{\lE } H^{[1]}_{3 } \lu
\nonumber\\&\quad{}
 -4 \zone^2 (7 \zone^2+1) (8 \zone^4+18 \zone^2+13)\sqrt{-1}(\omega^{[1]}_{0}\ExB)_{\lE+1 } S_{21}(1,1)_{1 } \lu
\nonumber\\&\quad{}
+8 \zone^4 (2 \zone^2+1) (7 \zone^2+1)\sqrt{-1}E_{\lE } S_{21}(1,1)_{1 } \omega^{[2]}_{1 } \lu
\nonumber\\&\quad{}
+4 \zone^2 (7 \zone^2+1) (12 \zone^4+20 \zone^2+13)\sqrt{-1}E_{\lE } S_{21}(1,1)_{1 } \omega^{[1]}_{1 } \lu
\nonumber\\&\quad{}
+2 \zone^2 (7 \zone^2+1) (48 \zone^6-26 \zone^2-13)\sqrt{-1}E_{\lE } S_{21}(1,3)_{3 } \lu,
\end{align}
\begin{align}
\label{eq:z12t+1)2((9z1^2-1)t+14-4}
0&=
-2 \zone^2 (\lE+1)^2 ((512 \zone^6+256 \zone^4+62 \zone^2+6) \lE^2\nonumber\\&\qquad{}+(1456 \zone^6+1040 \zone^4+211 \zone^2) \lE+980 \zone^6+896 \zone^4+120 \zone^2-24)E_{\lE } \lu
\nonumber\\&\quad{}
+(\lE+1) ((1024 \zone^6+640 \zone^4+132 \zone^2+9) \lE^2\nonumber\\&\qquad{}+(2832 \zone^6+2192 \zone^4+530 \zone^2+36) \lE+1960 \zone^6+2006 \zone^4+550 \zone^2+36)(\omega^{[1]}_{0}\ExB)_{\lE+1 } \lu
\nonumber\\&\quad{}
-4 \zone^4 ((512 \zone^6+256 \zone^4+62 \zone^2+6) \lE^2\nonumber\\&\qquad{}+(1456 \zone^6+1040 \zone^4+211 \zone^2) \lE+980 \zone^6+896 \zone^4+120 \zone^2-24)E_{\lE } \omega^{[ 2]}_{1 } \lu
\nonumber\\&\quad{}
+2 \zone^2 ((1024 \zone^8+1536 \zone^6+1060 \zone^4+260 \zone^2+18) \lE^2\nonumber\\&\qquad{}+(2912 \zone^8+5248 \zone^6+3910 \zone^4+1190 \zone^2+129) \lE\nonumber\\&\qquad{}+1960 \zone^8+4088 \zone^6+3246 \zone^4+1046 \zone^2+120)E_{\lE } \omega^{[1]}_{1 } \lu
\nonumber\\&\quad{}
+2 \zone^2 ((1024 \zone^8+1024 \zone^6+380 \zone^4+74 \zone^2+6) \lE^2\nonumber\\&\qquad{}+(2912 \zone^8+3536 \zone^6+1462 \zone^4+211 \zone^2) \lE\nonumber\\&\qquad{}+1960 \zone^8+3108 \zone^6+1744 \zone^4+278 \zone^2-6)\sqrt{-1}E_{\lE } S_{ 21}(1,1)_{1 } \lu
\nonumber\\&\quad{}
+2 \zone^2 ((128 \zone^4+8 \zone^2-3) \lE\nonumber\\&\qquad{}-336 \zone^6-384 \zone^4+24 \zone^2+36)(\omega^{[1]}_{0}\ExB)_{\lE+1 } \omega^{[ 2]}_{1 } \lu
\nonumber\\&\quad{}
+2 ((128 \zone^6-384 \zone^4-185 \zone^2-18) \lE\nonumber\\&\qquad{}+336 \zone^8+832 \zone^6+44 \zone^4-110 \zone^2-18)(\omega^{[1]}_{0}\ExB)_{\lE+1 } \omega^{[1]}_{1 } \lu
\nonumber\\&\quad{}
+48 \zone^4 (2 \zone^2+1) (7 \zone^2+1)E_{\lE } \omega^{[1]}_{1 } \omega^{[ 2]}_{1 } \lu
\nonumber\\&\quad{}
-8 \zone^2 (7 \zone^2+1) (12 \zone^4+34 \zone^2+9)E_{\lE } \omega^{[1]}_{1 } \omega^{[1]}_{1 } \lu
\nonumber\\&\quad{}
-\zone^2 ((256 \zone^6+368 \zone^4+76 \zone^2+3) \lE-1232 \zone^6-1586 \zone^4-372 \zone^2-6)\sqrt{-1}E_{\lE } S_{ 21}(1,2)_{2 } \lu
\nonumber\\&\quad{}
+4 \zone^2 (7 \zone^2+1) (32 \zone^4+100 \zone^2+27)E_{\lE } H^{[1]}_{3 } \lu
\nonumber\\&\quad{}
+4 \zone^2 (7 \zone^2+1) (24 \zone^4+40 \zone^2+9)\sqrt{-1}(\omega^{[1]}_{0}\ExB)_{\lE+1 } S_{ 21}(1,1)_{1 } \lu
\nonumber\\&\quad{}
-4 \zone^2 (7 \zone^2+1) (24 \zone^4+40 \zone^2+9)\sqrt{-1}E_{\lE } S_{ 21}(1,1)_{1 } \omega^{[1]}_{1 } \lu
\nonumber\\&\quad{}
+2 \zone^2 (4 \zone^2+3) (7 \zone^2+1) (16 \zone^2+3)\sqrt{-1}E_{\lE } S_{ 21}(1,3)_{3 } \lu.
\end{align}
Taking the $(\lE+3)$-th action of \eqref{eq:-2z12omega21Sk1110E-1} on $\lu$,
we have
\begin{align}
\label{eq:z12t+1)2((9z1^2-1)t+14-5}
0&=
	-2 \zone^2 (\lE+1) (\lE+2)(S_{k 1}(1,1)_{0}E)_{\lE+1 } \lu
\nonumber\\&\quad{}
-2 \zone^2 ((\zone^2+1) \lE-\zone^2+1)E_{\lE } S_{k 1}(1,1)_{1 } \lu
\nonumber\\&\quad{}
-2 \zone^2 ((\zone^2-1) \lE-\zone^2-1)\sqrt{-1}E_{\lE } S_{k i}(1,1)_{1 } \lu
\nonumber\\&\quad{}
-4 \zone^2(S_{k 1}(1,1)_{0}E)_{\lE+1 } \omega^{[i]}_{1 } \lu
-4 \zone^2(S_{k 1}(1,1)_{0}E)_{\lE+1 } \omega^{[1]}_{1 } \lu
\nonumber\\&\quad{}
+6 \zone^4E_{\lE } S_{k 1}(1,2)_{2 } \lu
+6 \zone^4\sqrt{-1}E_{\lE } S_{k i}(1,2)_{2 } \lu.
\end{align}
Taking the $(\lE+4)$-th action of \eqref{eq:-2z12omega21Sk1110E-2} on $\lu$,
we have
\begin{align}
\label{eq:z12t+1)2((9z1^2-1)t+14-6}
0&=
\zone^2  (\lE+1)  (\lE+2)  ((60  \zone^4-42  \zone^2-9)  \lE+8  \zone^6+170  \zone^4-115  \zone^2-27)(S_{k 1}(1,1)_{0}E)_{\lE+1 }\lu
\nonumber\\&\quad{}
+\zone^2  ((60  \zone^6-12  \zone^4-24  \zone^2-6)  \lE^2+(8  \zone^8+266  \zone^6-37  \zone^4-135  \zone^2-30)  \lE
\nonumber\\&\qquad{}
-24  \zone^8+120  \zone^6+18  \zone^4-90  \zone^2-24)E_{\lE } S_{k 1}(1,1)_{1 }\lu
\nonumber\\&\quad{}
+\zone^2  ((60  \zone^6-66  \zone^4+3  \zone^2+3)  \lE^2+(8  \zone^8+242  \zone^6-331  \zone^4+42  \zone^2+21)  \lE\nonumber\\&\qquad{}
-24  \zone^8+96  \zone^6-222  \zone^4+60  \zone^2+18)\sqrt{-1}E_{\lE } S_{k 2}(1,1)_{1 }\lu
\nonumber\\&\quad{}
-4  \zone^2  (\zone-1)  (\zone+1)  (4  \zone^4+26  \zone^2+3)  (\lE+1)(S_{k 1}(1,1)_{0}E)_{\lE+1 } \omega^{[k]}_{1 }\lu
\nonumber\\&\quad{}
+6  \zone^2  ((4  \zone^6+10  \zone^4-9  \zone^2-2)  \lE+8  \zone^6+38  \zone^4-27  \zone^2-7)(S_{k 1}(1,1)_{0}E)_{\lE+1 } \omega^{[2]}_{1 }\lu
\nonumber\\&\quad{}
+6  \zone^2  ((4  \zone^6+10  \zone^4-9  \zone^2-2)  \lE+8  \zone^6+38  \zone^4-27  \zone^2-7)(S_{k 1}(1,1)_{0}E)_{\lE+1 } \omega^{[1]}_{1 }\lu
\nonumber\\&\quad{}
-4  \zone^4  (\zone-1)  (\zone+1)  (4  \zone^4+26  \zone^2+3)E_{\lE } S_{k 1}(1,1)_{1 } \omega^{[k]}_{1 }\lu
\nonumber\\&\quad{}
-\zone^2  ((60  \zone^6-30  \zone^4-6  \zone^2+3)  \lE+40  \zone^8+256  \zone^6-146  \zone^4-45  \zone^2+3)E_{\lE } S_{k 1}(1,2)_{2 }\lu
\nonumber\\&\quad{}
+12  \zone^4  (\zone-1)  (\zone+1)  (2  \zone^2+1)(\omega^{[1]}_{0}\ExB)_{\lE+1 } S_{k 1}(1,1)_{1 }\lu
\nonumber\\&\quad{}
+12  \zone^4  (\zone-1)  (\zone+1)  (2  \zone^4+4  \zone^2+1)E_{\lE } S_{k 1}(1,1)_{1 } \omega^{[2]}_{1 }\lu
\nonumber\\&\quad{}
+24  \zone^6  (\zone-1)  (\zone+1)  (\zone^2+1)E_{\lE } S_{k 1}(1,1)_{1 } \omega^{[1]}_{1 }\lu
\nonumber\\&\quad{}
-4  \zone^4  (\zone-1)  (\zone+1)  (4  \zone^4+26  \zone^2+3)\sqrt{-1}E_{\lE } S_{k 2}(1,1)_{1 } \omega^{[k]}_{1 }\lu
\nonumber\\&\quad{}
-\zone^4  ((36  \zone^4-9)  \lE+40  \zone^6+244  \zone^4-122  \zone^2-54)\sqrt{-1}E_{\lE } S_{k 2}(1,2)_{2 }\lu
\nonumber\\&\quad{}
+12  \zone^4  (\zone-1)  (\zone+1)  (2  \zone^2+1)\sqrt{-1}(\omega^{[1]}_{0}\ExB)_{\lE+1 } S_{k 2}(1,1)_{1 }\lu
\nonumber\\&\quad{}
+6  \zone^4  (\zone-1)  (\zone+1)  (4  \zone^4+6  \zone^2+1)\sqrt{-1}E_{\lE } S_{k 2}(1,1)_{1 } \omega^{[2]}_{1 }\lu
\nonumber\\&\quad{}
+6  \zone^4  (\zone-1)  (\zone+1)  (4  \zone^4+2  \zone^2-1)\sqrt{-1}E_{\lE } S_{k 2}(1,1)_{1 } \omega^{[1]}_{1 }\lu
\nonumber\\&\quad{}
-6  \zone^4  (\zone-1)  (\zone+1)  (2  \zone^2+1)\sqrt{-1}E_{\lE } S_{k 2}(1,3)_{3 }\lu.
\end{align}

\end{lemma}

\subsubsection{The case that $K(0)\cong \C e^{\lambda}$ for some $\lambda\in\fh\setminus\{0\}$}
In this case, $\lu_{\lambda}$ denotes the element of $K(0)$ corresponding to $e^{\lambda}$.
\begin{lemma}
\label{lemma:intertwining-lambda-norm0}
Let $B$ be the vector space with $\lE=\epsilon_{I}(\ExB,\lu_{\lambda})$ in Lemma \ref{lemma:N-Basis}.
The following results hold:
\begin{enumerate}
\item If $\langle\alpha,\lambda\rangle=0$, then $\lE=-1$ or $-2$.

\item[(1-1)]
We further assume $\lE=-1$. Then
\begin{align}
B&\cong \C e^{\alpha+\lambda},
\C e^{\alpha-\lambda},\mbox{ or }
\C e^{\alpha+\lambda}\oplus 
\C e^{\alpha-\lambda}
\end{align}
as $A(M(1)^{+})$-modules and if $\lambda\neq\pm\alpha$, then
\begin{align}
\label{eq:norm0-mN=M(1)+cdotBnonumber-1}
\mN&=M(1)^{+}\cdot B\nonumber\\
&\cong M(1,\lambda+\alpha),M(1,\lambda-\alpha),\mbox{ or }M(1,\lambda+\alpha)\oplus M(1,\lambda-\alpha)
\end{align}
as  $M(1)^{+}$-modules.
\item[(1-2)] We further assume  $\lE=-2$.
Then, $\lambda= \pm\alpha$ and $B\cong M(1)^{-}(0)$
as  $A(M(1)^{+})$-modules.

\item
If $\langle\alpha,\lambda\rangle\neq 0$, then $(1+\lE)^2- \langle \alpha,\lambda\rangle^2=0$,
\begin{align}
B\cong \C e^{\alpha+\lambda}\mbox{ or }\C e^{\alpha-\lambda}
\end{align}
as $A(M(1)^{+})$-modules
and 
\begin{align}
\label{eq:norm0-mN=M(1)+cdotBnonumber-2}
\mN&\cong M(1,\lambda+\alpha),M(1,\lambda-\alpha),\mbox{ or }M(1,\lambda+\alpha)\oplus M(1,\lambda-\alpha)
\end{align}
as  $M(1)^{+}$-modules.
\end{enumerate}
\end{lemma}
\begin{proof}
Assume $\langle\alpha,\lambda\rangle=0$. Then, 
$\langle \lambda,h^{[2]}\rangle=\sqrt{-1}\langle \lambda,h^{[1]}\rangle$.
Deleting the terms including $\langle\lambda,h^{[1]}\rangle^i\ (i\in\Z_{>0})$ or $(\omega^{[1]}_{0}E)_{\lE+1}\lu_{\lambda}$
from \eqref{eq:z12t+1)2((9z1^2-1)t+14-0}--\eqref{eq:z12t+1)2((9z1^2-1)t+14-4},
we have
\begin{align}
0&= \zone^4 (\zone-1) (\zone+1) (2 \zone^2+1)^2 
(\lE+1) (\lE+2)^2 ((24 \zone^2+1) \lE+66 \zone^2)\label{eq:lE(lE+1)(lE+2)(lE+3)-1}
\end{align}
and hence $\lE=-1$ or $-2$ by Lemma \ref{lemma:infinite-set}. 
If $\lE=-2$, then substituting $\lE=-2$ into \eqref{eq:z12t+1)2((9z1^2-1)t+14-0}, we have
\begin{align}
\label{eq:z1=pmlanglelambdah[1]}
\langle\alpha,h^{[1]}\rangle&=\zone=\pm\langle\lambda,h^{[1]}\rangle
\end{align}
and hence $\lambda=\pm\alpha$ since \eqref{eq:z1=pmlanglelambdah[1]} holds for any
$h^{[1]}\in\fh$ with $\langle h^{[1]},h^{[1]}\rangle=1$ and 
$\langle\alpha,h^{[1]}\rangle\neq 0$.
By using the computation in Section \ref{section:Norm-0},
the same argument as in the proof of Lemma \eqref{lemma:intertwining-lambda}
shows that
the subspace spanned by
$(\omega^{[1]}_{0}E)_{-1}\lu_{\lambda}$, $E_{-2}\lu_{\lambda}$, and 
$(S_{k1}(1,1)_{0}E)_{-1}\lu_{\lambda}$, $k=3,\ldots,\rankL$,
is isomorphic to $M(1)^{-}(0)$ as $A(M(1)^{+})$-modules.
Assume $\lE=-1$.
We may assume that $\langle \lambda,h^{[1]}\rangle\neq 0$.
Taking the $(\lE+5)$-th action of \eqref{eq:R_1+R_2+R_3},
we have
\begin{align}
0&=
\zone^3 (52224 \zone^{14}-85056 \zone^{12}+85944 \zone^{10}
-25340 \zone^8-4046 \zone^6+899 \zone^4+224 \zone^2+108)\nonumber\\
&\quad{}\times \langle \lambda,h^{[1]}\rangle (\langle \lambda,h^{[1]}\rangle S_{k 1}(1,1)_{0}E-\langle \lambda,h^{[k]}\rangle \omega^{[1]}_{0}E)_{0}\lu_{\lambda}
\end{align}
and hence
\begin{align}
(S_{k 1}(1,1)_{0}E)_{0}\lu_{\lambda}&=\dfrac{\langle \lambda,h^{[k]}\rangle}{\langle \lambda,h^{[1]}\rangle}
(\omega^{[1]}_{0}E)_{0}\lu_{\lambda}.
\end{align}
A direct computation shows that
$\C(
\langle \alpha,h^{[1]}\rangle\langle \lambda,h^{[1]}\rangle\ExB_{-1}+(\omega^{[1]}_{0}E)_{0})\lu_{\lambda}\cong \C e^{\lambda+\alpha}$ or $0$,
and
$\C(\langle \alpha,h^{[1]}\rangle\langle \lambda,h^{[1]}\rangle\ExB_{-1}-(\omega^{[1]}_{0}E)_{0})\lu_{\lambda}\cong \C e^{\lambda-\alpha}$ or $0$
as $A(M(1)^{+})$-modules.

Assume $\langle\alpha,\lambda\rangle\neq 0$ and $\langle\lambda,\lambda\rangle=0$.
For any non-zero $\zone\in\C$ we can take an orthonormal basis 
 $h^{[1]},\ldots,h^{[\rankL]}$ of $\fh$ so that
\begin{align}
\label{eq:h[1]=dfrac12z1alpha}
h^{[1]}&=\dfrac{1}{2 \zone}\alpha+\dfrac{ \zone}{\langle\alpha,\lambda\rangle}\lambda,\nonumber\\ 
h^{[2]}&=\sqrt{-1}(\dfrac{-1}{2 \zone}\alpha+\dfrac{ \zone}{\langle\alpha,\lambda\rangle}\lambda),\mbox{ and }\nonumber\\
\langle h^{[k]},\lambda\rangle&\neq 0
\end{align} 
for all $k=3,\ldots,\rankL$.
Then,
\begin{align}
\langle \alpha,h^{[1]}\rangle&= \zone\mbox{ and }\langle \alpha,h^{[2]}\rangle=\sqrt{-1} \zone,\nonumber\\
\langle \lambda,h^{[1]}\rangle&=\dfrac{\langle\alpha,\lambda\rangle}{2 \zone}\mbox{ and }
\langle \lambda,h^{[2]}\rangle=-\sqrt{-1}\dfrac{\langle\alpha,\lambda\rangle}{2 \zone}.
\end{align} 
Deleting the terms including $(\omega^{[1]}_{0}E)_{\lE+1}\lu_{\lambda}$
from \eqref{eq:z12t+1)2((9z1^2-1)t+14-0}--\eqref{eq:z12t+1)2((9z1^2-1)t+14-4},
we have
\begin{align}
0&=(\lE-\langle\alpha,\lambda\rangle+1)(\lE+\langle\alpha,\lambda\rangle+1)g(\zone)
\end{align}
where
\begin{align}
\label{eq:-langlealphalambdarangle2lE+1lE+2}
g(\zone)&=\zone^{2}(7\zone^2+1)\big(
-\langle\alpha,\lambda\rangle^2 (\lE+1) (\lE+2)\nonumber\\&\qquad
+(\lE+1) (2 \lE^3+12 \lE^2+(9  \langle\alpha,\lambda\rangle^2+24) \lE+14  \langle\alpha,\lambda\rangle^2+16)\zone^2\nonumber\\&\qquad
-(6 \lE^4+54 \lE^3+(11  \langle\alpha,\lambda\rangle^2+164) \lE^2\nonumber\\
&\qquad\qquad{}+(65  \langle\alpha,\lambda\rangle^2+204) \lE+74  \langle\alpha,\lambda\rangle^2+88)\zone^4\nonumber\\&\qquad
-4 (\lE+1) (18 \lE^3+72 \lE^2+93 \lE+38)\zone^6\nonumber\\&\qquad
+4 (16 \lE^4+70 \lE^3+101 \lE^2+(-18  \langle\alpha,\lambda\rangle^2+49) \lE-28  \langle\alpha,\lambda\rangle^2+2)\zone^8\Big).
\end{align}
By Lemma \ref{lemma:infinite-set},
$\lE=\langle\alpha,\lambda\rangle-1$ or $-\langle\alpha,\lambda\rangle-1$.
Since  
$\langle\alpha,\lambda\rangle\neq 0$, $\lE\neq -1$.
If $\lE\neq -2$, then by \eqref{eq:z12t+1)2((9z1^2-1)t+14-0} and \eqref{eq:z12t+1)2((9z1^2-1)t+14-5},
\begin{align}
(\omega^{[1]}_{0}E)_{\lE+1}\lu_{\lambda}&=
\dfrac{1}{2(t+1) (t+2)}\nonumber\\
&\quad{}\times
((9 \zone^2-1) t^3+(32 \zone^2-4) t^2\nonumber\\
&\qquad{}+(-9 \zone^2 \langle\alpha,\lambda\rangle^2+37 \zone^2-5) t-14 \zone^2 \langle\alpha,\lambda\rangle^2+14 \zone^2-2)\ExB_{\lE}\lu_{\lambda},
\nonumber\\
(S_{k 1}(1,1)_{0}E)_{\lE+1}\lu_{\lambda}&=
\dfrac{-\zone\langle\lambda,h^{[k]}\rangle \langle\alpha,\lambda\rangle}{t+1}\ExB_{\lE}\lu_{\lambda}.
\end{align}
If $\lE=-2$, then 
$\langle\alpha,\lambda\rangle^2=1$
by \eqref{eq:z12t+1)2((9z1^2-1)t+14-0}
and hence
\begin{align}
(\omega^{[1]}_{0}E)_{\lE+1}\lu_{\lambda}
&=	\dfrac{1}{2}\ExB_{\lE}\lu_{\lambda},\nonumber\\
(S_{k 1}(1,1)_{0}E)_{\lE+1}\lu_{\lambda}&=
\zone\langle\lambda,h^{[k]}\rangle \langle\alpha,\lambda\rangle \ExB_{\lE}\lu_{\lambda}
\end{align}
by \eqref{eq:z12t+1)2((9z1^2-1)t+14-1}, \eqref{eq:z12t+1)2((9z1^2-1)t+14-5}, and \eqref{eq:h[1]=dfrac12z1alpha}.
A direct computation shows that $\C \ExB_{\lE}\lu_{\lambda}\cong \C e^{\lambda+\alpha}$ or $\C e^{\lambda+\alpha}$
as $A(M(1)^{+})$-modules.
The same argument as in the proof of Lemma \ref{lemma:intertwining-lambda}
shows that \eqref{eq:norm0-mN=M(1)+cdotBnonumber-1}
and \eqref{eq:norm0-mN=M(1)+cdotBnonumber-2}.

Assume $\langle\alpha,\lambda\rangle\neq 0$ and $\langle\lambda,\lambda\rangle\neq 0$.
We may assume $\lambda\in \C h^{[1]}$.
By taking the intertwining operator $\mK\times M(1,\alpha)\rightarrow \mN\db{x},
(\lv,a)\mapsto e^{x\omega_0}I(a,-x)\lv$,
this case is reduced to the case that $\langle\alpha,\alpha\rangle\neq 0$ in 
Section \ref{section:Thecasethatlanglealphaalpharangleneq0}.
\end{proof}

\subsubsection{The other cases}
\begin{lemma}
\label{lemma:intertwining-M1minus-norm0}
If $K(0)\cong M(1)^{-}(0)$, then
$B\cong\C e^{\alpha}$ as $A(M(1)^{+})$-modules,
and $\mN=M(1)^{+}\cdot B\cong M(1,\alpha)$.
\end{lemma}
\begin{proof}
For $i=1,\ldots,\rankL$,
$u^{[i]}$ denotes the element of $\mK(0)$ 
corresponding to $h^{[i]}(-1)\vac\in M(1)^{-}(0)$.
We write
\begin{align}
\lE=\epsilon(E,u^{[1]})
\end{align}
for simplicity.
Since $S_{21}(1,1)_{1}u^{[1]}=u^{[2]}$ and 
\begin{align}
&[ S_{21}(1,1)_{i},E_{j}]=(-\sqrt{-1}\omega_{0 } E
+2\sqrt{-1}\omega^{[1]}_{0}\ExB)_{i+j}
+\zone^2\sqrt{-1}E_{i+j-1}
\nonumber\\
&=(i+j)\sqrt{-1}E_{i+j-1}+2\sqrt{-1}([\omega^{[1]}_{1},\ExB_{i+j-1}]-\frac{\langle\alpha,h^{[1]}\rangle^2}{2}\ExB_{i+j-1})\nonumber\\
&\quad{}+\zone^2\sqrt{-1}E_{i+j-1}
\end{align}
for $i,j\in\Z$ by \eqref{eq:S21(11)0ExB=-sqrt-1omega0E+2sqrt-1(omega[1]0ExB)},
we have $\epsilon(E,u^{[2]})\leq \lE$, which implies that 
\begin{align}
\label{eq:epsilon(Ea)=lEmbox}
\epsilon(E,a)=\lE
\end{align}
for all non-zero $a\in \C u^{[1]}+\C u^{[2]}$.
 Substituting \eqref{eq:norm2Si1(1,1)1lu=-Si1(1,2)2lu-2}
into \eqref{eq:z12t+1)2((9z1^2-1)t+14-0}--\eqref{eq:z12t+1)2((9z1^2-1)t+14-4},
and then deleting the terms including 
 $(\omega^{[1]}_{0}E)_{\lE+1}\lu$, $E_{\lE}u^{[2]}$, or 
 $(\omega^{[1]}_{0}E)_{\lE+1}u^{[2]}$ from the obtained relations, we have
\begin{align}
0&=
\lE
\zone^4 (2 \zone^2+1)^2 (7 \zone^2+1) (2 \zone^4+1) \nonumber \\
&\quad{}\times
((9\lE+14) \zone^4-5 \zone^2-\lE-1)\nonumber\\
&\quad{}\times ((t+1)^3 (3 t^3+19 t^2+48 t+50)
\nonumber\\&\qquad
+2 (t+1) (2 t^6+15 t^5+64 t^4+218 t^3+498 t^2+610 t+297)\zone^2
\nonumber\\&\qquad
-2 (40 t^6+333 t^5+1105 t^4+1736 t^3+1116 t^2-107 t-335)\zone^4
\nonumber\\&\qquad
-2 (t+1) (192 t^4+1072 t^3+2305 t^2+2222 t+833)\zone^6)\ExB_{\lE}u^{[1]}.
\end{align}
Thus, by Lemma \ref{lemma:infinite-set} and \eqref{eq:epsilon(Ea)=lEmbox}, $\lE=0$. Substituting \eqref{eq:norm2Si1(1,1)1lu=-Si1(1,2)2lu-2} and $\lE=0$ into \eqref{eq:z12t+1)2((9z1^2-1)t+14-0}--\eqref{eq:z12t+1)2((9z1^2-1)t+14-4},
and then deleting the terms including $E_{0}u^{[2]}$ or 
 $(\omega^{[1]}_{0}E)_{1}u^{[2]}$ from the obtained relations,
we have
\begin{align}
0&=\zone^4 (2 \zone^2-1) (2 \zone^2+1)^2 (7 \zone^2+1) (203 \zone^4-110 \zone^2-47)\nonumber\\ 
&\quad{}\times((\omega^{[1]}_{0}\ExB)_{1}+\ExB_{0})u^{[1]}
\end{align}
and hence 
\begin{align}
(\omega^{[1]}_{0}\ExB)_{1}u^{[1]}=-\ExB_{0}u^{[1]}.
\end{align}
In the same manner, we have
\begin{align}
\ExB_{0}u^{[2]}&=\sqrt{-1}\ExB_{0}u^{[1]}
\mbox{ and }(\omega^{[1]}_{0}\ExB)_{1}u^{[2]}=0.
\end{align}
For $i=3,\ldots,\rankL$, by substituting 
\eqref{eq:norm2Si1(1,1)1lu=-Si1(1,2)2lu-2}
into \eqref{eq:z12t+1)2((9z1^2-1)t+14-0}--\eqref{eq:z12t+1)2((9z1^2-1)t+14-4},
the same argument as above shows that
\begin{align}
0&=(\omega^{[1]}_{0}\ExB)_{j+1}u^{[i]}=
\ExB_{j}u^{[i]}.
\end{align}
for all $j\in\Z_{\geq 0}$.
Thus $U=\C \ExB_{\lE}u^{[1]}$.
The same argument as in the proof of Lemma \ref{lemma:intertwining-M1minus}
shows that $U\cong \C e^{\alpha}$ and $N\cong M(1,\alpha)$.
\end{proof}

The same argument as in the proof of Lemmas \ref{lemma:intertwining-lambda-norm0} and 
\ref{lemma:intertwining-M1minus-norm0} shows the following result:
\begin{lemma}
\label{lemma:norm0-twist}
If $K(0)\cong M(\theta)^{\pm}$, then
$\epsilon(\ExB,\lu)\leq -1$ for all $\lu\in K(0)$.
\end{lemma}

\section{Proof of Theorem \ref{theorem:classification-weak-module}
\label{section:Proof of Theorem}
}
In this section, we will show Theorem \ref{theorem:classification-weak-module}.


\begin{lemma}
\label{lemma:irreducible-M(1)-submodule-of-M}
Let $\mK$ be an irreducible $M(1)^{+}$-submodule of a weak $V_{\lattice}^{+}$-module $\module$.
\begin{enumerate}
\item If $\mK\cong M(1)^{-}$, then $V_{\lattice}^{+}\cdot\mK \cong V_{\lattice}^{-}$.
\item
If $\mK\cong M(\theta)^{\pm}$, then
for any non-zero $\lu\in M(\theta)^{\pm}(0)$ and
homogeneous $a\in V_{\lattice}^{+}$,
$\epsilon(a,\lu)\leq \wt a-1$
and hence $V_{\lattice}^{+}\cdot \mK$ is a
completely reducible $\N$-graded weak $V_{\lattice}^{+}$-module.
\end{enumerate}
\end{lemma}
\begin{proof}
\begin{enumerate}
\item
It follows from
Lemmas \ref{lemma:intertwining-M1minus}, \ref{lemma:intertwining-M1minus-norm2}, and  \ref{lemma:intertwining-M1minus-norm0}
that
\begin{align}
\label{eq:Vlattice+cdotmKcongM(1)-}
V_{\lattice}^{+}\cdot\mK\cong M(1)^{-}\oplus \bigoplus_{\alpha\in\lattice\setminus\{0\}}M(1,\alpha)
\end{align}
as $M(1)^{+}$-modules and the $M(1)^{+}$-modules structure in the right hand side of \eqref{eq:Vlattice+cdotmKcongM(1)-}
uniquely determines the weak $V_{\lattice}^{+}$-module structure.
\item
It follows from
Lemmas \ref{lemma:structure-Vlattice-M1} (3),(4), \ref{lemma:structure-Vlattice-M1-norm2} (2), and
\ref{lemma:norm0-twist} that
$\epsilon(\ExB(\alpha),\lu)\leq \langle\alpha,\alpha\rangle/2-1$
for all $\alpha\in \lattice$.
By using Lemmas \ref{lemma:bound-H0-1E} and \ref{lemma:bound-index-S},
for homogeneous $a\in V_{L}^{+}$ an inductive argument on $\wt a$ 
shows that $\epsilon(a,\lu)\leq \wt a-1$.
Thus, $V_{\lattice}^{+}\cdot \mK$ is an $\N$-graded weak $V_{\lattice}^{+}$-module.
It follows from  \cite[Theorem 6.5]{DJL2012} and \cite[Theorems 3.8 and 3.16]{Yamsk2009}
that $V_{\lattice}^{+}\cdot \mK$ is completely reducible.

\end{enumerate}
\end{proof}

For a coset $\Gamma\in \lattice^{\perp}/\lattice$ with $\Gamma\neq\lattice$ and  $2\Gamma\subset\lattice$,
we take a subset $S_{\Gamma}$ of $\Gamma$ so that
$\Gamma=S_{\Gamma}\cup (-S_{\Gamma})$ and $S_{\Gamma}\cap (-S_{\Gamma})=\varnothing$,
where $-S_{\Gamma}=\{-\mu\ |\ \mu\in S_{\Gamma}\}$.

\begin{lemma}
\label{lemma:irreducible-M(1)-submodule-of-M-2}
Let $\mK$ be an irreducible $M(1)^{+}$-submodule of a weak $V_{\lattice}^{+}$-module $\module$
such that $\mK\cong M(1,\lambda)$ for some $\lambda\in\fh\setminus\{0\}$.
Then, $\lambda\in \lattice^{\perp}$.
If $2\lambda\not\in\lattice$, then
 $V_{\lattice}^{+}\cdot \mK$ is irreducible and 
\begin{align}
V_{\lattice}^{+}\cdot \mK&\cong 
\bigoplus\limits_{\mu\in\lambda+\lattice}M(1,\mu)
\end{align}
as $M(1)^{+}$-modules.
If $\lambda\not\in\lattice$ and $2\lambda\in\lattice$,
then $V_{\lattice}^{+}\cdot\mK\cong P$ or $P\oplus P$
where $P$ is an irreducible $V_{\lattice}^{+}$-module such that
\begin{align}
P&\cong \bigoplus\limits_{\mu\in S_{\lambda+\lattice}}M(1,\mu)
\end{align}
as $M(1)^{+}$-modules.
\end{lemma}
\begin{proof}
Let $\lu$ be the element of $\mK$ corresponding to $e^{\lambda}\in M(1,\lambda)$.
Let $\alpha\in\lattice$.

Assume $\langle\alpha,\alpha\rangle\neq 0$.
Setting $\omega^{[1]}=(1/\langle\alpha,\alpha\rangle)\alpha(-1)^2\vac$, we have
\begin{align}
\label{eq:omega11lu}
\omega^{[1]}_{1}\lu&
=\frac{\langle \alpha,\lambda\rangle^2}{2\langle \alpha,\alpha\rangle}\lu.
\end{align}
Since
\begin{align}
\omega^{[1]}_{1}\lu&
=\frac{(\epsilon(\ExB,\lu)+1)^2}{2\langle\alpha,\alpha\rangle}\lu, \frac{\langle\alpha,\alpha\rangle}{2}\lu, \mbox{ or }0
\end{align}
by Lemmas \ref{lemma:structure-Vlattice-M1} (1) and \ref{lemma:structure-Vlattice-M1-norm2} (1), 
we have $\langle \alpha,\lambda\rangle\in\Z$.
If $\langle\alpha,\alpha\rangle=0$, then
it follows from Lemma \ref{lemma:intertwining-lambda-norm0}
that $\langle\alpha,\lambda\rangle\in\Z$.

In the rest of the proof, we assume $\lambda\not\in\lattice$.
For a fixed non-zero $\alpha\in \lattice$, we define the $M(1)^{+}$-submodule
$N=M(1,\alpha)\cdot \mK$.
We write 
\begin{align}
E=E(\alpha)\mbox{ and }
\lE=\epsilon(E,\lu)
\end{align}
for simplicity.
By Lemmas \ref{lemma:intertwining-lambda}, \ref{lemma:intertwining-lambda-norm2}, and \ref{lemma:intertwining-lambda-norm0}, and the assumption that $\lambda\not\in\lattice$,
\begin{align}
N&\cong M(1,\lambda+\alpha),M(1,\lambda-\alpha),
\mbox{ or }M(1,\lambda+\alpha)\oplus M(1,\lambda-\alpha).
\end{align}

We also consider the intertwining operator $Y_{M}$ from $M(1,\alpha)\times N$ to $\module$
and define $N^{(1)}=M(1,\alpha)\cdot N$.
Since
\begin{align}
\label{eq:Ebetaebeta}
&Y(\ExB,x)\ExB\nonumber\\
&=
\big(\exp(\sum_{i=1}^{\infty}\dfrac{\alpha(-i)}{i}x^{i})e^{2\alpha}
+\exp(\sum_{i=1}^{\infty}\dfrac{-\alpha(-i)}{i}x^{i})e^{-2\alpha}\big)x^{\langle\alpha,\alpha\rangle}\nonumber\\
&\quad{}+\big(\exp(\sum_{i=1}^{\infty}\dfrac{\alpha(-i)}{i}x^{i})
+\exp(\sum_{i=1}^{\infty}\dfrac{-\alpha(-i)}{i}x^{i})\big)e^{0}x^{-\langle\alpha,\alpha\rangle},
\end{align}
we have $\mK\subset N^{(1)}$.

Assume $N\cong M(1,\lambda+\alpha)$ or $M(1,\lambda-\alpha)$.
We may assume $N\cong M(1,\lambda+\alpha)$.
Then, $Y_{M}$ induces an intertwining operator  
$M(1,\alpha)\times\mK\rightarrow N\cong M(1,\lambda+\alpha)$.
Thus, Lemmas \ref{lemma:intertwining-lambda}, \ref{lemma:intertwining-lambda-norm2}, \ref{lemma:intertwining-lambda-norm1-2},
and \ref{lemma:intertwining-lambda-norm1} show that
$N^{(1)}=\mK\oplus N^{(2)}$ where $N^{(2)}=0$ or  $N^{(2)}\cong M(1,\lambda+2\alpha)$ as $M(1)^{+}$-modules.
We have intertwining operators
\begin{align}
f^{\alpha,\lambda+\alpha,\lambda}(\mbox{ },x)&\in I_{M(1)^{+}}\binom{\mK}{M(1,\alpha)\quad \mN},\nonumber\\
f^{\alpha,\lambda+\alpha,\lambda+2\alpha}(\mbox{ },x)&\in I_{M(1)^{+}}\binom{N^{(2)}}{M(1,\alpha)\quad \mN}
\end{align}
such that
\begin{align}
f^{\alpha,\lambda+\alpha,\lambda}(\mbox{ },x)+
f^{\alpha,\lambda+\alpha,\lambda+2\alpha}(\mbox{ },x)&=
Y_{M}(\mbox{ },x).
\end{align}
Let $a,b\in M(1,\alpha)$ and $\lu\in\mK(0)$.
Since $\textcolor{black}{(x-y)^{k}}Y_{M}(a,x)Y_{M}(b,y)\lu=(x-y)^{k}Y_{M}(b,y)Y_{M}(a,x)\lu$ for sufficiently large $k$,
\begin{align}
(x-y)^{k}f^{\alpha,\lambda+\alpha,\lambda}(a,x)Y_{M}(b,y)\lu
&=(x-y)^{k}f^{\alpha,\lambda+\alpha,\lambda}(b,y)Y_{M}(a,x)\lu
\end{align}
for sufficiently large $k$.
In the following, we shall use the explicit \textcolor{black}{expressions} of 
intertwining operators for $M(1)^{+}$-modules. 
Since
\textcolor{black}{$f^{\alpha,\lambda+\alpha,\lambda}(\ExB,x)Y_{M}(\ExB,y)\lu$}
is a scalar multiple of 
\begin{align}
(x-y)^{-\langle\alpha,\alpha\rangle}
\exp(\sum_{n=1}^{\infty}\dfrac{-\alpha(-n)}{n}x^{n}+\sum_{n=1}^{\infty}\dfrac{\alpha(-n)}{n}y^{n})
x^{-\langle\alpha,\lambda\rangle}
y^{\langle\alpha,\lambda\rangle}\lu
\end{align}
and $f^{\alpha,\lambda+\alpha,\lambda}(\ExB,y)Y_{M}(\ExB,x)\lu$
is a scalar multiple of 
\begin{align}
(x-y)^{-\langle\alpha,\alpha\rangle}
\exp(\sum_{n=1}^{\infty}\dfrac{\alpha(-n)}{n}x^{n}+\sum_{n=1}^{\infty}\dfrac{-\alpha(-n)}{n}y^{n})
x^{\langle\alpha,\lambda\rangle}
y^{-\langle\alpha,\lambda\rangle}\lu,
\end{align}
we have 
$f^{\alpha,\lambda+\alpha,\lambda}(\mbox{ },x)\textcolor{black}{Y_{M}(\mbox{ },y)}=0$ 
and hence
\begin{align}
Y_{M}(\ExB,x)Y_{M}(\ExB,y)\lu=f^{\alpha,\lambda+\alpha,\lambda+2\alpha}(\ExB,x)Y_{M}(\ExB,y)\lu.
\end{align}
This leads to a contradiction
$\ExB(\cdot (\ExB\cdot \lu))\subset N^{(2)}$ since $\lu\in \ExB(\cdot (\ExB\cdot \lu))$.
Thus 
\begin{align}
	\label{eq:mNcong M(1,lambda+alpha) oplus M(1,lambda-alpha)}
	\mN&\cong M(1,\lambda+\alpha) \oplus M(1,\lambda-\alpha).
	\end{align}
For a non-zero $\alpha\in\lattice$, we write
\begin{align}
	\label{eq:mNcong M(1,lambda+alpha) oplus M(1,lambda-alpha)-2}
	M(1,\alpha)\cdot \mK&=M^{[\lambda+\alpha]} \oplus M^{[\lambda-\alpha]}
	\end{align}
where $M^{[\lambda\pm\alpha]}\cong  M(1,\lambda\pm\alpha)$ as $M(1)^{+}$-modules.

as $M(1)^{+}$-modules where for each $i\in I$ there exists $\alpha^{\langle i\rangle}\in \lambda+L$ such that $M^{\langle i\rangle}\cong M(1,\alpha^{\langle i\rangle})$
as $M(1)^{+}$-modules.

Assume $2\lambda\not\in \lattice$. Since $\lambda+\alpha\neq \pm (\lambda+\beta)$
for any pair of distinct elements $\alpha,\beta\in\lattice$,
it follows from	\eqref{eq:mNcong M(1,lambda+alpha) oplus M(1,lambda-alpha)} that 
\begin{align}
	\label{eq:Vlattice+cdot Kcong bigoplus-1}
	V_{\lattice}^{+}\cdot K=\sum_{\alpha\in\lattice}(M(1,\alpha)\cdot K)
\cong \bigoplus_{\mu\in \lambda+\lattice}M(1,\mu)
\end{align}
and $V_{\lattice}^{+}\cdot K$ is an irreducible$V_{\lattice}^{+}$-module.

Assume $\lambda\not\in\lattice$, $\lambda=\alpha/2$ for some $\alpha\in\lattice$, and
$\mK\cap (M(1,\alpha)\cdot \mK)\neq 0$.
Since $K=M^{[\lambda-\alpha]}$ in this case,
we can write $M(1,\alpha)\cdot \mK=\mK\oplus M^{[3\alpha/2]}$.
If $M^{[\alpha/2+\beta]}\cong \mK$ for some non-zero $\beta\in\lattice$,
then $\alpha/2+\beta=\pm \alpha/2$ and hence $\beta=-\alpha$, which leads
that $K=M^{[\alpha/2+\beta]}$.
If $M^{[\alpha/2+\beta]}\cong M^{[\alpha/2+\gamma]}$ for some pair of distinct non-zero elements $\beta,\gamma\in\lattice$,
then $\alpha/2+\beta=-\alpha/2-\gamma$, or equivalently $\beta=-\gamma-\alpha$.
Since $M(1,\beta)\cdot K\subset M(1,\gamma)\cdot M(1,\alpha)\cdot K
=M(1,\gamma)\cdot(\mK\oplus M^{[3\alpha/2]})$, it follows from	\eqref{eq:mNcong M(1,lambda+alpha) oplus M(1,lambda-alpha)} that $\alpha/2+\beta=\sigma_1\alpha/2+\sigma_2 \gamma$
or $\sigma_3(3\alpha/2)+\sigma_4\gamma$ where $\sigma_1,\sigma_2,\sigma_3,\sigma_4\in\{\pm 1\}$.
If $\alpha/2+\beta=\sigma_3(3\alpha/2)+\sigma_4\gamma$,
then a direct computations shows that  $\alpha=0$ or $\beta=0$, which is a contradiction.
Thus, $M^{[\alpha/2+\beta]}=M^{[\alpha/2+\gamma]}$ or $M^{[\alpha/2-\gamma]}$.
It follows from	\eqref{eq:mNcong M(1,lambda+alpha) oplus M(1,lambda-alpha)} that 
\begin{align}
	\label{eq:Vlattice+cdot Kcong bigoplus-2}
	V_{\lattice}^{+}\cdot K=\sum_{\beta\in\lattice}(M(1,\beta)\cdot K)
\cong \bigoplus_{\mu\in S_{\lambda+\lattice}}M(1,\mu)
\end{align}
and 
$V_{\lattice}^{+}\cdot K$ is an irreducible$V_{\lattice}^{+}$-module.

Assume $\lambda=\alpha/2$ for some $\alpha\in\lattice$ and $\mK\cap (M(1,\alpha)\cdot \mK)=0$.
We have $ M(1,\alpha)\cdot \mK=M^{[\alpha/2]}\oplus M^{[3\alpha/2]}$.
By \eqref{eq:Ebetaebeta}, we have $M(1,\alpha)\cdot M^{[\alpha/2]}=K\oplus N^{[3\alpha/2]}$ where $N^{[3\alpha/2]}\cong M(1,3\alpha/2)$ as $M(1)^{+}$-modules.
We take the following four intertwining operators
\begin{align}
I^{1,(\alpha,\alpha/2,\alpha/2)}(\ ,x) &  : M(1,\alpha)\times K\rightarrow M^{[\alpha/2]}\db{x},\nonumber\\ 
I^{1,(\alpha,\alpha/2,3\alpha/2)}(\ ,x)&  : M(1,\alpha)\times K\rightarrow M^{[3\alpha/2]}\db{x},
\nonumber\\
I^{2,(\alpha,\alpha/2,\alpha/2)}(\ ,x) & :  M(1,\alpha)\times M^{[\alpha/2]}\rightarrow K\db{x},\mbox{ and }\nonumber\\
I^{2,(\alpha,\alpha/2,3\alpha/2)}(\ ,x) &  : M(1,\alpha)\times M^{[\alpha/2]}\rightarrow N^{[3\alpha/2]}\db{x}
\end{align} 
such that 
\begin{align}
Y_{M}(\ ,x)|_{M(1,\alpha)\times K} &= I^{1,(\alpha,\alpha/2,\alpha/2)}(\ ,x) +I^{1,(\alpha,\alpha/2,3\alpha/2)}(\ ,x)\mbox{ and }\nonumber\\
Y_{M}(\ ,x)|_{M(1,\alpha)\times M^{[\alpha/2]}} &
= I^{2,(\alpha,\alpha/2,\alpha/2)}(\ ,x) +I^{2,(\alpha,\alpha/2,3\alpha/2)}(\ ,x).
\end{align}
We extend $I^{i,(\alpha,\alpha/2,\alpha/2)}(\ ,x),i=1,2,$ 
to the intertwining operators from $M(1,\alpha)\times (K\oplus M^{[\alpha/2]})$
to $K\oplus M^{[\alpha/2]}$
by setting
\begin{align}
I^{1,(\alpha,\alpha/2,\alpha/2)}(\ ,x)|_{M(1,\alpha)\times M^{[\alpha/2]}}&=I^{2,(\alpha,\alpha/2,\alpha/2)}(\ ,x)|_{M(1,\alpha)\times K}=0
\end{align}
and define
\begin{align}
&I^{(\alpha,\alpha/2,\alpha/2)}(\ ,x)\nonumber\\
&=I^{1,(\alpha,\alpha/2,\alpha/2)}(\ ,x)+I^{2,(\alpha,\alpha/2,\alpha/2)}(\ ,x):
K\oplus M^{[\alpha/2]}\rightarrow K\oplus M^{[\alpha/2]}.
\end{align}
Note that the $M(1)^{+}$-submodule $\mW=K+M^{[\alpha/2]}+M^{[3\alpha/2]}+N^{[3\alpha/2]}$
of $\module$ is completely reducible and
$I^{(\alpha,\alpha/2,\alpha/2)}(\ ,x)=\pr\circ Y_{\module}(\mbox{ },x)|_{K\oplus M^{[\alpha/2]}}$ where $\pr : \mW\rightarrow K\oplus M^{[\alpha/2]}$ is the projection.
We set $\lv=I^{1,(\alpha,\alpha/2,\alpha/2)}(E; \mn/2-1)\lu\in M^{[\alpha/2]}$.
Then, there exists a non-zero $m\in\C$ such that $I^{2,(\alpha,\alpha/2,\alpha/2)}(E; \mn/2-1)
I^{1,(\alpha,\alpha/2,\alpha/2)}(E; \mn/2-1)\lu=m^2\lu$.
Writing $\lw^{\pm}=\lv\pm m\lu\in \Omega_{M(1)^{+}}(\mK\oplus M^{[\alpha/2]})$,
we have $\C \lw^{\pm}\cong \C e^{\alpha/2}$ as $A(M(1)^{+})$-modules.
We set $K^{\pm}=M(1)^{+}\cdot \lw^{\pm}\cong M(1,\alpha/2)$ in $\module$.
Since $I^{(\alpha,\alpha/2,\alpha/2)}(E;\mn/2-1)\lw^{\pm}=m \lw^{\pm}$, we have 
 $M(1,\alpha)\cdot K^{\pm}=K^{\pm}\oplus P^{[3\alpha/2]}$
where $P^{[3\alpha/2]}\cong M(1,3\alpha/2)$ as $M(1)^{+}$-modules.
Replacing $\mK$ by $\mK^{\pm}$, we can obtain the following result
in the same way that we obtained \eqref{eq:Vlattice+cdot Kcong bigoplus-2}:
\begin{align}
	\label{eq:Vlattice+cdot Kcong bigoplus-3}
	V_{\lattice}^{+}\cdot \mK^{\pm}=\sum_{\beta\in\lattice}(M(1,\beta)\cdot \mK^{\pm})
\cong \bigoplus_{\mu\in S_{\lambda+\lattice}}M(1,\mu)
\end{align}
and $V_{\lattice}^{+}\cdot K^{\pm}$ are irreducible $V_{\lattice}^{+}$-modules.
Since $K\oplus M^{[\alpha/2]}=\mK^{+}\oplus \mK^{-}$,
we have $V_{\lattice}^{+}\mK=	V_{\lattice}^{+}\cdot \mK^{+}\oplus V_{\lattice}^{+}\cdot \mK^{-}$.
This completes the proof.
\end{proof}

\begin{lemma}\label{lemma:M(1)directsum}
Every weak $V_{\lattice}^{+}$-module contains an irreducible weak $V_{\lattice}^{+}$-submodule 
which is isomorphic to a submodule of a $\theta$-twisted irreducible $V_{\lattice}$-module or 
is a direct sum of pairwise non-isomorphic irreducible $M(1)^{+}$-modules.
\end{lemma}
\begin{proof}
It follows from Lemma \ref{lemma:Zhu-Omega}, Corollary \ref{corollary:verma-irreducible},
Lemma \ref{lemma:generalizedVermaModule-M(1)-}, and \cite[Theorem 6.2]{DLM1998t}
that 
for a weak $V_{\lattice}^{+}$-module $\module$,
there exists an irreducible $M(1)^{+}$-submodule $\mK$ of $\module$. 
Assume $\mK\cong M(1,\alpha)$ for some non-zero $\alpha\in \lattice$.
It follows from Lemmas \ref{lemma:intertwining-lambda}, \ref{lemma:intertwining-lambda-norm2},
 and 	\ref{lemma:intertwining-lambda-norm0}
that $M(1,\alpha)\cdot \mK$ includes an $A(M(1)^{+})$-module which is isomorphic to
$\C\vac$ or $M(1)^{-}(0)$. Hence 
there exists an irreducible $M(1)^{+}$-submodule of $\module$
which is isomorphic to $M(1)^{+}$ or $M(1)^{-}$
by Corollary \ref{corollary:verma-irreducible} and Lemma \ref{lemma:generalizedVermaModule-M(1)-}.   
Now the results follows from 
Lemma \ref{lemma:irreducible-M(1)-submodule-of-M},
Lemma \ref{lemma:irreducible-M(1)-submodule-of-M-2}, and  \cite[Proposition 4.7.7]{LL}.
\end{proof}

\begin{lemma}
	\label{lemma:st-unique-two}
Let $\lambda\in \lattice^{\perp}\setminus L$, $\mN$ an irreducible weak $V_{\lattice}^{+}$-modules such that
$\mN\cong\oplus_{\beta\in \lambda+\lattice}M(1,\beta)$ as  $M(1)^{+}$-modules.
If $2\lambda\not\in \lattice$, then $\mN\cong V_{\lambda+\lattice}$
and if $2\lambda\in \lattice$, then $\mN\cong V_{\lambda+\lattice}^{+}$ or $V_{\lambda+\lattice}^{-}$ 
 as $V_{\lattice}^{+}$-modules.
\end{lemma}
\begin{proof}
Let $\mN^{1}$ and $\mN^{2}$ be two irreducible weak $V_{\lattice}^{+}$-modules such that
\begin{align}
\label{eq:mNicongbigoplus}
	\mN^{i}&\cong\bigoplus_{\gamma\in \lambda+\lattice}M(1,\gamma), i=1,2,
	\end{align}
as  $M(1)^{+}$-modules.
We may assume $Y_{\mN^{1}}(a ,x)=Y_{\mN^{2}}(a, x)$ for all $a\in M(1)^{+}$.
For $i=1,2$, $\alpha \in\lattice\setminus\{0\}$ and $\beta\in\lambda+\lattice$, let
\begin{align}
\label{eq:fh}
f^{i,(\alpha,\beta,\alpha+\beta)}(\ ,x)& : M(1,\alpha)\times M(1,\beta)\rightarrow M(1,\alpha+\beta),\nonumber\\
f^{i,(\alpha,\beta,-\alpha+\beta)}(\ ,x)& : M(1,\alpha)\times M(1,\beta)\rightarrow M(1,-\alpha+\beta)
\end{align}
such that 
$Y_{\mN^{1}}(\ ,x)|_{M(1,\alpha)\times M(1,\beta)}
=f^{1,(\alpha,\beta,\alpha+\beta)}(\ ,x)+f^{1,(\alpha,\beta,-\alpha+\beta)}(\ ,x)$
and
$Y_{\mN^{2}}(\ ,x)|_{M(1,\alpha)\times M(1,\beta)}
=f^{2,(\alpha,\beta,\alpha+\beta)}(\ ,x)+f^{2,(\alpha,\beta,-\alpha+\beta)}(\ ,x)$.
Set $f^{1,(0,\beta,\beta)}(\ ,x)=f^{2,(0,\beta,\beta)}(\ ,x)=Y_{N^1}(\mbox{ },x)|_{M(1,\beta)}$.
As we have seen in the proof of Lemma \ref{lemma:M(1)directsum},
$f^{i,(\alpha,\beta,\pm\alpha+\beta)}(\ ,x)\neq 0$ for all $i=1,2$, $\alpha\in \lattice$, and $\beta\in \lambda+\lattice$.
Note that for $i=1,2$ and $\alpha\in\lattice,\beta\in\lambda+\lattice$,  
\begin{align}
	\label{eq:fi(alpha,beta,alpha+beta)}
f^{i,(\alpha,\beta,\alpha+\beta)}(\ ,x)
&=f^{i,(-\alpha,\beta,\alpha+\beta)}(\ ,x)\nonumber\\
&=f^{i,(\alpha,-\beta,\alpha+\beta)}(\ ,x)\nonumber\\
&=f^{i,(\alpha,\beta,-\alpha-\beta)}(\ ,x).
\end{align}
For $i=1,2$, $j\in\Z$, and homogeneous $a\in M(1,\alpha)$, we define 
\begin{align}
o^{i}_{j}(a)&=Y_{i}(a;\wt a-j-1),\nonumber\\
o^{i,(\alpha,\beta,\alpha+\beta)}_{j}(a)&=f^{i,(\alpha,\beta,\alpha+\beta)}(a;\wt a-j-1),\nonumber\\
o^{i,(\alpha,\beta,-\alpha+\beta)}_{j}(a)&=f^{i,(\alpha,\beta,-\alpha+\beta)}(a;\wt a-j-1)
\end{align}
and extend $o^{i}_{j}(a)$ and $o^{i,(\alpha,\beta,\pm\alpha+\beta)}_{j}(a)$ for all $a\in M(1,\alpha)$ by linearity.
Let $k\in\Z$, $\lu\in M(1,\beta)(k)$ a non-zero element, $i,j\in \Z$ such that $M(1,\alpha+\beta)(i)\neq 0$ and $M(1,-\alpha+\beta)(j)\neq 0$.
By \cite[Theorem 4.9]{DLM1998v}, there exists $a,b\in M(1,\alpha)$ such that
\begin{align}
\label{eq:0neqomNj-k(b)lu-1}
0&\neq o^{1}_{i-k}(a)\lu\in M(1,\alpha+\beta)(i),\quad 0=o^{1,(\alpha,\beta,-\alpha+\beta)}_{i-k}(a)\lu, \nonumber\\
0&\neq o^{2}_{i-k}(a)\lu\in M(1,\alpha+\beta)(i),\quad 0=o^{2,(\alpha,\beta,-\alpha+\beta)}_{i-k}(a)\lu,
\end{align}
and 
\begin{align}
\label{eq:0neqomNj-k(b)lu-2}
0&\neq o^{1}_{j-k}(b)\lu\in M(1,-\alpha+\beta)(j),\quad 0=o^{1,(\alpha,\beta,\alpha+\beta)}_{j-k}(b)\lu, \nonumber\\
0&\neq o^{2}_{j-k}(b)\lu\in M(1,-\alpha+\beta)(j),\quad 0=o^{2,(\alpha,\beta,\alpha+\beta)}_{j-k}(b)\lu.
\end{align}
Hence
\begin{align}
\label{eq:0neqomNj-k(b)lu-3}
o^{1}_{i-k}(a)\lu&=o^{1,(\alpha,\beta,\alpha+\beta)}_{i-k}(a)\lu, 
\quad
o^{2}_{i-k}(a)\lu=o^{2,(\alpha,\beta,\alpha+\beta)}_{i-k}(a)\lu,\nonumber\\
o^{1}_{j-k}(b)\lu&=o^{1,(\alpha,\beta,-\alpha+\beta)}_{j-k}(b)\lu,
\quad
o^{2}_{j-k}(b)\lu=o^{2,(\alpha,\beta,-\alpha+\beta)}_{j-k}(b)\lu.
\end{align}
By \cite[Theorem 4.7]{ADL2005}, there exists $\varphi(\alpha,\beta,\beta\pm\alpha)\in\C\setminus\{0\}$ such that
$f^{2,(\alpha,\beta,\alpha+\beta)}(\ ,x)=\cocy(\alpha,\beta,\alpha+\beta)f^{1,(\alpha,\beta,\alpha+\beta)}(\ ,x)$
and
$f^{2,(\alpha,\beta,-\alpha+\beta)}(\ ,x)=\cocy(\alpha,\beta,-\alpha+\beta)f^{1,(\alpha,\beta,-\alpha+\beta)}(\ ,x)$.
By \eqref{eq:0neqomNj-k(b)lu-1}--\eqref{eq:0neqomNj-k(b)lu-3} and Borcherds identity,
for $\alpha_1,\alpha_2\in\lattice$,
\begin{align}
	\label{eq:cocy(alpha_2,beta,alpha_2+beta)}
\cocy(\alpha_2,\beta,\alpha_2+\beta)
\cocy(\alpha_1,\alpha_2+\beta,\alpha_1+\alpha_2+\beta)
&=\cocy(\alpha_1+\alpha_2,\beta,\alpha_1+\alpha_2+\beta).
\end{align}

Assume $2\lambda\not\in\lattice$.
We fix $\beta\in \lattice$. Note that $\beta+\alpha \neq \pm(\beta+\gamma)$
for any pair of distinct elements $\alpha,\gamma\in\lattice$.
Then, the map $\varphi : N^{1}\rightarrow N^{2}$ defined by
$\varphi|_{M(1,\alpha+\beta)}=\varphi(\alpha,\beta,\alpha+\beta)\id_{M(1,\alpha+\beta)}$
for $\alpha\in\lattice$
is a weak $V_{\lattice}^{+}$-module isomorphism.

Assume $2\lambda\in\lattice$ and set $\beta=2\lambda\in\lattice$.
It follows from 	\eqref{eq:fi(alpha,beta,alpha+beta)} and \eqref{eq:cocy(alpha_2,beta,alpha_2+beta)}
that
\begin{align}
1&=\varphi(0,\beta/2,\beta/2)=\varphi(-\beta,\beta/2,-\beta/2)\varphi(\beta,-\beta/2,\beta/2)\nonumber\\
&=\varphi(\beta,\beta/2,\beta/2)^2
\end{align}
and hence $\varphi(\beta,\beta/2,\beta/2)=\pm 1$.
If $\varphi(\beta,\beta/2,\beta/2)=1$, then 
the map $\varphi : N^{1}\rightarrow N^{2}$ defined by
$\varphi|_{M(1,\alpha+\beta/2)}=\varphi(\alpha,\beta/2,\alpha+\beta/2)\id_{M(1,\alpha+\beta/2)}$
for $\alpha\in\lattice$
is a weak $V_{\lattice}^{+}$-module isomorphism.
This implies that there are at most two possible weak $V_{\lattice}^{+}$-module structure
on the $M(1)^{+}$-modules with the decomposition in the right hand side of \eqref{eq:Vlattice+cdotmKcongM(1)-}.
This completes the proof.
\end{proof}

\begin{proof}[(Proof of Theorem \ref{theorem:classification-weak-module})]
Combining Lemmas \ref{lemma:irreducible-M(1)-submodule-of-M}, \ref{lemma:irreducible-M(1)-submodule-of-M-2}, \ref{lemma:M(1)directsum}, and \ref{lemma:st-unique-two}, we  have the result.
\end{proof}

\section{Appendix}
\label{section:appendix}
In this appendix, for some $a,b\in V_{\lattice}$, we put the computations of 
$a_{k}b$ for $k\in\Z_{\geq 0}$.
For $k\in\Z_{\geq 0}$ not listed below, $a_{k}b=0$.
Using these results, we can compute the commutation relation $[a_i,b_j]=\sum_{k=0}^{\infty}\binom{i}{k}(a_kb)_{i+j-k}$.
In particular, we see that the space $B$ in Lemma \ref{lemma:N-Basis} is spanned by the elements described in 
the lemma (1)--(5).
Let $h^{[1]},\ldots,h^{[\rankL]}$ be an orthonormal basis of $\fh$.
The rest of this appendix, $i,j,k,l$ are distinct elements of $\{1,\ldots,\rankL\}$
and 
\begin{align}
E=E(\alpha)
\end{align}
for some $\alpha\in\fh\setminus\{0\}$ with $\langle\alpha,\alpha\rangle=\mn$.

\subsection{Computations in $M(1)$}
\label{section:normal-total}
\subsubsection{The case of rank one}
\label{section:normal-first}
\begin{align}
	\omega_{0}\omega&=
	\omega_{0 } \omega _{-1 } \vac
	,
&
	\omega_{1}\omega&=
	2\omega _{-1 } \vac
	,
&
	\omega_{2}\omega&=0,&
	\omega_{3}\omega&=
	\frac{1}{2}\vac
	,
\end{align}

\begin{align}
	\omega_{0}\Har&=
	\omega_{0 } \Har _{-1 } \vac
	,
&	\omega_{1}\Har&=
	4\Har _{-1 } \vac
	,
&	\omega_{2}\Har&=
	\frac{-1}{3}\omega_{0 } \omega _{-1 } \vac
	,
	\nonumber\\
	\omega_{3}\Har&=
	2\omega _{-1 } \vac
	,
&	\omega_{4}\Har&=0,&
	\omega_{5}\Har&=
	\frac{-1}{3}\vac
	,
\end{align}
\begin{align}
	\label{eq:Har0Har=2omega0}
	\Har_{0}\Har&=2\omega_{0 } \omega _{-2 } \omega _{-2 } \vac
	+\frac{24}{5}\omega_{0 } \omega _{-1 } \Har _{-1 } \vac
	\nonumber\\&\quad{}
	+\frac{-2}{5}\omega_{0 }^3 \omega _{-1 } \omega _{-1 } \vac
	+\frac{-3}{10}\omega_{0 }^3 \Har _{-1 } \vac
	+\frac{1}{20}\omega_{0 }^5\omega _{-1 } \vac,\nonumber\\
	\Har_{1}\Har&=
	4\omega _{-2 } \omega _{-2 } \vac
	+\frac{48}{5}\omega _{-1 } \Har _{-1 } \vac
	\nonumber\\&\quad{}
	+\frac{-4}{5}\omega_{0 }^2 \omega _{-1 } \omega _{-1 } \vac
	+\frac{16}{15}\omega_{0 }^2 \Har _{-1 } \vac
	+\frac{7}{45}\omega_{0 }^4\omega _{-1 } \vac,\nonumber\\
	\Har_{2}\Har&=
	10\omega_{0 } \Har _{-1 } \vac
	+\frac{7}{18}\omega_{0 }^3 \omega _{-1 } \vac,\nonumber\\
	\Har_{3}\Har&=
	20\Har_{-1 } \vac
	+\frac{4}{3}\omega_{0 }^2 \omega_{-1 } \vac,\nonumber\\
	\Har_{4}\Har&=
	\frac{10}{3}\omega_{0 } \omega_{-1 } \vac,\nonumber\\
	\Har_{5}\Har&=\frac{20}{3}\omega_{-1 } \vac,\nonumber\\
	\Har_{6}\Har&=0,\nonumber\\
	\Har_{7}\Har&=\frac{5}{3}\vac.
\end{align}
\subsubsection{The general case}
\label{section:appendix-The general case}

\begin{align}
\label{eq:omega[j]0Sij(1,1)=Sij(1,2)-1vac}
	\omega^{[j]}_{0}S_{ij}(1,1)&=
	S_{ij}(1,2)_{-1 } \vac
	,
&
	\omega^{[j]}_{1}S_{ij}(1,1)&=
	S_{ij}(1,1)_{-1 } \vac
	,
\end{align}

\begin{align}
	\omega^{[j]}_{0}S_{ij}(1,2)&=
	2S_{ij}(1,3)_{-1 } \vac
	,
	\nonumber\\
	\omega^{[j]}_{1}S_{ij}(1,2)&=
	2S_{ij}(1,2)_{-1 } \vac
	,
	\nonumber\\
	\omega^{[j]}_{2}S_{ij}(1,2)&=
	2S_{ij}(1,1)_{-1 } \vac
	,
\end{align}

\begin{align}
	\omega^{[j]}_{0}S_{ij}(1,3)&=
	-\omega^{[j]}_{-2 } S_{ij}(1,1)_{-1 } \vac
	+2\omega^{[j]}_{-1 } S_{ij}(1,2)_{-1 } \vac
	,
	\nonumber\\
	\omega^{[j]}_{1}S_{ij}(1,3)&=
	3S_{ij}(1,3)_{-1 } \vac
	,
	\nonumber\\
	\omega^{[j]}_{2}S_{ij}(1,3)&=
	3S_{ij}(1,2)_{-1 } \vac
	,
	\nonumber\\
	\omega^{[j]}_{3}S_{ij}(1,3)&=
	3S_{ij}(1,1)_{-1 } \vac
	,
\end{align}

\begin{align}
\label{eq:Har[j]0Sij(1,1)=-2omega[j]-2Sij(1,1)-1vac}
	\Har^{[j]}_{0}S_{ij}(1,1)&=
	-2\omega^{[j]}_{-2 } S_{ij}(1,1)_{-1 } \vac
	+4\omega^{[j]}_{-1 } S_{ij}(1,2)_{-1 } \vac
	,
	\nonumber\\
	\Har^{[j]}_{1}S_{ij}(1,1)&=
	4S_{ij}(1,3)_{-1 } \vac
	,
	\nonumber\\
	\Har^{[j]}_{2}S_{ij}(1,1)&=
	\frac{7}{3}S_{ij}(1,2)_{-1 } \vac
	,
	\nonumber\\
	\Har^{[j]}_{3}S_{ij}(1,1)&=
	S_{ij}(1,1)_{-1 } \vac
	,
\end{align}

\begin{align}
	\Har^{[j]}_{0}S_{ij}(1,2)&=
	-6\omega_{0 } \omega^{[j]}_{-2 } S_{ij}(1,1)_{-1 } \vac
	+6\omega^{[i]}_{-2 } S_{ij}(1,2)_{-1 } \vac
	\nonumber\\&\quad{}
	-4\omega_{0 } \omega^{[i]}_{-1 } S_{ij}(1,2)_{-1 } \vac
	+12\omega_{0 } \omega^{[j]}_{-1 } S_{ij}(1,2)_{-1 } \vac
	\nonumber\\&\quad{}
	+\omega_{0 } \omega_{0 } \omega_{0 } S_{ij}(1,2)_{-1 } \vac
	+8\omega^{[i]}_{-1 } S_{ij}(1,3)_{-1 } \vac
	\nonumber\\&\quad{}
	-6\omega_{0 } \omega_{0 } S_{ij}(1,3)_{-1 } \vac
	,
	\nonumber\\
	\Har^{[j]}_{1}S_{ij}(1,2)&=
	-6\omega^{[j]}_{-2 } S_{ij}(1,1)_{-1 } \vac
	+12\omega^{[j]}_{-1 } S_{ij}(1,2)_{-1 } \vac
	,
	\nonumber\\
	\Har^{[j]}_{2}S_{ij}(1,2)&=
	\frac{38}{3}S_{ij}(1,3)_{-1 } \vac
	,
	\nonumber\\
	\Har^{[j]}_{3}S_{ij}(1,2)&=
	8S_{ij}(1,2)_{-1 } \vac
	,
	\nonumber\\
	\Har^{[j]}_{4}S_{ij}(1,2)&=
	4S_{ij}(1,1)_{-1 } \vac
	,
\end{align}

\begin{align}
	\Har^{[j]}_{0}S_{ij}(1,3)&=
	\frac{40}{29}\omega^{[j]}_{-3 } S_{ij}(1,2)_{-1 } \vac
	+\frac{60}{29}\Har^{[j]}_{-1 } S_{ij}(1,2)_{-1 } \vac
	\nonumber\\&\quad{}
	+\frac{-60}{29}\omega^{[j]}_{-2 } S_{ij}(1,3)_{-1 } \vac
	,
	\nonumber\\
	\Har^{[j]}_{1}S_{ij}(1,3)&=
	-12\omega_{0 } \omega^{[j]}_{-2 } S_{ij}(1,1)_{-1 } \vac
	+12\omega^{[i]}_{-2 } S_{ij}(1,2)_{-1 } \vac
	\nonumber\\&\quad{}
	-8\omega_{0 } \omega^{[i]}_{-1 } S_{ij}(1,2)_{-1 } \vac
	+24\omega_{0 } \omega^{[j]}_{-1 } S_{ij}(1,2)_{-1 } \vac
	\nonumber\\&\quad{}
	+2\omega_{0 } \omega_{0 } \omega_{0 } S_{ij}(1,2)_{-1 } \vac
	+16\omega^{[i]}_{-1 } S_{ij}(1,3)_{-1 } \vac
	\nonumber\\&\quad{}
	-12\omega_{0 } \omega_{0 } S_{ij}(1,3)_{-1 } \vac
	,
	\nonumber\\
	\Har^{[j]}_{2}S_{ij}(1,3)&=
	\frac{-37}{3}\omega^{[j]}_{-2 } S_{ij}(1,1)_{-1 } \vac
	+\frac{74}{3}\omega^{[j]}_{-1 } S_{ij}(1,2)_{-1 } \vac
	,
	\nonumber\\
	\Har^{[j]}_{3}S_{ij}(1,3)&=
	27S_{ij}(1,3)_{-1 } \vac
	,
	\nonumber\\
	\Har^{[j]}_{4}S_{ij}(1,3)&=
	18S_{ij}(1,2)_{-1 } \vac
	,
	\nonumber\\
	\Har^{[j]}_{5}S_{ij}(1,3)&=
	10S_{ij}(1,1)_{-1 } \vac
	,
\end{align}

\begin{align}
	\omega^{[i]}_{0}S_{ij}(1,1)&=
	\omega_{0 } S_{i1}(1,1)_{-1 } \vac
	-S_{i1}(1,2)_{-1 } \vac
	,
	\nonumber\\
	\omega^{[i]}_{1}S_{ij}(1,1)&=
	S_{i1}(1,1)_{-1 } \vac
	,
\end{align}

\begin{align}
	\omega^{[i]}_{0}S_{ij}(1,2)&=
	\omega_{0 } S_{i1}(1,2)_{-1 } \vac
	-2S_{i1}(1,3)_{-1 } \vac
	,
	\nonumber\\
	\omega^{[i]}_{1}S_{ij}(1,2)&=
	S_{i1}(1,2)_{-1 } \vac
	,
\end{align}

\begin{align}
	\omega^{[i]}_{0}S_{ij}(1,3)&=
	\omega^{[j]}_{-2 } S_{i1}(1,1)_{-1 } \vac
	-2\omega^{[j]}_{-1 } S_{i1}(1,2)_{-1 } \vac
	\nonumber\\&\quad{}
	+\omega_{0 } S_{i1}(1,3)_{-1 } \vac
	,
	\nonumber\\
	\omega^{[i]}_{1}S_{ij}(1,3)&=
	S_{i1}(1,3)_{-1 } \vac
	,
\end{align}

\begin{align}
	\Har^{[i]}_{0}S_{ij}(1,1)&=
	2\omega^{[j]}_{-2 } S_{i1}(1,1)_{-1 } \vac
	+\omega_{0 } \omega_{0 } \omega_{0 } S_{i1}(1,1)_{-1 } \vac
	\nonumber\\&\quad{}
	-4\omega^{[j]}_{-1 } S_{i1}(1,2)_{-1 } \vac
	-3\omega_{0 } \omega_{0 } S_{i1}(1,2)_{-1 } \vac
	\nonumber\\&\quad{}
	+6\omega_{0 } S_{i1}(1,3)_{-1 } \vac
	,
	\nonumber\\
	\Har^{[i]}_{1}S_{ij}(1,1)&=
	2\omega_{0 } \omega_{0 } S_{i1}(1,1)_{-1 } \vac
	-4\omega_{0 } S_{i1}(1,2)_{-1 } \vac
	\nonumber\\&\quad{}
	+4S_{i1}(1,3)_{-1 } \vac
	,
	\nonumber\\
	\Har^{[i]}_{2}S_{ij}(1,1)&=
	\frac{7}{3}\omega_{0 } S_{i1}(1,1)_{-1 } \vac
	+\frac{-7}{3}S_{i1}(1,2)_{-1 } \vac
	,
	\nonumber\\
	\Har^{[i]}_{3}S_{ij}(1,1)&=
	S_{i1}(1,1)_{-1 } \vac
	,
\end{align}

\begin{align}
	\Har^{[i]}_{0}S_{ij}(1,2)&=
	-6\omega^{[i]}_{-2 } S_{i1}(1,2)_{-1 } \vac
	+4\omega_{0 } \omega^{[i]}_{-1 } S_{i1}(1,2)_{-1 } \vac
	\nonumber\\&\quad{}
	-8\omega^{[i]}_{-1 } S_{i1}(1,3)_{-1 } \vac
	,
	\nonumber\\
	\Har^{[i]}_{1}S_{ij}(1,2)&=
	-4\omega^{[j]}_{-2 } S_{i1}(1,1)_{-1 } \vac
	+8\omega^{[j]}_{-1 } S_{i1}(1,2)_{-1 } \vac
	\nonumber\\&\quad{}
	+2\omega_{0 } \omega_{0 } S_{i1}(1,2)_{-1 } \vac
	-8\omega_{0 } S_{i1}(1,3)_{-1 } \vac
	,
	\nonumber\\
	\Har^{[i]}_{2}S_{ij}(1,2)&=
	\frac{7}{3}\omega_{0 } S_{i1}(1,2)_{-1 } \vac
	+\frac{-14}{3}S_{i1}(1,3)_{-1 } \vac
	,
	\nonumber\\
	\Har^{[i]}_{3}S_{ij}(1,2)&=
	S_{i1}(1,2)_{-1 } \vac
	,
\end{align}

\begin{align}
	\Har^{[i]}_{0}S_{ij}(1,3)&=
	\frac{-4}{3}\Har^{[i]}_{-1 } S_{i1}(1,2)_{-1 } \vac
	+\frac{16}{9}\omega^{[i]}_{-3 } S_{i1}(1,2)_{-1 } \vac
	\nonumber\\&\quad{}
	+\frac{-11}{3}\omega_{0 } \omega^{[i]}_{-2 } S_{i1}(1,2)_{-1 } \vac
	+2\omega_{0 } \omega_{0 } \omega^{[i]}_{-1 } S_{i1}(1,2)_{-1 } \vac
	\nonumber\\&\quad{}
	+\frac{4}{3}\omega^{[i]}_{-2 } S_{i1}(1,3)_{-1 } \vac
	-4\omega_{0 } \omega^{[i]}_{-1 } S_{i1}(1,3)_{-1 } \vac
	,
	\nonumber\\
	\Har^{[i]}_{1}S_{ij}(1,3)&=
	-2\omega_{0 } \omega^{[j]}_{-2 } S_{i1}(1,1)_{-1 } \vac
	+6\omega^{[i]}_{-2 } S_{i1}(1,2)_{-1 } \vac
	\nonumber\\&\quad{}
	-4\omega_{0 } \omega^{[i]}_{-1 } S_{i1}(1,2)_{-1 } \vac
	+4\omega_{0 } \omega^{[j]}_{-1 } S_{i1}(1,2)_{-1 } \vac
	\nonumber\\&\quad{}
	+\omega_{0 } \omega_{0 } \omega_{0 } S_{i1}(1,2)_{-1 } \vac
	+8\omega^{[i]}_{-1 } S_{i1}(1,3)_{-1 } \vac
	\nonumber\\&\quad{}
	-4\omega_{0 } \omega_{0 } S_{i1}(1,3)_{-1 } \vac
	,
	\nonumber\\
	\Har^{[i]}_{2}S_{ij}(1,3)&=
	\frac{7}{3}\omega^{[j]}_{-2 } S_{i1}(1,1)_{-1 } \vac
	+\frac{-14}{3}\omega^{[j]}_{-1 } S_{i1}(1,2)_{-1 } \vac
	\nonumber\\&\quad{}
	+\frac{7}{3}\omega_{0 } S_{i1}(1,3)_{-1 } \vac
	,
	\nonumber\\
	\Har^{[i]}_{3}S_{ij}(1,3)&=
	S_{i1}(1,3)_{-1 } \vac
	,
\end{align}

\begin{align}
	\label{eq:Sij(1,1)0Sij(1,1)-1}
S_{ij}(1,1)_{0}S_{ij}(1,1)&=
\omega_{0 } \omega^{[i]}_{-1 } \vac
+\omega_{0 } \omega^{[j]}_{-1 } \vac
,
\nonumber\\
S_{ij}(1,1)_{1}S_{ij}(1,1)&=
2\omega^{[i]}_{-1 } \vac
+2\omega^{[j]}_{-1 } \vac
,
\nonumber\\
S_{ij}(1,1)_{2}S_{ij}(1,1)&=0,\nonumber\\
S_{ij}(1,1)_{3}S_{ij}(1,1)&=
\vac
,
\end{align}

\begin{align}
	\label{eq:Sij(1,1)0Sij(1,1)-2}
S_{ij}(1,1)_{0}S_{ij}(1,2)&=
2\Har^{[i]}_{-1 } \vac
-2\Har^{[j]}_{-1 } \vac
+2\omega_{0 }^{2}\omega^{[i]}_{-1 } \vac
+\omega_{0 }^{2}\omega^{[j]}_{-1 } \vac
,
\nonumber\\
S_{ij}(1,1)_{1}S_{ij}(1,2)&=
2\omega_{0 } \omega^{[i]}_{-1 } \vac
+\omega_{0 } \omega^{[j]}_{-1 } \vac
,
\nonumber\\
S_{ij}(1,1)_{2}S_{ij}(1,2)&=
4\omega^{[i]}_{-1 } \vac
,
\nonumber\\
S_{ij}(1,1)_{3}S_{ij}(1,2)&=0,\nonumber\\
S_{ij}(1,1)_{4}S_{ij}(1,2)&=
2\vac
,
\end{align}

\begin{align}
	\label{eq:Sij(1,1)0Sij(1,1)-3}
S_{ij}(1,1)_{0}S_{ij}(1,3)&=
3\omega_{0 } \Har^{[i]}_{-1 } \vac
-\omega_{0 } \Har^{[j]}_{-1 } \vac
+\omega_{0 }^{3}\omega^{[i]}_{-1 } \vac
+\omega_{0 }^{3}\omega^{[j]}_{-1 } \vac
,
\nonumber\\
S_{ij}(1,1)_{1}S_{ij}(1,3)&=
3\Har^{[i]}_{-1 } \vac
+\Har^{[j]}_{-1 } \vac
+\omega_{0 }^{2}\omega^{[i]}_{-1 } \vac
+\omega_{0 }^{2}\omega^{[j]}_{-1 } \vac
,
\nonumber\\
S_{ij}(1,1)_{2}S_{ij}(1,3)&=
3\omega_{0 } \omega^{[i]}_{-1 } \vac
,
\nonumber\\
S_{ij}(1,1)_{3}S_{ij}(1,3)&=
6\omega^{[i]}_{-1 } \vac
,
\nonumber\\
S_{ij}(1,1)_{4}S_{ij}(1,3)&=0,\nonumber\\
S_{ij}(1,1)_{5}S_{ij}(1,3)&=
3\vac
,
\end{align}

\begin{align}
S_{ij}(1,1)_{0}S_{k j}(1,1)&=
S_{k i}(1,2)_{-1 } \vac
,
\nonumber\\
S_{ij}(1,1)_{1}S_{k j}(1,1)&=
S_{k i}(1,1)_{-1 } \vac
,
\end{align}

\begin{align}
S_{ij}(1,1)_{0}S_{k j}(1,2)&=
2S_{k i}(1,3)_{-1 } \vac
,
\nonumber\\
S_{ij}(1,1)_{1}S_{k j}(1,2)&=
2S_{k i}(1,2)_{-1 } \vac
,
\nonumber\\
S_{ij}(1,1)_{2}S_{k j}(1,2)&=
2S_{k i}(1,1)_{-1 } \vac
,
\end{align}

\begin{align}
S_{ij}(1,1)_{0}S_{k j}(1,3)&=
-3\omega^{[k]}_{-1 } S_{k i}(1,1)_{-2 } \vac
+\omega_{0 } \omega^{[k]}_{-1 } S_{k i}(1,1)_{-1 } \vac
\nonumber\\&\quad{}
+\omega_{0 }^{3}S_{k i}(1,1)_{-1 } \vac
+2\omega^{[k]}_{-1 } S_{k i}(1,2)_{-1 } \vac
\nonumber\\&\quad{}
-3S_{k i}(1,2)_{-3 } \vac
+3\omega_{0 } S_{k i}(1,3)_{-1 } \vac
,
\nonumber\\
S_{ij}(1,1)_{1}S_{k j}(1,3)&=
3S_{k i}(1,3)_{-1 } \vac
,
\nonumber\\
S_{ij}(1,1)_{2}S_{k j}(1,3)&=
3S_{k i}(1,2)_{-1 } \vac
,
\nonumber\\
S_{ij}(1,1)_{3}S_{k j}(1,3)&=
3S_{k i}(1,1)_{-1 } \vac
,
\end{align}

\begin{align}
S_{ij}(1,1)_{0}S_{k i}(1,1)&=
S_{k j}(1,2)_{-1 } \vac
,
\nonumber\\
S_{ij}(1,1)_{1}S_{k i}(1,1)&=
S_{k j}(1,1)_{-1 } \vac
,
\end{align}

\begin{align}
S_{ij}(1,1)_{0}S_{k i}(1,2)&=
2S_{k j}(1,3)_{-1 } \vac
,
\nonumber\\
S_{ij}(1,1)_{1}S_{k i}(1,2)&=
2S_{k j}(1,2)_{-1 } \vac
,
\nonumber\\
S_{ij}(1,1)_{2}S_{k i}(1,2)&=
2S_{k j}(1,1)_{-1 } \vac
,
\end{align}

\begin{align}
S_{ij}(1,1)_{0}S_{k i}(1,3)&=
3\omega^{[k]}_{-2 } S_{k j}(1,1)_{-1 } \vac
-2\omega_{0 } \omega^{[k]}_{-1 } S_{k j}(1,1)_{-1 } \vac
\nonumber\\&\quad{}
+\omega_{0 }^{3}S_{k j}(1,1)_{-1 } \vac
+2\omega^{[k]}_{-1 } S_{k j}(1,2)_{-1 } \vac
\nonumber\\&\quad{}
-3\omega_{0 } S_{k j}(1,2)_{-2 } \vac
+3\omega_{0 } S_{k j}(1,3)_{-1 } \vac
,
\nonumber\\
S_{ij}(1,1)_{1}S_{k i}(1,3)&=
3S_{k j}(1,3)_{-1 } \vac
,
\nonumber\\
S_{ij}(1,1)_{2}S_{k i}(1,3)&=
3S_{k j}(1,2)_{-1 } \vac
,
\nonumber\\
S_{ij}(1,1)_{3}S_{k i}(1,3)&=
3S_{k j}(1,1)_{-1 } \vac
,
\end{align}

\begin{align}
S_{ij}(1,2)_{0}S_{ij}(1,2)&=
-3\omega_{0 } \Har^{[i]}_{-1 } \vac
-\omega_{0 } \Har^{[j]}_{-1 } \vac
-\omega_{0 }^{3}\omega^{[i]}_{-1 } \vac
+\omega_{0 }^{3}\omega^{[j]}_{-1 } \vac
,
\nonumber\\
S_{ij}(1,2)_{1}S_{ij}(1,2)&=
-6\Har^{[i]}_{-1 } \vac
-2\Har^{[j]}_{-1 } \vac
-2\omega_{0 }^{2}\omega^{[i]}_{-1 } \vac
+\omega_{0 }^{2}\omega^{[j]}_{-1 } \vac
,
\nonumber\\
S_{ij}(1,2)_{2}S_{ij}(1,2)&=
-6\omega_{0 } \omega^{[i]}_{-1 } \vac
,
\nonumber\\
S_{ij}(1,2)_{3}S_{ij}(1,2)&=
-12\omega^{[i]}_{-1 } \vac
,
\nonumber\\
S_{ij}(1,2)_{4}S_{ij}(1,2)&=0,\nonumber\\
S_{ij}(1,2)_{5}S_{ij}(1,2)&=
-6\vac
,
\end{align}

\begin{align}
S_{ij}(1,2)_{0}S_{ij}(1,3)&=
-8\omega^{[i]}_{-1 } \Har^{[i]}_{-1 } \vac
-2\omega^{[i]}_{-2 } \omega^{[i]}_{-2 } \vac
\nonumber\\&\quad{}
+2\omega^{[j]}_{-2 } \omega^{[j]}_{-2 } \vac
+8\omega^{[j]}_{-1 } \Har^{[j]}_{-1 } \vac
\nonumber\\&\quad{}
-7\omega_{0 }^{2}\Har^{[i]}_{-1 } \vac
+2\omega_{0 }^{2}\omega^{[i]}_{-1 } \omega^{[i]}_{-1 } \vac
\nonumber\\&\quad{}
-2\omega_{0 }^{2}\omega^{[j]}_{-1 } \omega^{[j]}_{-1 } \vac
-3\omega_{0 }^{2}\Har^{[j]}_{-1 } \vac
\nonumber\\&\quad{}
-19\omega_{0 }^{4}\omega^{[i]}_{-1 } \vac
+2\omega_{0 }^{4}\omega^{[j]}_{-1 } \vac
,
\nonumber\\
S_{ij}(1,2)_{1}S_{ij}(1,3)&=
-6\omega_{0 } \Har^{[i]}_{-1 } \vac
-\omega_{0 } \Har^{[j]}_{-1 } \vac
-\omega_{0 }^{3}\omega^{[i]}_{-1 } \vac
+\omega_{0 }^{3}\omega^{[j]}_{-1 } \vac
,
\nonumber\\
S_{ij}(1,2)_{2}S_{ij}(1,3)&=
-12\Har^{[i]}_{-1 } \vac
-4\omega_{0 }^{2}\omega^{[i]}_{-1 } \vac
,
\nonumber\\
S_{ij}(1,2)_{3}S_{ij}(1,3)&=
-12\omega_{0 } \omega^{[i]}_{-1 } \vac
,
\nonumber\\
S_{ij}(1,2)_{4}S_{ij}(1,3)&=
-24\omega^{[i]}_{-1 } \vac
,
\nonumber\\
S_{ij}(1,2)_{5}S_{ij}(1,3)&=0,\nonumber\\
S_{ij}(1,2)_{6}S_{ij}(1,3)&=
-12\vac
,
\end{align}

\begin{align}
S_{ij}(1,2)_{0}S_{k j}(1,1)&=
-2S_{k i}(1,3)_{-1 } \vac
,
\nonumber\\
S_{ij}(1,2)_{1}S_{k j}(1,1)&=
-2S_{k i}(1,2)_{-1 } \vac
,
\nonumber\\
S_{ij}(1,2)_{2}S_{k j}(1,1)&=
-2S_{k i}(1,1)_{-1 } \vac
,
\end{align}

\begin{align}
S_{ij}(1,2)_{0}S_{k j}(1,2)&=
6\omega^{[k]}_{-1 } S_{k i}(1,1)_{-2 } \vac
-2\omega_{0 } \omega^{[k]}_{-1 } S_{k i}(1,1)_{-1 } \vac
\nonumber\\&\quad{}
 -\omega_{0 }^{3}S_{k i}(1,1)_{-1 } \vac
-4\omega^{[k]}_{-1 } S_{k i}(1,2)_{-1 } \vac
\nonumber\\&\quad{}
+6S_{k i}(1,2)_{-3 } \vac
-6\omega_{0 } S_{k i}(1,3)_{-1 } \vac
,
\nonumber\\
S_{ij}(1,2)_{1}S_{k j}(1,2)&=
-6S_{k i}(1,3)_{-1 } \vac
,
\nonumber\\
S_{ij}(1,2)_{2}S_{k j}(1,2)&=
-6S_{k i}(1,2)_{-1 } \vac
,
\nonumber\\
S_{ij}(1,2)_{3}S_{k j}(1,2)&=
-6S_{k i}(1,1)_{-1 } \vac
,
\end{align}

\begin{align}
S_{ij}(1,2)_{0}S_{k j}(1,3)&=
16\omega^{[k]}_{-1 } S_{k i}(1,1)_{-3 } \vac
-4\Har^{[k]}_{-1 } S_{k i}(1,1)_{-1 } \vac
\nonumber\\&\quad{}
+44\omega_{0 } \omega^{[k]}_{-1 } S_{k i}(1,1)_{-2 } \vac
+2\omega_{0 } \omega^{[i]}_{-2 } S_{k i}(1,1)_{-1 } \vac
\nonumber\\&\quad{}
-16\omega_{0 }^{2}\omega^{[k]}_{-1 } S_{k i}(1,1)_{-1 } \vac
 -\omega_{0 }^{4}S_{k i}(1,1)_{-1 } \vac
\nonumber\\&\quad{}
-10\omega^{[k]}_{-2 } S_{k i}(1,2)_{-1 } \vac
-4\omega^{[k]}_{-1 } S_{k i}(1,2)_{-2 } \vac
\nonumber\\&\quad{}
+2\omega_{0 } S_{k i}(1,2)_{-3 } \vac
-4\omega_{0 } \omega^{[i]}_{-1 } S_{k i}(1,2)_{-1 } \vac
,
\nonumber\\
S_{ij}(1,2)_{1}S_{k j}(1,3)&=
12\omega^{[k]}_{-1 } S_{k i}(1,1)_{-2 } \vac
-4\omega_{0 } \omega^{[k]}_{-1 } S_{k i}(1,1)_{-1 } \vac
\nonumber\\&\quad{}
-2\omega_{0 }^{3}S_{k i}(1,1)_{-1 } \vac
-8\omega^{[k]}_{-1 } S_{k i}(1,2)_{-1 } \vac
\nonumber\\&\quad{}
+12S_{k i}(1,2)_{-3 } \vac
-12\omega_{0 } S_{k i}(1,3)_{-1 } \vac
,
\nonumber\\
S_{ij}(1,2)_{2}S_{k j}(1,3)&=
-12S_{k i}(1,3)_{-1 } \vac
,
\nonumber\\
S_{ij}(1,2)_{3}S_{k j}(1,3)&=
-12S_{k i}(1,2)_{-1 } \vac
,
\nonumber\\
S_{ij}(1,2)_{4}S_{k j}(1,3)&=
-12S_{k i}(1,1)_{-1 } \vac
,
\end{align}

\begin{align}
S_{ij}(1,2)_{0}S_{k i}(1,1)&=
2S_{k j}(1,3)_{-1 } \vac
,
\nonumber\\
S_{ij}(1,2)_{1}S_{k i}(1,1)&=
S_{k j}(1,2)_{-1 } \vac
,
\end{align}

\begin{align}
S_{ij}(1,2)_{0}S_{k i}(1,2)&=
6\omega^{[k]}_{-2 } S_{k j}(1,1)_{-1 } \vac
-4\omega_{0 } \omega^{[k]}_{-1 } S_{k j}(1,1)_{-1 } \vac
\nonumber\\&\quad{}
+\omega_{0 }^{3}S_{k j}(1,1)_{-1 } \vac
+4\omega^{[k]}_{-1 } S_{k j}(1,2)_{-1 } \vac
\nonumber\\&\quad{}
-3\omega_{0 } S_{k j}(1,2)_{-2 } \vac
+6\omega_{0 } S_{k j}(1,3)_{-1 } \vac
,
\nonumber\\
S_{ij}(1,2)_{1}S_{k i}(1,2)&=
4S_{k j}(1,3)_{-1 } \vac
,
\nonumber\\
S_{ij}(1,2)_{2}S_{k i}(1,2)&=
2S_{k j}(1,2)_{-1 } \vac
,
\end{align}

\begin{align}
S_{ij}(1,2)_{0}S_{k i}(1,3)&=
-16\omega^{[k]}_{-1 } S_{k j}(1,1)_{-3 } \vac
+4\Har^{[k]}_{-1 } S_{k j}(1,1)_{-1 } \vac
\nonumber\\&\quad{}
-98\omega_{0 } \omega^{[k]}_{-1 } S_{k j}(1,1)_{-2 } \vac
+34\omega_{0 }^{2}\omega^{[k]}_{-1 } S_{k j}(1,1)_{-1 } \vac
\nonumber\\&\quad{}
+3\omega_{0 }^{4}S_{k j}(1,1)_{-1 } \vac
+2\omega^{[k]}_{-1 } S_{k j}(1,2)_{-2 } \vac
\nonumber\\&\quad{}
+22\omega_{0 } \omega^{[k]}_{-1 } S_{k j}(1,2)_{-1 } \vac
-4\omega_{0 }^{3}S_{k j}(1,2)_{-1 } \vac
\nonumber\\&\quad{}
+6\omega_{0 } S_{k j}(1,3)_{-2 } \vac
,
\nonumber\\
S_{ij}(1,2)_{1}S_{k i}(1,3)&=
9\omega^{[k]}_{-2 } S_{k j}(1,1)_{-1 } \vac
-6\omega_{0 } \omega^{[k]}_{-1 } S_{k j}(1,1)_{-1 } \vac
\nonumber\\&\quad{}
+3\omega_{0 }^{3}S_{k j}(1,1)_{-1 } \vac
+6\omega^{[k]}_{-1 } S_{k j}(1,2)_{-1 } \vac
\nonumber\\&\quad{}
-9\omega_{0 } S_{k j}(1,2)_{-2 } \vac
+9\omega_{0 } S_{k j}(1,3)_{-1 } \vac
,
\nonumber\\
S_{ij}(1,2)_{2}S_{k i}(1,3)&=
6S_{k j}(1,3)_{-1 } \vac
,
\nonumber\\
S_{ij}(1,2)_{3}S_{k i}(1,3)&=
3S_{k j}(1,2)_{-1 } \vac
,
\end{align}

\begin{align}
S_{ij}(1,3)_{0}S_{ij}(1,3)&=
2\omega_{0 } \omega^{[i]}_{-1 } \Har^{[i]}_{-1 } \vac
+5\omega_{0 } \omega^{[i]}_{-2 } \omega^{[i]}_{-2 } \vac
\nonumber\\&\quad{}
+\omega_{0 } \omega^{[j]}_{-2 } \omega^{[j]}_{-2 } \vac
+2\omega_{0 } \omega^{[j]}_{-1 } \Har^{[j]}_{-1 } \vac
\nonumber\\&\quad{}
+\omega_{0 }^{3}\Har^{[i]}_{-1 } \vac
 -\omega_{0 }^{3}\omega^{[i]}_{-1 } \omega^{[i]}_{-1 } \vac
\nonumber\\&\quad{}
 -\omega_{0 }^{3}\omega^{[j]}_{-1 } \omega^{[j]}_{-1 } \vac
-3\omega_{0 }^{3}\Har^{[j]}_{-1 } \vac
\nonumber\\&\quad{}
+7\omega_{0 }^{5} \omega^{[i]}_{-1 } \vac
+\omega_{0 }^{5} \omega^{[j]}_{-1 } \vac
,
\nonumber\\
S_{ij}(1,3)_{1}S_{ij}(1,3)&=
4\omega^{[i]}_{-1 } \Har^{[i]}_{-1 } \vac
+5\omega^{[i]}_{-2 } \omega^{[i]}_{-2 } \vac
\nonumber\\&\quad{}
+\omega^{[j]}_{-2 } \omega^{[j]}_{-2 } \vac
+4\omega^{[j]}_{-1 } \Har^{[j]}_{-1 } \vac
\nonumber\\&\quad{}
+7\omega_{0 }^{2}\Har^{[i]}_{-1 } \vac
 -\omega_{0 }^{2}\omega^{[i]}_{-1 } \omega^{[i]}_{-1 } \vac
\nonumber\\&\quad{}
 -\omega_{0 }^{2}\omega^{[j]}_{-1 } \omega^{[j]}_{-1 } \vac
-3\omega_{0 }^{2}\Har^{[j]}_{-1 } \vac
\nonumber\\&\quad{}
+19\omega_{0 }^{4}\omega^{[i]}_{-1 } \vac
+\omega_{0 }^{4}\omega^{[j]}_{-1 } \vac
,
\nonumber\\
S_{ij}(1,3)_{2}S_{ij}(1,3)&=
15\omega_{0 } \Har^{[i]}_{-1 } \vac
+5\omega_{0 }^{3}\omega^{[i]}_{-1 } \vac
,
\nonumber\\
S_{ij}(1,3)_{3}S_{ij}(1,3)&=
30\Har^{[i]}_{-1 } \vac
+10\omega_{0 }^{2}\omega^{[i]}_{-1 } \vac
,
\nonumber\\
S_{ij}(1,3)_{4}S_{ij}(1,3)&=
30\omega_{0 } \omega^{[i]}_{-1 } \vac
,
\nonumber\\
S_{ij}(1,3)_{5}S_{ij}(1,3)&=
60\omega^{[i]}_{-1 } \vac
,
\nonumber\\
S_{ij}(1,3)_{6}S_{ij}(1,3)&=0,\nonumber\\
S_{ij}(1,3)_{7}S_{ij}(1,3)&=
30\vac
,
\end{align}

\begin{align}
S_{ij}(1,3)_{0}S_{k j}(1,1)&=
-3\omega^{[k]}_{-1 } S_{k i}(1,1)_{-2 } \vac
+\omega_{0 } \omega^{[k]}_{-1 } S_{k i}(1,1)_{-1 } \vac
\nonumber\\&\quad{}
+\omega_{0 }^{3}S_{k i}(1,1)_{-1 } \vac
+2\omega^{[k]}_{-1 } S_{k i}(1,2)_{-1 } \vac
\nonumber\\&\quad{}
-3S_{k i}(1,2)_{-3 } \vac
+3\omega_{0 } S_{k i}(1,3)_{-1 } \vac
,
\nonumber\\
S_{ij}(1,3)_{1}S_{k j}(1,1)&=
3S_{k i}(1,3)_{-1 } \vac
,
\nonumber\\
S_{ij}(1,3)_{2}S_{k j}(1,1)&=
3S_{k i}(1,2)_{-1 } \vac
,
\nonumber\\
S_{ij}(1,3)_{3}S_{k j}(1,1)&=
3S_{k i}(1,1)_{-1 } \vac
,
\end{align}

\begin{align}
S_{ij}(1,3)_{0}S_{k j}(1,2)&=
-16\omega^{[k]}_{-1 } S_{k i}(1,1)_{-3 } \vac
+4\Har^{[k]}_{-1 } S_{k i}(1,1)_{-1 } \vac
\nonumber\\&\quad{}
-44\omega_{0 } \omega^{[k]}_{-1 } S_{k i}(1,1)_{-2 } \vac
-2\omega_{0 } \omega^{[i]}_{-2 } S_{k i}(1,1)_{-1 } \vac
\nonumber\\&\quad{}
+16\omega_{0 }^{2}\omega^{[k]}_{-1 } S_{k i}(1,1)_{-1 } \vac
+\omega_{0 }^{4}S_{k i}(1,1)_{-1 } \vac
\nonumber\\&\quad{}
+10\omega^{[k]}_{-2 } S_{k i}(1,2)_{-1 } \vac
+4\omega^{[k]}_{-1 } S_{k i}(1,2)_{-2 } \vac
\nonumber\\&\quad{}
-2\omega_{0 } S_{k i}(1,2)_{-3 } \vac
+4\omega_{0 } \omega^{[i]}_{-1 } S_{k i}(1,2)_{-1 } \vac
,
\nonumber\\
S_{ij}(1,3)_{1}S_{k j}(1,2)&=
-12\omega^{[k]}_{-1 } S_{k i}(1,1)_{-2 } \vac
+4\omega_{0 } \omega^{[k]}_{-1 } S_{k i}(1,1)_{-1 } \vac
\nonumber\\&\quad{}
+2\omega_{0 }^{3}S_{k i}(1,1)_{-1 } \vac
+8\omega^{[k]}_{-1 } S_{k i}(1,2)_{-1 } \vac
\nonumber\\&\quad{}
-12S_{k i}(1,2)_{-3 } \vac
+12\omega_{0 } S_{k i}(1,3)_{-1 } \vac
,
\nonumber\\
S_{ij}(1,3)_{2}S_{k j}(1,2)&=
12S_{k i}(1,3)_{-1 } \vac
,
\nonumber\\
S_{ij}(1,3)_{3}S_{k j}(1,2)&=
12S_{k i}(1,2)_{-1 } \vac
,
\nonumber\\
S_{ij}(1,3)_{4}S_{k j}(1,2)&=
12S_{k i}(1,1)_{-1 } \vac
,
\end{align}

\begin{align}
S_{ij}(1,3)_{0}S_{k j}(1,3)&=
30\Har^{[i]}_{-1 } S_{k i}(1,2)_{-1 } \vac
+20\omega^{[i]}_{-3 } S_{k i}(1,2)_{-1 } \vac
\nonumber\\&\quad{}
-30\omega^{[i]}_{-2 } S_{k i}(1,3)_{-1 } \vac
,
\nonumber\\
S_{ij}(1,3)_{1}S_{k j}(1,3)&=
-40\omega^{[k]}_{-1 } S_{k i}(1,1)_{-3 } \vac
+10\Har^{[k]}_{-1 } S_{k i}(1,1)_{-1 } \vac
\nonumber\\&\quad{}
-110\omega_{0 } \omega^{[k]}_{-1 } S_{k i}(1,1)_{-2 } \vac
-5\omega_{0 } \omega^{[i]}_{-2 } S_{k i}(1,1)_{-1 } \vac
\nonumber\\&\quad{}
+40\omega_{0 }^{2}\omega^{[k]}_{-1 } S_{k i}(1,1)_{-1 } \vac
+5\omega_{0 }^{4}S_{k i}(1,1)_{-1 } \vac
\nonumber\\&\quad{}
+25\omega^{[k]}_{-2 } S_{k i}(1,2)_{-1 } \vac
+10\omega^{[k]}_{-1 } S_{k i}(1,2)_{-2 } \vac
\nonumber\\&\quad{}
-5\omega_{0 } S_{k i}(1,2)_{-3 } \vac
+10\omega_{0 } \omega^{[i]}_{-1 } S_{k i}(1,2)_{-1 } \vac
,
\nonumber\\
S_{ij}(1,3)_{2}S_{k j}(1,3)&=
-30\omega^{[k]}_{-1 } S_{k i}(1,1)_{-2 } \vac
+10\omega_{0 } \omega^{[k]}_{-1 } S_{k i}(1,1)_{-1 } \vac
\nonumber\\&\quad{}
+5\omega_{0 }^{3}S_{k i}(1,1)_{-1 } \vac
+20\omega^{[k]}_{-1 } S_{k i}(1,2)_{-1 } \vac
\nonumber\\&\quad{}
-30S_{k i}(1,2)_{-3 } \vac
+30\omega_{0 } S_{k i}(1,3)_{-1 } \vac
,
\nonumber\\
S_{ij}(1,3)_{3}S_{k j}(1,3)&=
30S_{k i}(1,3)_{-1 } \vac
,
\nonumber\\
S_{ij}(1,3)_{4}S_{k j}(1,3)&=
30S_{k i}(1,2)_{-1 } \vac
,
\nonumber\\
S_{ij}(1,3)_{5}S_{k j}(1,3)&=
30S_{k i}(1,1)_{-1 } \vac
,
\end{align}

\begin{align}
S_{ij}(1,3)_{0}S_{k i}(1,1)&=
3\omega^{[k]}_{-2 } S_{k j}(1,1)_{-1 } \vac
-2\omega_{0 } \omega^{[k]}_{-1 } S_{k j}(1,1)_{-1 } \vac
\nonumber\\&\quad{}
+\omega_{0 }^{3}S_{k j}(1,1)_{-1 } \vac
+2\omega^{[k]}_{-1 } S_{k j}(1,2)_{-1 } \vac
\nonumber\\&\quad{}
-3\omega_{0 } S_{k j}(1,2)_{-2 } \vac
+3\omega_{0 } S_{k j}(1,3)_{-1 } \vac
,
\nonumber\\
S_{ij}(1,3)_{1}S_{k i}(1,1)&=
S_{k j}(1,3)_{-1 } \vac
,
\end{align}

\begin{align}
S_{ij}(1,3)_{0}S_{k i}(1,2)&=
-16\omega^{[k]}_{-1 } S_{k j}(1,1)_{-3 } \vac
+4\Har^{[k]}_{-1 } S_{k j}(1,1)_{-1 } \vac
\nonumber\\&\quad{}
-98\omega_{0 } \omega^{[k]}_{-1 } S_{k j}(1,1)_{-2 } \vac
+34\omega_{0 }^{2}\omega^{[k]}_{-1 } S_{k j}(1,1)_{-1 } \vac
\nonumber\\&\quad{}
+3\omega_{0 }^{4}S_{k j}(1,1)_{-1 } \vac
+2\omega^{[k]}_{-1 } S_{k j}(1,2)_{-2 } \vac
\nonumber\\&\quad{}
+22\omega_{0 } \omega^{[k]}_{-1 } S_{k j}(1,2)_{-1 } \vac
-4\omega_{0 }^{3}S_{k j}(1,2)_{-1 } \vac
\nonumber\\&\quad{}
+6\omega_{0 } S_{k j}(1,3)_{-2 } \vac
,
\nonumber\\
S_{ij}(1,3)_{1}S_{k i}(1,2)&=
6\omega^{[k]}_{-2 } S_{k j}(1,1)_{-1 } \vac
-4\omega_{0 } \omega^{[k]}_{-1 } S_{k j}(1,1)_{-1 } \vac
\nonumber\\&\quad{}
+\omega_{0 }^{3}S_{k j}(1,1)_{-1 } \vac
+4\omega^{[k]}_{-1 } S_{k j}(1,2)_{-1 } \vac
\nonumber\\&\quad{}
-3\omega_{0 } S_{k j}(1,2)_{-2 } \vac
+6\omega_{0 } S_{k j}(1,3)_{-1 } \vac
,
\nonumber\\
S_{ij}(1,3)_{2}S_{k i}(1,2)&=
2S_{k j}(1,3)_{-1 } \vac
,
\end{align}

\begin{align}
S_{ij}(1,3)_{0}S_{k i}(1,3)&=
-120\omega^{[k]}_{-1 } \omega^{[j]}_{-1 } S_{k j}(1,1)_{-2 } \vac
-300\omega^{[j]}_{-4 } S_{k j}(1,1)_{-1 } \vac
\nonumber\\&\quad{}
-180\Har^{[j]}_{-1 } S_{k j}(1,1)_{-2 } \vac
-60\Har^{[j]}_{-2 } S_{k j}(1,1)_{-1 } \vac
\nonumber\\&\quad{}
+40\omega_{0 } \omega^{[k]}_{-1 } \omega^{[j]}_{-1 } S_{k j}(1,1)_{-1 } \vac
+240\omega_{0 } \omega^{[j]}_{-3 } S_{k j}(1,1)_{-1 } \vac
\nonumber\\&\quad{}
-90\omega_{0 }^{2}\omega^{[j]}_{-2 } S_{k j}(1,1)_{-1 } \vac
+20\omega_{0 }^{3}\omega^{[j]}_{-1 } S_{k j}(1,1)_{-1 } \vac
\nonumber\\&\quad{}
+330\omega^{[k]}_{-2 } S_{k j}(1,2)_{-2 } \vac
+780\omega^{[k]}_{-1 } S_{k j}(1,2)_{-3 } \vac
\nonumber\\&\quad{}
-90\Har^{[k]}_{-1 } S_{k j}(1,2)_{-1 } \vac
+180\Har^{[j]}_{-1 } S_{k j}(1,2)_{-1 } \vac
\nonumber\\&\quad{}
-120\omega_{0 }^{2}\omega^{[k]}_{-1 } S_{k j}(1,2)_{-1 } \vac
-135\omega_{0 }^{4}S_{k j}(1,2)_{-1 } \vac
\nonumber\\&\quad{}
-470\omega^{[k]}_{-2 } S_{k j}(1,3)_{-1 } \vac
-500\omega^{[k]}_{-1 } S_{k j}(1,3)_{-2 } \vac
\nonumber\\&\quad{}
+750S_{k j}(1,3)_{-4 } \vac
,
\nonumber\\
S_{ij}(1,3)_{1}S_{k i}(1,3)&=
-8\omega^{[k]}_{-1 } S_{k j}(1,1)_{-3 } \vac
+2\Har^{[k]}_{-1 } S_{k j}(1,1)_{-1 } \vac
\nonumber\\&\quad{}
-49\omega_{0 } \omega^{[k]}_{-1 } S_{k j}(1,1)_{-2 } \vac
+17\omega_{0 }^{2}\omega^{[k]}_{-1 } S_{k j}(1,1)_{-1 } \vac
\nonumber\\&\quad{}
+9\omega_{0 }^{4}S_{k j}(1,1)_{-1 } \vac
+\omega^{[k]}_{-1 } S_{k j}(1,2)_{-2 } \vac
\nonumber\\&\quad{}
+11\omega_{0 } \omega^{[k]}_{-1 } S_{k j}(1,2)_{-1 } \vac
-6\omega_{0 }^{3}S_{k j}(1,2)_{-1 } \vac
\nonumber\\&\quad{}
+9\omega_{0 } S_{k j}(1,3)_{-2 } \vac
,
\nonumber\\
S_{ij}(1,3)_{2}S_{k i}(1,3)&=
9\omega^{[k]}_{-2 } S_{k j}(1,1)_{-1 } \vac
-6\omega_{0 } \omega^{[k]}_{-1 } S_{k j}(1,1)_{-1 } \vac
\nonumber\\&\quad{}
+3\omega_{0 }^{3}S_{k j}(1,1)_{-1 } \vac
+6\omega^{[k]}_{-1 } S_{k j}(1,2)_{-1 } \vac
\nonumber\\&\quad{}
-9\omega_{0 } S_{k j}(1,2)_{-2 } \vac
+9\omega_{0 } S_{k j}(1,3)_{-1 } \vac
,
\nonumber\\
S_{ij}(1,3)_{3}S_{k i}(1,3)&=
3S_{k j}(1,3)_{-1 } \vac
,
\end{align}

\subsection{The case that $\langle \alpha,\alpha\rangle\neq 0,1/2,1$, and $2$}
\label{section:Norm-nonzero}
Let $\mn$ be a complex number with $\mn\neq 0,1/2,1,2$ and $\alpha\in\fh$ such that $\langle\alpha,\alpha\rangle=\mn$.
Let $h^{[1]},\ldots,h^{[\rankL]}$ be an orthonormal basis of $\fh$
such that $\langle\alpha,h^{[1]}\rangle\neq 0$ and
$\langle\alpha,h^{[i]}\rangle=0$ for all $i=2,\ldots,\rankL$.
In the following computation, $i,j,k$ are distinct elements of $\{2,\ldots,\rankL\}$.

\begin{align}
	\label{eq:Har[1]0ExB=frac2mnmn-2omega-2E}
\Har^{[1]}_{0}\ExB&=
\frac{2  \mn }{\mn -2}\omega^{[1]}_{-2 } E
+\frac{-4  \mn }{(\mn -2)  (2  \mn -1)}\omega_{0 } \omega^{[1]}_{-1 } E
+\frac{2}{(\mn -2)  (2  \mn -1)}\omega_{0 } \omega_{0 } \omega_{0 } E,\nonumber\\
\Har^{[1]}_{1}\ExB&=
\frac{2  \mn }{2  \mn -1}\omega^{[1]}_{-1 } E
+\frac{-1}{2  \mn -1}\omega_{0 } \omega_{0 } E,
\nonumber\\
\Har^{[1]}_{2}\ExB&=
\frac{1}{3}\omega_{0 } E,
\end{align}

\begin{align}
\omega^{[1]}_{0}S_{i1}(1,1)_{0}\ExB&=
\frac{\mn }{\mn -1}\omega_{0 } (S_{i1}(1,1)_{0}E)
+\frac{-\mn }{\mn -1}S_{i1}(1,1)_{-1 } E,\nonumber\\
\omega^{[1]}_{1}S_{i1}(1,1)_{0}\ExB&=
\frac{\mn }{2}(S_{i1}(1,1)_{0}E),
\end{align}

\begin{align}
\Har^{[1]}_{0}S_{i1}(1,1)_{0}\ExB&=
\frac{-2\mn^2}{3(\mn-1)}S_{i1}(1,1)_{-3 } \ExB 
+\frac{\mn^2}{3(\mn-1)^2}\omega_{0 } S_{i1}(1,1)_{-2 } \ExB 
\nonumber\\&\quad{}
+\frac{-\mn^2}{3(\mn-1)^3}\omega_{0 } \omega_{0 } S_{i1}(1,1)_{-1 } \ExB 
+\frac{\mn}{3(\mn-1)^3}\omega_{0 } \omega_{0 } \omega_{0 } (S_{i1}(1,1)_{0}\ExB) 
\nonumber\\&\quad{}
+\frac{2\mn^2}{3(\mn-1)}S_{i1}(1,2)_{-2 } \ExB 
+\frac{-\mn^2}{3(\mn-1)^2}\omega_{0 } S_{i1}(1,2)_{-1 } \ExB 
\nonumber\\&\quad{}
+\frac{2\mn(2\mn-3)}{3(\mn-1)}S_{i1}(1,3)_{-1 } \ExB,\end{align}

\begin{align}
\Har^{[1]}_{1}S_{i1}(1,1)_{0}\ExB&=
\frac{\mn }{2  (\mn -1)^2}\omega_{0 } \omega_{0 } (S_{i1}(1,1)_{0}E)
+\frac{\mn ^2}{2  (\mn -1)}S_{i1}(1,1)_{-2 } E\nonumber\\&\quad{}
+\frac{-\mn ^2}{2  (\mn -1)^2}\omega_{0 } S_{i1}(1,1)_{-1 } E
+\frac{\mn   (\mn -2)}{2  (\mn -1)}S_{i1}(1,2)_{-1 } E,
\nonumber\\
\Har^{[1]}_{2}S_{i1}(1,1)_{0}\ExB&=
\frac{\mn }{3  (\mn -1)}\omega_{0 } (S_{i1}(1,1)_{0}E)
+\frac{-\mn }{3  (\mn -1)}S_{i1}(1,1)_{-1 } E,
\end{align}

\begin{align}
\omega^{[i]}_{0}S_{i1}(1,1)_{0}\ExB&=
\frac{-1}{\mn -1}\omega_{0 } (S_{i1}(1,1)_{0}E)
+\frac{\mn }{\mn -1}S_{i1}(1,1)_{-1 } E,\nonumber\\
\omega^{[i]}_{1}S_{i1}(1,1)_{0}\ExB&=
S_{i1}(1,1)_{0}E,
\end{align}

\begin{align}
\Har^{[i]}_{0}S_{i1}(1,1)_{0}\ExB&=
\frac{-2  (\mn -3)}{\mn -1}\omega^{[i]}_{-2 } (S_{i1}(1,1)_{0}E)
+\frac{-4}{\mn -1}\omega_{0 } \omega^{[i]}_{-1 } (S_{i1}(1,1)_{0}E)\nonumber\\&\quad{}
+\frac{4  \mn }{\mn -1}\omega^{[i]}_{-1 } S_{i1}(1,1)_{-1 } E,\nonumber\\
\Har^{[i]}_{1}S_{i1}(1,1)_{0}\ExB&=
\frac{2}{(\mn -1)^2}\omega_{0 } \omega_{0 } (S_{i1}(1,1)_{0}E)
+\frac{2  \mn }{\mn -1}S_{i1}(1,1)_{-2 } E\nonumber\\&\quad{}
+\frac{-2  \mn }{(\mn -1)^2}\omega_{0 } S_{i1}(1,1)_{-1 } E
+\frac{-2  \mn }{\mn -1}S_{i1}(1,2)_{-1 } E,
\nonumber\\
\Har^{[i]}_{2}S_{i1}(1,1)_{0}\ExB&=
\frac{-7}{3  (\mn -1)}\omega_{0 } (S_{i1}(1,1)_{0}E)
+\frac{7  \mn }{3  (\mn -1)}S_{i1}(1,1)_{-1 } E,
\nonumber\\
\Har^{[i]}_{3}S_{i1}(1,1)_{0}\ExB&=
S_{i1}(1,1)_{0}E,
\end{align}

\begin{align}
S_{i1}(1,1)_{0}S_{i1}(1,1)_{0}\ExB&=
2  \mn \omega^{[i]}_{-1 } E
+\frac{2  \mn }{2  \mn -1}\omega^{[1]}_{-1 } E
+\frac{-1}{2  \mn -1}\omega_{0 } \omega_{0 } E,\nonumber\\
S_{i1}(1,1)_{1}S_{i1}(1,1)_{0}\ExB&=
\omega_{0 } E,
\nonumber\\
S_{i1}(1,1)_{2}S_{i1}(1,1)_{0}\ExB&=
\mn E,
\end{align}

\begin{align}
S_{i1}(1,2)_{0}\ExB&=
\frac{1}{\mn -1}\omega_{0 } (S_{i1}(1,1)_{0}E)
+\frac{-\mn }{\mn -1}S_{i1}(1,1)_{-1 } E,\nonumber\\
S_{i1}(1,2)_{1}\ExB&=
-S_{i1}(1,1)_{0}E,
\end{align}

\begin{align}
S_{i1}(1,2)_{0}S_{i1}(1,1)_{0}\ExB&=
-\mn \omega^{[i]}_{-2 } E
+\frac{2  \mn }{\mn -2}\omega^{[1]}_{-2 } E\nonumber\\&\quad{}
+\frac{-4  \mn }{(\mn -2)  (2  \mn -1)}\omega_{0 } \omega^{[1]}_{-1 } E
+\frac{2}{(\mn -2)  (2  \mn -1)}\omega_{0 } \omega_{0 } \omega_{0 } E,\nonumber\\
S_{i1}(1,2)_{1}S_{i1}(1,1)_{0}\ExB&=
-2  \mn \omega^{[i]}_{-1 } E
+\frac{2  \mn }{2  \mn -1}\omega^{[1]}_{-1 } E
+\frac{-1}{2  \mn -1}\omega_{0 } \omega_{0 } E,
\nonumber\\
S_{i1}(1,2)_{2}S_{i1}(1,1)_{0}\ExB&=0,\nonumber\\
S_{i1}(1,2)_{3}S_{i1}(1,1)_{0}\ExB&=
-\mn E,
\end{align}

\begin{align}
S_{i1}(1,3)_{0}\ExB&=
\frac{1}{2  (\mn -1)^2}\omega_{0 } \omega_{0 } (S_{i1}(1,1)_{0}E)
+\frac{\mn }{2  (\mn -1)}S_{i1}(1,1)_{-2 } E\nonumber\\&\quad{}
+\frac{-\mn }{2  (\mn -1)^2}\omega_{0 } S_{i1}(1,1)_{-1 } E
+\frac{-\mn }{2  (\mn -1)}S_{i1}(1,2)_{-1 } E,\nonumber\\
S_{i1}(1,3)_{1}\ExB&=
\frac{-1}{\mn -1}\omega_{0 } (S_{i1}(1,1)_{0}E)
+\frac{\mn }{\mn -1}S_{i1}(1,1)_{-1 } E,
\nonumber\\
S_{i1}(1,3)_{2}\ExB&=
S_{i1}(1,1)_{0}E.
\end{align}
\begin{align}
&S_{i1}(1,3)_{0}S_{i1}(1,1)_{0}\ExB\nonumber\\
&=
\frac{-\mn}{(\mn-2)(2\mn-1)}\omega^{[1]}_{-1 } \omega^{[1]}_{-1 } \ExB 
+\mn\Har^{[i]}_{-1 } \ExB 
+\frac{2\mn}{3}\omega^{[i]}_{-3 } \ExB 
\nonumber\\&\quad{}
+\frac{16\mn^2-14\mn+9}{6(\mn-2)(2\mn-1)}\omega^{[1]}_{-3 } \ExB 
+\frac{\mn}{2(\mn-2)}\Har^{[1]}_{-1 } \ExB 
\nonumber\\&\quad{}
+\frac{-(\mn+1)(6\mn-5)}{2(\mn-2)^2(2\mn-1)}\omega_{0 } \omega^{[1]}_{-2 } \ExB 
\nonumber\\&\quad{}
+\frac{4\mn-1}{(\mn-2)^2(2\mn-1)}\omega_{0 } \omega_{0 } \omega^{[1]}_{-1 } \ExB 
\nonumber\\&\quad{}
+\frac{-7}{4(\mn-2)^2(2\mn-1)}\omega_{0 } \omega_{0 } \omega_{0 } \omega_{0 } \ExB,
\label{eq:Si1(1,3)0Si1-1}
\end{align}
\begin{align}
S_{i1}(1,3)_{1}S_{i1}(1,1)_{0}\ExB&=
\mn \omega^{[i]}_{-2 } E
+\frac{\mn }{\mn -2}\omega^{[1]}_{-2 } E
+\frac{-2  \mn }{(\mn -2)  (2  \mn -1)}\omega_{0 } \omega^{[1]}_{-1 } E\nonumber\\&\quad{}
+\frac{1}{(\mn -2)  (2  \mn -1)}\omega_{0 } \omega_{0 } \omega_{0 } E,
\nonumber\\
S_{i1}(1,3)_{2}S_{i1}(1,1)_{0}\ExB&=
2  \mn \omega^{[i]}_{-1 } E,
\nonumber\\
S_{i1}(1,3)_{3}S_{i1}(1,1)_{0}\ExB&=0,\nonumber\\
S_{i1}(1,3)_{4}S_{i1}(1,1)_{0}\ExB&=
\mn E,
\end{align}
\begin{align}
S_{k 1}(1,1)_{0}S_{i1}(1,1)_{0}\ExB&=
\mn S_{k i}(1,1)_{-1 } \ExB 
,\end{align}

\begin{align}
S_{k 1}(1,2)_{0}\ExB&=
\frac{-\mn }{\mn -1}S_{k 1}(1,1)_{-1 } \ExB 
+\frac{1}{\mn -1}\omega_{0 } (S_{k 1}(1,1)_{0}\ExB) 
,\nonumber\\
S_{k 1}(1,2)_{1}\ExB&=
-(S_{k 1}(1,1)_{0}\ExB) 
,
\end{align}

\begin{align}
S_{k 1}(1,2)_{0}S_{i1}(1,1)_{0}\ExB&=
-\mn S_{k i}(1,1)_{-2 } \ExB 
+\mn S_{k i}(1,2)_{-1 } \ExB 
,\nonumber\\
S_{k 1}(1,2)_{1}S_{i1}(1,1)_{0}\ExB&=
-\mn S_{k i}(1,1)_{-1 } \ExB 
,
\end{align}

\begin{align}
S_{k 1}(1,3)_{0}\ExB&=
\frac{\mn }{2 (\mn -1)}S_{k 1}(1,1)_{-2 } \ExB 
+\frac{-\mn }{2 (\mn -1)^2}\omega_{0 } S_{k 1}(1,1)_{-1 } \ExB 
\nonumber\\&\quad{}
+\frac{1}{2 (\mn -1)^2}\omega_{0 } \omega_{0 } (S_{k 1}(1,1)_{0}\ExB) 
+\frac{-\mn }{2 (\mn -1)}S_{k 1}(1,2)_{-1 } \ExB 
,\nonumber\\
S_{k 1}(1,3)_{1}\ExB&=
\frac{\mn }{\mn -1}S_{k 1}(1,1)_{-1 } \ExB 
+\frac{-1}{\mn -1}\omega_{0 } (S_{k 1}(1,1)_{0}\ExB) 
,
\nonumber\\
S_{k 1}(1,3)_{2}\ExB&=
(S_{k 1}(1,1)_{0}\ExB) 
,
\end{align}

\begin{align}
S_{k 1}(1,3)_{0}S_{i1}(1,1)_{0}\ExB&=
\mn S_{k i}(1,1)_{-3 } \ExB 
-\mn S_{k i}(1,2)_{-2 } \ExB 
+\mn S_{k i}(1,3)_{-1 } \ExB 
,\nonumber\\
S_{k 1}(1,3)_{1}S_{i1}(1,1)_{0}\ExB&=
\mn S_{k i}(1,1)_{-2 } \ExB 
-\mn S_{k i}(1,2)_{-1 } \ExB 
,
\nonumber\\
S_{k 1}(1,3)_{2}S_{i1}(1,1)_{0}\ExB&=
\mn S_{k i}(1,1)_{-1 } \ExB 
,
\end{align}

\begin{align}
S_{k i}(1,1)_{0}S_{i1}(1,1)_{0}\ExB&=
\frac{\mn }{\mn -1}S_{k 1}(1,1)_{-1 } \ExB 
+\frac{-1}{\mn -1}\omega_{0 } (S_{k 1}(1,1)_{0}\ExB) 
,\nonumber\\
S_{k i}(1,1)_{1}S_{i1}(1,1)_{0}\ExB&=
(S_{k 1}(1,1)_{0}\ExB) 
,
\end{align}

\begin{align}
S_{k i}(1,2)_{0}S_{i1}(1,1)_{0}\ExB&=
\frac{-\mn }{\mn -1}S_{k 1}(1,1)_{-2 } \ExB 
+\frac{\mn }{(\mn -1)^2}\omega_{0 } S_{k 1}(1,1)_{-1 } \ExB 
\nonumber\\&\quad{}
+\frac{-1}{(\mn -1)^2}\omega_{0 } \omega_{0 } (S_{k 1}(1,1)_{0}\ExB) 
+\frac{\mn }{\mn -1}S_{k 1}(1,2)_{-1 } \ExB 
,\nonumber\\
S_{k i}(1,2)_{1}S_{i1}(1,1)_{0}\ExB&=
\frac{-2 \mn }{\mn -1}S_{k 1}(1,1)_{-1 } \ExB 
+\frac{2}{\mn -1}\omega_{0 } (S_{k 1}(1,1)_{0}\ExB) 
,
\nonumber\\
S_{k i}(1,2)_{2}S_{i1}(1,1)_{0}\ExB&=
-2(S_{k 1}(1,1)_{0}\ExB) 
,
\end{align}

\begin{align}
S_{k i}(1,3)_{0}S_{i1}(1,1)_{0}\ExB&=
\frac{-(\mn -3)}{\mn -1}\omega^{[k]}_{-2 } (S_{k 1}(1,1)_{0}\ExB) 
+\frac{2 \mn }{\mn -1}\omega^{[k]}_{-1 } S_{k 1}(1,1)_{-1 } \ExB 
\nonumber\\&\quad{}
+\frac{-2}{\mn -1}\omega_{0 } \omega^{[k]}_{-1 } (S_{k 1}(1,1)_{0}\ExB) 
,\nonumber\\
S_{k i}(1,3)_{1}S_{i1}(1,1)_{0}\ExB&=
\frac{3 \mn }{2 (\mn -1)}S_{k 1}(1,1)_{-2 } \ExB 
+\frac{-3 \mn }{2 (\mn -1)^2}\omega_{0 } S_{k 1}(1,1)_{-1 } \ExB 
\nonumber\\&\quad{}
+\frac{3}{2 (\mn -1)^2}\omega_{0 } \omega_{0 } (S_{k 1}(1,1)_{0}\ExB) 
+\frac{-3 \mn }{2 (\mn -1)}S_{k 1}(1,2)_{-1 } \ExB 
,
\nonumber\\
S_{k i}(1,3)_{2}S_{i1}(1,1)_{0}\ExB&=
\frac{3 \mn }{\mn -1}S_{k 1}(1,1)_{-1 } \ExB 
+\frac{-3}{\mn -1}\omega_{0 } (S_{k 1}(1,1)_{0}\ExB) 
,
\nonumber\\
S_{k i}(1,3)_{3}S_{i1}(1,1)_{0}\ExB&=
3(S_{k 1}(1,1)_{0}\ExB).
\end{align}

\subsection{The case that $\langle \alpha,\alpha\rangle=2$}
\label{section:normal-2}
Let $\alpha\in\fh$ with $\langle\alpha,\alpha\rangle=2$.
Let $h^{[1]},\ldots,h^{[\rankL]}$ be an orthonormal basis of $\fh$ such that $\alpha\in\C h^{[1]}$.
In the following computation, $i,j,k$ are distinct elements of $\{2,\ldots,\rankL\}$.

\begin{align}
\label{eq:omega[1]0Har[1]0ExB=omega0}
\omega^{[1]}_{0}\Har^{[1]}_{0}\ExB&=
\omega_{0 } (\Har^{[1]}_{0}E),&
\omega^{[1]}_{1}\Har^{[1]}_{0}\ExB&=
4(\Har^{[1]}_{0}E) ,
\nonumber\\
\omega^{[1]}_{2}\Har^{[1]}_{0}\ExB&=
8\omega^{[1]}_{-1 } E
-2\omega_{0 } \omega_{0 } E,&
\omega^{[1]}_{3}\Har^{[1]}_{0}\ExB&=
6\omega_{0 } E,
\nonumber\\
\omega^{[1]}_{4}\Har^{[1]}_{0}\ExB&=
12E,
\end{align}
\begin{align}
\Har^{[1]}_{0}\ExB&=
\Har^{[1]}_{0}E,\nonumber\\
\Har^{[1]}_{1}\ExB&=
\frac{4}{3}\omega^{[1]}_{-1 } E
+\frac{-1}{3}\omega_{0 } \omega_{0 } E,
\nonumber\\
\Har^{[1]}_{2}\ExB&=
\frac{1}{3}\omega_{0 } E,
\end{align}

\begin{align}
\Har^{[1]}_{0}\Har^{[1]}_{0}\ExB&=
\frac{4096}{5145}\omega^{[1]}_{-1 } \omega^{[1]}_{-1 } \omega^{[1]}_{-1 } E
+\frac{210412}{46305}\omega^{[1]}_{-2 } (\Har^{[1]}_{0}E) \nonumber\\&\quad{}
+\frac{18944}{5145}\omega^{[1]}_{-1 } \Har^{[1]}_{-1 } E
+\frac{-124504}{15435}\Har^{[1]}_{-3 } E\nonumber\\&\quad{}
+\frac{5888}{36015}\omega_{0 } \omega^{[1]}_{-1 } (\Har^{[1]}_{0}E) 
+\frac{3032168}{324135}\omega_{0 } \Har^{[1]}_{-2 } E\nonumber\\&\quad{}
+\frac{-42680128}{72930375}\omega_{0 } \omega_{0 } \omega^{[1]}_{-1 } \omega^{[1]}_{-1 } E
+\frac{-42960376}{24310125}\omega_{0 } \omega_{0 } \Har^{[1]}_{-1 } E\nonumber\\&\quad{}
+\frac{-4044646}{24310125}\omega_{0 } \omega_{0 } \omega_{0 } (\Har^{[1]}_{0}E) 
+\frac{32226464}{218791125}\omega_{0 } \omega_{0 } \omega_{0 } \omega_{0 } \omega^{[1]}_{-1 } E\nonumber\\&\quad{}
+\frac{-2775692}{218791125}\omega_{0 } \omega_{0 } \omega_{0 } \omega_{0 } \omega_{0 } \omega_{0 } E,\end{align}

\begin{align}
\Har^{[1]}_{1}\Har^{[1]}_{0}\ExB&=
\frac{36}{7}\omega^{[1]}_{-1 } (\Har^{[1]}_{0}E) 
+\frac{80}{7}\Har^{[1]}_{-2 } E\nonumber\\&\quad{}
+\frac{1088}{1575}\omega_{0 } \omega^{[1]}_{-1 } \omega^{[1]}_{-1 } E
+\frac{1496}{525}\omega_{0 } \Har^{[1]}_{-1 } E\nonumber\\&\quad{}
+\frac{-709}{525}\omega_{0 } \omega_{0 } (\Har^{[1]}_{0}E) 
+\frac{-544}{4725}\omega_{0 } \omega_{0 } \omega_{0 } \omega^{[1]}_{-1 } E\nonumber\\&\quad{}
+\frac{-68}{4725}\omega_{0 } \omega_{0 } \omega_{0 } \omega_{0 } \omega_{0 } E,
\nonumber\\
\Har^{[1]}_{2}\Har^{[1]}_{0}\ExB&=
\frac{128}{25}\omega^{[1]}_{-1 } \omega^{[1]}_{-1 } E
+\frac{528}{25}\Har^{[1]}_{-1 } E\nonumber\\&\quad{}
+\frac{-11}{75}\omega_{0 } (\Har^{[1]}_{0}E) 
+\frac{-64}{75}\omega_{0 } \omega_{0 } \omega^{[1]}_{-1 } E\nonumber\\&\quad{}
+\frac{-8}{75}\omega_{0 } \omega_{0 } \omega_{0 } \omega_{0 } E,
\nonumber\\
\Har^{[1]}_{3}\Har^{[1]}_{0}\ExB&=
27(\Har^{[1]}_{0}E) ,
\nonumber\\
\Har^{[1]}_{4}\Har^{[1]}_{0}\ExB&=
48\omega^{[1]}_{-1 } E-12\omega_{0 } \omega_{0 } E,
\nonumber\\
\Har^{[1]}_{5}\Har^{[1]}_{0}\ExB&=
20\omega_{0 } E,
\end{align}
\begin{align}
\omega^{[1]}_{0}S_{i1}(1,1)_{0}\ExB&=
-2S_{i1}(1,1)_{-1 } E+2\omega_{0 } (S_{i1}(1,1)_{0}E),\nonumber\\
\omega^{[1]}_{1}S_{i1}(1,1)_{0}\ExB&=
S_{i1}(1,1)_{0}E,
\end{align}

\begin{align}
\Har^{[1]}_{0}S_{i1}(1,1)_{0}\ExB&=
\frac{40}{21}S_{i1}(1,1)_{-3 } E
+\frac{-32}{21}\omega^{[1]}_{-1 } S_{i1}(1,1)_{-1 } E\nonumber\\&\quad{}
+\frac{-4}{7}\omega_{0 } S_{i1}(1,1)_{-2 } E
+\frac{-4}{21}\omega_{0 } \omega^{[1]}_{-1 } (S_{i1}(1,1)_{0}E)\nonumber\\&\quad{}
+\frac{5}{21}\omega_{0 } \omega_{0 } \omega_{0 } (S_{i1}(1,1)_{0}E)
+\frac{26}{21}\omega_{0 } S_{i1}(1,2)_{-1 } E\nonumber\\&\quad{}
+\frac{44}{21}S_{i1}(1,3)_{-1 } E,\nonumber\\
\Har^{[1]}_{1}S_{i1}(1,1)_{0}\ExB&=
2S_{i1}(1,1)_{-2 } E
-2\omega_{0 } S_{i1}(1,1)_{-1 } E
+\omega_{0 } \omega_{0 } (S_{i1}(1,1)_{0}E),
\nonumber\\
\Har^{[1]}_{2}S_{i1}(1,1)_{0}\ExB&=
\frac{-2}{3}S_{i1}(1,1)_{-1 } E
+\frac{2}{3}\omega_{0 } (S_{i1}(1,1)_{0}E),
\end{align}

\begin{align}
\omega^{[i]}_{0}S_{i1}(1,1)_{0}\ExB&=
2S_{i1}(1,1)_{-1 } E-\omega_{0 } (S_{i1}(1,1)_{0}E),\nonumber\\
\omega^{[i]}_{1}S_{i1}(1,1)_{0}\ExB&=
S_{i1}(1,1)_{0}E,
\end{align}

\begin{align}
\Har^{[i]}_{0}S_{i1}(1,1)_{0}\ExB&=
\frac{-20}{7}S_{i1}(1,1)_{-3 } E
+\frac{16}{7}\omega^{[1]}_{-1 } S_{i1}(1,1)_{-1 } E\nonumber\\&\quad{}
+\frac{6}{7}\omega_{0 } S_{i1}(1,1)_{-2 } E
+\frac{2}{7}\omega_{0 } \omega^{[1]}_{-1 } (S_{i1}(1,1)_{0}E)\nonumber\\&\quad{}
+\frac{-5}{14}\omega_{0 } \omega_{0 } \omega_{0 } (S_{i1}(1,1)_{0}E)
+\frac{-13}{7}\omega_{0 } S_{i1}(1,2)_{-1 } E\nonumber\\&\quad{}
+\frac{20}{7}S_{i1}(1,3)_{-1 } E,\nonumber\\
\Har^{[i]}_{1}S_{i1}(1,1)_{0}\ExB&=
4S_{i1}(1,1)_{-2 } E
-4\omega_{0 } S_{i1}(1,1)_{-1 } E\nonumber\\&\quad{}
+2\omega_{0 } \omega_{0 } (S_{i1}(1,1)_{0}E)
-4S_{i1}(1,2)_{-1 } E,
\nonumber\\
\Har^{[i]}_{2}S_{i1}(1,1)_{0}\ExB&=
\frac{14}{3}S_{i1}(1,1)_{-1 } E
+\frac{-7}{3}\omega_{0 } (S_{i1}(1,1)_{0}E),
\nonumber\\
\Har^{[i]}_{3}S_{i1}(1,1)_{0}\ExB&=
S_{i1}(1,1)_{0}E,
\end{align}

\begin{align}
S_{i1}(1,1)_{0}S_{i1}(1,1)_{0}\ExB&=
4\omega^{[i]}_{-1 } E 
+\frac{4}{3}\omega^{[1]}_{-1 } E 
+\frac{-1}{3}\omega_{0 } \omega_{0 } E 
,\nonumber\\
S_{i1}(1,1)_{1}S_{i1}(1,1)_{0}\ExB&=
\omega_{0 } E 
,
\nonumber\\
S_{i1}(1,1)_{2}S_{i1}(1,1)_{0}\ExB&=
2E 
,
\end{align}

\begin{align}
S_{i1}(1,1)_{0}\Har^{[1]}_{0}\ExB&=
\frac{-20}{21}S_{i1}(1,1)_{-3 } E
+\frac{16}{21}\omega^{[1]}_{-1 } S_{i1}(1,1)_{-1 } E\nonumber\\&\quad{}
+\frac{2}{7}\omega_{0 } S_{i1}(1,1)_{-2 } E
+\frac{2}{21}\omega_{0 } \omega^{[1]}_{-1 } (S_{i1}(1,1)_{0}E)\nonumber\\&\quad{}
+\frac{-5}{42}\omega_{0 } \omega_{0 } \omega_{0 } (S_{i1}(1,1)_{0}E)
+\frac{-13}{21}\omega_{0 } S_{i1}(1,2)_{-1 } E\nonumber\\&\quad{}
+\frac{104}{21}S_{i1}(1,3)_{-1 } E,\nonumber\\
S_{i1}(1,1)_{1}\Har^{[1]}_{0}\ExB&=
6S_{i1}(1,1)_{-2 } E
-6\omega_{0 } S_{i1}(1,1)_{-1 } E\nonumber\\&\quad{}
+3\omega_{0 } \omega_{0 } (S_{i1}(1,1)_{0}E)
-6S_{i1}(1,2)_{-1 } E,
\nonumber\\
S_{i1}(1,1)_{2}\Har^{[1]}_{0}\ExB&=
12S_{i1}(1,1)_{-1 } E
-6\omega_{0 } (S_{i1}(1,1)_{0}E),
\nonumber\\
S_{i1}(1,1)_{3}\Har^{[1]}_{0}\ExB&=
6(S_{i1}(1,1)_{0}E),
\end{align}

\begin{align}
S_{i1}(1,2)_{0}\ExB&=
-2S_{i1}(1,1)_{-1 } E+\omega_{0 } (S_{i1}(1,1)_{0}E),\nonumber\\
S_{i1}(1,2)_{1}\ExB&=-S_{i1}(1,1)_{0}E,
\end{align}

\begin{align}
S_{i1}(1,2)_{0}S_{i1}(1,1)_{0}\ExB&=
-2\omega^{[i]}_{-2 } E 
+(\Har_{0}E) 
,\nonumber\\
S_{i1}(1,2)_{1}S_{i1}(1,1)_{0}\ExB&=
-4\omega^{[i]}_{-1 } E 
+\frac{4}{3}\omega^{[1]}_{-1 } E 
+\frac{-1}{3}\omega_{0 } \omega_{0 } E 
,
\nonumber\\
S_{i1}(1,2)_{2}S_{i1}(1,1)_{0}\ExB&=0,\nonumber\\
S_{i1}(1,2)_{3}S_{i1}(1,1)_{0}\ExB&=
-2E 
,
\end{align}

\begin{align}
S_{i1}(1,2)_{0}\Har^{[1]}_{0}\ExB&=
12S_{i1}(1,1)_{-4 } E
+\frac{-80}{9}\omega^{[i]}_{-2 } S_{i1}(1,1)_{-1 } E\nonumber\\&\quad{}
+\frac{-484}{27}\omega^{[i]}_{-3 } (S_{i1}(1,1)_{0}E)
+\frac{-80}{9}\Har^{[i]}_{-1 } (S_{i1}(1,1)_{0}E)\nonumber\\&\quad{}
+8\omega^{[i]}_{-1 } \omega^{[1]}_{-1 } (S_{i1}(1,1)_{0}E)
+\frac{94}{9}\omega_{0 } \omega^{[i]}_{-2 } (S_{i1}(1,1)_{0}E)\nonumber\\&\quad{}
-2\omega_{0 } \omega_{0 } \omega^{[i]}_{-1 } (S_{i1}(1,1)_{0}E)
-12S_{i1}(1,2)_{-3 } E\nonumber\\&\quad{}
-12\omega^{[i]}_{-1 } S_{i1}(1,2)_{-1 } E
+8S_{i1}(1,3)_{-2 } E,\end{align}

\begin{align}
S_{i1}(1,2)_{1}\Har^{[1]}_{0}\ExB&=
\frac{200}{21}S_{i1}(1,1)_{-3 } E
+\frac{-160}{21}\omega^{[1]}_{-1 } S_{i1}(1,1)_{-1 } E\nonumber\\&\quad{}
+\frac{-20}{7}\omega_{0 } S_{i1}(1,1)_{-2 } E
+\frac{-20}{21}\omega_{0 } \omega^{[1]}_{-1 } (S_{i1}(1,1)_{0}E)\nonumber\\&\quad{}
+\frac{25}{21}\omega_{0 } \omega_{0 } \omega_{0 } (S_{i1}(1,1)_{0}E)
+\frac{130}{21}\omega_{0 } S_{i1}(1,2)_{-1 } E\nonumber\\&\quad{}
+\frac{-284}{21}S_{i1}(1,3)_{-1 } E,
\nonumber\\
S_{i1}(1,2)_{2}\Har^{[1]}_{0}\ExB&=
-24S_{i1}(1,1)_{-2 } E
+24\omega_{0 } S_{i1}(1,1)_{-1 } E\nonumber\\&\quad{}
-12\omega_{0 } \omega_{0 } (S_{i1}(1,1)_{0}E)
+24S_{i1}(1,2)_{-1 } E,
\nonumber\\
S_{i1}(1,2)_{3}\Har^{[1]}_{0}\ExB&=
-48S_{i1}(1,1)_{-1 } E\nonumber\\&\quad{}
+24\omega_{0 } (S_{i1}(1,1)_{0}E),
\nonumber\\
S_{i1}(1,2)_{4}\Har^{[1]}_{0}\ExB&=
-24(S_{i1}(1,1)_{0}E),
\end{align}

\begin{align}
S_{i1}(1,3)_{0}\ExB&=
S_{i1}(1,1)_{-2 } E
-\omega_{0 } S_{i1}(1,1)_{-1 } E\nonumber\\&\quad{}
+\frac{1}{2}\omega_{0 } \omega_{0 } (S_{i1}(1,1)_{0}E)
-S_{i1}(1,2)_{-1 } E,\nonumber\\
S_{i1}(1,3)_{1}\ExB&=
2S_{i1}(1,1)_{-1 } E
-\omega_{0 } (S_{i1}(1,1)_{0}E),
\nonumber\\
S_{i1}(1,3)_{2}\ExB&=
S_{i1}(1,1)_{0}E,
\end{align}

\begin{align}
S_{i1}(1,3)_{0}S_{i1}(1,1)_{0}\ExB&=
\frac{4}{3}\omega^{[i]}_{-3 } E 
+2\Har^{[i]}_{-1 } E 
+\frac{16}{75}\omega^{[1]}_{-1 } \omega^{[1]}_{-1 } E 
+\frac{22}{25}\Har^{[1]}_{-1 } E 
\nonumber\\&\quad{}
+\frac{-1}{50}\omega_{0 } (\Har_{0}E) 
+\frac{-8}{225}\omega_{0 } \omega_{0 } \omega^{[1]}_{-1 } E 
\nonumber\\&\quad{}
+\frac{-1}{225}\omega_{0 } \omega_{0 } \omega_{0 } \omega_{0 } E 
,\nonumber\\
S_{i1}(1,3)_{1}S_{i1}(1,1)_{0}\ExB&=
2\omega^{[i]}_{-2 } E 
+\frac{1}{2}(\Har_{0}E) 
,
\nonumber\\
S_{i1}(1,3)_{2}S_{i1}(1,1)_{0}\ExB&=
4\omega^{[i]}_{-1 } E 
,
\nonumber\\
S_{i1}(1,3)_{3}S_{i1}(1,1)_{0}\ExB&=0,\nonumber\\
S_{i1}(1,3)_{4}S_{i1}(1,1)_{0}\ExB&=
2E 
,
\end{align}

\begin{align}
S_{i1}(1,3)_{0}\Har^{[1]}_{0}\ExB&=
\frac{1732}{213}S_{i1}(1,1)_{-5 } E
+\frac{-1856}{781}\omega^{[i]}_{-4 } (S_{i1}(1,1)_{0}E)\nonumber\\&\quad{}
+\frac{49860}{10153}\Har^{[i]}_{-1 } S_{i1}(1,1)_{-1 } E
+\frac{1856}{781}\Har^{[i]}_{-2 } (S_{i1}(1,1)_{0}E)\nonumber\\&\quad{}
+\frac{-8}{213}\omega^{[1]}_{-1 } S_{i1}(1,1)_{-3 } E
+\frac{16620}{10153}\omega^{[i]}_{-2 } \omega^{[1]}_{-1 } (S_{i1}(1,1)_{0}E)\nonumber\\&\quad{}
+\frac{836}{213}\omega^{[i]}_{-1 } \omega^{[1]}_{-2 } (S_{i1}(1,1)_{0}E)
+\frac{-32}{213}\omega^{[1]}_{-3 } S_{i1}(1,1)_{-1 } E\nonumber\\&\quad{}
+\frac{-28}{71}\Har^{[1]}_{-1 } S_{i1}(1,1)_{-1 } E\nonumber\\&\quad{}
+\frac{46}{71}\omega_{0 } S_{i1}(1,1)_{-4 } E\nonumber\\&\quad{}
+\frac{123554}{91377}\omega_{0 } \omega^{[i]}_{-2 } S_{i1}(1,1)_{-1 } E\nonumber\\&\quad{}
+\frac{235819}{274131}\omega_{0 } \omega^{[i]}_{-3 } (S_{i1}(1,1)_{0}E)\nonumber\\&\quad{}
+\frac{-429718}{91377}\omega_{0 } \Har^{[i]}_{-1 } (S_{i1}(1,1)_{0}E)\nonumber\\&\quad{}
+\frac{130}{213}\omega_{0 } \omega^{[i]}_{-1 } \omega^{[1]}_{-1 } (S_{i1}(1,1)_{0}E)\nonumber\\&\quad{}
+\frac{1}{71}\omega_{0 } \omega^{[1]}_{-2 } S_{i1}(1,1)_{-1 } E\nonumber\\&\quad{}
+\frac{-1}{213}\omega_{0 } \omega_{0 } S_{i1}(1,1)_{-3 } E\nonumber\\&\quad{}
+\frac{-114689}{182754}\omega_{0 } \omega_{0 } \omega^{[i]}_{-2 } (S_{i1}(1,1)_{0}E)\nonumber\\&\quad{}
+\frac{-65}{426}\omega_{0 } \omega_{0 } \omega_{0 } \omega^{[i]}_{-1 } (S_{i1}(1,1)_{0}E)\nonumber\\&\quad{}
+\frac{-460}{213}S_{i1}(1,2)_{-4 } E\nonumber\\&\quad{}
+\frac{34844}{10153}\omega^{[i]}_{-2 } S_{i1}(1,2)_{-1 } E\nonumber\\&\quad{}
+\frac{-49}{71}\omega_{0 } S_{i1}(1,2)_{-3 } E\nonumber\\&\quad{}
+\frac{-1031}{213}\omega_{0 } \omega^{[i]}_{-1 } S_{i1}(1,2)_{-1 } E\nonumber\\&\quad{}
+\frac{-2}{71}\omega_{0 } \omega^{[1]}_{-1 } S_{i1}(1,2)_{-1 } E\nonumber\\&\quad{}
+\frac{49}{71}\omega_{0 } S_{i1}(1,3)_{-2 } E\nonumber\\&\quad{}
+\frac{10}{71}\omega_{0 } \omega_{0 } S_{i1}(1,3)_{-1 } E,\end{align}

\begin{align}
S_{i1}(1,3)_{1}\Har^{[1]}_{0}\ExB&=
-12S_{i1}(1,1)_{-4 } E
+\frac{116}{9}\omega^{[i]}_{-2 } S_{i1}(1,1)_{-1 } E\nonumber\\&\quad{}
+\frac{556}{27}\omega^{[i]}_{-3 } (S_{i1}(1,1)_{0}E)
+\frac{116}{9}\Har^{[i]}_{-1 } (S_{i1}(1,1)_{0}E)\nonumber\\&\quad{}
-8\omega^{[i]}_{-1 } \omega^{[1]}_{-1 } (S_{i1}(1,1)_{0}E)
+\frac{-112}{9}\omega_{0 } \omega^{[i]}_{-2 } (S_{i1}(1,1)_{0}E)\nonumber\\&\quad{}
+2\omega_{0 } \omega_{0 } \omega^{[i]}_{-1 } (S_{i1}(1,1)_{0}E)
+12S_{i1}(1,2)_{-3 } E\nonumber\\&\quad{}
+12\omega^{[i]}_{-1 } S_{i1}(1,2)_{-1 } E
-8S_{i1}(1,3)_{-2 } E,
\nonumber\\
S_{i1}(1,3)_{2}\Har^{[1]}_{0}\ExB&=
\frac{-80}{3}S_{i1}(1,1)_{-3 } E
+\frac{64}{3}\omega^{[1]}_{-1 } S_{i1}(1,1)_{-1 } E\nonumber\\&\quad{}
+8\omega_{0 } S_{i1}(1,1)_{-2 } E
+\frac{8}{3}\omega_{0 } \omega^{[1]}_{-1 } (S_{i1}(1,1)_{0}E)\nonumber\\&\quad{}
+\frac{-10}{3}\omega_{0 } \omega_{0 } \omega_{0 } (S_{i1}(1,1)_{0}E)
+\frac{-52}{3}\omega_{0 } S_{i1}(1,2)_{-1 } E\nonumber\\&\quad{}
+\frac{92}{3}S_{i1}(1,3)_{-1 } E,
\nonumber\\
S_{i1}(1,3)_{3}\Har^{[1]}_{0}\ExB&=
60S_{i1}(1,1)_{-2 } E
-60\omega_{0 } S_{i1}(1,1)_{-1 } E\nonumber\\&\quad{}
+30\omega_{0 } \omega_{0 } (S_{i1}(1,1)_{0}E)
-60S_{i1}(1,2)_{-1 } E,
\nonumber\\
S_{i1}(1,3)_{4}\Har^{[1]}_{0}\ExB&=
120S_{i1}(1,1)_{-1 } E
-60\omega_{0 } (S_{i1}(1,1)_{0}E),
\nonumber\\
S_{i1}(1,3)_{5}\Har^{[1]}_{0}\ExB&=
60(S_{i1}(1,1)_{0}E),
\end{align}
\begin{align}
	S_{k 1}(1,1)_{0}\ExB&=
	(S_{k 1}(1,1)_{0}\ExB) 
	,\end{align}

\begin{align}
	S_{k 1}(1,1)_{0}S_{i1}(1,1)_{0}\ExB&=
	2S_{k i}(1,1)_{-1 } \ExB 
	,\end{align}

\begin{align}
	S_{k 1}(1,1)_{0}\Har^{[1]}_{0}\ExB&=
	\frac{2}{3}\omega^{[k]}_{-2 } (S_{k 1}(1,1)_{0}\ExB) 
	+\frac{8}{3}\omega^{[k]}_{-1 } S_{k 1}(1,1)_{-1 } \ExB 
	\nonumber\\&\quad{}
	+\frac{-4}{3}\omega_{0 } \omega^{[k]}_{-1 } (S_{k 1}(1,1)_{0}\ExB) 
	+4S_{k 1}(1,3)_{-1 } \ExB 
	,\nonumber\\
	S_{k 1}(1,1)_{1}\Har^{[1]}_{0}\ExB&=
	3\omega^{[1]}_{-1 } (S_{k 1}(1,1)_{0}\ExB) 
	+3\omega_{0 } S_{k 1}(1,1)_{-1 } \ExB 
	\nonumber\\&\quad{}
	+\frac{-9}{4}\omega_{0 } \omega_{0 } (S_{k 1}(1,1)_{0}\ExB) 
	+\frac{-9}{2}S_{k 1}(1,2)_{-1 } \ExB 
	,
	\nonumber\\
	S_{k 1}(1,1)_{2}\Har^{[1]}_{0}\ExB&=
	12S_{k 1}(1,1)_{-1 } \ExB 
	-6\omega_{0 } (S_{k 1}(1,1)_{0}\ExB) 
	,
	\nonumber\\
	S_{k 1}(1,1)_{3}\Har^{[1]}_{0}\ExB&=
	6(S_{k 1}(1,1)_{0}\ExB) 
	,
\end{align}

\begin{align}
	S_{k 1}(1,2)_{0}\ExB&=
	-2S_{k 1}(1,1)_{-1 } \ExB 
	+\omega_{0 } (S_{k 1}(1,1)_{0}\ExB) 
	,\nonumber\\
	S_{k 1}(1,2)_{1}\ExB&=
	-(S_{k 1}(1,1)_{0}\ExB) 
	,
\end{align}

\begin{align}
	S_{k 1}(1,2)_{0}S_{i1}(1,1)_{0}\ExB&=
	-2S_{k i}(1,1)_{-2 } \ExB 
	+2S_{k i}(1,2)_{-1 } \ExB 
	,\nonumber\\
	S_{k 1}(1,2)_{1}S_{i1}(1,1)_{0}\ExB&=
	-2S_{k i}(1,1)_{-1 } \ExB 
	,
\end{align}

\begin{align}
	S_{k 1}(1,3)_{0}S_{i1}(1,1)_{0}\ExB&=
	2S_{k i}(1,1)_{-3 } \ExB 
	-2S_{k i}(1,2)_{-2 } \ExB 
	\nonumber\\&\quad{}
	+2S_{k i}(1,3)_{-1 } \ExB 
	,\end{align}

\begin{align}
	S_{k 1}(1,3)_{1}S_{i1}(1,1)_{0}\ExB&=
	2S_{k i}(1,1)_{-2 } \ExB 
	-2S_{k i}(1,2)_{-1 } \ExB 
	,
	\nonumber\\
	S_{k 1}(1,3)_{2}S_{i1}(1,1)_{0}\ExB&=
	2S_{k i}(1,1)_{-1 } \ExB 
	,
\end{align}

\begin{align}
	S_{k i}(1,1)_{0}S_{i1}(1,1)_{0}\ExB&=
	2S_{k 1}(1,1)_{-1 } \ExB 
	-\omega_{0 } (S_{k 1}(1,1)_{0}\ExB) 
	,\end{align}

\begin{align}
	S_{k i}(1,1)_{1}S_{i1}(1,1)_{0}\ExB&=
	(S_{k 1}(1,1)_{0}\ExB) 
	,
\end{align}

\begin{align}
	S_{k i}(1,2)_{0}S_{i1}(1,1)_{0}\ExB&=
	-\omega^{[1]}_{-1 } (S_{k 1}(1,1)_{0}\ExB) 
	-\omega_{0 } S_{k 1}(1,1)_{-1 } \ExB 
	\nonumber\\&\quad{}
	+\frac{3}{4}\omega_{0 } \omega_{0 } (S_{k 1}(1,1)_{0}\ExB) 
	+\frac{3}{2}S_{k 1}(1,2)_{-1 } \ExB 
	,\end{align}

\begin{align}
	S_{k i}(1,2)_{1}S_{i1}(1,1)_{0}\ExB&=
	-4S_{k 1}(1,1)_{-1 } \ExB 
	+2\omega_{0 } (S_{k 1}(1,1)_{0}\ExB) 
	,
	\nonumber\\
	S_{k i}(1,2)_{2}S_{i1}(1,1)_{0}\ExB&=
	-2(S_{k 1}(1,1)_{0}\ExB) 
	,
\end{align}

\begin{align}
	S_{k i}(1,3)_{0}S_{i1}(1,1)_{0}\ExB&=
	\omega^{[k]}_{-2 } (S_{k 1}(1,1)_{0}\ExB) 
	+4\omega^{[k]}_{-1 } S_{k 1}(1,1)_{-1 } \ExB 
	\nonumber\\&\quad{}
	-2\omega_{0 } \omega^{[k]}_{-1 } (S_{k 1}(1,1)_{0}\ExB) 
	,\end{align}

\begin{align}
	S_{k i}(1,3)_{1}S_{i1}(1,1)_{0}\ExB&=
	\frac{3}{2}\omega^{[1]}_{-1 } (S_{k 1}(1,1)_{0}\ExB) 
	+\frac{3}{2}\omega_{0 } S_{k 1}(1,1)_{-1 } \ExB 
	\nonumber\\&\quad{}
	+\frac{-9}{8}\omega_{0 } \omega_{0 } (S_{k 1}(1,1)_{0}\ExB) 
	+\frac{-9}{4}S_{k 1}(1,2)_{-1 } \ExB 
	,
	\nonumber\\
	S_{k i}(1,3)_{2}S_{i1}(1,1)_{0}\ExB&=
	6S_{k 1}(1,1)_{-1 } \ExB 
	-3\omega_{0 } (S_{k 1}(1,1)_{0}\ExB) 
	,
\end{align}

\subsection{The case that $\langle \alpha,\alpha\rangle=1/2$}
\label{section:norm-1-2}
Let $\alpha\in\fh$ with $\langle\alpha,\alpha\rangle=1/2$.
Let $h^{[1]},\ldots,h^{[\rankL]}$ be an orthonormal basis of $\fh$
such that $\alpha\in\C h^{[1]}$.
In the following computation, $i,j,k$ are distinct elements of $\{2,\ldots,\rankL\}$.

\begin{align}
\omega^{[1]}_{0}\ExB&=
\omega_{0 } E,&
\omega^{[1]}_{1}\ExB&=
\frac{1}{4}E,
\end{align}

\begin{align}
\omega^{[1]}_{0}S_{i1}(1,1)_{0}\ExB&=
S_{i1}(1,1)_{-1 } E-\omega_{0 } (S_{i1}(1,1)_{0}E),\nonumber\\
\omega^{[1]}_{1}S_{i1}(1,1)_{0}\ExB&=
\frac{1}{4}(S_{i1}(1,1)_{0}E),
\end{align}

\begin{align}
\label{eq:norm1-2-omega[1]0Har[1]1ExB}
\omega^{[1]}_{0}\Har^{[1]}_{1}\ExB&=
\omega_{0 } (\Har^{[1]}_{1}\ExB),&
\omega^{[1]}_{1}\Har^{[1]}_{1}\ExB&=
\frac{9}{4}(\Har^{[1]}_{1}\ExB),
\nonumber\\
\omega^{[1]}_{2}\Har^{[1]}_{1}\ExB&=
2\omega_{0 } E,
&
\omega^{[1]}_{3}\Har^{[1]}_{1}\ExB&=
E,
\end{align}

\begin{align}
\Har^{[1]}_{0}\ExB&=
\frac{-2}{3}\omega^{[1]}_{-2 } E+\frac{4}{3}\omega_{0 } (\Har^{[1]}_{1}\ExB),\nonumber\\
\Har^{[1]}_{1}\ExB&=
(\Har^{[1]}_{1}\ExB),
\nonumber\\
\Har^{[1]}_{2}\ExB&=
\frac{1}{3}\omega_{0 } E,
\end{align}

\begin{align}
\Har^{[1]}_{0}S_{i1}(1,1)_{0}\ExB&=
-2S_{i1}(1,1)_{-3 } E+\frac{8}{3}\omega^{[i]}_{-1 } S_{i1}(1,1)_{-1 } E\nonumber\\&\quad{}
+\frac{20}{3}\omega^{[i]}_{-2 } (S_{i1}(1,1)_{0}E)+\omega^{[1]}_{-1 } S_{i1}(1,1)_{-1 } E\nonumber\\&\quad{}
+5\omega_{0 } S_{i1}(1,1)_{-2 } E+\frac{-16}{3}\omega_{0 } \omega^{[i]}_{-1 } (S_{i1}(1,1)_{0}E)\nonumber\\&\quad{}
+4\omega_{0 } \omega^{[1]}_{-1 } (S_{i1}(1,1)_{0}E)-3\omega_{0 } \omega_{0 } S_{i1}(1,1)_{-1 } E\nonumber\\&\quad{}
-3\omega_{0 } S_{i1}(1,2)_{-1 } E,\end{align}

\begin{align}
\Har^{[1]}_{1}S_{i1}(1,1)_{0}\ExB&=
\frac{-1}{4}S_{i1}(1,1)_{-2 } E+\frac{-1}{2}\omega_{0 } S_{i1}(1,1)_{-1 } E\nonumber\\&\quad{}
+\omega_{0 } \omega_{0 } (S_{i1}(1,1)_{0}E)+\frac{3}{4}S_{i1}(1,2)_{-1 } E,
\nonumber\\
\Har^{[1]}_{2}S_{i1}(1,1)_{0}\ExB&=
\frac{1}{3}S_{i1}(1,1)_{-1 } E+\frac{-1}{3}\omega_{0 } (S_{i1}(1,1)_{0}E),
\end{align}

\begin{align}
\Har^{[1]}_{0}\Har^{[1]}_{1}\ExB&=
\frac{-172}{25}\Har^{[1]}_{-2 } E+\frac{42}{25}\omega^{[1]}_{-2 } (\Har^{[1]}_{1}\ExB)\nonumber\\&\quad{}
+\frac{126}{25}\omega_{0 } \Har^{[1]}_{-1 } E+\frac{22}{25}\omega_{0 } \omega^{[1]}_{-1 } (\Har^{[1]}_{1}\ExB)\nonumber\\&\quad{}
+\frac{-4}{75}\omega_{0 } \omega_{0 } \omega^{[1]}_{-2 } E+\frac{-58}{75}\omega_{0 } \omega_{0 } \omega_{0 } (\Har^{[1]}_{1}\ExB),\end{align}

\begin{align}
\Har^{[1]}_{1}\Har^{[1]}_{1}\ExB&=
\frac{-3}{2}\Har^{[1]}_{-1 } E+\frac{5}{2}\omega^{[1]}_{-1 } (\Har^{[1]}_{1}\ExB)\nonumber\\&\quad{}
+\frac{-11}{3}\omega_{0 } \omega^{[1]}_{-2 } E+\frac{29}{6}\omega_{0 } \omega_{0 } (\Har^{[1]}_{1}\ExB),
\nonumber\\
\Har^{[1]}_{2}\Har^{[1]}_{1}\ExB&=
-4\omega^{[1]}_{-2 } E+\frac{25}{3}\omega_{0 } (\Har^{[1]}_{1}\ExB),
\nonumber\\
\Har^{[1]}_{3}\Har^{[1]}_{1}\ExB&=
8(\Har^{[1]}_{1}\ExB),
\nonumber\\
\Har^{[1]}_{4}\Har^{[1]}_{1}\ExB&=
4\omega_{0 } E,
\nonumber\\
\Har^{[1]}_{5}\Har^{[1]}_{1}\ExB&=
\frac{1}{3}E,
\end{align}

\begin{align}
\omega^{[i]}_{0}S_{i1}(1,1)_{0}\ExB&=
-S_{i1}(1,1)_{-1 } E+2\omega_{0 } (S_{i1}(1,1)_{0}E),\nonumber\\
\omega^{[i]}_{1}S_{i1}(1,1)_{0}\ExB&=
(S_{i1}(1,1)_{0}E),
\end{align}

\begin{align}
\Har^{[i]}_{0}S_{i1}(1,1)_{0}\ExB&=
-4\omega^{[i]}_{-1 } S_{i1}(1,1)_{-1 } E-10\omega^{[i]}_{-2 } (S_{i1}(1,1)_{0}E)\nonumber\\&\quad{}
+8\omega_{0 } \omega^{[i]}_{-1 } (S_{i1}(1,1)_{0}E),\end{align}

\begin{align}
\Har^{[i]}_{1}S_{i1}(1,1)_{0}\ExB&=
-2S_{i1}(1,1)_{-2 } E-4\omega_{0 } S_{i1}(1,1)_{-1 } E\nonumber\\&\quad{}
+8\omega_{0 } \omega_{0 } (S_{i1}(1,1)_{0}E)+2S_{i1}(1,2)_{-1 } E,
\nonumber\\
\Har^{[i]}_{2}S_{i1}(1,1)_{0}\ExB&=
\frac{-7}{3}S_{i1}(1,1)_{-1 } E+\frac{14}{3}\omega_{0 } (S_{i1}(1,1)_{0}E),
\nonumber\\
\Har^{[i]}_{3}S_{i1}(1,1)_{0}\ExB&=
(S_{i1}(1,1)_{0}E),
\end{align}

\begin{align}
S_{i1}(1,1)_{0}\ExB&=
(S_{i1}(1,1)_{0}E),\end{align}

\begin{align}
S_{i1}(1,1)_{0}S_{i1}(1,1)_{0}\ExB&=
\omega^{[i]}_{-1 } E+(\Har^{[1]}_{1}\ExB),\nonumber\\
S_{i1}(1,1)_{1}S_{i1}(1,1)_{0}\ExB&=
\omega_{0 } E,
\nonumber\\
S_{i1}(1,1)_{2}S_{i1}(1,1)_{0}\ExB&=
\frac{1}{2}E,
\end{align}

\begin{align}
S_{i1}(1,1)_{0}\Har^{[1]}_{1}\ExB&=
\frac{-5}{4}S_{i1}(1,1)_{-2 } E+\frac{-5}{2}\omega_{0 } S_{i1}(1,1)_{-1 } E\nonumber\\&\quad{}
+5\omega_{0 } \omega_{0 } (S_{i1}(1,1)_{0}E)+\frac{7}{4}S_{i1}(1,2)_{-1 } E,\nonumber\\
S_{i1}(1,1)_{1}\Har^{[1]}_{1}\ExB&=
-2S_{i1}(1,1)_{-1 } E+4\omega_{0 } (S_{i1}(1,1)_{0}E),
\nonumber\\
S_{i1}(1,1)_{2}\Har^{[1]}_{1}\ExB&=
2(S_{i1}(1,1)_{0}E),
\end{align}

\begin{align}
S_{i1}(1,2)_{0}\ExB&=
S_{i1}(1,1)_{-1 } E-2\omega_{0 } (S_{i1}(1,1)_{0}E),\nonumber\\
S_{i1}(1,2)_{1}\ExB&=-(S_{i1}(1,1)_{0}E),
\end{align}

\begin{align}
S_{i1}(1,2)_{0}S_{i1}(1,1)_{0}\ExB&=
\frac{-1}{2}\omega^{[i]}_{-2 } E+\frac{-2}{3}\omega^{[1]}_{-2 } E
+\frac{4}{3}\omega_{0 } (\Har^{[1]}_{1}\ExB),\end{align}

\begin{align}
S_{i1}(1,2)_{1}S_{i1}(1,1)_{0}\ExB&=
-\omega^{[i]}_{-1 } E+(\Har^{[1]}_{1}\ExB),
\nonumber\\
S_{i1}(1,2)_{2}S_{i1}(1,1)_{0}\ExB&=0,\nonumber\\
S_{i1}(1,2)_{3}S_{i1}(1,1)_{0}\ExB&=
\frac{-1}{2}E,
\end{align}

\begin{align}
S_{i1}(1,2)_{0}\Har^{[1]}_{1}\ExB&=
\frac{-3}{2}S_{i1}(1,1)_{-3 } E+7\omega^{[i]}_{-1 } S_{i1}(1,1)_{-1 } E\nonumber\\&\quad{}
+\frac{35}{2}\omega^{[i]}_{-2 } (S_{i1}(1,1)_{0}E)+\frac{1}{2}\omega^{[1]}_{-1 } S_{i1}(1,1)_{-1 } E\nonumber\\&\quad{}
+4\omega_{0 } S_{i1}(1,1)_{-2 } E-14\omega_{0 } \omega^{[i]}_{-1 } (S_{i1}(1,1)_{0}E)\nonumber\\&\quad{}
+4\omega_{0 } \omega^{[1]}_{-1 } (S_{i1}(1,1)_{0}E)+\frac{-5}{2}\omega_{0 } \omega_{0 } S_{i1}(1,1)_{-1 } E\nonumber\\&\quad{}
-3\omega_{0 } S_{i1}(1,2)_{-1 } E,\end{align}

\begin{align}
S_{i1}(1,2)_{1}\Har^{[1]}_{1}\ExB&=
\frac{13}{4}S_{i1}(1,1)_{-2 } E+\frac{13}{2}\omega_{0 } S_{i1}(1,1)_{-1 } E\nonumber\\&\quad{}
-13\omega_{0 } \omega_{0 } (S_{i1}(1,1)_{0}E)+\frac{-15}{4}S_{i1}(1,2)_{-1 } E,
\nonumber\\
S_{i1}(1,2)_{2}\Har^{[1]}_{1}\ExB&=
6S_{i1}(1,1)_{-1 } E-12\omega_{0 } (S_{i1}(1,1)_{0}E),
\nonumber\\
S_{i1}(1,2)_{3}\Har^{[1]}_{1}\ExB&=
-6(S_{i1}(1,1)_{0}E),
\end{align}

\begin{align}
S_{i1}(1,3)_{0}\ExB&=
\frac{-1}{2}S_{i1}(1,1)_{-2 } E-\omega_{0 } S_{i1}(1,1)_{-1 } E\nonumber\\&\quad{}
+2\omega_{0 } \omega_{0 } (S_{i1}(1,1)_{0}E)+\frac{1}{2}S_{i1}(1,2)_{-1 } E,\end{align}

\begin{align}
S_{i1}(1,3)_{1}\ExB&=
-S_{i1}(1,1)_{-1 } E+2\omega_{0 } (S_{i1}(1,1)_{0}E),
\nonumber\\
S_{i1}(1,3)_{2}\ExB&=
(S_{i1}(1,1)_{0}E),
\end{align}

\begin{align}
S_{i1}(1,3)_{0}S_{i1}(1,1)_{0}\ExB&=
\frac{1}{2}\Har^{[i]}_{-1 } E+\frac{1}{6}\omega^{[i]}_{-1 } \omega^{[1]}_{-1 } E
+\frac{3}{8}\omega^{[1]}_{-1 } (\Har^{[1]}_{1}\ExB)+\frac{1}{3}\omega_{0 } \omega^{[i]}_{-2 } E\nonumber\\&\quad{}
+\frac{-1}{2}\omega_{0 } \omega^{[1]}_{-2 } E+\frac{-1}{6}\omega_{0 } \omega_{0 } \omega^{[i]}_{-1 } E\nonumber\\&\quad{}
+\frac{5}{8}\omega_{0 } \omega_{0 } (\Har^{[1]}_{1}\ExB),\end{align}

\begin{align}
S_{i1}(1,3)_{1}S_{i1}(1,1)_{0}\ExB&=
\frac{1}{2}\omega^{[i]}_{-2 } E+\frac{-1}{3}\omega^{[1]}_{-2 } E
+\frac{2}{3}\omega_{0 } (\Har^{[1]}_{1}\ExB),
\nonumber\\
S_{i1}(1,3)_{2}S_{i1}(1,1)_{0}\ExB&=
\omega^{[i]}_{-1 } E,
\nonumber\\
S_{i1}(1,3)_{3}S_{i1}(1,1)_{0}\ExB&=0,\nonumber\\
S_{i1}(1,3)_{4}S_{i1}(1,1)_{0}\ExB&=
\frac{1}{2}E,
\end{align}

\begin{align}
S_{i1}(1,3)_{0}\Har^{[1]}_{1}\ExB&=
6S_{i1}(1,1)_{-4 } E+\frac{-3}{4}\omega^{[1]}_{-1 } S_{i1}(1,1)_{-2 } E\nonumber\\&\quad{}
+\frac{7}{4}\omega^{[1]}_{-2 } S_{i1}(1,1)_{-1 } E+\frac{19}{4}\omega^{[1]}_{-3 } (S_{i1}(1,1)_{0}E)\nonumber\\&\quad{}
+\frac{9}{8}\omega^{[1]}_{-1 } \omega^{[1]}_{-1 } (S_{i1}(1,1)_{0}E)+\frac{-145}{3}\omega_{0 } \omega^{[i]}_{-1 } S_{i1}(1,1)_{-1 } E\nonumber\\&\quad{}
+\frac{-725}{6}\omega_{0 } \omega^{[i]}_{-2 } (S_{i1}(1,1)_{0}E)+\frac{-65}{4}\omega_{0 } S_{i1}(1,1)_{-1 } (\Har^{[1]}_{1}\ExB)\nonumber\\&\quad{}
+3\omega_{0 } \omega_{0 } S_{i1}(1,1)_{-2 } E+\frac{290}{3}\omega_{0 } \omega_{0 } \omega^{[i]}_{-1 } (S_{i1}(1,1)_{0}E)\nonumber\\&\quad{}
+\frac{-9}{8}\omega_{0 } \omega_{0 } \omega_{0 } \omega_{0 } (S_{i1}(1,1)_{0}E)+\frac{-83}{8}\omega_{0 } S_{i1}(1,2)_{-2 } E\nonumber\\&\quad{}
+\frac{51}{4}\omega_{0 } \omega_{0 } S_{i1}(1,2)_{-1 } E+\frac{-103}{8}\omega_{0 } S_{i1}(1,3)_{-1 } E,\end{align}

\begin{align}
S_{i1}(1,3)_{1}\Har^{[1]}_{1}\ExB&=
\frac{3}{2}S_{i1}(1,1)_{-3 } E-11\omega^{[i]}_{-1 } S_{i1}(1,1)_{-1 } E\nonumber\\&\quad{}
+\frac{-55}{2}\omega^{[i]}_{-2 } (S_{i1}(1,1)_{0}E)+\frac{-1}{2}\omega^{[1]}_{-1 } S_{i1}(1,1)_{-1 } E\nonumber\\&\quad{}
-4\omega_{0 } S_{i1}(1,1)_{-2 } E+22\omega_{0 } \omega^{[i]}_{-1 } (S_{i1}(1,1)_{0}E)\nonumber\\&\quad{}
-4\omega_{0 } \omega^{[1]}_{-1 } (S_{i1}(1,1)_{0}E)+\frac{5}{2}\omega_{0 } \omega_{0 } S_{i1}(1,1)_{-1 } E\nonumber\\&\quad{}
+3\omega_{0 } S_{i1}(1,2)_{-1 } E,
\nonumber\\
S_{i1}(1,3)_{2}\Har^{[1]}_{1}\ExB&=
\frac{-25}{4}S_{i1}(1,1)_{-2 } E+\frac{-25}{2}\omega_{0 } S_{i1}(1,1)_{-1 } E\nonumber\\&\quad{}
+25\omega_{0 } \omega_{0 } (S_{i1}(1,1)_{0}E)+\frac{27}{4}S_{i1}(1,2)_{-1 } E,
\nonumber\\
S_{i1}(1,3)_{3}\Har^{[1]}_{1}\ExB&=
-12S_{i1}(1,1)_{-1 } E+24\omega_{0 } (S_{i1}(1,1)_{0}E),
\nonumber\\
S_{i1}(1,3)_{4}\Har^{[1]}_{1}\ExB&=
12(S_{i1}(1,1)_{0}E),
\end{align}
\begin{align}
	S_{k 1}(1,1)_{0}S_{i1}(1,1)_{0}\ExB&=
	\frac{1}{2}S_{k i}(1,1)_{-1 } E,\end{align}

\begin{align}
	S_{k 1}(1,2)_{0}S_{i1}(1,1)_{0}\ExB&=
	\frac{-1}{2}S_{k i}(1,1)_{-2 } E+\frac{1}{2}S_{k i}(1,2)_{-1 } E,\nonumber\\
	S_{k 1}(1,2)_{1}S_{i1}(1,1)_{0}\ExB&=
	\frac{-1}{2}S_{k i}(1,1)_{-1 } E,
\end{align}

\begin{align}
	S_{k 1}(1,3)_{0}S_{i1}(1,1)_{0}\ExB&=
	\frac{1}{2}S_{k i}(1,1)_{-3 } E+\frac{-1}{2}S_{k i}(1,2)_{-2 } E\nonumber\\&\quad{}
	+\frac{1}{2}S_{k i}(1,3)_{-1 } E,\nonumber\\
	S_{k 1}(1,3)_{1}S_{i1}(1,1)_{0}\ExB&=
	\frac{1}{2}S_{k i}(1,1)_{-2 } E+\frac{-1}{2}S_{k i}(1,2)_{-1 } E,
	\nonumber\\
	S_{k 1}(1,3)_{2}S_{i1}(1,1)_{0}\ExB&=
	\frac{1}{2}S_{k i}(1,1)_{-1 } E,
\end{align}

\begin{align}
	S_{k i}(1,1)_{0}S_{i1}(1,1)_{0}\ExB&=
	-S_{k 1}(1,1)_{-1 } E+2\omega_{0 } (S_{k 1}(1,1)_{0}E),\nonumber\\
	S_{k i}(1,1)_{1}S_{i1}(1,1)_{0}\ExB&=
	(S_{k 1}(1,1)_{0}E),
\end{align}

\begin{align}
	S_{k i}(1,2)_{0}S_{i1}(1,1)_{0}\ExB&=
	-1\omega^{[1]}_{-1 } (S_{k 1}(1,1)_{0}E)+2\omega_{0 } S_{k 1}(1,1)_{-1 } E\nonumber\\&\quad{}
	-3\omega_{0 } \omega_{0 } (S_{k 1}(1,1)_{0}E),\nonumber\\
	S_{k i}(1,2)_{1}S_{i1}(1,1)_{0}\ExB&=
	2S_{k 1}(1,1)_{-1 } E-4\omega_{0 } (S_{k 1}(1,1)_{0}E),
	\nonumber\\
	S_{k i}(1,2)_{2}S_{i1}(1,1)_{0}\ExB&=
	-2(S_{k 1}(1,1)_{0}E),
\end{align}

\begin{align}
	S_{k i}(1,3)_{0}S_{i1}(1,1)_{0}\ExB&=
	\frac{-3}{8}S_{k 1}(1,1)_{-3 } E+\frac{3}{4}S_{k 1}(1,1)_{-1 } (\Har^{[1]}_{1}E)\nonumber\\&\quad{}
	+\frac{3}{4}S_{k 1}(1,2)_{-2 } E+\frac{-3}{4}\omega_{0 } S_{k 1}(1,2)_{-1 } E\nonumber\\&\quad{}
	+\frac{3}{8}S_{k 1}(1,3)_{-1 } E,\nonumber\\
	S_{k i}(1,3)_{1}S_{i1}(1,1)_{0}\ExB&=
	\frac{3}{2}\omega^{[1]}_{-1 } (S_{k 1}(1,1)_{0}E)-3\omega_{0 } S_{k 1}(1,1)_{-1 } E\nonumber\\&\quad{}
	+\frac{9}{2}\omega_{0 } \omega_{0 } (S_{k 1}(1,1)_{0}E),
	\nonumber\\
	S_{k i}(1,3)_{2}S_{i1}(1,1)_{0}\ExB&=
	-3S_{k 1}(1,1)_{-1 } E+6\omega_{0 } (S_{k 1}(1,1)_{0}E),
\end{align}

\subsection{The case that $\langle \alpha,\alpha\rangle=1$}
\label{section:Norm-1}
Let $\alpha\in\fh$ with $\langle\alpha,\alpha\rangle=1$.
Let $h^{[1]},\ldots,h^{[\rankL]}$ be an orthonormal basis of $\fh$
such that $\alpha\in\C h^{[1]}$.
In the following computation, $i,j,k$ are distinct elements of $\{2,\ldots,\rankL\}$.

\begin{align}
\omega^{[1]}_{0}\ExB&=
\omega_{0 } \ExB,
&
\omega^{[1]}_{1}\ExB&=
\frac{1}{2}\ExB 
,
\end{align}

\begin{align}
\omega^{[1]}_{0}S_{i1}(1,1)_{0}\ExB&=
S_{i1}(1,1)_{-1 } \ExB 
+(S_{i1}(1,2)_{0}\ExB) 
,\nonumber\\
\omega^{[1]}_{1}S_{i1}(1,1)_{0}\ExB&=
\frac{1}{2}(S_{i1}(1,1)_{0}\ExB) 
,
\end{align}

\begin{align}
\omega^{[1]}_{0}S_{i1}(1,2)_{0}\ExB&=
\frac{-2}{3}S_{i1}(1,1)_{-2 } \ExB 
+\frac{4}{3}\omega^{[1]}_{-1 } (S_{i1}(1,1)_{0}\ExB) 
\nonumber\\&\quad{}
+\frac{-2}{3}\omega_{0 } S_{i1}(1,1)_{-1 } \ExB 
-\omega_{0 } (S_{i1}(1,2)_{0}\ExB) 
,\nonumber\\
\omega^{[1]}_{1}S_{i1}(1,2)_{0}\ExB&=
\frac{1}{2}(S_{i1}(1,2)_{0}\ExB) 
,
\end{align}

\begin{align}
\omega^{[1]}_{0}S_{k 1}(1,1)_{0}\ExB&=
S_{k 1}(1,1)_{-1 } \ExB 
+(S_{k 1}(1,2)_{0}E) 
,\nonumber\\
\omega^{[1]}_{1}S_{k 1}(1,1)_{0}\ExB&=
\frac{1}{2}(S_{k 1}(1,1)_{0}E) 
,
\end{align}

\begin{align}
\omega^{[1]}_{0}S_{k 1}(1,2)_{0}\ExB&=
\frac{4}{3}\omega^{[1]}_{-1 } (S_{k 1}(1,1)_{0}E) 
+\frac{-2}{3}\omega_{0 } S_{k 1}(1,1)_{-1 } \ExB 
\nonumber\\&\quad{}
+\frac{-2}{3}S_{k 1}(1,2)_{-1 } \ExB 
+\frac{-1}{3}\omega_{0 } (S_{k 1}(1,2)_{0}E) 
,\nonumber\\
\omega^{[1]}_{1}S_{k 1}(1,2)_{0}\ExB&=
\frac{1}{2}(S_{k 1}(1,2)_{0}E) 
,
\end{align}

\begin{align}
\Har^{[1]}_{0}\ExB&=
-2\omega^{[1]}_{-2 } \ExB 
+4\omega_{0 } \omega^{[1]}_{-1 } \ExB 
-2\omega_{0 } \omega_{0 } \omega_{0 } \ExB 
,\nonumber\\
\Har^{[1]}_{1}\ExB&=
2\omega^{[1]}_{-1 } \ExB 
-\omega_{0 } \omega_{0 } \ExB 
,
\nonumber\\
\Har^{[1]}_{2}\ExB&=
\frac{1}{3}\omega_{0 } \ExB 
,
\end{align}

\begin{align}
\Har^{[1]}_{0}S_{i1}(1,1)_{0}\ExB&=
\frac{2}{3}\omega^{[i]}_{-2 } (S_{i1}(1,1)_{0}\ExB) 
+\frac{4}{3}\omega^{[i]}_{-1 } (S_{i1}(1,2)_{0}\ExB) 
\nonumber\\&\quad{}
+2S_{i1}(1,3)_{-1 } \ExB 
,\nonumber\\
\Har^{[1]}_{1}S_{i1}(1,1)_{0}\ExB&=
\frac{2}{3}S_{i1}(1,1)_{-2 } \ExB 
+\frac{2}{3}\omega^{[1]}_{-1 } (S_{i1}(1,1)_{0}\ExB) 
\nonumber\\&\quad{}
+\frac{-1}{3}\omega_{0 } S_{i1}(1,1)_{-1 } \ExB 
,
\nonumber\\
\Har^{[1]}_{2}S_{i1}(1,1)_{0}\ExB&=
\frac{1}{3}S_{i1}(1,1)_{-1 } \ExB 
+\frac{1}{3}(S_{i1}(1,2)_{0}\ExB) 
,
\end{align}

\begin{align}
\Har^{[1]}_{0}S_{i1}(1,2)_{0}\ExB&=
\frac{2}{87}\omega^{[1]}_{-1 } S_{i1}(1,1)_{-2 } \ExB 
+\omega^{[1]}_{-3 } (S_{i1}(1,1)_{0}\ExB) 
\nonumber\\&\quad{}
+\frac{2}{87}\omega^{[1]}_{-1 } \omega^{[1]}_{-1 } (S_{i1}(1,1)_{0}\ExB) 
+\frac{43}{29}\Har^{[1]}_{-1 } (S_{i1}(1,1)_{0}\ExB) 
\nonumber\\&\quad{}
+\frac{-16}{29}\omega_{0 } \omega^{[i]}_{-2 } (S_{i1}(1,1)_{0}\ExB) 
+\frac{-1}{87}\omega_{0 } \omega^{[1]}_{-1 } S_{i1}(1,1)_{-1 } \ExB 
\nonumber\\&\quad{}
+\frac{-1}{174}\omega_{0 } \omega^{[1]}_{-2 } (S_{i1}(1,1)_{0}\ExB) 
+\frac{-2}{29}\omega^{[1]}_{-1 } S_{i1}(1,2)_{-1 } \ExB 
\nonumber\\&\quad{}
+\frac{-3}{58}\omega_{0 } S_{i1}(1,2)_{-2 } \ExB 
+\frac{-32}{29}\omega_{0 } \omega^{[i]}_{-1 } (S_{i1}(1,2)_{0}\ExB) 
\nonumber\\&\quad{}
+\frac{1}{29}\omega_{0 } \omega_{0 } S_{i1}(1,2)_{-1 } \ExB 
+\frac{-12}{29}S_{i1}(1,3)_{-2 } \ExB 
\nonumber\\&\quad{}
+\frac{-3}{2}\omega_{0 } S_{i1}(1,3)_{-1 } \ExB 
,\nonumber\\
\Har^{[1]}_{1}S_{i1}(1,2)_{0}\ExB&=
\frac{-1}{2}S_{i1}(1,1)_{-3 } \ExB 
+\frac{2}{3}\omega^{[i]}_{-2 } (S_{i1}(1,1)_{0}\ExB) 
\nonumber\\&\quad{}
+\frac{-1}{2}\omega^{[1]}_{-1 } S_{i1}(1,1)_{-1 } \ExB 
+\frac{-1}{2}\omega_{0 } S_{i1}(1,1)_{-2 } \ExB 
\nonumber\\&\quad{}
+\frac{1}{2}\omega_{0 } \omega^{[1]}_{-1 } (S_{i1}(1,1)_{0}\ExB) 
+\frac{4}{3}\omega^{[i]}_{-1 } (S_{i1}(1,2)_{0}\ExB) 
\nonumber\\&\quad{}
+\frac{-1}{2}\omega_{0 } \omega_{0 } (S_{i1}(1,2)_{0}\ExB) 
+S_{i1}(1,3)_{-1 } \ExB 
,
\nonumber\\
\Har^{[1]}_{2}S_{i1}(1,2)_{0}\ExB&=
\frac{-2}{9}S_{i1}(1,1)_{-2 } \ExB 
+\frac{4}{9}\omega^{[1]}_{-1 } (S_{i1}(1,1)_{0}\ExB) 
\nonumber\\&\quad{}
+\frac{-2}{9}\omega_{0 } S_{i1}(1,1)_{-1 } \ExB 
+\frac{-1}{3}\omega_{0 } (S_{i1}(1,2)_{0}\ExB) 
,
\end{align}

\begin{align}
\Har^{[1]}_{0}S_{k 1}(1,1)_{0}\ExB&=
-8\omega^{[1]}_{-2 } (S_{k 1}(1,1)_{0}E) 
+2\omega_{0 } S_{k 1}(1,1)_{-2 } \ExB 
\nonumber\\&\quad{}
-6S_{k 1}(1,2)_{-2 } \ExB 
+4\omega^{[1]}_{-1 } (S_{k 1}(1,2)_{0}E) 
\nonumber\\&\quad{}
+6\omega_{0 } S_{k 1}(1,2)_{-1 } \ExB 
+6S_{k 1}(1,3)_{-1 } \ExB 
,\nonumber\\
\Har^{[1]}_{1}S_{k 1}(1,1)_{0}\ExB&=
\frac{2}{3}\omega^{[1]}_{-1 } (S_{k 1}(1,1)_{0}E) 
+\frac{-1}{3}\omega_{0 } S_{k 1}(1,1)_{-1 } \ExB 
\nonumber\\&\quad{}
+\frac{2}{3}S_{k 1}(1,2)_{-1 } \ExB 
+\frac{-2}{3}\omega_{0 } (S_{k 1}(1,2)_{0}E) 
,
\nonumber\\
\Har^{[1]}_{2}S_{k 1}(1,1)_{0}\ExB&=
\frac{1}{3}S_{k 1}(1,1)_{-1 } \ExB 
+\frac{1}{3}(S_{k 1}(1,2)_{0}E) 
,
\end{align}

\begin{align}
\Har^{[1]}_{0}S_{k 1}(1,2)_{0}\ExB&=
\frac{196}{243}\omega^{[k]}_{-3 } (S_{k 1}(1,1)_{0}E) 
+\frac{4}{81}\omega^{[k]}_{-2 } S_{k 1}(1,1)_{-1 } \ExB 
\nonumber\\&\quad{}
+\frac{86}{81}\omega^{[k]}_{-1 } S_{k 1}(1,1)_{-2 } \ExB 
+\frac{-172}{81}\Har^{[k]}_{-1 } (S_{k 1}(1,1)_{0}E) 
\nonumber\\&\quad{}
+2\omega^{[1]}_{-2 } S_{k 1}(1,1)_{-1 } \ExB 
+\frac{8}{3}\omega^{[1]}_{-1 } \omega^{[1]}_{-1 } (S_{k 1}(1,1)_{0}E) 
\nonumber\\&\quad{}
+\frac{-2}{27}\omega_{0 } \omega^{[1]}_{-2 } (S_{k 1}(1,1)_{0}E) 
+\frac{-2}{3}\omega_{0 } \omega_{0 } \omega_{0 } \omega_{0 } (S_{k 1}(1,1)_{0}E) 
\nonumber\\&\quad{}
+\frac{-86}{81}\omega^{[k]}_{-1 } S_{k 1}(1,2)_{-1 } \ExB 
+\frac{-28}{9}S_{k 1}(1,2)_{-3 } \ExB 
\nonumber\\&\quad{}
+\frac{26}{9}\omega^{[1]}_{-2 } (S_{k 1}(1,2)_{0}E) 
+\frac{94}{81}\omega_{0 } \omega^{[k]}_{-1 } (S_{k 1}(1,2)_{0}E) 
\nonumber\\&\quad{}
+\frac{-116}{27}\omega_{0 } \omega^{[1]}_{-1 } (S_{k 1}(1,2)_{0}E) 
+\frac{-16}{27}\omega_{0 } \omega_{0 } S_{k 1}(1,2)_{-1 } \ExB 
\nonumber\\&\quad{}
+\frac{-14}{27}\omega_{0 } \omega_{0 } \omega_{0 } (S_{k 1}(1,2)_{0}E) 
+\frac{88}{27}\omega_{0 } S_{k 1}(1,3)_{-1 } \ExB 
,\nonumber\\
\Har^{[1]}_{1}S_{k 1}(1,2)_{0}\ExB&=
-12\omega^{[1]}_{-2 } (S_{k 1}(1,1)_{0}E) 
+3\omega_{0 } S_{k 1}(1,1)_{-2 } \ExB 
\nonumber\\&\quad{}
-10S_{k 1}(1,2)_{-2 } \ExB 
+6\omega^{[1]}_{-1 } (S_{k 1}(1,2)_{0}E) 
\nonumber\\&\quad{}
+9\omega_{0 } S_{k 1}(1,2)_{-1 } \ExB 
+8S_{k 1}(1,3)_{-1 } \ExB 
,
\nonumber\\
\Har^{[1]}_{2}S_{k 1}(1,2)_{0}\ExB&=
\frac{4}{9}\omega^{[1]}_{-1 } (S_{k 1}(1,1)_{0}E) 
+\frac{-2}{9}\omega_{0 } S_{k 1}(1,1)_{-1 } \ExB 
\nonumber\\&\quad{}
+\frac{-2}{9}S_{k 1}(1,2)_{-1 } \ExB 
+\frac{-1}{9}\omega_{0 } (S_{k 1}(1,2)_{0}E) 
,
\end{align}

\begin{align}
\omega^{[i]}_{0}S_{i1}(1,1)_{0}\ExB&=
-(S_{i1}(1,2)_{0}\ExB) 
,\nonumber\\
\omega^{[i]}_{1}S_{i1}(1,1)_{0}\ExB&=
(S_{i1}(1,1)_{0}\ExB) 
,
\end{align}

\begin{align}
\omega^{[i]}_{0}S_{i1}(1,2)_{0}\ExB&=
\frac{2}{3}S_{i1}(1,1)_{-2 } \ExB 
+\frac{-4}{3}\omega^{[1]}_{-1 } (S_{i1}(1,1)_{0}\ExB) 
\nonumber\\&\quad{}
+\frac{2}{3}\omega_{0 } S_{i1}(1,1)_{-1 } \ExB 
+2\omega_{0 } (S_{i1}(1,2)_{0}\ExB) 
,\nonumber\\
\omega^{[i]}_{1}S_{i1}(1,2)_{0}\ExB&=
2(S_{i1}(1,2)_{0}\ExB) 
,
\nonumber\\
\omega^{[i]}_{2}S_{i1}(1,2)_{0}\ExB&=
-2(S_{i1}(1,1)_{0}\ExB) 
,
\end{align}

\begin{align}
\Har^{[i]}_{0}S_{i1}(1,1)_{0}\ExB&=
-2\omega^{[i]}_{-2 } (S_{i1}(1,1)_{0}\ExB) 
-4\omega^{[i]}_{-1 } (S_{i1}(1,2)_{0}\ExB) 
,\nonumber\\
\Har^{[i]}_{1}S_{i1}(1,1)_{0}\ExB&=
\frac{-4}{3}S_{i1}(1,1)_{-2 } \ExB 
+\frac{8}{3}\omega^{[1]}_{-1 } (S_{i1}(1,1)_{0}\ExB) 
\nonumber\\&\quad{}
+\frac{-4}{3}\omega_{0 } S_{i1}(1,1)_{-1 } \ExB 
-4\omega_{0 } (S_{i1}(1,2)_{0}\ExB) 
,
\nonumber\\
\Har^{[i]}_{2}S_{i1}(1,1)_{0}\ExB&=
\frac{-7}{3}(S_{i1}(1,2)_{0}\ExB) 
,
\nonumber\\
\Har^{[i]}_{3}S_{i1}(1,1)_{0}\ExB&=
(S_{i1}(1,1)_{0}\ExB) 
,
\end{align}

\begin{align}
\Har^{[i]}_{0}S_{i1}(1,2)_{0}\ExB&=
\frac{-32}{29}\omega^{[1]}_{-1 } S_{i1}(1,1)_{-2 } \ExB 
+\frac{-32}{29}\omega^{[1]}_{-1 } \omega^{[1]}_{-1 } (S_{i1}(1,1)_{0}\ExB) 
\nonumber\\&\quad{}
+\frac{24}{29}\Har^{[1]}_{-1 } (S_{i1}(1,1)_{0}\ExB) 
+\frac{72}{29}\omega_{0 } \omega^{[i]}_{-2 } (S_{i1}(1,1)_{0}\ExB) 
\nonumber\\&\quad{}
+\frac{16}{29}\omega_{0 } \omega^{[1]}_{-1 } S_{i1}(1,1)_{-1 } \ExB 
+\frac{8}{29}\omega_{0 } \omega^{[1]}_{-2 } (S_{i1}(1,1)_{0}\ExB) 
\nonumber\\&\quad{}
+\frac{96}{29}\omega^{[1]}_{-1 } S_{i1}(1,2)_{-1 } \ExB 
+\frac{72}{29}\omega_{0 } S_{i1}(1,2)_{-2 } \ExB 
\nonumber\\&\quad{}
+\frac{144}{29}\omega_{0 } \omega^{[i]}_{-1 } (S_{i1}(1,2)_{0}\ExB) 
+\frac{-48}{29}\omega_{0 } \omega_{0 } S_{i1}(1,2)_{-1 } \ExB 
\nonumber\\&\quad{}
+\frac{-120}{29}S_{i1}(1,3)_{-2 } \ExB 
,\end{align}

\begin{align}
\Har^{[i]}_{1}S_{i1}(1,2)_{0}\ExB&=
6\omega^{[i]}_{-2 } (S_{i1}(1,1)_{0}\ExB) 
+12\omega^{[i]}_{-1 } (S_{i1}(1,2)_{0}\ExB) 
,
\nonumber\\
\Har^{[i]}_{2}S_{i1}(1,2)_{0}\ExB&=
\frac{38}{9}S_{i1}(1,1)_{-2 } \ExB 
+\frac{-76}{9}\omega^{[1]}_{-1 } (S_{i1}(1,1)_{0}\ExB) 
\nonumber\\&\quad{}
+\frac{38}{9}\omega_{0 } S_{i1}(1,1)_{-1 } \ExB 
+\frac{38}{3}\omega_{0 } (S_{i1}(1,2)_{0}\ExB) 
,
\nonumber\\
\Har^{[i]}_{3}S_{i1}(1,2)_{0}\ExB&=
8(S_{i1}(1,2)_{0}\ExB) 
,
\nonumber\\
\Har^{[i]}_{4}S_{i1}(1,2)_{0}\ExB&=
-4(S_{i1}(1,1)_{0}\ExB) 
,
\end{align}

\begin{align}
S_{i1}(1,1)_{0}\ExB&=
(S_{i1}(1,1)_{0}\ExB) 
,\end{align}

\begin{align}
S_{i1}(1,1)_{0}S_{i1}(1,1)_{0}\ExB&=
2\omega^{[i]}_{-1 } \ExB 
+2\omega^{[1]}_{-1 } \ExB 
-\omega_{0 } \omega_{0 } \ExB 
,\nonumber\\
S_{i1}(1,1)_{1}S_{i1}(1,1)_{0}\ExB&=
\omega_{0 } \ExB 
,
\nonumber\\
S_{i1}(1,1)_{2}S_{i1}(1,1)_{0}\ExB&=
\ExB 
,
\end{align}

\begin{align}
S_{i1}(1,1)_{0}S_{i1}(1,2)_{0}\ExB&=
-\omega^{[i]}_{-2 } \ExB 
+2\omega^{[1]}_{-2 } \ExB 
-4\omega_{0 } \omega^{[1]}_{-1 } \ExB 
+2\omega_{0 } \omega_{0 } \omega_{0 } \ExB 
,\nonumber\\
S_{i1}(1,1)_{1}S_{i1}(1,2)_{0}\ExB&=
-4\omega^{[1]}_{-1 } \ExB 
+2\omega_{0 } \omega_{0 } \ExB 
,
\nonumber\\
S_{i1}(1,1)_{2}S_{i1}(1,2)_{0}\ExB&=
-2\omega_{0 } \ExB 
,
\nonumber\\
S_{i1}(1,1)_{3}S_{i1}(1,2)_{0}\ExB&=
-2\ExB 
,
\end{align}

\begin{align}
S_{i1}(1,1)_{0}S_{k 1}(1,1)_{0}\ExB&=
S_{k i}(1,1)_{-1 } \ExB 
,\end{align}

\begin{align}
S_{i1}(1,1)_{0}S_{k 1}(1,2)_{0}\ExB&=
-S_{k i}(1,1)_{-2 } \ExB 
+S_{k i}(1,2)_{-1 } \ExB 
,\end{align}

\begin{align}
S_{i1}(1,2)_{0}\ExB&=
(S_{i1}(1,2)_{0}\ExB) 
,\nonumber\\
S_{i1}(1,2)_{1}\ExB&=
-(S_{i1}(1,1)_{0}\ExB) 
,
\end{align}

\begin{align}
S_{i1}(1,2)_{0}S_{i1}(1,1)_{0}\ExB&=
-\omega^{[i]}_{-2 } \ExB 
-2\omega^{[1]}_{-2 } \ExB 
\nonumber\\&\quad{}
+4\omega_{0 } \omega^{[1]}_{-1 } \ExB 
-2\omega_{0 } \omega_{0 } \omega_{0 } \ExB 
,\nonumber\\
S_{i1}(1,2)_{1}S_{i1}(1,1)_{0}\ExB&=
-2\omega^{[i]}_{-1 } \ExB 
+2\omega^{[1]}_{-1 } \ExB 
-\omega_{0 } \omega_{0 } \ExB 
,
\nonumber\\
S_{i1}(1,2)_{2}S_{i1}(1,1)_{0}\ExB&=0,\nonumber\\
S_{i1}(1,2)_{3}S_{i1}(1,1)_{0}\ExB&=
-\ExB 
,
\end{align}

\begin{align}
S_{i1}(1,2)_{0}S_{i1}(1,2)_{0}\ExB&=
-2\Har^{[i]}_{-1 } \ExB 
+\frac{2}{3}\omega^{[i]}_{-3 } \ExB 
\nonumber\\&\quad{}
+\frac{-24}{23}\omega^{[1]}_{-1 } \omega^{[1]}_{-1 } \ExB 
+\frac{-54}{23}\Har^{[1]}_{-1 } \ExB 
\nonumber\\&\quad{}
+\frac{-42}{23}\omega_{0 } \omega^{[1]}_{-2 } \ExB 
+\frac{60}{23}\omega_{0 } \omega_{0 } \omega^{[1]}_{-1 } \ExB 
\nonumber\\&\quad{}
+\frac{-24}{23}\omega_{0 } \omega_{0 } \omega_{0 } \omega_{0 } \ExB 
,\nonumber\\
S_{i1}(1,2)_{1}S_{i1}(1,2)_{0}\ExB&=
\omega^{[i]}_{-2 } \ExB 
+4\omega^{[1]}_{-2 } \ExB 
-8\omega_{0 } \omega^{[1]}_{-1 } \ExB 
+4\omega_{0 } \omega_{0 } \omega_{0 } \ExB 
,
\nonumber\\
S_{i1}(1,2)_{2}S_{i1}(1,2)_{0}\ExB&=
-4\omega^{[1]}_{-1 } \ExB 
+2\omega_{0 } \omega_{0 } \ExB 
,
\nonumber\\
S_{i1}(1,2)_{3}S_{i1}(1,2)_{0}\ExB&=0,\nonumber\\
S_{i1}(1,2)_{4}S_{i1}(1,2)_{0}\ExB&=
2\ExB 
,
\end{align}

\begin{align}
S_{i1}(1,2)_{0}S_{k 1}(1,1)_{0}\ExB&=
-S_{k i}(1,2)_{-1 } \ExB,\nonumber\\
S_{i1}(1,2)_{1}S_{k 1}(1,1)_{0}\ExB&=
-S_{k i}(1,1)_{-1 } \ExB 
,
\end{align}

\begin{align}
S_{i1}(1,2)_{0}S_{k 1}(1,2)_{0}\ExB&=
S_{k i}(1,2)_{-2 } \ExB 
-2S_{k i}(1,3)_{-1 } \ExB 
,\nonumber\\
S_{i1}(1,2)_{1}S_{k 1}(1,2)_{0}\ExB&=
S_{k i}(1,1)_{-2 } \ExB 
-S_{k i}(1,2)_{-1 } \ExB 
,
\end{align}

\begin{align}
S_{i1}(1,3)_{0}\ExB&=
\frac{-1}{3}S_{i1}(1,1)_{-2 } \ExB 
+\frac{2}{3}\omega^{[1]}_{-1 } (S_{i1}(1,1)_{0}\ExB) 
\nonumber\\&\quad{}
+\frac{-1}{3}\omega_{0 } S_{i1}(1,1)_{-1 } \ExB 
-\omega_{0 } (S_{i1}(1,2)_{0}\ExB) 
,\nonumber\\
S_{i1}(1,3)_{1}\ExB&=
-(S_{i1}(1,2)_{0}\ExB) 
,
\nonumber\\
S_{i1}(1,3)_{2}\ExB&=
(S_{i1}(1,1)_{0}\ExB) 
,
\end{align}

\begin{align}
S_{i1}(1,3)_{0}S_{i1}(1,1)_{0}\ExB&=
\Har^{[i]}_{-1 } \ExB 
+\frac{2}{3}\omega^{[i]}_{-3 } \ExB 
\nonumber\\&\quad{}
+\frac{12}{23}\omega^{[1]}_{-1 } \omega^{[1]}_{-1 } \ExB 
+\frac{27}{23}\Har^{[1]}_{-1 } \ExB 
\nonumber\\&\quad{}
+\frac{21}{23}\omega_{0 } \omega^{[1]}_{-2 } \ExB 
+\frac{-30}{23}\omega_{0 } \omega_{0 } \omega^{[1]}_{-1 } \ExB 
\nonumber\\&\quad{}
+\frac{12}{23}\omega_{0 } \omega_{0 } \omega_{0 } \omega_{0 } \ExB 
,\nonumber\\
S_{i1}(1,3)_{1}S_{i1}(1,1)_{0}\ExB&=
\omega^{[i]}_{-2 } \ExB 
-\omega^{[1]}_{-2 } \ExB 
+2\omega_{0 } \omega^{[1]}_{-1 } \ExB 
-\omega_{0 } \omega_{0 } \omega_{0 } \ExB 
,
\nonumber\\
S_{i1}(1,3)_{2}S_{i1}(1,1)_{0}\ExB&=
2\omega^{[i]}_{-1 } \ExB 
,
\nonumber\\
S_{i1}(1,3)_{3}S_{i1}(1,1)_{0}\ExB&=0,\nonumber\\
S_{i1}(1,3)_{4}S_{i1}(1,1)_{0}\ExB&=
\ExB 
,
\end{align}

\begin{align}
S_{i1}(1,3)_{0}S_{i1}(1,2)_{0}\ExB&=
\frac{1}{2}\Har^{[i]}_{-2 } \ExB 
+\frac{-1}{2}\omega^{[i]}_{-4 } \ExB 
\nonumber\\&\quad{}
+\frac{24}{13}\omega^{[1]}_{-2 } \omega^{[1]}_{-1 } \ExB 
+\frac{36}{13}\Har^{[1]}_{-2 } \ExB 
\nonumber\\&\quad{}
+\frac{-480}{299}\omega_{0 } \omega^{[1]}_{-1 } \omega^{[1]}_{-1 } \ExB 
+\frac{-1908}{299}\omega_{0 } \Har^{[1]}_{-1 } \ExB 
\nonumber\\&\quad{}
+\frac{-192}{23}\omega_{0 } \omega_{0 } \omega^{[1]}_{-2 } \ExB 
+\frac{3960}{299}\omega_{0 } \omega_{0 } \omega_{0 } \omega^{[1]}_{-1 } \ExB 
\nonumber\\&\quad{}
+\frac{-1860}{299}\omega_{0 } \omega_{0 } \omega_{0 } \omega_{0 } \omega_{0 } \ExB 
,\nonumber\\
S_{i1}(1,3)_{1}S_{i1}(1,2)_{0}\ExB&=
2\Har^{[i]}_{-1 } \ExB 
+\frac{-2}{3}\omega^{[i]}_{-3 } \ExB 
\nonumber\\&\quad{}
+\frac{-24}{23}\omega^{[1]}_{-1 } \omega^{[1]}_{-1 } \ExB 
+\frac{-54}{23}\Har^{[1]}_{-1 } \ExB 
\nonumber\\&\quad{}
+\frac{-42}{23}\omega_{0 } \omega^{[1]}_{-2 } \ExB 
+\frac{60}{23}\omega_{0 } \omega_{0 } \omega^{[1]}_{-1 } \ExB 
\nonumber\\&\quad{}
+\frac{-24}{23}\omega_{0 } \omega_{0 } \omega_{0 } \omega_{0 } \ExB 
,
\nonumber\\
S_{i1}(1,3)_{2}S_{i1}(1,2)_{0}\ExB&=
-\omega^{[i]}_{-2 } \ExB 
+2\omega^{[1]}_{-2 } \ExB 
-4\omega_{0 } \omega^{[1]}_{-1 } \ExB 
+2\omega_{0 } \omega_{0 } \omega_{0 } \ExB 
,\nonumber\\
S_{i1}(1,3)_{4}S_{i1}(1,2)_{0}\ExB&=0,\nonumber\\
S_{i1}(1,3)_{5}S_{i1}(1,2)_{0}\ExB&=
-2\ExB 
,
\end{align}

\begin{align}
S_{i1}(1,3)_{0}S_{k 1}(1,1)_{0}\ExB&=
S_{k i}(1,3)_{-1 } \ExB 
,\nonumber\\
S_{i1}(1,3)_{1}S_{k 1}(1,1)_{0}\ExB&=
S_{k i}(1,2)_{-1 } \ExB 
,
\nonumber\\
S_{i1}(1,3)_{2}S_{k 1}(1,1)_{0}\ExB&=
S_{k i}(1,1)_{-1 } \ExB 
,
\end{align}

\begin{align}
S_{i1}(1,3)_{0}S_{k 1}(1,2)_{0}\ExB&=
\omega^{[k]}_{-2 } S_{k i}(1,1)_{-1 } \ExB 
-2\omega^{[k]}_{-1 } S_{k i}(1,1)_{-2 } \ExB 
\nonumber\\&\quad{}
+3S_{k i}(1,1)_{-4 } \ExB 
+2\omega^{[k]}_{-1 } S_{k i}(1,2)_{-1 } \ExB 
\nonumber\\&\quad{}
-3S_{k i}(1,2)_{-3 } \ExB 
+2S_{k i}(1,3)_{-2 } \ExB 
,\nonumber\\
S_{i1}(1,3)_{1}S_{k 1}(1,2)_{0}\ExB&=
-S_{k i}(1,2)_{-2 } \ExB 
+2S_{k i}(1,3)_{-1 } \ExB 
,
\nonumber\\
S_{i1}(1,3)_{2}S_{k 1}(1,2)_{0}\ExB&=
-S_{k i}(1,1)_{-2 } \ExB 
+S_{k i}(1,2)_{-1 } \ExB 
,
\end{align}

\begin{align}
S_{k 1}(1,1)_{0}\ExB&=
(S_{k 1}(1,1)_{0}E) 
,\end{align}

\begin{align}
S_{k 1}(1,1)_{0}S_{i1}(1,1)_{0}\ExB&=
S_{k i}(1,1)_{-1 } \ExB 
,\end{align}

\begin{align}
S_{k 1}(1,1)_{0}S_{i1}(1,2)_{0}\ExB&=
-S_{k i}(1,2)_{-1 } \ExB 
,\end{align}

\begin{align}
S_{k 1}(1,2)_{0}\ExB&=
(S_{k 1}(1,2)_{0}E), \nonumber\\
S_{k 1}(1,2)_{1}\ExB&=
-(S_{k 1}(1,1)_{0}E) 
,
\end{align}

\begin{align}
S_{k 1}(1,2)_{0}S_{i1}(1,1)_{0}\ExB&=
-S_{k i}(1,1)_{-2 } \ExB 
+S_{k i}(1,2)_{-1 } \ExB,
\nonumber\\
S_{k 1}(1,2)_{1}S_{i1}(1,1)_{0}\ExB&=
-S_{k i}(1,1)_{-1 } \ExB 
,
\end{align}

\begin{align}
S_{k 1}(1,2)_{0}S_{i1}(1,2)_{0}\ExB&=
S_{k i}(1,2)_{-2 } \ExB 
-2S_{k i}(1,3)_{-1 } \ExB, 
\nonumber\\
S_{k 1}(1,2)_{1}S_{i1}(1,2)_{0}\ExB&=
S_{k i}(1,2)_{-1 } \ExB 
,
\end{align}

\begin{align}
S_{k 1}(1,2)_{0}S_{k 1}(1,1)_{0}\ExB&=
-\omega^{[k]}_{-2 } \ExB 
-2\omega^{[1]}_{-2 } \ExB 
+4\omega_{0 } \omega^{[1]}_{-1 } \ExB 
-2\omega_{0 } \omega_{0 } \omega_{0 } \ExB 
,\nonumber\\
S_{k 1}(1,2)_{1}S_{k 1}(1,1)_{0}\ExB&=
-2\omega^{[k]}_{-1 } \ExB 
+2\omega^{[1]}_{-1 } \ExB 
-\omega_{0 } \omega_{0 } \ExB 
,
\end{align}

\begin{align}
S_{k 1}(1,2)_{0}S_{k 1}(1,2)_{0}\ExB&=
\frac{2}{3}\omega^{[k]}_{-3 } \ExB 
-2\Har^{[k]}_{-1 } \ExB 
+\frac{-24}{23}\omega^{[1]}_{-1 } \omega^{[1]}_{-1 } \ExB 
+\frac{-54}{23}\Har^{[1]}_{-1 } \ExB 
\nonumber\\&\quad{}
+\frac{-42}{23}\omega_{0 } \omega^{[1]}_{-2 } \ExB 
+\frac{60}{23}\omega_{0 } \omega_{0 } \omega^{[1]}_{-1 } \ExB 
+\frac{-24}{23}\omega_{0 } \omega_{0 } \omega_{0 } \omega_{0 } \ExB 
,\nonumber\\
S_{k 1}(1,2)_{1}S_{k 1}(1,2)_{0}\ExB&=
\omega^{[k]}_{-2 } \ExB 
+4\omega^{[1]}_{-2 } \ExB 
-8\omega_{0 } \omega^{[1]}_{-1 } \ExB 
+4\omega_{0 } \omega_{0 } \omega_{0 } \ExB 
,
\nonumber\\
S_{k 1}(1,2)_{2}S_{k 1}(1,2)_{0}\ExB&=
-4\omega^{[1]}_{-1 } \ExB 
+2\omega_{0 } \omega_{0 } \ExB 
,
\end{align}

\begin{align}
S_{k 1}(1,3)_{0}\ExB&=
\frac{2}{3}\omega^{[1]}_{-1 } (S_{k 1}(1,1)_{0}E) 
+\frac{-1}{3}\omega_{0 } S_{k 1}(1,1)_{-1 } \ExB 
\nonumber\\&\quad{}
+\frac{-1}{3}S_{k 1}(1,2)_{-1 } \ExB 
+\frac{-2}{3}\omega_{0 } (S_{k 1}(1,2)_{0}E) 
,\nonumber\\
S_{k 1}(1,3)_{1}\ExB&=
-(S_{k 1}(1,2)_{0}E) 
,
\nonumber\\
S_{k 1}(1,3)_{2}\ExB&=
(S_{k 1}(1,1)_{0}E) 
,
\end{align}

\begin{align}
S_{k 1}(1,3)_{0}S_{i1}(1,1)_{0}\ExB&=
S_{k i}(1,1)_{-3 } \ExB 
-S_{k i}(1,2)_{-2 } \ExB 
+S_{k i}(1,3)_{-1 } \ExB 
,\nonumber\\
S_{k 1}(1,3)_{1}S_{i1}(1,1)_{0}\ExB&=
S_{k i}(1,1)_{-2 } \ExB 
-S_{k i}(1,2)_{-1 } \ExB 
,
\nonumber\\
S_{k 1}(1,3)_{2}S_{i1}(1,1)_{0}\ExB&=
S_{k i}(1,1)_{-1 } \ExB 
,
\end{align}

\begin{align}
S_{k 1}(1,3)_{0}S_{i1}(1,2)_{0}\ExB&=
-\omega^{[k]}_{-2 } S_{k i}(1,1)_{-1 } \ExB 
+2\omega^{[k]}_{-1 } S_{k i}(1,1)_{-2 } \ExB 
\nonumber\\&\quad{}
-3S_{k i}(1,1)_{-4 } \ExB 
-2\omega^{[k]}_{-1 } S_{k i}(1,2)_{-1 } \ExB 
\nonumber\\&\quad{}
+2S_{k i}(1,2)_{-3 } \ExB 
-S_{k i}(1,3)_{-2 } \ExB 
,\nonumber\\
S_{k 1}(1,3)_{1}S_{i1}(1,2)_{0}\ExB&=
-S_{k i}(1,2)_{-2 } \ExB 
+2S_{k i}(1,3)_{-1 } \ExB 
,
\nonumber\\
S_{k 1}(1,3)_{2}S_{i1}(1,2)_{0}\ExB&=
-S_{k i}(1,2)_{-1 } \ExB 
,
\end{align}

\begin{align}
S_{k 1}(1,3)_{0}S_{k 1}(1,1)_{0}\ExB&=
\frac{2}{3}\omega^{[k]}_{-3 } \ExB 
+\Har^{[k]}_{-1 } \ExB 
+\frac{12}{23}\omega^{[1]}_{-1 } \omega^{[1]}_{-1 } \ExB 
+\frac{27}{23}\Har^{[1]}_{-1 } \ExB 
\nonumber\\&\quad{}
+\frac{21}{23}\omega_{0 } \omega^{[1]}_{-2 } \ExB 
+\frac{-30}{23}\omega_{0 } \omega_{0 } \omega^{[1]}_{-1 } \ExB 
\nonumber\\&\quad{}
+\frac{12}{23}\omega_{0 } \omega_{0 } \omega_{0 } \omega_{0 } \ExB 
,\nonumber\\
S_{k 1}(1,3)_{1}S_{k 1}(1,1)_{0}\ExB&=
\omega^{[k]}_{-2 } \ExB 
-\omega^{[1]}_{-2 } \ExB 
+2\omega_{0 } \omega^{[1]}_{-1 } \ExB 
-\omega_{0 } \omega_{0 } \omega_{0 } \ExB 
,
\nonumber\\
S_{k 1}(1,3)_{2}S_{k 1}(1,1)_{0}\ExB&=
2\omega^{[k]}_{-1 } \ExB 
,
\end{align}

\begin{align}
S_{k 1}(1,3)_{0}S_{k 1}(1,2)_{0}\ExB&=
\frac{-1}{2}\omega^{[k]}_{-4 } \ExB 
+\frac{1}{2}\Har^{[k]}_{-2 } \ExB 
+\frac{24}{13}\omega^{[1]}_{-2 } \omega^{[1]}_{-1 } \ExB 
+\frac{36}{13}\Har^{[1]}_{-2 } \ExB 
\nonumber\\&\quad{}
+\frac{-480}{299}\omega_{0 } \omega^{[1]}_{-1 } \omega^{[1]}_{-1 } \ExB 
+\frac{-1908}{299}\omega_{0 } \Har^{[1]}_{-1 } \ExB 
\nonumber\\&\quad{}
+\frac{-192}{23}\omega_{0 } \omega_{0 } \omega^{[1]}_{-2 } \ExB 
+\frac{3960}{299}\omega_{0 } \omega_{0 } \omega_{0 } \omega^{[1]}_{-1 } \ExB 
\nonumber\\&\quad{}
+\frac{-1860}{299}\omega_{0 } \omega_{0 } \omega_{0 } \omega_{0 } \omega_{0 } \ExB 
,\nonumber\\
S_{k 1}(1,3)_{1}S_{k 1}(1,2)_{0}\ExB&=
\frac{-2}{3}\omega^{[k]}_{-3 } \ExB 
+2\Har^{[k]}_{-1 } \ExB 
\nonumber\\&\quad{}
+\frac{-24}{23}\omega^{[1]}_{-1 } \omega^{[1]}_{-1 } \ExB 
+\frac{-54}{23}\Har^{[1]}_{-1 } \ExB 
\nonumber\\&\quad{}
+\frac{-42}{23}\omega_{0 } \omega^{[1]}_{-2 } \ExB 
+\frac{60}{23}\omega_{0 } \omega_{0 } \omega^{[1]}_{-1 } \ExB 
\nonumber\\&\quad{}
+\frac{-24}{23}\omega_{0 } \omega_{0 } \omega_{0 } \omega_{0 } \ExB 
,
\nonumber\\
S_{k 1}(1,3)_{2}S_{k 1}(1,2)_{0}\ExB&=
-\omega^{[k]}_{-2 } \ExB 
+2\omega^{[1]}_{-2 } \ExB 
-4\omega_{0 } \omega^{[1]}_{-1 } \ExB 
+2\omega_{0 } \omega_{0 } \omega_{0 } \ExB 
,
\end{align}

\begin{align}
S_{k i}(1,1)_{0}S_{i1}(1,1)_{0}\ExB&=
-(S_{k 1}(1,2)_{0}E) 
,\nonumber\\
S_{k i}(1,1)_{1}S_{i1}(1,1)_{0}\ExB&=
(S_{k 1}(1,1)_{0}E) 
,
\end{align}
\begin{align}
S_{k i}(1,1)_{0}S_{i1}(1,2)_{0}\ExB&=
\frac{-4}{3}\omega^{[1]}_{-1 } (S_{k 1}(1,1)_{0}E) 
+\frac{2}{3}\omega_{0 } S_{k 1}(1,1)_{-1 } \ExB 
\nonumber\\&\quad{}
+\frac{2}{3}S_{k 1}(1,2)_{-1 } \ExB 
+\frac{4}{3}\omega_{0 } (S_{k 1}(1,2)_{0}E) 
,\nonumber\\
S_{k i}(1,1)_{1}S_{i1}(1,2)_{0}\ExB&=
2(S_{k 1}(1,2)_{0}E) 
,
\nonumber\\
S_{k i}(1,1)_{2}S_{i1}(1,2)_{0}\ExB&=
-2(S_{k 1}(1,1)_{0}E) 
,
\end{align}

\begin{align}
S_{k i}(1,1)_{0}S_{k 1}(1,1)_{0}\ExB&=
-(S_{i1}(1,2)_{0}\ExB) 
,\nonumber\\
S_{k i}(1,1)_{1}S_{k 1}(1,1)_{0}\ExB&=
(S_{i1}(1,1)_{0}\ExB) 
,
\end{align}

\begin{align}
S_{k i}(1,1)_{0}S_{k 1}(1,2)_{0}\ExB&=
\frac{2}{3}S_{i1}(1,1)_{-2 } \ExB 
+\frac{-4}{3}\omega^{[1]}_{-1 } (S_{i1}(1,1)_{0}\ExB) 
\nonumber\\&\quad{}
+\frac{2}{3}\omega_{0 } S_{i1}(1,1)_{-1 } \ExB 
+2\omega_{0 } (S_{i1}(1,2)_{0}\ExB) 
,\nonumber\\
S_{k i}(1,1)_{1}S_{k 1}(1,2)_{0}\ExB&=
2(S_{i1}(1,2)_{0}\ExB) 
,
\nonumber\\
S_{k i}(1,1)_{2}S_{k 1}(1,2)_{0}\ExB&=
-2(S_{i1}(1,1)_{0}\ExB) 
,
\end{align}

\begin{align}
S_{k i}(1,2)_{0}S_{i1}(1,1)_{0}\ExB&=
\frac{-4}{3}\omega^{[1]}_{-1 } (S_{k 1}(1,1)_{0}E) 
+\frac{2}{3}\omega_{0 } S_{k 1}(1,1)_{-1 } \ExB 
\nonumber\\&\quad{}
+\frac{2}{3}S_{k 1}(1,2)_{-1 } \ExB 
+\frac{4}{3}\omega_{0 } (S_{k 1}(1,2)_{0}E) 
,\nonumber\\
S_{k i}(1,2)_{1}S_{i1}(1,1)_{0}\ExB&=
2(S_{k 1}(1,2)_{0}E) 
,
\nonumber\\
S_{k i}(1,2)_{2}S_{i1}(1,1)_{0}\ExB&=
-2(S_{k 1}(1,1)_{0}E) 
,
\end{align}

\begin{align}
S_{k i}(1,2)_{0}S_{i1}(1,2)_{0}\ExB&=
24\omega^{[1]}_{-2 } (S_{k 1}(1,1)_{0}E) 
-6\omega_{0 } S_{k 1}(1,1)_{-2 } \ExB 
\nonumber\\&\quad{}
+18S_{k 1}(1,2)_{-2 } \ExB 
-12\omega^{[1]}_{-1 } (S_{k 1}(1,2)_{0}E) 
\nonumber\\&\quad{}
-18\omega_{0 } S_{k 1}(1,2)_{-1 } \ExB 
-12S_{k 1}(1,3)_{-1 } \ExB 
,\nonumber\\
S_{k i}(1,2)_{1}S_{i1}(1,2)_{0}\ExB&=
4\omega^{[1]}_{-1 } (S_{k 1}(1,1)_{0}E) 
-2\omega_{0 } S_{k 1}(1,1)_{-1 } \ExB 
\nonumber\\&\quad{}
-2S_{k 1}(1,2)_{-1 } \ExB 
-4\omega_{0 } (S_{k 1}(1,2)_{0}E) 
,
\nonumber\\
S_{k i}(1,2)_{2}S_{i1}(1,2)_{0}\ExB&=
-6(S_{k 1}(1,2)_{0}E) 
,
\nonumber\\
S_{k i}(1,2)_{3}S_{i1}(1,2)_{0}\ExB&=
6(S_{k 1}(1,1)_{0}E) 
,
\end{align}

\begin{align}
S_{k i}(1,2)_{0}S_{k 1}(1,1)_{0}\ExB&=
\frac{-2}{3}S_{i1}(1,1)_{-2 } \ExB 
+\frac{4}{3}\omega^{[1]}_{-1 } (S_{i1}(1,1)_{0}\ExB) 
\nonumber\\&\quad{}
+\frac{-2}{3}\omega_{0 } S_{i1}(1,1)_{-1 } \ExB 
-2\omega_{0 } (S_{i1}(1,2)_{0}\ExB) 
,\nonumber\\
S_{k i}(1,2)_{1}S_{k 1}(1,1)_{0}\ExB&=
-(S_{i1}(1,2)_{0}\ExB) 
,
\end{align}

\begin{align}
S_{k i}(1,2)_{0}S_{k 1}(1,2)_{0}\ExB&=
2\omega^{[i]}_{-2 } (S_{i1}(1,1)_{0}\ExB) 
+4\omega^{[i]}_{-1 } (S_{i1}(1,2)_{0}\ExB) 
,\nonumber\\
S_{k i}(1,2)_{1}S_{k 1}(1,2)_{0}\ExB&=
\frac{4}{3}S_{i1}(1,1)_{-2 } \ExB 
+\frac{-8}{3}\omega^{[1]}_{-1 } (S_{i1}(1,1)_{0}\ExB) 
\nonumber\\&\quad{}
+\frac{4}{3}\omega_{0 } S_{i1}(1,1)_{-1 } \ExB 
+4\omega_{0 } (S_{i1}(1,2)_{0}\ExB) 
,
\nonumber\\
S_{k i}(1,2)_{2}S_{k 1}(1,2)_{0}\ExB&=
2(S_{i1}(1,2)_{0}\ExB) 
,
\end{align}

\begin{align}
S_{k i}(1,3)_{0}S_{i1}(1,1)_{0}\ExB&=
12\omega^{[1]}_{-2 } (S_{k 1}(1,1)_{0}E) 
-3\omega_{0 } S_{k 1}(1,1)_{-2 } \ExB 
\nonumber\\&\quad{}
+9S_{k 1}(1,2)_{-2 } \ExB 
-6\omega^{[1]}_{-1 } (S_{k 1}(1,2)_{0}E) 
\nonumber\\&\quad{}
-9\omega_{0 } S_{k 1}(1,2)_{-1 } \ExB 
-6S_{k 1}(1,3)_{-1 } \ExB 
,\nonumber\\
S_{k i}(1,3)_{1}S_{i1}(1,1)_{0}\ExB&=
2\omega^{[1]}_{-1 } (S_{k 1}(1,1)_{0}E) 
-\omega_{0 } S_{k 1}(1,1)_{-1 } \ExB 
\nonumber\\&\quad{}
-S_{k 1}(1,2)_{-1 } \ExB 
-2\omega_{0 } (S_{k 1}(1,2)_{0}E) 
,
\nonumber\\
S_{k i}(1,3)_{2}S_{i1}(1,1)_{0}\ExB&=
-3(S_{k 1}(1,2)_{0}E) 
,
\nonumber\\
S_{k i}(1,3)_{3}S_{i1}(1,1)_{0}\ExB&=
3(S_{k 1}(1,1)_{0}E) 
,
\end{align}

\begin{align}
S_{k i}(1,3)_{0}S_{i1}(1,2)_{0}\ExB&=
\frac{4}{9}\omega^{[k]}_{-3 } (S_{k 1}(1,1)_{0}E) 
+\frac{2}{3}\omega^{[k]}_{-1 } S_{k 1}(1,1)_{-2 } \ExB 
\nonumber\\&\quad{}
+\frac{-4}{3}\Har^{[k]}_{-1 } (S_{k 1}(1,1)_{0}E) 
+\frac{-2}{3}\omega^{[k]}_{-1 } S_{k 1}(1,2)_{-1 } \ExB 
\nonumber\\&\quad{}
+\frac{2}{3}\omega_{0 } \omega^{[k]}_{-1 } (S_{k 1}(1,2)_{0}E) 
,\nonumber\\
S_{k i}(1,3)_{1}S_{i1}(1,2)_{0}\ExB&=
-48\omega^{[1]}_{-2 } (S_{k 1}(1,1)_{0}E) 
+12\omega_{0 } S_{k 1}(1,1)_{-2 } \ExB 
\nonumber\\&\quad{}
-36S_{k 1}(1,2)_{-2 } \ExB 
+24\omega^{[1]}_{-1 } (S_{k 1}(1,2)_{0}E) 
\nonumber\\&\quad{}
+36\omega_{0 } S_{k 1}(1,2)_{-1 } \ExB 
+24S_{k 1}(1,3)_{-1 } \ExB 
,
\nonumber\\
S_{k i}(1,3)_{2}S_{i1}(1,2)_{0}\ExB&=
-8\omega^{[1]}_{-1 } (S_{k 1}(1,1)_{0}E) 
+4\omega_{0 } S_{k 1}(1,1)_{-1 } \ExB 
\nonumber\\&\quad{}
+4S_{k 1}(1,2)_{-1 } \ExB 
+8\omega_{0 } (S_{k 1}(1,2)_{0}E) 
,
\nonumber\\
S_{k i}(1,3)_{3}S_{i1}(1,2)_{0}\ExB&=
12(S_{k 1}(1,2)_{0}E) 
,
\nonumber\\
S_{k i}(1,3)_{4}S_{i1}(1,2)_{0}\ExB&=
-12(S_{k 1}(1,1)_{0}E) 
,
\end{align}

\begin{align}
S_{k i}(1,3)_{0}S_{k 1}(1,1)_{0}\ExB&=
-\omega^{[i]}_{-2 } (S_{i1}(1,1)_{0}\ExB) 
-2\omega^{[i]}_{-1 } (S_{i1}(1,2)_{0}\ExB) 
,\nonumber\\
S_{k i}(1,3)_{1}S_{k 1}(1,1)_{0}\ExB&=
\frac{-1}{3}S_{i1}(1,1)_{-2 } \ExB 
+\frac{2}{3}\omega^{[1]}_{-1 } (S_{i1}(1,1)_{0}\ExB) 
\nonumber\\&\quad{}
+\frac{-1}{3}\omega_{0 } S_{i1}(1,1)_{-1 } \ExB 
-\omega_{0 } (S_{i1}(1,2)_{0}\ExB) 
,
\end{align}

\begin{align}
S_{k i}(1,3)_{0}S_{k 1}(1,2)_{0}\ExB&=
\frac{-16}{29}\omega^{[1]}_{-1 } S_{i1}(1,1)_{-2 } \ExB 
+\frac{-16}{29}\omega^{[1]}_{-1 } \omega^{[1]}_{-1 } (S_{i1}(1,1)_{0}\ExB) 
\nonumber\\&\quad{}
+\frac{12}{29}\Har^{[1]}_{-1 } (S_{i1}(1,1)_{0}\ExB) 
+\frac{36}{29}\omega_{0 } \omega^{[i]}_{-2 } (S_{i1}(1,1)_{0}\ExB) 
\nonumber\\&\quad{}
+\frac{8}{29}\omega_{0 } \omega^{[1]}_{-1 } S_{i1}(1,1)_{-1 } \ExB 
+\frac{4}{29}\omega_{0 } \omega^{[1]}_{-2 } (S_{i1}(1,1)_{0}\ExB) 
\nonumber\\&\quad{}
+\frac{48}{29}\omega^{[1]}_{-1 } S_{i1}(1,2)_{-1 } \ExB 
+\frac{36}{29}\omega_{0 } S_{i1}(1,2)_{-2 } \ExB 
\nonumber\\&\quad{}
+\frac{72}{29}\omega_{0 } \omega^{[i]}_{-1 } (S_{i1}(1,2)_{0}\ExB) 
+\frac{-24}{29}\omega_{0 } \omega_{0 } S_{i1}(1,2)_{-1 } \ExB 
\nonumber\\&\quad{}
+\frac{-60}{29}S_{i1}(1,3)_{-2 } \ExB 
,\nonumber\\
S_{k i}(1,3)_{1}S_{k 1}(1,2)_{0}\ExB&=
2\omega^{[i]}_{-2 } (S_{i1}(1,1)_{0}\ExB) 
+4\omega^{[i]}_{-1 } (S_{i1}(1,2)_{0}\ExB) 
,
\nonumber\\
S_{k i}(1,3)_{2}S_{k 1}(1,2)_{0}\ExB&=
\frac{2}{3}S_{i1}(1,1)_{-2 } \ExB 
+\frac{-4}{3}\omega^{[1]}_{-1 } (S_{i1}(1,1)_{0}\ExB) 
\nonumber\\&\quad{}
+\frac{2}{3}\omega_{0 } S_{i1}(1,1)_{-1 } \ExB 
+2\omega_{0 } (S_{i1}(1,2)_{0}\ExB) 
,
\end{align}

\subsection{The case that $\langle \alpha,\alpha\rangle=0$}
\label{section:Norm-0}

Let $\alpha\in\fh\setminus\{0\}$ with $\langle\alpha,\alpha\rangle=0$.
Let  $h^{[1]},h^{[2]},\ldots,h^{[\rankL]}$ be an orthonormal basis of $\fh$ such that
\begin{align}
0\neq \langle\alpha,h^{[2]}\rangle=\sqrt{-1}\langle\alpha,h^{[1]}\rangle
\mbox{ and }\langle\alpha,h^{[i]}\rangle=0
\end{align}
for all $i=3,\ldots,\rankL$. We write
\begin{align}
\zone&=\langle\alpha,h^{[1]}\rangle
\end{align}
for simplicity. By Lemma \ref{lemma:infinite-set}, we may assume all the denominators in the following formulas
are non-zero.
In the following computation, $k,l$ are distinct elements of $\{3,\ldots,\rankL\}$.


\begin{align}
\omega^{[1]}_{0}\omega^{[1]}_{0}\ExB&=
-2 \zone^4 (2 \zone^2-1)\omega^{[2]}_{-1 } E
+2 \zone^4 (2 \zone^2+1)\omega^{[1]}_{-1 } E\nonumber\\&\quad{}
-\zone^2 (2 \zone^2-1)\omega_{0 } \omega_{0 } E
-(2 \zone^2-1)\omega_{0 } (\omega^{[1]}_{0}\ExB)\nonumber\\&\quad{}
+\zone^2 (2 \zone^2-1) (2 \zone^2+1)\sqrt{-1}S_{21}(1,1)_{-1 } E,\nonumber\\
\omega^{[1]}_{1}\omega^{[1]}_{0}\ExB&=
\frac{\zone^2+2}{2}(\omega^{[1]}_{0}\ExB),
\nonumber\\
\omega^{[1]}_{2}\omega^{[1]}_{0}\ExB&=
\zone^2E,
\end{align}

\begin{align}
\omega^{[1]}_{0}S_{k 1}(1,1)_{0}\ExB&=
-\zone^2\omega_{0 } (S_{k 1}(1,1)_{0}E)
+\zone^2 (\zone^2+1)S_{k 1}(1,1)_{-1 } E\nonumber\\&\quad{}
+\zone^4\sqrt{-1}S_{k 2}(1,1)_{-1 } E,\nonumber\\
\omega^{[1]}_{1}S_{k 1}(1,1)_{0}\ExB&=
\frac{\zone^2}{2}(S_{k 1}(1,1)_{0}E),
\end{align}

\begin{align}
\Har^{[1]}_{0}\ExB&=
\frac{-\zone^2 (6 \zone^2-1)}{7 \zone^2+1}\omega^{[1]}_{-2 } E
+\frac{16 \zone^6}{7 \zone^2+1}\omega_{0 } \omega^{[2]}_{-1 } E\nonumber\\&\quad{}
+\frac{-4 \zone^2 (2 \zone^2+1)^2}{7 \zone^2+1}\omega_{0 } \omega^{[1]}_{-1 } E
+\frac{8 \zone^4}{7 \zone^2+1}\omega_{0 } \omega_{0 } \omega_{0 } E\nonumber\\&\quad{}
+\frac{2 \zone^2}{7 \zone^2+1}\omega^{[2]}_{-1 } (\omega^{[1]}_{0}\ExB)
+\frac{2 \zone^2}{7 \zone^2+1}\omega^{[1]}_{-1 } (\omega^{[1]}_{0}\ExB)\nonumber\\&\quad{}
+\frac{8 \zone^2+1}{7 \zone^2+1}\omega_{0 } \omega_{0 } (\omega^{[1]}_{0}\ExB)\nonumber\\&\quad{}
+\frac{-\zone^2 (2 \zone^2+1)}{7 \zone^2+1}\sqrt{-1}S_{21}(1,1)_{-2 } E\nonumber\\&\quad{}
+\frac{-8 \zone^4 (2 \zone^2+1)}{7 \zone^2+1}\sqrt{-1}\omega_{0 } S_{21}(1,1)_{-1 } E\nonumber\\&\quad{}
+\frac{-\zone^2 (2 \zone^2+1)}{7 \zone^2+1}\sqrt{-1}S_{21}(1,2)_{-1 } E,\nonumber\\
\Har^{[1]}_{1}\ExB&=
2 \zone^4\omega^{[2]}_{-1 } E
-2 \zone^2 (\zone^2+1)\omega^{[1]}_{-1 } E\nonumber\\&\quad{}
+\zone^2\omega_{0 } \omega_{0 } E
+\omega_{0 } (\omega^{[1]}_{0}\ExB)\nonumber\\&\quad{}
-\zone^2 (2 \zone^2+1)\sqrt{-1}S_{21}(1,1)_{-1 } E,
\nonumber\\
\Har^{[1]}_{2}\ExB&=
\frac{1}{3}(\omega^{[1]}_{0}\ExB),
\end{align}

\begin{align}
\Har^{[1]}_{0}\omega^{[1]}_{0}\ExB&=
\frac{-8 \zone^4 (4 \zone^2-3)}{16 \zone^2+3}\omega^{[2]}_{-1 } \omega^{[1]}_{-1 } E\nonumber\\&\quad{}
+\frac{-8 \zone^2 (2 \zone^2-9)}{48 \zone^2+9}\omega^{[1]}_{-3 } E\nonumber\\&\quad{}
+\frac{8 \zone^2 (\zone^2+1) (4 \zone^2-3)}{16 \zone^2+3}\omega^{[1]}_{-1 } \omega^{[1]}_{-1 } E\nonumber\\&\quad{}
+\frac{-4 \zone^2 (2 \zone^2-9)}{16 \zone^2+3}\Har^{[1]}_{-1 } E\nonumber\\&\quad{}
+\frac{\zone^2 (4 \zone^2-3) (7 \zone^2+1)}{16 \zone^2+3}\omega_{0 } \omega^{[2]}_{-2 } E\nonumber\\&\quad{}
+\frac{-(4 \zone^2-3) (7 \zone^4-9 \zone^2-4)}{16 \zone^2+3}\omega_{0 } \omega^{[1]}_{-2 } E\nonumber\\&\quad{}
+\frac{2 \zone^2 (3 \zone-1) (3 \zone+1) (4 \zone^2-3)}{16 \zone^2+3}\omega_{0 } \omega_{0 } \omega^{[2]}_{-1 } E\nonumber\\&\quad{}
+\frac{-2 (4 \zone^2-3) (9 \zone^4+3 \zone^2+2)}{16 \zone^2+3}\omega_{0 } \omega_{0 } \omega^{[1]}_{-1 } E\nonumber\\&\quad{}
+\frac{(3 \zone-1) (3 \zone+1) (4 \zone^2-3)}{16 \zone^2+3}\omega_{0 } \omega_{0 } \omega_{0 } \omega_{0 } E\nonumber\\&\quad{}
+\frac{6 (4 \zone^2-3)}{16 \zone^2+3}\omega^{[1]}_{-2 } (\omega^{[1]}_{0}\ExB)\nonumber\\&\quad{}
+\frac{4 (4 \zone^2-3)}{16 \zone^2+3}\omega_{0 } \omega^{[2]}_{-1 } (\omega^{[1]}_{0}\ExB)\nonumber\\&\quad{}
+\frac{2 (4 \zone^2-3)}{16 \zone^2+3}\omega_{0 } \omega_{0 } \omega_{0 } (\omega^{[1]}_{0}\ExB)\nonumber\\&\quad{}
+\frac{4 \zone^2 (2 \zone^2+1) (4 \zone^2-3)}{16 \zone^2+3}\sqrt{-1}\omega^{[1]}_{-1 } S_{21}(1,1)_{-1 } E\nonumber\\&\quad{}
+\frac{-(4 \zone^2-3) (7 \zone^4+5 \zone^2+2)}{16 \zone^2+3}\sqrt{-1}\omega_{0 } S_{21}(1,1)_{-2 } E\nonumber\\&\quad{}
+\frac{-2 (3 \zone^2-1) (3 \zone^2+1) (4 \zone^2-3)}{16 \zone^2+3}\sqrt{-1}\omega_{0 } \omega_{0 } S_{21}(1,1)_{-1 } E\nonumber\\&\quad{}
+\frac{10 \zone^2 (4 \zone^2-3)}{16 \zone^2+3}\sqrt{-1}\omega_{0 } S_{21}(1,2)_{-1 } E,\end{align}

\begin{align}
\Har^{[1]}_{1}\omega^{[1]}_{0}\ExB&=
\frac{\zone^2 (6 \zone^4-11 \zone^2+6)}{14 \zone^2+2}\omega^{[1]}_{-2 } E\nonumber\\&\quad{}
+\frac{-8 \zone^6 (\zone-2) (\zone+2)}{7 \zone^2+1}\omega_{0 } \omega^{[2]}_{-1 } E\nonumber\\&\quad{}
+\frac{2 \zone^2 (\zone-2) (\zone+2) (2 \zone^2+1)^2}{7 \zone^2+1}\omega_{0 } \omega^{[1]}_{-1 } E\nonumber\\&\quad{}
+\frac{-4 \zone^4 (\zone-2) (\zone+2)}{7 \zone^2+1}\omega_{0 } \omega_{0 } \omega_{0 } E\nonumber\\&\quad{}
+\frac{-\zone^2 (\zone-2) (\zone+2)}{7 \zone^2+1}\omega^{[2]}_{-1 } (\omega^{[1]}_{0}\ExB)\nonumber\\&\quad{}
+\frac{-\zone^2 (\zone-2) (\zone+2)}{7 \zone^2+1}\omega^{[1]}_{-1 } (\omega^{[1]}_{0}\ExB)\nonumber\\&\quad{}
+\frac{-(\zone-2) (\zone+2) (8 \zone^2+1)}{14 \zone^2+2}\omega_{0 } \omega_{0 } (\omega^{[1]}_{0}\ExB)\nonumber\\&\quad{}
+\frac{\zone^2 (\zone-2) (\zone+2) (2 \zone^2+1)}{14 \zone^2+2}\sqrt{-1}S_{21}(1,1)_{-2 } E\nonumber\\&\quad{}
+\frac{4 \zone^4 (\zone-2) (\zone+2) (2 \zone^2+1)}{7 \zone^2+1}\sqrt{-1}\omega_{0 } S_{21}(1,1)_{-1 } E\nonumber\\&\quad{}
+\frac{\zone^2 (\zone-2) (\zone+2) (2 \zone^2+1)}{14 \zone^2+2}\sqrt{-1}S_{21}(1,2)_{-1 } E,
\nonumber\\
\Har^{[1]}_{2}\omega^{[1]}_{0}\ExB&=
\frac{-2 \zone^4 (2 \zone^2-7)}{3}\omega^{[2]}_{-1 } E\nonumber\\&\quad{}
+\frac{2 \zone^2 (2 \zone^4-5 \zone^2-6)}{3}\omega^{[1]}_{-1 } E\nonumber\\&\quad{}
+\frac{-\zone^2 (2 \zone^2-7)}{3}\omega_{0 } \omega_{0 } E\nonumber\\&\quad{}
+\frac{-(2 \zone^2-7)}{3}\omega_{0 } (\omega^{[1]}_{0}\ExB)\nonumber\\&\quad{}
+\frac{\zone^2 (2 \zone^2-7) (2 \zone^2+1)}{3}\sqrt{-1}S_{21}(1,1)_{-1 } E,
\nonumber\\
\Har^{[1]}_{3}\omega^{[1]}_{0}\ExB&=
(\omega^{[1]}_{0}\ExB),
\end{align}

\begin{align}
&\Har^{[1]}_{0}S_{k 1}(1,1)_{0}\ExB\nonumber\\&=
\frac{2 (\zone-1) (\zone+1) (4 \zone^4-10 \zone^2-3)}{4 \zone^4+6 \zone^2+1}\omega^{[k]}_{-2 } (S_{k 1}(1,1)_{0}E)\nonumber\\&\quad{}
+\frac{2 \zone^2 (2 \zone^2+1)}{4 \zone^4+6 \zone^2+1}\omega^{[2]}_{-2 } (S_{k 1}(1,1)_{0}E)
+\frac{2 \zone^2 (2 \zone^2+1)}{4 \zone^4+6 \zone^2+1}\omega^{[1]}_{-2 } (S_{k 1}(1,1)_{0}E)\nonumber\\&\quad{}
+\frac{-4 (\zone-1) (\zone+1) (4 \zone^4-10 \zone^2-3)}{12 \zone^4+18 \zone^2+3}\omega_{0 } \omega^{[k]}_{-1 } (S_{k 1}(1,1)_{0}E)\nonumber\\&\quad{}
+\frac{-2 \zone^2 (12 \zone^4-2 \zone^2-1)}{4 \zone^6+2 \zone^4-5 \zone^2-1}\omega_{0 } \omega^{[2]}_{-1 } (S_{k 1}(1,1)_{0}E)
+\frac{-2 \zone^2 (12 \zone^4-2 \zone^2-1)}{4 \zone^6+2 \zone^4-5 \zone^2-1}\omega_{0 } \omega^{[1]}_{-1 } (S_{k 1}(1,1)_{0}E)\nonumber\\&\quad{}
+\frac{-(2 \zone^2-1) (12 \zone^4-2 \zone^2-1)}{4 \zone^6+2 \zone^4-5 \zone^2-1}\omega_{0 } \omega_{0 } \omega_{0 } (S_{k 1}(1,1)_{0}E)\nonumber\\&\quad{}
+\frac{4 \zone^2 (\zone-1) (\zone+1) (4 \zone^4-10 \zone^2-3)}{12 \zone^4+18 \zone^2+3}\omega^{[k]}_{-1 } S_{k 1}(1,1)_{-1 } E\nonumber\\&\quad{}
+\frac{2 \zone^2 (\zone^2-3) (2 \zone^2+1)}{4 \zone^4+6 \zone^2+1}S_{k 1}(1,1)_{-3 } E\nonumber\\&\quad{}
+\frac{-4 \zone^4}{4 \zone^4+6 \zone^2+1}\omega^{[2]}_{-1 } S_{k 1}(1,1)_{-1 } E
+\frac{4 \zone^4}{4 \zone^4+6 \zone^2+1}\omega^{[1]}_{-1 } S_{k 1}(1,1)_{-1 } E\nonumber\\&\quad{}
+\frac{-2 \zone^2 (10 \zone^4-2 \zone^2+1)}{4 \zone^6+2 \zone^4-5 \zone^2-1}\omega_{0 } S_{k 1}(1,1)_{-2 } E\nonumber\\&\quad{}
+\frac{2 \zone^2 (2 \zone^2+1) (6 \zone^4-4 \zone^2+1)}{4 \zone^6+2 \zone^4-5 \zone^2-1}\omega_{0 } \omega_{0 } S_{k 1}(1,1)_{-1 } E\nonumber\\&\quad{}
+\frac{2 \zone^2 (8 \zone^2+1)}{4 \zone^4+6 \zone^2+1}S_{k 1}(1,2)_{-2 } E
+\frac{\zone^2 (24 \zone^6-4 \zone^4+8 \zone^2-1)}{4 \zone^6+2 \zone^4-5 \zone^2-1}\omega_{0 } S_{k 1}(1,2)_{-1 } E\nonumber\\&\quad{}
+\frac{4 \zone^2 (\zone-1) (\zone+1) (4 \zone^4-10 \zone^2-3)}{12 \zone^4+18 \zone^2+3}\sqrt{-1}\omega^{[k]}_{-1 } S_{k 2}(1,1)_{-1 } E\nonumber\\&\quad{}
+\frac{2 \zone^2 (\zone-1) (\zone+1) (2 \zone^2+1)}{4 \zone^4+6 \zone^2+1}\sqrt{-1}S_{k 2}(1,1)_{-3 } E\nonumber\\&\quad{}
+\frac{8 \zone^4}{4 \zone^4+6 \zone^2+1}\sqrt{-1}\omega^{[1]}_{-1 } S_{k 2}(1,1)_{-1 } E
+\frac{2 \zone^2 (12 \zone^4-2 \zone^2-1)}{4 \zone^4+6 \zone^2+1}\sqrt{-1}\omega_{0 } \omega_{0 } S_{k 2}(1,1)_{-1 } E\nonumber\\&\quad{}
+\frac{3 \zone^2 (2 \zone^2-1) (2 \zone^2+1)^2}{4 \zone^6+2 \zone^4-5 \zone^2-1}\sqrt{-1}\omega_{0 } S_{k 2}(1,2)_{-1 } E\nonumber\\&\quad{}
+\frac{-2 \zone^2 (2 \zone^2+1)^2}{4 \zone^4+6 \zone^2+1}\sqrt{-1}S_{k 2}(1,3)_{-1 } E,\end{align}

\begin{align}
\Har^{[1]}_{1}S_{k 1}(1,1)_{0}\ExB&=
\frac{\zone^2+2}{3}\omega^{[2]}_{-1 } (S_{k 1}(1,1)_{0}E)
+\frac{\zone^2+2}{3}\omega^{[1]}_{-1 } (S_{k 1}(1,1)_{0}E)\nonumber\\&\quad{}
+\frac{2 \zone^2+1}{3}\omega_{0 } \omega_{0 } (S_{k 1}(1,1)_{0}E)
+\frac{-(\zone^2-2) (2 \zone^2+1)}{6}S_{k 1}(1,1)_{-2 } E\nonumber\\&\quad{}
+\frac{-(4 \zone^4+3 \zone^2+2)}{6}\omega_{0 } S_{k 1}(1,1)_{-1 } E
+\frac{-(2 \zone^4-\zone^2+2)}{6}\sqrt{-1}S_{k 2}(1,1)_{-2 } E\nonumber\\&\quad{}
+\frac{-(4 \zone^4+\zone^2-2)}{6}\sqrt{-1}\omega_{0 } S_{k 2}(1,1)_{-1 } E
-\zone^2\sqrt{-1}S_{k 2}(1,2)_{-1 } E,
\nonumber\\
\Har^{[1]}_{2}S_{k 1}(1,1)_{0}\ExB&=
\frac{-\zone^2}{3}\omega_{0 } (S_{k 1}(1,1)_{0}E)
+\frac{\zone^2 (\zone^2+1)}{3}S_{k 1}(1,1)_{-1 } E\nonumber\\&\quad{}
+\frac{\zone^4}{3}\sqrt{-1}S_{k 2}(1,1)_{-1 } E,
\end{align}

\begin{align}
\omega^{[2]}_{0}\ExB&=
\omega_{0 } E
-(\omega^{[1]}_{0}\ExB),&
\omega^{[2]}_{1}\ExB&=
\frac{-\zone^2}{2}E,
\end{align}

\begin{align}
\omega^{[2]}_{0}\omega^{[1]}_{0}\ExB&=
2 \zone^4 (2 \zone^2-1)\omega^{[2]}_{-1 } E
-2 \zone^4 (2 \zone^2+1)\omega^{[1]}_{-1 } E\nonumber\\&\quad{}
+\zone^2 (2 \zone^2-1)\omega_{0 } \omega_{0 } E
+2 \zone^2\omega_{0 } (\omega^{[1]}_{0}\ExB)\nonumber\\&\quad{}
-\zone^2 (2 \zone^2-1) (2 \zone^2+1)\sqrt{-1}S_{21}(1,1)_{-1 } E,\nonumber\\
\omega^{[2]}_{1}\omega^{[1]}_{0}\ExB&=
\frac{-\zone^2}{2}(\omega^{[1]}_{0}\ExB),
\end{align}

\begin{align}
\omega^{[2]}_{0}S_{k 1}(1,1)_{0}\ExB&=
\zone^2\omega_{0 } (S_{k 1}(1,1)_{0}E)
-\zone^4S_{k 1}(1,1)_{-1 } E\nonumber\\&\quad{}
-\zone^2 (\zone-1) (\zone+1)\sqrt{-1}S_{k 2}(1,1)_{-1 } E,\nonumber\\
\omega^{[2]}_{1}S_{k 1}(1,1)_{0}\ExB&=
\frac{-\zone^2}{2}(S_{k 1}(1,1)_{0}E),
\end{align}

\begin{align}
\Har^{[2]}_{0}\ExB&=
\frac{6 \zone^4-3 \zone^2-4}{7 \zone^2+1}\omega^{[1]}_{-2 } E
+\frac{-4 \zone^2 (4 \zone^4+\zone^2-1)}{7 \zone^2+1}\omega_{0 } \omega^{[2]}_{-1 } E\nonumber\\&\quad{}
+\frac{4 (4 \zone^6+5 \zone^4+\zone^2+1)}{7 \zone^2+1}\omega_{0 } \omega^{[1]}_{-1 } E
+\frac{-2 (4 \zone^4+\zone^2-1)}{7 \zone^2+1}\omega_{0 } \omega_{0 } \omega_{0 } E\nonumber\\&\quad{}
+\frac{-2 (\zone^2+2)}{7 \zone^2+1}\omega^{[2]}_{-1 } (\omega^{[1]}_{0}\ExB)
+\frac{-2 (\zone^2+2)}{7 \zone^2+1}\omega^{[1]}_{-1 } (\omega^{[1]}_{0}\ExB)\nonumber\\&\quad{}
+\frac{-(8 \zone^2+3)}{7 \zone^2+1}\omega_{0 } \omega_{0 } (\omega^{[1]}_{0}\ExB)
+\frac{(\zone^2+2) (2 \zone^2+1)}{7 \zone^2+1}\sqrt{-1}S_{21}(1,1)_{-2 } E\nonumber\\&\quad{}
+\frac{2 (2 \zone^2+1) (4 \zone^4+\zone^2-1)}{7 \zone^2+1}\sqrt{-1}\omega_{0 } S_{21}(1,1)_{-1 } E\nonumber\\&\quad{}
+\frac{\zone^2 (2 \zone^2-9)}{7 \zone^2+1}\sqrt{-1}S_{21}(1,2)_{-1 } E,\nonumber\\
\Har^{[2]}_{1}\ExB&=
-2 \zone^2 (\zone-1) (\zone+1)\omega^{[2]}_{-1 } E
+2 \zone^4\omega^{[1]}_{-1 } E\nonumber\\&\quad{}
-(\zone-1) (\zone+1)\omega_{0 } \omega_{0 } E
-\omega_{0 } (\omega^{[1]}_{0}\ExB)\nonumber\\&\quad{}
+\zone^2 (2 \zone^2-1)\sqrt{-1}S_{21}(1,1)_{-1 } E,
\nonumber\\
\Har^{[2]}_{2}\ExB&=
\frac{1}{3}\omega_{0 } E+\frac{-1}{3}(\omega^{[1]}_{0}\ExB),
\end{align}

\begin{align}
\Har^{[2]}_{0}\omega^{[1]}_{0}\ExB&=
\frac{8 \zone^4 (12 \zone^4-11 \zone^2-3)}{48 \zone^4+25 \zone^2+3}\omega^{[2]}_{-1 } \omega^{[1]}_{-1 } E\nonumber\\&\quad{}
+\frac{8 \zone^2 (38 \zone^4-34 \zone^2-9)}{144 \zone^4+75 \zone^2+9}\omega^{[1]}_{-3 } E\nonumber\\&\quad{}
+\frac{-8 \zone^4 (12 \zone^4+\zone^2+2)}{48 \zone^4+25 \zone^2+3}\omega^{[1]}_{-1 } \omega^{[1]}_{-1 } E\nonumber\\&\quad{}
+\frac{-20 \zone^4 (2 \zone^2-1)}{48 \zone^4+25 \zone^2+3}\Har^{[1]}_{-1 } E\nonumber\\&\quad{}
+\frac{-\zone^4 (184 \zone^4+161 \zone^2+22)}{96 \zone^4+50 \zone^2+6}\omega_{0 } \omega^{[2]}_{-2 } E\nonumber\\&\quad{}
+\frac{\zone^2 (184 \zone^6-111 \zone^4+76 \zone^2+8)}{96 \zone^4+50 \zone^2+6}\omega_{0 } \omega^{[1]}_{-2 } E\nonumber\\&\quad{}
+\frac{-\zone^4 (200 \zone^4+127 \zone^2+2)}{48 \zone^4+25 \zone^2+3}\omega_{0 } \omega_{0 } \omega^{[2]}_{-1 } E\nonumber\\&\quad{}
+\frac{\zone^2 (200 \zone^6+191 \zone^4+20 \zone^2+8)}{48 \zone^4+25 \zone^2+3}\omega_{0 } \omega_{0 } \omega^{[1]}_{-1 } E\nonumber\\&\quad{}
+\frac{-\zone^2 (200 \zone^4+127 \zone^2+2)}{96 \zone^4+50 \zone^2+6}\omega_{0 } \omega_{0 } \omega_{0 } \omega_{0 } E\nonumber\\&\quad{}
+\frac{2 \zone^2}{3 \zone^2+1}\omega^{[2]}_{-2 } (\omega^{[1]}_{0}\ExB)\nonumber\\&\quad{}
+\frac{-8 \zone^2 (5 \zone^2+3)}{48 \zone^4+25 \zone^2+3}\omega^{[1]}_{-2 } (\omega^{[1]}_{0}\ExB)\nonumber\\&\quad{}
+\frac{-4 \zone^2 (12 \zone^2+5)}{48 \zone^4+25 \zone^2+3}\omega_{0 } \omega^{[2]}_{-1 } (\omega^{[1]}_{0}\ExB)\nonumber\\&\quad{}
+\frac{-\zone^2 (8 \zone^2+7)}{48 \zone^4+25 \zone^2+3}\omega_{0 } \omega_{0 } \omega_{0 } (\omega^{[1]}_{0}\ExB)\nonumber\\&\quad{}
+\frac{2 \zone^2 (2 \zone^2-1)}{3 \zone^2+1}\sqrt{-1}S_{21}(1,1)_{-3 } E\nonumber\\&\quad{}
+\frac{-4 \zone^4 (2 \zone^2-1) (12 \zone^2+1)}{48 \zone^4+25 \zone^2+3}\sqrt{-1}\omega^{[1]}_{-1 } S_{21}(1,1)_{-1 } E\nonumber\\&\quad{}
+\frac{\zone^2 (184 \zone^6+257 \zone^4+174 \zone^2+32)}{96 \zone^4+50 \zone^2+6}\sqrt{-1}\omega_{0 } S_{21}(1,1)_{-2 } E\nonumber\\&\quad{}
+\frac{5 \zone^2 (40 \zone^6+27 \zone^4-3 \zone^2-2)}{48 \zone^4+25 \zone^2+3}\sqrt{-1}\omega_{0 } \omega_{0 } S_{21}(1,1)_{-1 } E\nonumber\\&\quad{}
+\frac{-\zone^2 (136 \zone^4+85 \zone^2+6)}{48 \zone^4+25 \zone^2+3}\sqrt{-1}\omega_{0 } S_{21}(1,2)_{-1 } E\nonumber\\&\quad{}
+\frac{-2 \zone^2 (2 \zone^2-1)}{3 \zone^2+1}\sqrt{-1}S_{21}(1,3)_{-1 } E,\end{align}

\begin{align}
\Har^{[2]}_{1}\omega^{[1]}_{0}\ExB&=
\frac{-\zone^2 (6 \zone^4-5 \zone^2-8)}{14 \zone^2+2}\omega^{[1]}_{-2 } E\
+\frac{4 \zone^4 (\zone^2+1) (2 \zone^2-1)}{7 \zone^2+1}\omega_{0 } \omega^{[2]}_{-1 } E\nonumber\\&\quad{}
+\frac{-2 \zone^2 (4 \zone^6+6 \zone^4+\zone^2+2)}{7 \zone^2+1}\omega_{0 } \omega^{[1]}_{-1 } E
+\frac{2 \zone^2 (\zone^2+1) (2 \zone^2-1)}{7 \zone^2+1}\omega_{0 } \omega_{0 } \omega_{0 } E\nonumber\\&\quad{}
+\frac{\zone^2 (\zone^2+4)}{7 \zone^2+1}\omega^{[2]}_{-1 } (\omega^{[1]}_{0}\ExB)
+\frac{\zone^2 (\zone^2+4)}{7 \zone^2+1}\omega^{[1]}_{-1 } (\omega^{[1]}_{0}\ExB)\nonumber\\&\quad{}
+\frac{\zone^2 (8 \zone^2+5)}{14 \zone^2+2}\omega_{0 } \omega_{0 } (\omega^{[1]}_{0}\ExB)\nonumber\\&\quad{}
+\frac{-\zone^2 (\zone^2-2) (2 \zone^2-1)}{14 \zone^2+2}\sqrt{-1}S_{21}(1,1)_{-2 } E\nonumber\\&\quad{}
+\frac{-2 \zone^2 (\zone^2+1) (2 \zone^2-1) (2 \zone^2+1)}{7 \zone^2+1}\sqrt{-1}\omega_{0 } S_{21}(1,1)_{-1 } E\nonumber\\&\quad{}
+\frac{-\zone^2 (\zone^2-2) (2 \zone^2-1)}{14 \zone^2+2}\sqrt{-1}S_{21}(1,2)_{-1 } E,
\nonumber\\
\Har^{[2]}_{2}\omega^{[1]}_{0}\ExB&=
\frac{2 \zone^4 (2 \zone^2-1)}{3}\omega^{[2]}_{-1 } E
+\frac{-2 \zone^4 (2 \zone^2+1)}{3}\omega^{[1]}_{-1 } E\nonumber\\&\quad{}
+\frac{\zone^2 (2 \zone^2-1)}{3}\omega_{0 } \omega_{0 } E
+\frac{2 \zone^2}{3}\omega_{0 } (\omega^{[1]}_{0}\ExB)\nonumber\\&\quad{}
+\frac{-\zone^2 (2 \zone^2-1) (2 \zone^2+1)}{3}\sqrt{-1}S_{21}(1,1)_{-1 } E,
\end{align}

\begin{align}
\Har^{[2]}_{0}S_{k 1}(1,1)_{0}\ExB&=
-2 \zone^2\omega^{[k]}_{-2 } (S_{k 1}(1,1)_{0}E)
+\frac{4 \zone^2}{3}\omega_{0 } \omega^{[k]}_{-1 } (S_{k 1}(1,1)_{0}E)\nonumber\\&\quad{}
+\frac{-4 \zone^4}{3}\omega^{[k]}_{-1 } S_{k 1}(1,1)_{-1 } E
+\frac{-4 \zone^4}{3}\sqrt{-1}\omega^{[k]}_{-1 } S_{k 2}(1,1)_{-1 } E\nonumber\\&\quad{}
+2 \zone^2\sqrt{-1}S_{k 2}(1,3)_{-1 } E,\end{align}

\begin{align}
\Har^{[2]}_{1}S_{k 1}(1,1)_{0}\ExB&=
\frac{-\zone^2}{3}\omega^{[2]}_{-1 } (S_{k 1}(1,1)_{0}E)
+\frac{-\zone^2}{3}\omega^{[1]}_{-1 } (S_{k 1}(1,1)_{0}E)\nonumber\\&\quad{}
+\frac{-2 \zone^2}{3}\omega_{0 } \omega_{0 } (S_{k 1}(1,1)_{0}E)
+\frac{\zone^2 (2 \zone^2-1)}{6}S_{k 1}(1,1)_{-2 } E\nonumber\\&\quad{}
+\frac{\zone^2 (4 \zone^2+1)}{6}\omega_{0 } S_{k 1}(1,1)_{-1 } E
+\frac{\zone^2 (2 \zone^2+1)}{6}\sqrt{-1}S_{k 2}(1,1)_{-2 } E\nonumber\\&\quad{}
+\frac{\zone^2 (2 \zone-1) (2 \zone+1)}{6}\sqrt{-1}\omega_{0 } S_{k 2}(1,1)_{-1 } E\nonumber\\&\quad{}
+\zone^2\sqrt{-1}S_{k 2}(1,2)_{-1 } E,
\nonumber\\
\Har^{[2]}_{2}S_{k 1}(1,1)_{0}\ExB&=
\frac{\zone^2}{3}\omega_{0 } (S_{k 1}(1,1)_{0}E)
+\frac{-\zone^4}{3}S_{k 1}(1,1)_{-1 } E\nonumber\\&\quad{}
+\frac{-\zone^2 (\zone-1) (\zone+1)}{3}\sqrt{-1}S_{k 2}(1,1)_{-1 } E,
\end{align}

\begin{align}
\label{eq:S21(11)0ExB=-sqrt-1omega0E+2sqrt-1(omega[1]0ExB)}
S_{21}(1,1)_{0}\ExB&=
-\sqrt{-1}\omega_{0 } E
+2\sqrt{-1}(\omega^{[1]}_{0}\ExB),\nonumber\\
S_{21}(1,1)_{1}\ExB&=
\zone^2\sqrt{-1}E,
\end{align}

\begin{align}
S_{21}(1,1)_{0}\omega^{[1]}_{0}\ExB&=
-2 \zone^2 (4 \zone^4-2 \zone^2+1)\sqrt{-1}\omega^{[2]}_{-1 } E\nonumber\\&\quad{}
+2 \zone^2 (4 \zone^4+2 \zone^2+1)\sqrt{-1}\omega^{[1]}_{-1 } E\nonumber\\&\quad{}
-(4 \zone^4-2 \zone^2+1)\sqrt{-1}\omega_{0 } \omega_{0 } E\nonumber\\&\quad{}
-(2 \zone-1) (2 \zone+1)\sqrt{-1}\omega_{0 } (\omega^{[1]}_{0}\ExB)\nonumber\\&\quad{}
-8 \zone^6S_{21}(1,1)_{-1 } E,\nonumber\\
S_{21}(1,1)_{1}\omega^{[1]}_{0}\ExB&=
-\sqrt{-1}\omega_{0 } E
+(\zone^2+1)\sqrt{-1}(\omega^{[1]}_{0}\ExB),
\nonumber\\
S_{21}(1,1)_{2}\omega^{[1]}_{0}\ExB&=
\zone^2\sqrt{-1}E,
\end{align}

\begin{align}
S_{21}(1,1)_{0}S_{k 1}(1,1)_{0}\ExB&=
-2 \zone^2\sqrt{-1}\omega_{0 } (S_{k 1}(1,1)_{0}E)\nonumber\\&\quad{}
+\zone^2 (2 \zone^2+1)\sqrt{-1}S_{k 1}(1,1)_{-1 } E\nonumber\\&\quad{}
-\zone^2 (2 \zone^2-1)S_{k 2}(1,1)_{-1 } E,\nonumber\\
S_{21}(1,1)_{1}S_{k 1}(1,1)_{0}\ExB&=
\zone^2\sqrt{-1}(S_{k 1}(1,1)_{0}E),
\end{align}

\begin{align}
S_{21}(1,2)_{0}\ExB&=
2 \zone^2\sqrt{-1}\omega^{[2]}_{-1 } E
-2 \zone^2\sqrt{-1}\omega^{[1]}_{-1 } E
+\sqrt{-1}\omega_{0 } \omega_{0 } E\nonumber\\&\quad{}
+2 \zone^2S_{21}(1,1)_{-1 } E,\nonumber\\
S_{21}(1,2)_{1}\ExB&=
\sqrt{-1}\omega_{0 } E-\sqrt{-1}(\omega^{[1]}_{0}\ExB),
\nonumber\\
S_{21}(1,2)_{2}\ExB&=
-\zone^2\sqrt{-1}E,
\end{align}

\begin{align}
S_{21}(1,2)_{0}\omega^{[1]}_{0}\ExB&=
\frac{15 \zone^4+2 \zone^2-4}{7 \zone^2+1}\sqrt{-1}\omega^{[1]}_{-2 } E\nonumber\\&\quad{}
+\frac{-4 \zone^2 (\zone^2+1) (3 \zone^2-1)}{7 \zone^2+1}\sqrt{-1}\omega_{0 } \omega^{[2]}_{-1 } E\nonumber\\&\quad{}
+\frac{4 (3 \zone^6+5 \zone^4+1)}{7 \zone^2+1}\sqrt{-1}\omega_{0 } \omega^{[1]}_{-1 } E\nonumber\\&\quad{}
+\frac{-2 (\zone^2+1) (3 \zone^2-1)}{7 \zone^2+1}\sqrt{-1}\omega_{0 } \omega_{0 } \omega_{0 } E\nonumber\\&\quad{}
+\frac{2 (\zone^2-2)}{7 \zone^2+1}\sqrt{-1}\omega^{[2]}_{-1 } (\omega^{[1]}_{0}\ExB)\nonumber\\&\quad{}
+\frac{2 (\zone^2-2)}{7 \zone^2+1}\sqrt{-1}\omega^{[1]}_{-1 } (\omega^{[1]}_{0}\ExB)\nonumber\\&\quad{}
+\frac{-3 (2 \zone^2+1)}{7 \zone^2+1}\sqrt{-1}\omega_{0 } \omega_{0 } (\omega^{[1]}_{0}\ExB)\nonumber\\&\quad{}
+\frac{-(5 \zone^4+4 \zone^2+2)}{7 \zone^2+1}S_{21}(1,1)_{-2 } E\nonumber\\&\quad{}
+\frac{-2 (\zone^2+1) (2 \zone^2+1) (3 \zone^2-1)}{7 \zone^2+1}\omega_{0 } S_{21}(1,1)_{-1 } E\nonumber\\&\quad{}
+\frac{-5 \zone^2 (\zone^2-2)}{7 \zone^2+1}S_{21}(1,2)_{-1 } E,\end{align}

\begin{align}
S_{21}(1,2)_{1}\omega^{[1]}_{0}\ExB&=
2 \zone^2 (2 \zone^4-3 \zone^2+2)\sqrt{-1}\omega^{[2]}_{-1 } E
-2 \zone^4 (2 \zone^2-1)\sqrt{-1}\omega^{[1]}_{-1 } E\nonumber\\&\quad{}
+2 \zone^4-3 \zone^2+2\sqrt{-1}\omega_{0 } \omega_{0 } E
+2 (\zone-1) (\zone+1)\sqrt{-1}\omega_{0 } (\omega^{[1]}_{0}\ExB)\nonumber\\&\quad{}
+\zone^2 (2 \zone^2-1)^2S_{21}(1,1)_{-1 } E,
\nonumber\\
S_{21}(1,2)_{2}\omega^{[1]}_{0}\ExB&=
2\sqrt{-1}\omega_{0 } E
-(\zone^2+2)\sqrt{-1}(\omega^{[1]}_{0}\ExB),
\nonumber\\
S_{21}(1,2)_{3}\omega^{[1]}_{0}\ExB&=
-2 \zone^2\sqrt{-1}E,
\end{align}

\begin{align}
S_{21}(1,2)_{0}S_{k 1}(1,1)_{0}\ExB&=
\frac{2}{3}\sqrt{-1}\omega^{[2]}_{-1 } (S_{k 1}(1,1)_{0}E)
+\frac{2}{3}\sqrt{-1}\omega^{[1]}_{-1 } (S_{k 1}(1,1)_{0}E)\nonumber\\&\quad{}
+\frac{1}{3}\sqrt{-1}\omega_{0 } \omega_{0 } (S_{k 1}(1,1)_{0}E)
+\frac{\zone^2+1}{3}\sqrt{-1}S_{k 1}(1,1)_{-2 } E\nonumber\\&\quad{}
+\frac{-(\zone^2+1)}{3}\sqrt{-1}\omega_{0 } S_{k 1}(1,1)_{-1 } E
+\frac{-(\zone-1) (\zone+1)}{3}S_{k 2}(1,1)_{-2 } E\nonumber\\&\quad{}
+\frac{(\zone-1) (\zone+1)}{3}\omega_{0 } S_{k 2}(1,1)_{-1 } E,\nonumber\\
S_{21}(1,2)_{1}S_{k 1}(1,1)_{0}\ExB&=
\zone^2\sqrt{-1}\omega_{0 } (S_{k 1}(1,1)_{0}E)
-\zone^4\sqrt{-1}S_{k 1}(1,1)_{-1 } E\nonumber\\&\quad{}
+\zone^2 (\zone-1) (\zone+1)S_{k 2}(1,1)_{-1 } E,
\nonumber\\
S_{21}(1,2)_{2}S_{k 1}(1,1)_{0}\ExB&=
-\zone^2\sqrt{-1}(S_{k 1}(1,1)_{0}E),
\end{align}

\begin{align}
S_{21}(1,3)_{0}\ExB&=
\frac{-2 (3 \zone^4-\zone^2-1)}{7 \zone^2+1}\sqrt{-1}\omega^{[1]}_{-2 } E\nonumber\\&\quad{}
+\frac{2 \zone^2 (8 \zone^4+\zone^2-1)}{7 \zone^2+1}\sqrt{-1}\omega_{0 } \omega^{[2]}_{-1 } E\nonumber\\&\quad{}
+\frac{-2 (\zone^2+1) (8 \zone^4+\zone^2+1)}{7 \zone^2+1}\sqrt{-1}\omega_{0 } \omega^{[1]}_{-1 } E\nonumber\\&\quad{}
+\frac{8 \zone^4+\zone^2-1}{7 \zone^2+1}\sqrt{-1}\omega_{0 } \omega_{0 } \omega_{0 } E\nonumber\\&\quad{}
+\frac{2 (\zone^2+1)}{7 \zone^2+1}\sqrt{-1}\omega^{[2]}_{-1 } (\omega^{[1]}_{0}\ExB)\nonumber\\&\quad{}
+\frac{2 (\zone^2+1)}{7 \zone^2+1}\sqrt{-1}\omega^{[1]}_{-1 } (\omega^{[1]}_{0}\ExB)\nonumber\\&\quad{}
+\frac{2 (4 \zone^2+1)}{7 \zone^2+1}\sqrt{-1}\omega_{0 } \omega_{0 } (\omega^{[1]}_{0}\ExB)\nonumber\\&\quad{}
+\frac{(\zone^2+1) (2 \zone^2+1)}{7 \zone^2+1}S_{21}(1,1)_{-2 } E\nonumber\\&\quad{}
+\frac{(2 \zone^2+1) (8 \zone^4+\zone^2-1)}{7 \zone^2+1}\omega_{0 } S_{21}(1,1)_{-1 } E\nonumber\\&\quad{}
+\frac{2 \zone^2 (\zone^2-2)}{7 \zone^2+1}S_{21}(1,2)_{-1 } E,\nonumber\\
S_{21}(1,3)_{1}\ExB&=
2 \zone^2 (\zone-1) (\zone+1)\sqrt{-1}\omega^{[2]}_{-1 } E
-2 \zone^4\sqrt{-1}\omega^{[1]}_{-1 } E\nonumber\\&\quad{}
+(\zone-1) (\zone+1)\sqrt{-1}\omega_{0 } \omega_{0 } E
+\sqrt{-1}\omega_{0 } (\omega^{[1]}_{0}\ExB)\nonumber\\&\quad{}
+\zone^2 (2 \zone^2-1)S_{21}(1,1)_{-1 } E,
\nonumber\\
S_{21}(1,3)_{2}\ExB&=-\sqrt{-1}\omega_{0 } E+\sqrt{-1}(\omega^{[1]}_{0}\ExB),
\nonumber\\
S_{21}(1,3)_{3}\ExB&=
\zone^2\sqrt{-1}E,
\end{align}

\begin{align}
S_{21}(1,3)_{0}\omega^{[1]}_{0}\ExB&=
\frac{-4 \zone^2 (24 \zone^6-16 \zone^4+21 \zone^2+9)}{48 \zone^4+25 \zone^2+3}\sqrt{-1}\omega^{[2]}_{-1 } \omega^{[1]}_{-1 } E\nonumber\\&\quad{}
+\frac{-4 (44 \zone^6+13 \zone^4+87 \zone^2+27)}{144 \zone^4+75 \zone^2+9}\sqrt{-1}\omega^{[1]}_{-3 } E\nonumber\\&\quad{}
+\frac{4 (24 \zone^8+8 \zone^6+21 \zone^4+33 \zone^2+9)}{48 \zone^4+25 \zone^2+3}\sqrt{-1}\omega^{[1]}_{-1 } \omega^{[1]}_{-1 } E\nonumber\\&\quad{}
+\frac{2 (4 \zone^6-52 \zone^4-96 \zone^2-27)}{48 \zone^4+25 \zone^2+3}\sqrt{-1}\Har^{[1]}_{-1 } E\nonumber\\&\quad{}
+\frac{352 \zone^8-29 \zone^6+273 \zone^4+165 \zone^2+18}{192 \zone^4+100 \zone^2+12}\sqrt{-1}\omega_{0 } \omega^{[2]}_{-2 } E\nonumber\\&\quad{}
+\frac{-(352 \zone^{10}-541 \zone^8+619 \zone^6-117 \zone^4-336 \zone^2-72)}{192 \zone^6+100 \zone^4+12 \zone^2}\sqrt{-1}\omega_{0 } \omega^{[1]}_{-2 } E\nonumber\\&\quad{}
+\frac{416 \zone^8+157 \zone^6+519 \zone^4+123 \zone^2-18}{96 \zone^4+50 \zone^2+6}\sqrt{-1}\omega_{0 } \omega_{0 } \omega^{[2]}_{-1 } E\nonumber\\&\quad{}
+\frac{-(416 \zone^{10}+317 \zone^8+833 \zone^6+453 \zone^4+168 \zone^2+36)}{96 \zone^6+50 \zone^4+6 \zone^2}\sqrt{-1}\omega_{0 } \omega_{0 } \omega^{[1]}_{-1 } E\nonumber\\&\quad{}
+\frac{416 \zone^8+157 \zone^6+519 \zone^4+123 \zone^2-18}{192 \zone^6+100 \zone^4+12 \zone^2}\sqrt{-1}\omega_{0 } \omega_{0 } \omega_{0 } \omega_{0 } E\nonumber\\&\quad{}
+\frac{-\zone^2}{3 \zone^2+1}\sqrt{-1}\omega^{[2]}_{-2 } (\omega^{[1]}_{0}\ExB)\nonumber\\&\quad{}
+\frac{56 \zone^6-3 \zone^4+72 \zone^2+27}{48 \zone^6+25 \zone^4+3 \zone^2}\sqrt{-1}\omega^{[1]}_{-2 } (\omega^{[1]}_{0}\ExB)\nonumber\\&\quad{}
+\frac{6 (8 \zone^6+8 \zone^2+3)}{48 \zone^6+25 \zone^4+3 \zone^2}\sqrt{-1}\omega_{0 } \omega^{[2]}_{-1 } (\omega^{[1]}_{0}\ExB)\nonumber\\&\quad{}
+\frac{32 \zone^6+141 \zone^4+123 \zone^2+27}{96 \zone^6+50 \zone^4+6 \zone^2}\sqrt{-1}\omega_{0 } \omega_{0 } \omega_{0 } (\omega^{[1]}_{0}\ExB)\nonumber\\&\quad{}
+\frac{\zone^2 (2 \zone^2-1)}{3 \zone^2+1}S_{21}(1,1)_{-3 } E\nonumber\\&\quad{}
+\frac{-2 (48 \zone^8-8 \zone^6+42 \zone^4+42 \zone^2+9)}{48 \zone^4+25 \zone^2+3}\omega^{[1]}_{-1 } S_{21}(1,1)_{-1 } E\nonumber\\&\quad{}
+\frac{352 \zone^{10}+163 \zone^8+433 \zone^6+369 \zone^4+186 \zone^2+36}{192 \zone^6+100 \zone^4+12 \zone^2}\omega_{0 } S_{21}(1,1)_{-2 } E\nonumber\\&\quad{}
+\frac{416 \zone^{10}+189 \zone^8+612 \zone^6+246 \zone^4-39 \zone^2-18}{96 \zone^6+50 \zone^4+6 \zone^2}\omega_{0 } \omega_{0 } S_{21}(1,1)_{-1 } E\nonumber\\&\quad{}
+\frac{-(256 \zone^6-109 \zone^4+171 \zone^2+81)}{96 \zone^4+50 \zone^2+6}\omega_{0 } S_{21}(1,2)_{-1 } E\nonumber\\&\quad{}
+\frac{-(2 \zone^4-10 \zone^2-3)}{3 \zone^2+1}S_{21}(1,3)_{-1 } E,\end{align}

\begin{align}
S_{21}(1,3)_{1}\omega^{[1]}_{0}\ExB&=
\frac{6 \zone^6-23 \zone^4+\zone^2+12}{14 \zone^2+2}\sqrt{-1}\omega^{[1]}_{-2 } E\nonumber\\&\quad{}
+\frac{-2 \zone^2 (4 \zone^6-10 \zone^4-5 \zone^2+3)}{7 \zone^2+1}\sqrt{-1}\omega_{0 } \omega^{[2]}_{-1 } E\nonumber\\&\quad{}
+\frac{2 (4 \zone^8-6 \zone^6-14 \zone^4-\zone^2-3)}{7 \zone^2+1}\sqrt{-1}\omega_{0 } \omega^{[1]}_{-1 } E\nonumber\\&\quad{}
+\frac{-(4 \zone^6-10 \zone^4-5 \zone^2+3)}{7 \zone^2+1}\sqrt{-1}\omega_{0 } \omega_{0 } \omega_{0 } E\nonumber\\&\quad{}
+\frac{-(\zone^2-2) (\zone^2+3)}{7 \zone^2+1}\sqrt{-1}\omega^{[2]}_{-1 } (\omega^{[1]}_{0}\ExB)\nonumber\\&\quad{}
+\frac{-(\zone^2-2) (\zone^2+3)}{7 \zone^2+1}\sqrt{-1}\omega^{[1]}_{-1 } (\omega^{[1]}_{0}\ExB)\nonumber\\&\quad{}
+\frac{-(8 \zone^4-19 \zone^2-9)}{14 \zone^2+2}\sqrt{-1}\omega_{0 } \omega_{0 } (\omega^{[1]}_{0}\ExB)\nonumber\\&\quad{}
+\frac{-(2 \zone^6-11 \zone^4-13 \zone^2-6)}{14 \zone^2+2}S_{21}(1,1)_{-2 } E\nonumber\\&\quad{}
+\frac{-(2 \zone^2+1) (4 \zone^6-10 \zone^4-5 \zone^2+3)}{7 \zone^2+1}\omega_{0 } S_{21}(1,1)_{-1 } E\nonumber\\&\quad{}
+\frac{-\zone^2 (2 \zone^4-11 \zone^2+29)}{14 \zone^2+2}S_{21}(1,2)_{-1 } E,
\nonumber\\
S_{21}(1,3)_{2}\omega^{[1]}_{0}\ExB&=
-2 \zone^2 (2 \zone^4-4 \zone^2+3)\sqrt{-1}\omega^{[2]}_{-1 } E\nonumber\\&\quad{}
+4 \zone^4 (\zone-1) (\zone+1)\sqrt{-1}\omega^{[1]}_{-1 } E\nonumber\\&\quad{}
-(2 \zone^4-4 \zone^2+3)\sqrt{-1}\omega_{0 } \omega_{0 } E\nonumber\\&\quad{}
-(2 \zone^2-3)\sqrt{-1}\omega_{0 } (\omega^{[1]}_{0}\ExB)\nonumber\\&\quad{}
-2 \zone^2 (\zone-1) (\zone+1) (2 \zone^2-1)S_{21}(1,1)_{-1 } E,
\nonumber\\
S_{21}(1,3)_{3}\omega^{[1]}_{0}\ExB&=
-3\sqrt{-1}\omega_{0 } E
+(\zone^2+3)\sqrt{-1}(\omega^{[1]}_{0}\ExB),
\nonumber\\
S_{21}(1,3)_{4}\omega^{[1]}_{0}\ExB&=
3 \zone^2\sqrt{-1}E,
\end{align}

\begin{align}
&S_{21}(1,3)_{0}S_{k 1}(1,1)_{0}\ExB\nonumber\\
&=
\frac{8 \zone^6-8 \zone^4+8 \zone^2+3}{4 \zone^4+6 \zone^2+1}\sqrt{-1}\omega^{[k]}_{-2 } (S_{k 1}(1,1)_{0}E)\nonumber\\&\quad{}
+\frac{\zone^2 (2 \zone^2+1)}{4 \zone^4+6 \zone^2+1}\sqrt{-1}\omega^{[2]}_{-2 } (S_{k 1}(1,1)_{0}E)
+\frac{\zone^2 (2 \zone^2+1)}{4 \zone^4+6 \zone^2+1}\sqrt{-1}\omega^{[1]}_{-2 } (S_{k 1}(1,1)_{0}E)\nonumber\\&\quad{}
+\frac{-2 (8 \zone^6-8 \zone^4+8 \zone^2+3)}{12 \zone^4+18 \zone^2+3}\sqrt{-1}\omega_{0 } \omega^{[k]}_{-1 } (S_{k 1}(1,1)_{0}E)\nonumber\\&\quad{}
+\frac{-\zone^2 (12 \zone^4-2 \zone^2-1)}{4 \zone^6+2 \zone^4-5 \zone^2-1}\sqrt{-1}\omega_{0 } \omega^{[2]}_{-1 } (S_{k 1}(1,1)_{0}E)\nonumber\\&\quad{}
+\frac{-\zone^2 (12 \zone^4-2 \zone^2-1)}{4 \zone^6+2 \zone^4-5 \zone^2-1}\sqrt{-1}\omega_{0 } \omega^{[1]}_{-1 } (S_{k 1}(1,1)_{0}E)\nonumber\\&\quad{}
+\frac{-(2 \zone^2-1) (12 \zone^4-2 \zone^2-1)}{8 \zone^6+4 \zone^4-10 \zone^2-2}\sqrt{-1}\omega_{0 } \omega_{0 } \omega_{0 } (S_{k 1}(1,1)_{0}E)\nonumber\\&\quad{}
+\frac{2 \zone^2 (8 \zone^6-8 \zone^4+8 \zone^2+3)}{12 \zone^4+18 \zone^2+3}\sqrt{-1}\omega^{[k]}_{-1 } S_{k 1}(1,1)_{-1 } E\nonumber\\&\quad{}
+\frac{\zone^2 (\zone^2-3) (2 \zone^2+1)}{4 \zone^4+6 \zone^2+1}\sqrt{-1}S_{k 1}(1,1)_{-3 } E\nonumber\\&\quad{}
+\frac{-2 \zone^4}{4 \zone^4+6 \zone^2+1}\sqrt{-1}\omega^{[2]}_{-1 } S_{k 1}(1,1)_{-1 } E
+\frac{2 \zone^4}{4 \zone^4+6 \zone^2+1}\sqrt{-1}\omega^{[1]}_{-1 } S_{k 1}(1,1)_{-1 } E\nonumber\\&\quad{}
+\frac{-\zone^2 (10 \zone^4-2 \zone^2+1)}{4 \zone^6+2 \zone^4-5 \zone^2-1}\sqrt{-1}\omega_{0 } S_{k 1}(1,1)_{-2 } E\nonumber\\&\quad{}
+\frac{\zone^2 (2 \zone^2+1) (6 \zone^4-4 \zone^2+1)}{4 \zone^6+2 \zone^4-5 \zone^2-1}\sqrt{-1}\omega_{0 } \omega_{0 } S_{k 1}(1,1)_{-1 } E\nonumber\\&\quad{}
+\frac{\zone^2 (8 \zone^2+1)}{4 \zone^4+6 \zone^2+1}\sqrt{-1}S_{k 1}(1,2)_{-2 } E\nonumber\\&\quad{}
+\frac{\zone^2 (24 \zone^6-4 \zone^4+8 \zone^2-1)}{8 \zone^6+4 \zone^4-10 \zone^2-2}\sqrt{-1}\omega_{0 } S_{k 1}(1,2)_{-1 } E\nonumber\\&\quad{}
+\frac{-2 \zone^2 (8 \zone^6-8 \zone^4+8 \zone^2+3)}{12 \zone^4+18 \zone^2+3}\omega^{[k]}_{-1 } S_{k 2}(1,1)_{-1 } E\nonumber\\&\quad{}
+\frac{-\zone^2 (\zone-1) (\zone+1) (2 \zone^2+1)}{4 \zone^4+6 \zone^2+1}S_{k 2}(1,1)_{-3 } E\nonumber\\&\quad{}
+\frac{-4 \zone^4}{4 \zone^4+6 \zone^2+1}\omega^{[1]}_{-1 } S_{k 2}(1,1)_{-1 } E
+\frac{-\zone^2 (12 \zone^4-2 \zone^2-1)}{4 \zone^4+6 \zone^2+1}\omega_{0 } \omega_{0 } S_{k 2}(1,1)_{-1 } E\nonumber\\&\quad{}
+\frac{-3 \zone^2 (2 \zone^2-1) (2 \zone^2+1)^2}{8 \zone^6+4 \zone^4-10 \zone^2-2}\omega_{0 } S_{k 2}(1,2)_{-1 } E\nonumber\\&\quad{}
+\frac{2 \zone^2 (\zone^2+1) (4 \zone^2+1)}{4 \zone^4+6 \zone^2+1}S_{k 2}(1,3)_{-1 } E,\end{align}

\begin{align}
S_{21}(1,3)_{1}S_{k 1}(1,1)_{0}\ExB&=
\frac{\zone^2}{3}\sqrt{-1}\omega^{[2]}_{-1 } (S_{k 1}(1,1)_{0}E)\nonumber\\&\quad{}
+\frac{\zone^2}{3}\sqrt{-1}\omega^{[1]}_{-1 } (S_{k 1}(1,1)_{0}E)\nonumber\\&\quad{}
+\frac{2 \zone^2}{3}\sqrt{-1}\omega_{0 } \omega_{0 } (S_{k 1}(1,1)_{0}E)\nonumber\\&\quad{}
+\frac{-\zone^2 (2 \zone^2-1)}{6}\sqrt{-1}S_{k 1}(1,1)_{-2 } E\nonumber\\&\quad{}
+\frac{-\zone^2 (4 \zone^2+1)}{6}\sqrt{-1}\omega_{0 } S_{k 1}(1,1)_{-1 } E\nonumber\\&\quad{}
+\frac{\zone^2 (2 \zone^2+1)}{6}S_{k 2}(1,1)_{-2 } E\nonumber\\&\quad{}
+\frac{\zone^2 (2 \zone-1) (2 \zone+1)}{6}\omega_{0 } S_{k 2}(1,1)_{-1 } E\nonumber\\&\quad{}
+\zone^2S_{k 2}(1,2)_{-1 } E,
\nonumber\\
S_{21}(1,3)_{2}S_{k 1}(1,1)_{0}\ExB&=
-\zone^2\sqrt{-1}\omega_{0 } (S_{k 1}(1,1)_{0}E)
+\zone^4\sqrt{-1}S_{k 1}(1,1)_{-1 } E\nonumber\\&\quad{}
-\zone^2 (\zone-1) (\zone+1)S_{k 2}(1,1)_{-1 } E,
\nonumber\\
S_{21}(1,3)_{3}S_{k 1}(1,1)_{0}\ExB&=
\zone^2\sqrt{-1}(S_{k 1}(1,1)_{0}E),
\end{align}

\begin{align}
\omega^{\textcolor{black}{[k]}}_{0}S_{k 1}(1,1)_{0}\ExB&=
\omega_{0 } (S_{k 1}(1,1)_{0}E)
-\zone^2S_{k 1}(1,1)_{-1 } E\nonumber\\&\quad{}
-\zone^2\sqrt{-1}S_{k 2}(1,1)_{-1 } E,\nonumber\\
\omega^{\textcolor{black}{[k]}}_{1}S_{k 1}(1,1)_{0}\ExB&=
(S_{k 1}(1,1)_{0}E),
\end{align}

\begin{align}
\Har^{\textcolor{black}{[k]}}_{0}S_{k 1}(1,1)_{0}\ExB&=
-6\omega^{[k]}_{-2 } (S_{k 1}(1,1)_{0}E)
+4\omega_{0 } \omega^{[k]}_{-1 } (S_{k 1}(1,1)_{0}E)\nonumber\\&\quad{}
-4 \zone^2\omega^{[k]}_{-1 } S_{k 1}(1,1)_{-1 } E
-4 \zone^2\sqrt{-1}\omega^{[k]}_{-1 } S_{k 2}(1,1)_{-1 } E,\end{align}

\begin{align}
\Har^{\textcolor{black}{[k]}}_{1}S_{k 1}(1,1)_{0}\ExB&=
\frac{4}{3}\omega^{[2]}_{-1 } (S_{k 1}(1,1)_{0}E)
+\frac{4}{3}\omega^{[1]}_{-1 } (S_{k 1}(1,1)_{0}E)\nonumber\\&\quad{}
+\frac{8}{3}\omega_{0 } \omega_{0 } (S_{k 1}(1,1)_{0}E)
+\frac{-2 (2 \zone^2-1)}{3}S_{k 1}(1,1)_{-2 } E\nonumber\\&\quad{}
+\frac{-2 (4 \zone^2+1)}{3}\omega_{0 } S_{k 1}(1,1)_{-1 } E\nonumber\\&\quad{}
+\frac{-2 (2 \zone^2+1)}{3}\sqrt{-1}S_{k 2}(1,1)_{-2 } E\nonumber\\&\quad{}
+\frac{-2 (2 \zone-1) (2 \zone+1)}{3}\sqrt{-1}\omega_{0 } S_{k 2}(1,1)_{-1 } E,
\nonumber\\
\Har^{\textcolor{black}{[k]}}_{2}S_{k 1}(1,1)_{0}\ExB&=
\frac{7}{3}\omega_{0 } (S_{k 1}(1,1)_{0}E)
+\frac{-7 \zone^2}{3}S_{k 1}(1,1)_{-1 } E
+\frac{-7 \zone^2}{3}\sqrt{-1}S_{k 2}(1,1)_{-1 } E,
\nonumber\\
\Har^{\textcolor{black}{[k]}}_{3}S_{k 1}(1,1)_{0}\ExB&=
(S_{k 1}(1,1)_{0}E),
\end{align}


\begin{align}
S_{k 1}(1,1)_{0}\omega^{[1]}_{0}\ExB&=
-(\zone-1) (\zone+1)\omega_{0 } (S_{k 1}(1,1)_{0}E)
+\zone^4S_{k 1}(1,1)_{-1 } E\nonumber\\&\quad{}
+\zone^2 (\zone-1) (\zone+1)\sqrt{-1}S_{k 2}(1,1)_{-1 } E,\nonumber\\
S_{k 1}(1,1)_{1}\omega^{[1]}_{0}\ExB&=
(S_{k 1}(1,1)_{0}E),
\end{align}

\begin{align}
S_{k 1}(1,1)_{0}S_{k 1}(1,1)_{0}\ExB&=
2 \zone^2\omega^{[k]}_{-1 } E
+2 \zone^4\omega^{[2]}_{-1 } E
-2 \zone^2 (\zone^2+1)\omega^{[1]}_{-1 } E\nonumber\\&\quad{}
+\zone^2\omega_{0 } \omega_{0 } E
+\omega_{0 } (\omega^{[1]}_{0}\ExB)
-\zone^2 (2 \zone^2+1)\sqrt{-1}S_{21}(1,1)_{-1 } E,\nonumber\\
S_{k 1}(1,1)_{1}S_{k 1}(1,1)_{0}\ExB&=
(\omega^{[1]}_{0}\ExB),
\nonumber\\
S_{k 1}(1,1)_{2}S_{k 1}(1,1)_{0}\ExB&=
\zone^2E,
\end{align}

\begin{align}
S_{k 1}(1,2)_{0}\ExB&=
-\omega_{0 } (S_{k 1}(1,1)_{0}E)
+\zone^2S_{k 1}(1,1)_{-1 } E
+\zone^2\sqrt{-1}S_{k 2}(1,1)_{-1 } E,\nonumber\\
S_{k 1}(1,2)_{1}\ExB&=
-(S_{k 1}(1,1)_{0}E),
\end{align}

\begin{align}
S_{k 1}(1,2)_{0}\omega^{[1]}_{0}\ExB&=
\frac{2 \zone^2}{3}\omega^{[2]}_{-1 } (S_{k 1}(1,1)_{0}E)
+\frac{2 \zone^2}{3}\omega^{[1]}_{-1 } (S_{k 1}(1,1)_{0}E)\nonumber\\&\quad{}
+\frac{4 \zone^2-3}{3}\omega_{0 } \omega_{0 } (S_{k 1}(1,1)_{0}E)
+\frac{-\zone^2 (2 \zone^2-1)}{3}S_{k 1}(1,1)_{-2 } E\nonumber\\&\quad{}
+\frac{-2 \zone^2 (2 \zone^2-1)}{3}\omega_{0 } S_{k 1}(1,1)_{-1 } E\nonumber\\&\quad{}
+\frac{-2 \zone^2 (\zone-1) (\zone+1)}{3}\sqrt{-1}S_{k 2}(1,1)_{-2 } E\nonumber\\&\quad{}
+\frac{-4 \zone^2 (\zone-1) (\zone+1)}{3}\sqrt{-1}\omega_{0 } S_{k 2}(1,1)_{-1 } E\nonumber\\&\quad{}
-\zone^2\sqrt{-1}S_{k 2}(1,2)_{-1 } E,\end{align}

\begin{align}
S_{k 1}(1,2)_{1}\omega^{[1]}_{0}\ExB&=
\zone^2-2\omega_{0 } (S_{k 1}(1,1)_{0}E)
-\zone^2 (\zone-1) (\zone+1)S_{k 1}(1,1)_{-1 } E\nonumber\\&\quad{}
-\zone^2 (\zone^2-2)\sqrt{-1}S_{k 2}(1,1)_{-1 } E,
\nonumber\\
S_{k 1}(1,2)_{2}\omega^{[1]}_{0}\ExB&=
-2(S_{k 1}(1,1)_{0}E),
\end{align}

\begin{align}
S_{k 1}(1,2)_{0}S_{k 1}(1,1)_{0}\ExB&=
-\zone^2\omega^{[k]}_{-2 } E
+\frac{-\zone^2 (6 \zone^2-1)}{7 \zone^2+1}\omega^{[1]}_{-2 } E\nonumber\\&\quad{}
+\frac{16 \zone^6}{7 \zone^2+1}\omega_{0 } \omega^{[2]}_{-1 } E
+\frac{-4 \zone^2 (2 \zone^2+1)^2}{7 \zone^2+1}\omega_{0 } \omega^{[1]}_{-1 } E\nonumber\\&\quad{}
+\frac{8 \zone^4}{7 \zone^2+1}\omega_{0 } \omega_{0 } \omega_{0 } E
+\frac{2 \zone^2}{7 \zone^2+1}\omega^{[2]}_{-1 } (\omega^{[1]}_{0}\ExB)\nonumber\\&\quad{}
+\frac{2 \zone^2}{7 \zone^2+1}\omega^{[1]}_{-1 } (\omega^{[1]}_{0}\ExB)
+\frac{8 \zone^2+1}{7 \zone^2+1}\omega_{0 } \omega_{0 } (\omega^{[1]}_{0}\ExB)\nonumber\\&\quad{}
+\frac{-\zone^2 (2 \zone^2+1)}{7 \zone^2+1}\sqrt{-1}S_{21}(1,1)_{-2 } E\nonumber\\&\quad{}
+\frac{-8 \zone^4 (2 \zone^2+1)}{7 \zone^2+1}\sqrt{-1}\omega_{0 } S_{21}(1,1)_{-1 } E\nonumber\\&\quad{}
+\frac{-\zone^2 (2 \zone^2+1)}{7 \zone^2+1}\sqrt{-1}S_{21}(1,2)_{-1 } E,\end{align}

\begin{align}
S_{k 1}(1,2)_{1}S_{k 1}(1,1)_{0}\ExB&=
-2 \zone^2\omega^{[k]}_{-1 } E
+2 \zone^4\omega^{[2]}_{-1 } E\nonumber\\&\quad{}
-2 \zone^2 (\zone^2+1)\omega^{[1]}_{-1 } E
+\zone^2\omega_{0 } \omega_{0 } E\nonumber\\&\quad{}
+\omega_{0 } (\omega^{[1]}_{0}\ExB)
-\zone^2 (2 \zone^2+1)\sqrt{-1}S_{21}(1,1)_{-1 } E,
\nonumber\\
S_{k 1}(1,2)_{2}S_{k 1}(1,1)_{0}\ExB&=0,\nonumber\\
S_{k 1}(1,2)_{3}S_{k 1}(1,1)_{0}\ExB&=
-\zone^2E,
\end{align}

\begin{align}
S_{k 1}(1,3)_{0}\ExB&=
\frac{1}{3}\omega^{[2]}_{-1 } (S_{k 1}(1,1)_{0}E)
+\frac{1}{3}\omega^{[1]}_{-1 } (S_{k 1}(1,1)_{0}E)\nonumber\\&\quad{}
+\frac{2}{3}\omega_{0 } \omega_{0 } (S_{k 1}(1,1)_{0}E)
+\frac{-(2 \zone^2-1)}{6}S_{k 1}(1,1)_{-2 } E\nonumber\\&\quad{}
+\frac{-(4 \zone^2+1)}{6}\omega_{0 } S_{k 1}(1,1)_{-1 } E
+\frac{-(2 \zone^2+1)}{6}\sqrt{-1}S_{k 2}(1,1)_{-2 } E\nonumber\\&\quad{}
+\frac{-(2 \zone-1) (2 \zone+1)}{6}\sqrt{-1}\omega_{0 } S_{k 2}(1,1)_{-1 } E,\nonumber\\
S_{k 1}(1,3)_{1}\ExB&=
\omega_{0 } (S_{k 1}(1,1)_{0}E)
-\zone^2S_{k 1}(1,1)_{-1 } E
-\zone^2\sqrt{-1}S_{k 2}(1,1)_{-1 } E,
\nonumber\\
S_{k 1}(1,3)_{2}\ExB&=
(S_{k 1}(1,1)_{0}E),
\end{align}

\begin{align}
&S_{k 1}(1,3)_{0}\omega^{[1]}_{0}\ExB\nonumber\\
&=
\frac{\zone^2 (12 \zone^4-14 \zone^2-9)}{4 \zone^4+6 \zone^2+1}\omega^{[k]}_{-2 } (S_{k 1}(1,1)_{0}E)
+\frac{\zone^2 (2 \zone^2+1)}{4 \zone^4+6 \zone^2+1}\omega^{[2]}_{-2 } (S_{k 1}(1,1)_{0}E)\nonumber\\&\quad{}
+\frac{\zone^2 (2 \zone^2+1)}{4 \zone^4+6 \zone^2+1}\omega^{[1]}_{-2 } (S_{k 1}(1,1)_{0}E)
+\frac{-2 \zone^2 (12 \zone^4-14 \zone^2-9)}{12 \zone^4+18 \zone^2+3}\omega_{0 } \omega^{[k]}_{-1 } (S_{k 1}(1,1)_{0}E)\nonumber\\&\quad{}
+\frac{-\zone^2 (12 \zone^4-2 \zone^2-1)}{4 \zone^6+2 \zone^4-5 \zone^2-1}\omega_{0 } \omega^{[2]}_{-1 } (S_{k 1}(1,1)_{0}E)
+\frac{-\zone^2 (12 \zone^4-2 \zone^2-1)}{4 \zone^6+2 \zone^4-5 \zone^2-1}\omega_{0 } \omega^{[1]}_{-1 } (S_{k 1}(1,1)_{0}E)\nonumber\\&\quad{}
+\frac{-(2 \zone^2-1) (12 \zone^4-2 \zone^2-1)}{8 \zone^6+4 \zone^4-10 \zone^2-2}\omega_{0 } \omega_{0 } \omega_{0 } (S_{k 1}(1,1)_{0}E)\nonumber\\&\quad{}
+\frac{2 \zone^4 (12 \zone^4-14 \zone^2-9)}{12 \zone^4+18 \zone^2+3}\omega^{[k]}_{-1 } S_{k 1}(1,1)_{-1 } E
+\frac{\zone^2 (2 \zone^2-1) (3 \zone^2+2)}{4 \zone^4+6 \zone^2+1}S_{k 1}(1,1)_{-3 } E\nonumber\\&\quad{}
+\frac{-2 \zone^4}{4 \zone^4+6 \zone^2+1}\omega^{[2]}_{-1 } S_{k 1}(1,1)_{-1 } E\nonumber\\&\quad{}
+\frac{2 \zone^4}{4 \zone^4+6 \zone^2+1}\omega^{[1]}_{-1 } S_{k 1}(1,1)_{-1 } E\nonumber\\&\quad{}
+\frac{-\zone^2 (10 \zone^4-2 \zone^2+1)}{4 \zone^6+2 \zone^4-5 \zone^2-1}\omega_{0 } S_{k 1}(1,1)_{-2 } E\nonumber\\&\quad{}
+\frac{\zone^2 (2 \zone^2+1) (6 \zone^4-4 \zone^2+1)}{4 \zone^6+2 \zone^4-5 \zone^2-1}\omega_{0 } \omega_{0 } S_{k 1}(1,1)_{-1 } E\nonumber\\&\quad{}
+\frac{-2 \zone^4 (2 \zone^2-1)}{4 \zone^4+6 \zone^2+1}S_{k 1}(1,2)_{-2 } E\nonumber\\&\quad{}
+\frac{\zone^2 (24 \zone^6-4 \zone^4+8 \zone^2-1)}{8 \zone^6+4 \zone^4-10 \zone^2-2}\omega_{0 } S_{k 1}(1,2)_{-1 } E\nonumber\\&\quad{}
+\frac{2 \zone^4 (12 \zone^4-14 \zone^2-9)}{12 \zone^4+18 \zone^2+3}\sqrt{-1}\omega^{[k]}_{-1 } S_{k 2}(1,1)_{-1 } E\nonumber\\&\quad{}
+\frac{\zone^2 (\zone-1) (\zone+1) (2 \zone^2+1)}{4 \zone^4+6 \zone^2+1}\sqrt{-1}S_{k 2}(1,1)_{-3 } E\nonumber\\&\quad{}
+\frac{4 \zone^4}{4 \zone^4+6 \zone^2+1}\sqrt{-1}\omega^{[1]}_{-1 } S_{k 2}(1,1)_{-1 } E\nonumber\\&\quad{}
+\frac{\zone^2 (12 \zone^4-2 \zone^2-1)}{4 \zone^4+6 \zone^2+1}\sqrt{-1}\omega_{0 } \omega_{0 } S_{k 2}(1,1)_{-1 } E\nonumber\\&\quad{}
+\frac{3 \zone^2 (2 \zone^2-1) (2 \zone^2+1)^2}{8 \zone^6+4 \zone^4-10 \zone^2-2}\sqrt{-1}\omega_{0 } S_{k 2}(1,2)_{-1 } E\nonumber\\&\quad{}
+\frac{-\zone^2 (2 \zone^2+1)^2}{4 \zone^4+6 \zone^2+1}\sqrt{-1}S_{k 2}(1,3)_{-1 } E,\end{align}

\begin{align}
S_{k 1}(1,3)_{1}\omega^{[1]}_{0}\ExB
&=
\frac{-(2 \zone^2-1)}{3}\omega^{[2]}_{-1 } (S_{k 1}(1,1)_{0}E)
+\frac{-(2 \zone^2-1)}{3}\omega^{[1]}_{-1 } (S_{k 1}(1,1)_{0}E)\nonumber\\&\quad{}
+\frac{-(4 \zone^2-5)}{3}\omega_{0 } \omega_{0 } (S_{k 1}(1,1)_{0}E)\nonumber\\&\quad{}
+\frac{(2 \zone^2-1)^2}{6}S_{k 1}(1,1)_{-2 } E\nonumber\\&\quad{}
+\frac{8 \zone^4-8 \zone^2-1}{6}\omega_{0 } S_{k 1}(1,1)_{-1 } E\nonumber\\&\quad{}
+\frac{4 \zone^4-6 \zone^2-1}{6}\sqrt{-1}S_{k 2}(1,1)_{-2 } E\nonumber\\&\quad{}
+\frac{8 \zone^4-12 \zone^2+1}{6}\sqrt{-1}\omega_{0 } S_{k 2}(1,1)_{-1 } E\nonumber\\&\quad{}
+\zone^2\sqrt{-1}S_{k 2}(1,2)_{-1 } E,
\nonumber\\
S_{k 1}(1,3)_{2}\omega^{[1]}_{0}\ExB&=
-(\zone^2-3)\omega_{0 } (S_{k 1}(1,1)_{0}E)
+\zone^2 (\zone^2-2)S_{k 1}(1,1)_{-1 } E\nonumber\\&\quad{}
+\zone^2 (\zone^2-3)\sqrt{-1}S_{k 2}(1,1)_{-1 } E,
\nonumber\\
S_{k 1}(1,3)_{3}\omega^{[1]}_{0}\ExB&=
3(S_{k 1}(1,1)_{0}E),
\end{align}

\begin{align}
S_{k 1}(1,3)_{0}S_{k 1}(1,1)_{0}\ExB&=
\frac{2 \zone^2}{3}\omega^{[k]}_{-3 } E
+\zone^2\Har^{[k]}_{-1 } E\nonumber\\&\quad{}
+\frac{12 \zone^4}{16 \zone^2+3}\omega^{[2]}_{-1 } \omega^{[1]}_{-1 } E\nonumber\\&\quad{}
+\frac{10 \zone^2}{16 \zone^2+3}\omega^{[1]}_{-3 } E\nonumber\\&\quad{}
+\frac{-12 \zone^2 (\zone^2+1)}{16 \zone^2+3}\omega^{[1]}_{-1 } \omega^{[1]}_{-1 } E\nonumber\\&\quad{}
+\frac{15 \zone^2}{16 \zone^2+3}\Har^{[1]}_{-1 } E\nonumber\\&\quad{}
+\frac{-3 \zone^2 (7 \zone^2+1)}{32 \zone^2+6}\omega_{0 } \omega^{[2]}_{-2 } E\nonumber\\&\quad{}
+\frac{3 (7 \zone^4-9 \zone^2-4)}{32 \zone^2+6}\omega_{0 } \omega^{[1]}_{-2 } E\nonumber\\&\quad{}
+\frac{-3 \zone^2 (3 \zone-1) (3 \zone+1)}{16 \zone^2+3}\omega_{0 } \omega_{0 } \omega^{[2]}_{-1 } E\nonumber\\&\quad{}
+\frac{3 (9 \zone^4+3 \zone^2+2)}{16 \zone^2+3}\omega_{0 } \omega_{0 } \omega^{[1]}_{-1 } E\nonumber\\&\quad{}
+\frac{-3 (3 \zone-1) (3 \zone+1)}{32 \zone^2+6}\omega_{0 } \omega_{0 } \omega_{0 } \omega_{0 } E\nonumber\\&\quad{}
+\frac{-9}{16 \zone^2+3}\omega^{[1]}_{-2 } (\omega^{[1]}_{0}\ExB)\nonumber\\&\quad{}
+\frac{-6}{16 \zone^2+3}\omega_{0 } \omega^{[2]}_{-1 } (\omega^{[1]}_{0}\ExB)\nonumber\\&\quad{}
+\frac{-3}{16 \zone^2+3}\omega_{0 } \omega_{0 } \omega_{0 } (\omega^{[1]}_{0}\ExB)\nonumber\\&\quad{}
+\frac{-6 \zone^2 (2 \zone^2+1)}{16 \zone^2+3}\sqrt{-1}\omega^{[1]}_{-1 } S_{21}(1,1)_{-1 } E\nonumber\\&\quad{}
+\frac{3 (7 \zone^4+5 \zone^2+2)}{32 \zone^2+6}\sqrt{-1}\omega_{0 } S_{21}(1,1)_{-2 } E\nonumber\\&\quad{}
+\frac{3 (3 \zone^2-1) (3 \zone^2+1)}{16 \zone^2+3}\sqrt{-1}\omega_{0 } \omega_{0 } S_{21}(1,1)_{-1 } E\nonumber\\&\quad{}
+\frac{-15 \zone^2}{16 \zone^2+3}\sqrt{-1}\omega_{0 } S_{21}(1,2)_{-1 } E,\end{align}

\begin{align}
S_{k 1}(1,3)_{1}S_{k 1}(1,1)_{0}\ExB&=
\zone^2\omega^{[k]}_{-2 } E\nonumber\\&\quad{}
+\frac{-\zone^2 (6 \zone^2-1)}{14 \zone^2+2}\omega^{[1]}_{-2 } E\nonumber\\&\quad{}
+\frac{8 \zone^6}{7 \zone^2+1}\omega_{0 } \omega^{[2]}_{-1 } E\nonumber\\&\quad{}
+\frac{-2 \zone^2 (2 \zone^2+1)^2}{7 \zone^2+1}\omega_{0 } \omega^{[1]}_{-1 } E\nonumber\\&\quad{}
+\frac{4 \zone^4}{7 \zone^2+1}\omega_{0 } \omega_{0 } \omega_{0 } E\nonumber\\&\quad{}
+\frac{\zone^2}{7 \zone^2+1}\omega^{[2]}_{-1 } (\omega^{[1]}_{0}\ExB)\nonumber\\&\quad{}
+\frac{\zone^2}{7 \zone^2+1}\omega^{[1]}_{-1 } (\omega^{[1]}_{0}\ExB)\nonumber\\&\quad{}
+\frac{8 \zone^2+1}{14 \zone^2+2}\omega_{0 } \omega_{0 } (\omega^{[1]}_{0}\ExB)\nonumber\\&\quad{}
+\frac{-\zone^2 (2 \zone^2+1)}{14 \zone^2+2}\sqrt{-1}S_{21}(1,1)_{-2 } E\nonumber\\&\quad{}
+\frac{-4 \zone^4 (2 \zone^2+1)}{7 \zone^2+1}\sqrt{-1}\omega_{0 } S_{21}(1,1)_{-1 } E\nonumber\\&\quad{}
+\frac{-\zone^2 (2 \zone^2+1)}{14 \zone^2+2}\sqrt{-1}S_{21}(1,2)_{-1 } E,
\nonumber\\
S_{k 1}(1,3)_{2}S_{k 1}(1,1)_{0}\ExB&=
2 \zone^2\omega^{[k]}_{-1 } E,
\nonumber\\
S_{k 1}(1,3)_{3}S_{k 1}(1,1)_{0}\ExB&=0,\nonumber\\
S_{k 1}(1,3)_{4}S_{k 1}(1,1)_{0}\ExB&=
\zone^2E,
\end{align}

\begin{align}
S_{k 2}(1,1)_{0}\ExB&=
\sqrt{-1}(S_{k 1}(1,1)_{0}E),\end{align}

\begin{align}
S_{k 2}(1,1)_{0}\omega^{[1]}_{0}\ExB&=
-\zone^2\sqrt{-1}\omega_{0 } (S_{k 1}(1,1)_{0}E)
+\zone^2 (\zone^2+1)\sqrt{-1}S_{k 1}(1,1)_{-1 } E\nonumber\\&\quad{}
-\zone^4S_{k 2}(1,1)_{-1 } E,\end{align}

\begin{align}
S_{k 2}(1,1)_{0}S_{k 1}(1,1)_{0}\ExB&=
2 \zone^2\sqrt{-1}\omega^{[k]}_{-1 } E
+2 \zone^2 (\zone-1) (\zone+1)\sqrt{-1}\omega^{[2]}_{-1 } E\nonumber\\&\quad{}
-2 \zone^4\sqrt{-1}\omega^{[1]}_{-1 } E
+(\zone-1) (\zone+1)\sqrt{-1}\omega_{0 } \omega_{0 } E\nonumber\\&\quad{}
+\sqrt{-1}\omega_{0 } (\omega^{[1]}_{0}\ExB)
+\zone^2 (2 \zone^2-1)S_{21}(1,1)_{-1 } E,\nonumber\\
S_{k 2}(1,1)_{1}S_{k 1}(1,1)_{0}\ExB&=
-\sqrt{-1}\omega_{0 } E+\sqrt{-1}(\omega^{[1]}_{0}\ExB),
\nonumber\\
S_{k 2}(1,1)_{2}S_{k 1}(1,1)_{0}\ExB&=
\zone^2\sqrt{-1}E,
\end{align}

\begin{align}
S_{k 2}(1,2)_{0}\ExB&=
-\sqrt{-1}\omega_{0 } (S_{k 1}(1,1)_{0}E)
+\zone^2\sqrt{-1}S_{k 1}(1,1)_{-1 } E
-\zone^2S_{k 2}(1,1)_{-1 } E,\nonumber\\
S_{k 2}(1,2)_{1}\ExB&=
-\sqrt{-1}(S_{k 1}(1,1)_{0}E),
\end{align}

\begin{align}
S_{k 2}(1,2)_{0}\omega^{[1]}_{0}\ExB&=
\frac{2 (\zone^2+1)}{3}\sqrt{-1}\omega^{[2]}_{-1 } (S_{k 1}(1,1)_{0}E)\nonumber\\&\quad{}
+\frac{2 (\zone^2+1)}{3}\sqrt{-1}\omega^{[1]}_{-1 } (S_{k 1}(1,1)_{0}E)\nonumber\\&\quad{}
+\frac{4 \zone^2+1}{3}\sqrt{-1}\omega_{0 } \omega_{0 } (S_{k 1}(1,1)_{0}E)\nonumber\\&\quad{}
+\frac{-(\zone^2+1) (2 \zone^2-1)}{3}\sqrt{-1}S_{k 1}(1,1)_{-2 } E\nonumber\\&\quad{}
+\frac{-(4 \zone^4+2 \zone^2+1)}{3}\sqrt{-1}\omega_{0 } S_{k 1}(1,1)_{-1 } E\nonumber\\&\quad{}
+\frac{2 \zone^4+1}{3}S_{k 2}(1,1)_{-2 } E\nonumber\\&\quad{}
+\frac{(2 \zone^2-1) (2 \zone^2+1)}{3}\omega_{0 } S_{k 2}(1,1)_{-1 } E\nonumber\\&\quad{}
+\zone^2S_{k 2}(1,2)_{-1 } E,\nonumber\\
S_{k 2}(1,2)_{1}\omega^{[1]}_{0}\ExB&=
\zone^2\sqrt{-1}\omega_{0 } (S_{k 1}(1,1)_{0}E)
-\zone^2 (\zone^2+1)\sqrt{-1}S_{k 1}(1,1)_{-1 } E\nonumber\\&\quad{}
+\zone^4S_{k 2}(1,1)_{-1 } E,
\end{align}

\begin{align}
S_{k 2}(1,2)_{0}S_{k 1}(1,1)_{0}\ExB&=
-\zone^2\sqrt{-1}\omega^{[k]}_{-2 } E\nonumber\\&\quad{}
+\frac{-(6 \zone^4-3 \zone^2-4)}{7 \zone^2+1}\sqrt{-1}\omega^{[1]}_{-2 } E\nonumber\\&\quad{}
+\frac{4 \zone^2 (4 \zone^4+\zone^2-1)}{7 \zone^2+1}\sqrt{-1}\omega_{0 } \omega^{[2]}_{-1 } E\nonumber\\&\quad{}
+\frac{-4 (4 \zone^6+5 \zone^4+\zone^2+1)}{7 \zone^2+1}\sqrt{-1}\omega_{0 } \omega^{[1]}_{-1 } E\nonumber\\&\quad{}
+\frac{2 (4 \zone^4+\zone^2-1)}{7 \zone^2+1}\sqrt{-1}\omega_{0 } \omega_{0 } \omega_{0 } E\nonumber\\&\quad{}
+\frac{2 (\zone^2+2)}{7 \zone^2+1}\sqrt{-1}\omega^{[2]}_{-1 } (\omega^{[1]}_{0}\ExB)\nonumber\\&\quad{}
+\frac{2 (\zone^2+2)}{7 \zone^2+1}\sqrt{-1}\omega^{[1]}_{-1 } (\omega^{[1]}_{0}\ExB)\nonumber\\&\quad{}
+\frac{8 \zone^2+3}{7 \zone^2+1}\sqrt{-1}\omega_{0 } \omega_{0 } (\omega^{[1]}_{0}\ExB)\nonumber\\&\quad{}
+\frac{(\zone^2+2) (2 \zone^2+1)}{7 \zone^2+1}S_{21}(1,1)_{-2 } E\nonumber\\&\quad{}
+\frac{2 (2 \zone^2+1) (4 \zone^4+\zone^2-1)}{7 \zone^2+1}\omega_{0 } S_{21}(1,1)_{-1 } E\nonumber\\&\quad{}
+\frac{\zone^2 (2 \zone^2-9)}{7 \zone^2+1}S_{21}(1,2)_{-1 } E,\end{align}

\begin{align}
S_{k 2}(1,2)_{1}S_{k 1}(1,1)_{0}\ExB&=
-2 \zone^2\sqrt{-1}\omega^{[k]}_{-1 } E
+2 \zone^2 (\zone-1) (\zone+1)\sqrt{-1}\omega^{[2]}_{-1 } E\nonumber\\&\quad{}
-2 \zone^4\sqrt{-1}\omega^{[1]}_{-1 } E
+(\zone-1) (\zone+1)\sqrt{-1}\omega_{0 } \omega_{0 } E\nonumber\\&\quad{}
+\sqrt{-1}\omega_{0 } (\omega^{[1]}_{0}\ExB)
+\zone^2 (2 \zone^2-1)S_{21}(1,1)_{-1 } E,
\nonumber\\
S_{k 2}(1,2)_{2}S_{k 1}(1,1)_{0}\ExB&=0,\nonumber\\
S_{k 2}(1,2)_{3}S_{k 1}(1,1)_{0}\ExB&=
-\zone^2\sqrt{-1}E,
\end{align}

\begin{align}
S_{k 2}(1,3)_{0}\ExB&=
\frac{1}{3}\sqrt{-1}\omega^{[2]}_{-1 } (S_{k 1}(1,1)_{0}E)
+\frac{1}{3}\sqrt{-1}\omega^{[1]}_{-1 } (S_{k 1}(1,1)_{0}E)\nonumber\\&\quad{}
+\frac{2}{3}\sqrt{-1}\omega_{0 } \omega_{0 } (S_{k 1}(1,1)_{0}E)
+\frac{-(2 \zone^2-1)}{6}\sqrt{-1}S_{k 1}(1,1)_{-2 } E\nonumber\\&\quad{}
+\frac{-(4 \zone^2+1)}{6}\sqrt{-1}\omega_{0 } S_{k 1}(1,1)_{-1 } E
+\frac{2 \zone^2+1}{6}S_{k 2}(1,1)_{-2 } E\nonumber\\&\quad{}
+\frac{(2 \zone-1) (2 \zone+1)}{6}\omega_{0 } S_{k 2}(1,1)_{-1 } E,\nonumber\\
S_{k 2}(1,3)_{1}\ExB&=
\sqrt{-1}\omega_{0 } (S_{k 1}(1,1)_{0}E)
-\zone^2\sqrt{-1}S_{k 1}(1,1)_{-1 } E\nonumber\\&\quad{}
+\zone^2S_{k 2}(1,1)_{-1 } E,
\nonumber\\
S_{k 2}(1,3)_{2}\ExB&=
\sqrt{-1}(S_{k 1}(1,1)_{0}E),
\end{align}

\begin{align}
&S_{k 2}(1,3)_{0}\omega^{[1]}_{0}\ExB\nonumber\\&=
\frac{12 \zone^6-2 \zone^4+9 \zone^2+3}{4 \zone^4+6 \zone^2+1}\sqrt{-1}\omega^{[k]}_{-2 } (S_{k 1}(1,1)_{0}E)
+\frac{\zone^2 (2 \zone^2+1)}{4 \zone^4+6 \zone^2+1}\sqrt{-1}\omega^{[2]}_{-2 } (S_{k 1}(1,1)_{0}E)\nonumber\\&\quad{}
+\frac{\zone^2 (2 \zone^2+1)}{4 \zone^4+6 \zone^2+1}\sqrt{-1}\omega^{[1]}_{-2 } (S_{k 1}(1,1)_{0}E)\nonumber\\&\quad{}
+\frac{-2 (12 \zone^6-2 \zone^4+9 \zone^2+3)}{12 \zone^4+18 \zone^2+3}\sqrt{-1}\omega_{0 } \omega^{[k]}_{-1 } (S_{k 1}(1,1)_{0}E)\nonumber\\&\quad{}
+\frac{-\zone^2 (12 \zone^4-2 \zone^2-1)}{4 \zone^6+2 \zone^4-5 \zone^2-1}\sqrt{-1}\omega_{0 } \omega^{[2]}_{-1 } (S_{k 1}(1,1)_{0}E)\nonumber\\&\quad{}
+\frac{-\zone^2 (12 \zone^4-2 \zone^2-1)}{4 \zone^6+2 \zone^4-5 \zone^2-1}\sqrt{-1}\omega_{0 } \omega^{[1]}_{-1 } (S_{k 1}(1,1)_{0}E)\nonumber\\&\quad{}
+\frac{-(2 \zone^2-1) (12 \zone^4-2 \zone^2-1)}{8 \zone^6+4 \zone^4-10 \zone^2-2}\sqrt{-1}\omega_{0 } \omega_{0 } \omega_{0 } (S_{k 1}(1,1)_{0}E)\nonumber\\&\quad{}
+\frac{2 \zone^2 (12 \zone^6-2 \zone^4+9 \zone^2+3)}{12 \zone^4+18 \zone^2+3}\sqrt{-1}\omega^{[k]}_{-1 } S_{k 1}(1,1)_{-1 } E\nonumber\\&\quad{}
+\frac{\zone^2 (2 \zone^2-1) (3 \zone^2+2)}{4 \zone^4+6 \zone^2+1}\sqrt{-1}S_{k 1}(1,1)_{-3 } E\nonumber\\&\quad{}
+\frac{-2 \zone^4}{4 \zone^4+6 \zone^2+1}\sqrt{-1}\omega^{[2]}_{-1 } S_{k 1}(1,1)_{-1 } E
+\frac{2 \zone^4}{4 \zone^4+6 \zone^2+1}\sqrt{-1}\omega^{[1]}_{-1 } S_{k 1}(1,1)_{-1 } E\nonumber\\&\quad{}
+\frac{-\zone^2 (10 \zone^4-2 \zone^2+1)}{4 \zone^6+2 \zone^4-5 \zone^2-1}\sqrt{-1}\omega_{0 } S_{k 1}(1,1)_{-2 } E\nonumber\\&\quad{}
+\frac{\zone^2 (2 \zone^2+1) (6 \zone^4-4 \zone^2+1)}{4 \zone^6+2 \zone^4-5 \zone^2-1}\sqrt{-1}\omega_{0 } \omega_{0 } S_{k 1}(1,1)_{-1 } E\nonumber\\&\quad{}
+\frac{-2 \zone^4 (2 \zone^2-1)}{4 \zone^4+6 \zone^2+1}\sqrt{-1}S_{k 1}(1,2)_{-2 } E
+\frac{\zone^2 (24 \zone^6-4 \zone^4+8 \zone^2-1)}{8 \zone^6+4 \zone^4-10 \zone^2-2}\sqrt{-1}\omega_{0 } S_{k 1}(1,2)_{-1 } E\nonumber\\&\quad{}
+\frac{-2 \zone^2 (12 \zone^6-2 \zone^4+9 \zone^2+3)}{12 \zone^4+18 \zone^2+3}\omega^{[k]}_{-1 } S_{k 2}(1,1)_{-1 } E\nonumber\\&\quad{}
+\frac{-\zone^2 (\zone-1) (\zone+1) (2 \zone^2+1)}{4 \zone^4+6 \zone^2+1}S_{k 2}(1,1)_{-3 } E\nonumber\\&\quad{}
+\frac{-4 \zone^4}{4 \zone^4+6 \zone^2+1}\omega^{[1]}_{-1 } S_{k 2}(1,1)_{-1 } E
+\frac{-\zone^2 (12 \zone^4-2 \zone^2-1)}{4 \zone^4+6 \zone^2+1}\omega_{0 } \omega_{0 } S_{k 2}(1,1)_{-1 } E\nonumber\\&\quad{}
+\frac{-3 \zone^2 (2 \zone^2-1) (2 \zone^2+1)^2}{8 \zone^6+4 \zone^4-10 \zone^2-2}\omega_{0 } S_{k 2}(1,2)_{-1 } E
+\frac{\zone^2 (2 \zone^2+1)^2}{4 \zone^4+6 \zone^2+1}S_{k 2}(1,3)_{-1 } E,\end{align}

\begin{align}
S_{k 2}(1,3)_{1}\omega^{[1]}_{0}\ExB&=
\frac{-2 (\zone^2+1)}{3}\sqrt{-1}\omega^{[2]}_{-1 } (S_{k 1}(1,1)_{0}E)\nonumber\\&\quad{}
+\frac{-2 (\zone^2+1)}{3}\sqrt{-1}\omega^{[1]}_{-1 } (S_{k 1}(1,1)_{0}E)\nonumber\\&\quad{}
+\frac{-(4 \zone^2+1)}{3}\sqrt{-1}\omega_{0 } \omega_{0 } (S_{k 1}(1,1)_{0}E)\nonumber\\&\quad{}
+\frac{(\zone^2+1) (2 \zone^2-1)}{3}\sqrt{-1}S_{k 1}(1,1)_{-2 } E\nonumber\\&\quad{}
+\frac{4 \zone^4+2 \zone^2+1}{3}\sqrt{-1}\omega_{0 } S_{k 1}(1,1)_{-1 } E\nonumber\\&\quad{}
+\frac{-(2 \zone^4+1)}{3}S_{k 2}(1,1)_{-2 } E\nonumber\\&\quad{}
+\frac{-(2 \zone^2-1) (2 \zone^2+1)}{3}\omega_{0 } S_{k 2}(1,1)_{-1 } E\nonumber\\&\quad{}
-\zone^2S_{k 2}(1,2)_{-1 } E,
\nonumber\\
S_{k 2}(1,3)_{2}\omega^{[1]}_{0}\ExB&=
-\zone^2\sqrt{-1}\omega_{0 } (S_{k 1}(1,1)_{0}E)\nonumber\\&\quad{}
+\zone^2 (\zone^2+1)\sqrt{-1}S_{k 1}(1,1)_{-1 } E\nonumber\\&\quad{}
-\zone^4S_{k 2}(1,1)_{-1 } E,
\end{align}

\begin{align}
S_{k 2}(1,3)_{0}S_{k 1}(1,1)_{0}\ExB&=
\frac{2 \zone^2}{3}\sqrt{-1}\omega^{[k]}_{-3 } E
+\zone^2\sqrt{-1}\Har^{[k]}_{-1 } E\nonumber\\&\quad{}
+\frac{12 \zone^2 (\zone^2-3)}{16 \zone^2+3}\sqrt{-1}\omega^{[2]}_{-1 } \omega^{[1]}_{-1 } E\nonumber\\&\quad{}
+\frac{-2 (11 \zone^2+18)}{16 \zone^2+3}\sqrt{-1}\omega^{[1]}_{-3 } E\nonumber\\&\quad{}
+\frac{-12 (\zone^2-3) (\zone^2+1)}{16 \zone^2+3}\sqrt{-1}\omega^{[1]}_{-1 } \omega^{[1]}_{-1 } E\nonumber\\&\quad{}
+\frac{-3 (11 \zone^2+18)}{16 \zone^2+3}\sqrt{-1}\Har^{[1]}_{-1 } E\nonumber\\&\quad{}
+\frac{-3 (30 \zone^4-37 \zone^2-6)}{64 \zone^2+12}\sqrt{-1}\omega_{0 } \omega^{[2]}_{-2 } E\nonumber\\&\quad{}
+\frac{3 (30 \zone^6-89 \zone^4+40 \zone^2+24)}{64 \zone^4+12 \zone^2}\sqrt{-1}\omega_{0 } \omega^{[1]}_{-2 } E\nonumber\\&\quad{}
+\frac{-3 (2 \zone^4-59 \zone^2+6)}{32 \zone^2+6}\sqrt{-1}\omega_{0 } \omega_{0 } \omega^{[2]}_{-1 } E\nonumber\\&\quad{}
+\frac{3 (2 \zone^6-83 \zone^4-20 \zone^2-12)}{32 \zone^4+6 \zone^2}\sqrt{-1}\omega_{0 } \omega_{0 } \omega^{[1]}_{-1 } E\nonumber\\&\quad{}
+\frac{-3 (2 \zone^4-59 \zone^2+6)}{64 \zone^4+12 \zone^2}\sqrt{-1}\omega_{0 } \omega_{0 } \omega_{0 } \omega_{0 } E\nonumber\\&\quad{}
+\frac{-9 (\zone^2-3)}{16 \zone^4+3 \zone^2}\sqrt{-1}\omega^{[1]}_{-2 } (\omega^{[1]}_{0}\ExB)\nonumber\\&\quad{}
+\frac{-6 (\zone^2-3)}{16 \zone^4+3 \zone^2}\sqrt{-1}\omega_{0 } \omega^{[2]}_{-1 } (\omega^{[1]}_{0}\ExB)\nonumber\\&\quad{}
+\frac{3 (14 \zone^2+9)}{32 \zone^4+6 \zone^2}\sqrt{-1}\omega_{0 } \omega_{0 } \omega_{0 } (\omega^{[1]}_{0}\ExB)\nonumber\\&\quad{}
+\frac{6 (\zone^2-3) (2 \zone^2+1)}{16 \zone^2+3}\omega^{[1]}_{-1 } S_{21}(1,1)_{-1 } E\nonumber\\&\quad{}
+\frac{-3 (30 \zone^6-29 \zone^4-26 \zone^2-12)}{64 \zone^4+12 \zone^2}\omega_{0 } S_{21}(1,1)_{-2 } E\nonumber\\&\quad{}
+\frac{-3 (2 \zone^6-73 \zone^4-5 \zone^2+6)}{32 \zone^4+6 \zone^2}\omega_{0 } \omega_{0 } S_{21}(1,1)_{-1 } E\nonumber\\&\quad{}
+\frac{3 (26 \zone^2-27)}{32 \zone^2+6}\omega_{0 } S_{21}(1,2)_{-1 } E\nonumber\\&\quad{}
+3S_{21}(1,3)_{-1 } E,\end{align}

\begin{align}
S_{k 2}(1,3)_{1}S_{k 1}(1,1)_{0}\ExB&=
\zone^2\sqrt{-1}\omega^{[k]}_{-2 } E\nonumber\\&\quad{}
+\frac{-(6 \zone^4-3 \zone^2-4)}{14 \zone^2+2}\sqrt{-1}\omega^{[1]}_{-2 } E\nonumber\\&\quad{}
+\frac{2 \zone^2 (4 \zone^4+\zone^2-1)}{7 \zone^2+1}\sqrt{-1}\omega_{0 } \omega^{[2]}_{-1 } E\nonumber\\&\quad{}
+\frac{-2 (4 \zone^6+5 \zone^4+\zone^2+1)}{7 \zone^2+1}\sqrt{-1}\omega_{0 } \omega^{[1]}_{-1 } E\nonumber\\&\quad{}
+\frac{4 \zone^4+\zone^2-1}{7 \zone^2+1}\sqrt{-1}\omega_{0 } \omega_{0 } \omega_{0 } E\nonumber\\&\quad{}
+\frac{\zone^2+2}{7 \zone^2+1}\sqrt{-1}\omega^{[2]}_{-1 } (\omega^{[1]}_{0}\ExB)\nonumber\\&\quad{}
+\frac{\zone^2+2}{7 \zone^2+1}\sqrt{-1}\omega^{[1]}_{-1 } (\omega^{[1]}_{0}\ExB)\nonumber\\&\quad{}
+\frac{8 \zone^2+3}{14 \zone^2+2}\sqrt{-1}\omega_{0 } \omega_{0 } (\omega^{[1]}_{0}\ExB)\nonumber\\&\quad{}
+\frac{(\zone^2+2) (2 \zone^2+1)}{14 \zone^2+2}S_{21}(1,1)_{-2 } E\nonumber\\&\quad{}
+\frac{(2 \zone^2+1) (4 \zone^4+\zone^2-1)}{7 \zone^2+1}\omega_{0 } S_{21}(1,1)_{-1 } E\nonumber\\&\quad{}
+\frac{\zone^2 (2 \zone^2-9)}{14 \zone^2+2}S_{21}(1,2)_{-1 } E,
\nonumber\\
S_{k 2}(1,3)_{2}S_{k 1}(1,1)_{0}\ExB&=
2 \zone^2\sqrt{-1}\omega^{[k]}_{-1 } E,
\nonumber\\
S_{k 2}(1,3)_{3}S_{k 1}(1,1)_{0}\ExB&=0,\nonumber\\
S_{k 2}(1,3)_{4}S_{k 1}(1,1)_{0}\ExB&=
\zone^2\sqrt{-1}\ExB,
\end{align}


\begin{align}
S_{l 1}(1,1)_{0}S_{k 1}(1,1)_{0}\ExB&=
\mz^2S_{l k}(1,1)_{-1 } \ExB 
,\end{align}

\begin{align}
S_{l 1}(1,2)_{0}S_{k 1}(1,1)_{0}\ExB&=
-\mz^2S_{l k}(1,1)_{-2 } \ExB 
+\mz^2S_{l k}(1,2)_{-1 } \ExB 
,\nonumber\\
S_{l 1}(1,2)_{1}S_{k 1}(1,1)_{0}\ExB&=
-\mz^2S_{l k}(1,1)_{-1 } \ExB 
,
\end{align}

\begin{align}
S_{l 1}(1,3)_{0}S_{k 1}(1,1)_{0}\ExB&=
\mz^2S_{l k}(1,1)_{-3 } \ExB 
-\mz^2S_{l k}(1,2)_{-2 } \ExB 
+\mz^2S_{l k}(1,3)_{-1 } \ExB,\nonumber\\
S_{l 1}(1,3)_{1}S_{k 1}(1,1)_{0}\ExB&=
\mz^2S_{l k}(1,1)_{-2 } \ExB 
-\mz^2S_{l k}(1,2)_{-1 } \ExB 
,
\nonumber\\
S_{l 1}(1,3)_{2}S_{k 1}(1,1)_{0}\ExB&=
\mz^2S_{l k}(1,1)_{-1 } \ExB 
,
\end{align}

\begin{align}
S_{l 2}(1,1)_{0}S_{k 1}(1,1)_{0}\ExB&=
\mz^2\sqrt{-1}S_{l k}(1,1)_{-1 } \ExB 
,\end{align}

\begin{align}
S_{l 2}(1,2)_{0}S_{k 1}(1,1)_{0}\ExB&=
-\mz^2\sqrt{-1}S_{l k}(1,1)_{-2 } \ExB 
+\mz^2\sqrt{-1}S_{l k}(1,2)_{-1 } \ExB,\nonumber\\
S_{l 2}(1,2)_{1}S_{k 1}(1,1)_{0}\ExB&=
-\mz^2\sqrt{-1}S_{l k}(1,1)_{-1 } \ExB 
,
\end{align}

\begin{align}
S_{l 2}(1,3)_{0}S_{k 1}(1,1)_{0}\ExB&=
\mz^2\sqrt{-1}S_{l k}(1,1)_{-3 } \ExB 
-\mz^2\sqrt{-1}S_{l k}(1,2)_{-2 } \ExB 
\nonumber\\&\quad{}
+\mz^2\sqrt{-1}S_{l k}(1,3)_{-1 } \ExB 
,\nonumber\\
S_{l 2}(1,3)_{1}S_{k 1}(1,1)_{0}\ExB&=
\mz^2\sqrt{-1}S_{l k}(1,1)_{-2 } \ExB 
-\mz^2\sqrt{-1}S_{l k}(1,2)_{-1 } \ExB 
,
\nonumber\\
S_{l 2}(1,3)_{2}S_{k 1}(1,1)_{0}\ExB&=
\mz^2\sqrt{-1}S_{l k}(1,1)_{-1 } \ExB 
,
\end{align}

\begin{align}
S_{l k}(1,1)_{0}S_{k 1}(1,1)_{0}\ExB&=
-\mz^2S_{l 1}(1,1)_{-1 } \ExB 
-\mz^2\sqrt{-1}S_{l 2}(1,1)_{-1 } \ExB 
\nonumber\\&\quad{}
-\sqrt{-1}\omega_{0 } (S_{l 2}(1,1)_{0}\ExB) 
,\nonumber\\
S_{l k}(1,1)_{1}S_{k 1}(1,1)_{0}\ExB&=
-\sqrt{-1}(S_{l 2}(1,1)_{0}\ExB) 
,
\end{align}

\begin{align}
S_{l k}(1,2)_{0}S_{k 1}(1,1)_{0}\ExB&=
\frac{2\mz^4}{(\mz-1)(\mz+1)}S_{l 1}(1,1)_{-2 } \ExB 
\nonumber\\&\quad{}
+\frac{-2\mz^2}{(\mz-1)(\mz+1)}\omega_{0 } S_{l 1}(1,1)_{-1 } \ExB 
\nonumber\\&\quad{}
+\frac{1}{(\mz-1)(\mz+1)}\omega_{0 } \omega_{0 } (S_{l 1}(1,1)_{0}\ExB) 
\nonumber\\&\quad{}
+\frac{-\mz^2(2\mz-1)(2\mz+1)}{(\mz-1)(\mz+1)}S_{l 1}(1,2)_{-1 } \ExB 
\nonumber\\&\quad{}
+2\mz^2\sqrt{-1}S_{l 2}(1,1)_{-2 } \ExB 
\nonumber\\&\quad{}
+\frac{-2\mz^2}{(\mz-1)(\mz+1)}\sqrt{-1}\omega^{[2]}_{-1 } (S_{l 2}(1,1)_{0}\ExB) 
\nonumber\\&\quad{}
+\frac{-2\mz^2}{(\mz-1)(\mz+1)}\sqrt{-1}\omega^{[1]}_{-1 } (S_{l 2}(1,1)_{0}\ExB) 
\nonumber\\&\quad{}
+\frac{-\mz^2(2\mz-1)(2\mz+1)}{(\mz-1)(\mz+1)}\sqrt{-1}S_{l 2}(1,2)_{-1 } \ExB 
,\nonumber\\
S_{l k}(1,2)_{1}S_{k 1}(1,1)_{0}\ExB&=
2\mz^2S_{l 1}(1,1)_{-1 } \ExB 
+2\mz^2\sqrt{-1}S_{l 2}(1,1)_{-1 } \ExB 
\nonumber\\&\quad{}
+2\sqrt{-1}\omega_{0 } (S_{l 2}(1,1)_{0}\ExB) 
,
\nonumber\\
S_{l k}(1,2)_{2}S_{k 1}(1,1)_{0}\ExB&=
2\sqrt{-1}(S_{l 2}(1,1)_{0}\ExB) 
,
\end{align}

\begin{align}
S_{l k}(1,3)_{0}S_{k 1}(1,1)_{0}\ExB&=
\frac{-3\mz^2(2\mz^4-9\mz^2+3)}{20\mz^4-30\mz^2+3}S_{l 1}(1,1)_{-3 } \ExB 
\nonumber\\&\quad{}
+\frac{18\mz^4}{20\mz^4-30\mz^2+3}\omega^{[2]}_{-1 } S_{l 1}(1,1)_{-1 } \ExB 
\nonumber\\&\quad{}
+\frac{-3\mz^2(2\mz^2-3)}{20\mz^4-30\mz^2+3}\omega^{[2]}_{-2 } (S_{l 1}(1,1)_{0}\ExB) 
\nonumber\\&\quad{}
+\frac{6\mz^4}{20\mz^4-30\mz^2+3}\omega^{[1]}_{-1 } S_{l 1}(1,1)_{-1 } \ExB 
\nonumber\\&\quad{}
+\frac{-3\mz^4(12\mz^4-12\mz^2-1)}{(\mz-1)(\mz+1)(20\mz^4-30\mz^2+3)}\omega_{0 } S_{l 1}(1,1)_{-2 } \ExB 
\nonumber\\&\quad{}
+\frac{-3\mz^2(12\mz^4-10\mz^2-3)}{(\mz-1)(\mz+1)(20\mz^4-30\mz^2+3)}\omega_{0 } \omega^{[2]}_{-1 } (S_{l 1}(1,1)_{0}\ExB) 
\nonumber\\&\quad{}
+\frac{3\mz^4(12\mz^2-13)}{(\mz-1)(\mz+1)(20\mz^4-30\mz^2+3)}\omega_{0 } \omega_{0 } S_{l 1}(1,1)_{-1 } \ExB 
\nonumber\\&\quad{}
+\frac{-3\mz^2(8\mz^2-3)}{20\mz^4-30\mz^2+3}S_{l 1}(1,2)_{-2 } \ExB 
\nonumber\\&\quad{}
+\frac{3\mz^2(48\mz^6-56\mz^4+2\mz^2+3)}{2(\mz-1)(\mz+1)(20\mz^4-30\mz^2+3)}\omega_{0 } S_{l 1}(1,2)_{-1 } \ExB 
\nonumber\\&\quad{}
+\frac{3\mz^2(4\mz^4+2\mz^2-1)}{20\mz^4-30\mz^2+3}S_{l 1}(1,3)_{-1 } \ExB 
\nonumber\\&\quad{}
+\frac{-3\mz^2(\mz^2-3)(2\mz^2+1)}{20\mz^4-30\mz^2+3}\sqrt{-1}S_{l 2}(1,1)_{-3 } \ExB 
\nonumber\\&\quad{}
+\frac{12\mz^4}{20\mz^4-30\mz^2+3}\sqrt{-1}\omega^{[2]}_{-1 } S_{l 2}(1,1)_{-1 } \ExB 
\nonumber\\&\quad{}
+\frac{3\mz^2(2\mz^2-3)}{20\mz^4-30\mz^2+3}\sqrt{-1}\omega^{[1]}_{-2 } (S_{l 2}(1,1)_{0}\ExB) 
\nonumber\\&\quad{}
+\frac{-3\mz^2(12\mz^4-14\mz^2+3)}{20\mz^4-30\mz^2+3}\sqrt{-1}\omega_{0 } S_{l 2}(1,1)_{-2 } \ExB 
\nonumber\\&\quad{}
+\frac{3\mz^2(12\mz^4-10\mz^2-3)}{(\mz-1)(\mz+1)(20\mz^4-30\mz^2+3)}\sqrt{-1}\omega_{0 } \omega^{[1]}_{-1 } (S_{l 2}(1,1)_{0}\ExB) 
\nonumber\\&\quad{}
+\frac{3(12\mz^4-14\mz^2+1)}{2(\mz-1)(\mz+1)(20\mz^4-30\mz^2+3)}\sqrt{-1}\omega_{0 } \omega_{0 } \omega_{0 } (S_{l 2}(1,1)_{0}\ExB) 
\nonumber\\&\quad{}
+\frac{9\mz^2(16\mz^6-24\mz^4+8\mz^2-1)}{2(\mz-1)(\mz+1)(20\mz^4-30\mz^2+3)}\sqrt{-1}\omega_{0 } S_{l 2}(1,2)_{-1 } \ExB 
\nonumber\\&\quad{}
+\frac{3\mz^2(2\mz^2-1)(2\mz^2+1)}{20\mz^4-30\mz^2+3}\sqrt{-1}S_{l 2}(1,3)_{-1 } \ExB 
,\end{align}

\begin{align}
S_{l k}(1,3)_{1}S_{k 1}(1,1)_{0}\ExB&=
\frac{-3\mz^4}{(\mz-1)(\mz+1)}S_{l 1}(1,1)_{-2 } \ExB 
\nonumber\\&\quad{}
+\frac{3\mz^2}{(\mz-1)(\mz+1)}\omega_{0 } S_{l 1}(1,1)_{-1 } \ExB 
\nonumber\\&\quad{}
+\frac{-3}{2(\mz-1)(\mz+1)}\omega_{0 } \omega_{0 } (S_{l 1}(1,1)_{0}\ExB) 
\nonumber\\&\quad{}
+\frac{3\mz^2(2\mz-1)(2\mz+1)}{2(\mz-1)(\mz+1)}S_{l 1}(1,2)_{-1 } \ExB 
\nonumber\\&\quad{}
-3\mz^2\sqrt{-1}S_{l 2}(1,1)_{-2 } \ExB 
\nonumber\\&\quad{}
+\frac{3\mz^2}{(\mz-1)(\mz+1)}\sqrt{-1}\omega^{[2]}_{-1 } (S_{l 2}(1,1)_{0}\ExB) 
\nonumber\\&\quad{}
+\frac{3\mz^2}{(\mz-1)(\mz+1)}\sqrt{-1}\omega^{[1]}_{-1 } (S_{l 2}(1,1)_{0}\ExB) 
\nonumber\\&\quad{}
+\frac{3\mz^2(2\mz-1)(2\mz+1)}{2(\mz-1)(\mz+1)}\sqrt{-1}S_{l 2}(1,2)_{-1 } \ExB 
,
\nonumber\\
S_{l k}(1,3)_{2}S_{k 1}(1,1)_{0}\ExB&=
-3\mz^2S_{l 1}(1,1)_{-1 } \ExB 
-3\mz^2\sqrt{-1}S_{l 2}(1,1)_{-1 } \ExB 
\nonumber\\&\quad{}
-3\sqrt{-1}\omega_{0 } (S_{l 2}(1,1)_{0}\ExB) 
,
\nonumber\\
S_{l k}(1,3)_{3}S_{k 1}(1,1)_{0}\ExB&=
-3\sqrt{-1}(S_{l 2}(1,1)_{0}\ExB).
\end{align}

 For $r,s\in\{1,2,3\}$, $i,j\in\Z_{\geq 0}$ and a pair of distinct elements $k,l\in\{3,\ldots,\rankL\}$, one can compute $S_{kl}(1,r)_{i}S_{k1}(1,s)_{j}\ExB$
and $S_{kl}(1,r)_{i}S_{k2}(1,s)_{j}\ExB$
using the fact that $S_{kl}(1,r)_{m}E=0$ for all $m\in\Z_{\geq 0}$.

\section{Notation}\label{section:notation}
\begin{tabular}{lp{13cm}}
$V$ & a vertex algebra.\\
$U$ & a subspace of a weak $V$-module.\\
$\Omega_{V}(U)$&$
=\{\lu\in U\ \Big|\ 
a_{i}\lu=0\ \mbox{for all homogeneous }a\in V\mbox{and }i>\wt a-1.\}$ \eqref{eq:OmegaV(U)=BigluinU}.\\
$\mn$ & a non-zero complex number.\\
$\hei$ & a finite dimensional vector space equipped with a nondegenerate symmetric bilinear form
$\langle \mbox{ }, \mbox{ }\rangle$.\\
$h$ & an element of $\fh$ with $\langle \wh,\wh\rangle=1$.\\
$h^{[1]},\ldots,h^{[\rankL]}$ & an orthonormal basis of $\fh$.\\
$\alpha$ & an element of $\fh$ with $\langle\alpha,\alpha\rangle=\mn$.\\
$M(1)$ & the vertex algebra associated to the Heisenberg algebra.\\
$\lattice$ & a non-degenerate even lattice of finite rank.\\
$\rankL$ & the rank of $\lattice$.\\
$V_{\lattice}$ & the vertex algebra associated to $\lattice$.\\
$\theta$ & the automorphism of $V_{\lattice}$ induced from the $-1$ symmetry of $\lattice$.\\
$M(1)^{+}$ & the fixed point subalgbra of $M(1)$ under the action of $\theta$.\\
$V_{\lattice}^{+}$ & the fixed point subalgbra of $V_{\lattice}$ under the action of $\theta$.\\
$\mK,\module,\mN,\mW$ & weak $M(1)^{+}$ (or $V_{\lattice}^{+}$)-modules.\\
$I(\mbox{ },x)$ & an intertwining operator for $M(1)^{+}$.\\
$\epsilon(\lu,\lv)$ & 
$\lu_{\epsilon(\lu,\lv)}\lv\neq 0\mbox{ and }\lu_{i}\lv=0\mbox{ for all }i>\epsilon(\lu,\lv)$
if $I(\lu,x)\lv\neq 0$ and $\epsilon(\lu,\lv)=+\infty$ if $I(\lu,x)\lv= 0$,
where $I : \module\times\mW\rightarrow \mN\db{x}$ is an intertwining operator and 
$\lu\in\module$, $\lv\in\mW$ \eqref{eqn:max-vanish}.\\
$\langle\omega_i\rangle X$ & the space spanned by the elements $\omega_i^{j}\lu, j\in\Z_{\geq 0}, \lu\in X$. \\
$A(V)$ & the Zhu algebra of a vertex operator algebra $V$.\\
$\omega$&$=(1/2)h(-1)^2\vac$ or $(1/2)\sum_{i=1}^{\rankL}h^{[i]}(-1)^2\vac$.\nonumber\\
$\Har$&$=(1/3)(h(-3)h(-1)\vac-h(-2)^2\vac)$.\\
$J$&$=h(-1)^4\vac-2h(-3)h(-1)\vac+(3/2)h(-2)^2\vac=-9\Har+4\omega_{-1}^2\vac-3\omega_{-3}\vac$.\\
$\ExB$&$=\ExB(\alpha)=e^{\alpha}+e^{-\alpha}$ where $\alpha\in\fh$.\\
$\lE$ & an integer such that $\lE\geq \epsilon(\ExB,\lu)$ or $\lE=\epsilon(\ExB,\lu)$ for a given element $\lu$.\\
$\omega^{[i]}$&$=(1/2)h^{[i]}(-1)^2$.\nonumber\\
$\Har^{[i]}$&$=(1/3)(h^{[i]}(-3)h^{[i]}(-1)\vac-h^{[i]}(-2)^2\vac)$.\\
$\nS_{ij}(l,m)$&$=h^{[i]}(-l)h^{[j]}(-m)$.\\
$\epsilon_{S}$& an integer such that $\epsilon_{S}\geq \epsilon(S_{ij},\lu)$ or 
$\epsilon_{S}=\epsilon(S_{ij},\lu)$ for a given element $\lu$ and 
$i,j\in\{1,\ldots,\rankL\}$.\\
$B$ & a subspace of a weak $M(1)^{+}$-module $\mW$ \eqref{eq:B=SpanC(ajExB)}.\\
$\zone$&$=\langle\alpha,h^{[1]}\rangle$.
\end{tabular}


\providecommand{\MR}{\relax\ifhmode\unskip\space\fi MR }
\providecommand{\MRhref}[2]{%
  \href{http://www.ams.org/mathscinet-getitem?mr=#1}{#2}
}
\providecommand{\href}[2]{#2}

\bibliographystyle{ijmart}

\end{document}